\documentclass[12pt]{article}

\usepackage{latexsym}
\usepackage{a4wide}
\usepackage{amscd}
\usepackage{graphics}
\usepackage{amsmath}
\usepackage{amssymb}
\usepackage{esint}
\usepackage{mathrsfs}
\usepackage{amsthm}
\usepackage{bbm}
\usepackage{color}
\usepackage{accents}
\usepackage{enumerate}
\usepackage[pdftex]{hyperref}
\usepackage{soul}

\usepackage{tikz-cd}

\newcommand{\N}{\mathbb{N}}                     
\newcommand{\Z}{\mathbb{Z}}                     
\newcommand{\Q}{\mathbb{Q}}                     
\newcommand{\R}{\mathbb{R}}                     
\newcommand{\C}{\mathbb{C}}                     
\newcommand{\F}{\mathbb{F}}                     
\newcommand{\p}{\partial}             		
\newcommand{\om}{\omega}             		
\newcommand{\x}{\times}            		
\newcommand{\coker}{\mathrm{coker\,}}           
\newcommand{\ind}{\mathrm{ind\,}}               
\newcommand{\supp}{\mathrm{supp\,}}             
\newcommand{\CZ}{\mathrm{CZ}}               
\newcommand{\CP}{\mathbb{CP}}                   
\newcommand{\sign}{\mathrm{sign\,}}             
\newcommand{\Det}{\mathrm{Det}}                 
\newcommand{\Crit}{\mathrm{Crit}\,}             
\newcommand{\RFH}{\mathrm{RFH}}             
\newcommand{\RFC}{\mathrm{RFC}}             
\newcommand{\FH}{\mathrm{FH}}             
\newcommand{\FC}{\mathrm{FC}}             
\renewcommand{\H}{\mathrm{H}}             
\newcommand{\ev}{\mathrm{ev}}               
\newcommand{\PD}{\mathrm{PD}}               
\newcommand{\cov}{\mathrm{cov}}               

\newcommand{\w}{\mathfrak{w}}  
\newcommand{\q}{\mathfrak{q}}  
\renewcommand{\j}{\mathfrak{j}} 
\renewcommand{\a}{\mathfrak{a}}  
\renewcommand{\r}{\mathfrak{r}} 
\newcommand{\Cont}{\mathrm{Cont}}

\newcommand{\OO}{\mathcal{O}}     
\newcommand{\BB}{\mathcal{B}}     
\newcommand{\CC}{\mathcal{C}}     
\newcommand{\EE}{\mathcal{E}}     
\newcommand{\PP}{\mathcal{P}}    
\newcommand{\JJ}{\mathcal{J}}     
\newcommand{\MM}{\mathcal{M}}     
\newcommand{\NN}{\mathcal{N}}

\newcommand{\into}{\hookrightarrow}

\renewcommand{\AA}{\mathcal{A}}            

\newtheorem{thm}{\sc Theorem}[section]               
\newtheorem*{thm*}{\sc Theorem}               
\newtheorem{cor}[thm]{\sc Corollary}        
\newtheorem*{cor*}{\sc Corollary}        
\newtheorem{lem}[thm]{\sc Lemma}            
\newtheorem{prop}[thm]{\sc Proposition}     
\newtheorem{rem}[thm]{\sc Remark}           
\newtheorem{ex}[thm]{\sc Example}           

\DeclareFontFamily{U}{mathx}{\hyphenchar\font45}
\DeclareFontShape{U}{mathx}{m}{n}{
      <5> <6> <7> <8> <9> <10>
      <10.95> <12> <14.4> <17.28> <20.74> <24.88>
      mathx10
      }{}
\DeclareSymbolFont{mathx}{U}{mathx}{m}{n}
\DeclareFontSubstitution{U}{mathx}{m}{n}
\DeclareMathAccent{\widecheck}{0}{mathx}{"71}
\DeclareMathAccent{\wideparen}{0}{mathx}{"75}

\AtEndDocument{\bigskip{\footnotesize
\noindent  \textsc{Mathematisches Institut, Universit\"at Heidelberg, 69120 Heidelberg, Germany} \par  
\noindent  \textit{E-mail address}: \texttt{\href{mailto:peter.albers@uni-heidelberg.de}{peter.albers@uni-heidelberg.de}} \par
  \addvspace{\medskipamount}
\noindent\textsc{Seoul National University, Department of Mathematical Sciences, Research Institute in Mathematics, 
	08826 Seoul, South Korea} \par  
 \noindent \textit{E-mail address}: \texttt{\href{mailto:jungsoo.kang@snu.ac.kr}{jungsoo.kang@snu.ac.kr}} \par
}}
\numberwithin{equation}{section}

\title{Rabinowitz Floer homology of negative line bundles and Floer Gysin sequence} 

\author{Peter Albers and Jungsoo Kang}

\date{}

\begin{document}
\maketitle

\begin{abstract}
This article is concerned with the Rabinowitz Floer homology of negative line bundles. We construct a refined version of Rabinowitz Floer homology and study its properties. In particular, we build a Gysin-type long exact sequence for this new invariant and discuss an application to the orderability problem for prequantization spaces. We also construct a  short exact sequence for the ordinary Rabinowitz Floer homology and provide computational results.
\end{abstract}

\setcounter{tocdepth}{1} 
\tableofcontents

\section{Introduction}

Let $(M,\om)$ be a closed connected symplectic manifold with integral symplectic form. We consider a negative line bundle over $M$, that is a complex line bundle 
\[
\wp:E\longrightarrow M
\]  
with first Chern class $c_1^E=-m[\om]$ for some natural number $m\in\N=\{1,2,\dots\}$. 
We choose a Hermitian metric on $E$ and denote by 
$r:E\to [0,+\infty)$
the induced radial coordinate. We consider circle subbundles 
\[
\Sigma_\tau:= \left\{e \in E \mid m\pi r^2(e)=\tau \,\right\},\qquad \tau >0\,.	
\]
If the size of the radius $\tau$ is not relevant, we simply write $\Sigma$. We also denote the restriction of $\wp$ to $\Sigma$ again by $\wp$. We choose a connection 1-form $\alpha$ on $\Sigma$ such that $\wp^*(m\om)=d\alpha$.
The 1-form $\alpha$ is a contact form and naturally extends to $E\setminus\OO_E$, where  $\OO_E$ denotes the zero section of $\wp$. 
The total space $E$ of the line bundle is endowed with the symplectic form
$\Omega=\wp^*\om+d(\pi r^2\alpha)$, which is not exact. Sometimes $(\Sigma,\alpha)$ is referred to as prequantization or Boothby--Wang bundle.

This article is concerned with the Rabinowitz Floer homology of $(E,\Sigma_\tau)$. Although very little is known for Rabinowitz Floer homology or symplectic homology beyond the exact setting, due to the simple Reeb dynamics of $(\Sigma_\tau,\alpha)$, there has been remarkable work on the Rabinowitz Floer homology or symplectic homology for negative line bundles as we will mention below. 
In this article, we construct a refined version of Rabinowitz Floer homology, denoted by $\RFH^{\w_0}(E,\Sigma_\tau)$, by studying the filtration, called winding numbers,  introduced by Frauenfelder \cite{Fra04}.   We then build a Gysin-type long exact sequence for this new invariant. This serves as an efficient tool to compute $\RFH^{\w_0}(E,\Sigma_\tau)$. Combining it with methods from \cite{AM18} leads to an application to the orderability problem for the contact manifold $\Sigma$.
We also investigate the ordinary Rabinowitz Floer homology $\RFH(E,\Sigma_\tau)$, which we call full Rabinowitz Floer homology in distinction from $\RFH^{\w_0}(E,\Sigma_\tau)$. The full Rabinowitz Floer homology $\RFH(E,\Sigma_\tau)$ fits in to a certain short exact sequence, and using this we compute $\RFH(E,\Sigma_\tau)$. This computational result shows that $\RFH(E,\Sigma_\tau)$ can change dramatically as the radius $\tau$ varies, see also Remark \ref{rem:RFH_full_intro}.(a).

 We denote by $c_1^{TM}$ the first Chern class of the tangent bundle $TM\to M$ and write  $\om(\pi_2(M))=\nu\Z$, where $\nu\in\N\cup\{0\}$ since $\omega$ is integral.  Throughout this article, we will assume that one of the following conditions holds.
\begin{enumerate}
	\item[(A1)] $\om$ vanishes on $\pi_2(M)$.
	\item[(A2)] $c_1^{TM}=\lambda\om$ on $\pi_2(M)$ for some $\lambda\in\R$ satisfying either $\lambda\nu\leq -\frac{1}{2}\dim M$ or $\lambda\nu\geq 2$.
	\item[(A3)] $c_1^{TM}=\lambda\om$ on $\pi_2(M)$ for some $\lambda\in\R$ satisfying either $\lambda\nu\leq -\frac{1}{2}\dim M$ or $(\lambda-1)\nu\geq 1$. 
\end{enumerate}
Note that (A3) is slightly stronger than (A2), and $\lambda\in\Q$ since $\om$ is integral.  Condition (A2) is equivalent to saying that $\omega$ does not vanish on $\pi_2(M)$ and every $A\in\pi_2(M)$ with $\omega(A)\neq 0$ has either $c_1^{TM}(A)\leq -\frac{1}{2}\dim M$ or $c_1^{TM}(A)\geq2$, see Remark \ref{rem:H}.(i). The assumption $\lambda\nu\geq2$ in (A2) can be weakened to $\lambda\nu\geq 1$ if we use a smaller class of almost complex structures to define $\RFH^{\w_0}(E,\Sigma_\tau)$ and $\RFH(E,\Sigma_\tau)$, see Remark \ref{rem:hope}.

Although we give the definition of the Novikov ring $\Lambda$ for a general situation  in \eqref{eq:nov},  $\Lambda$ is isomorphic to the Laurent polynomial ring $\Z[t,t^{-1}]$ under condition (A2) or (A3), and to $\Z$ if (A1) is assumed. In the former case, the degree of the formal variable $t$ is minus twice the minimal Chern number of $c_1^{TM}$, see Remark \ref{rem:H}.(ii).

\subsection*{Rabinowitz Floer homology with zero winding number}

We list properties and applications of the new invariant $\RFH_*^{\w_0}(E,\Sigma_\tau)$. Here the superscript $\w_0$ indicates that this invariant is built on generators with zero winding number, meaning that generators are periodic Reeb orbits on $\Sigma_\tau$ with (homotopy classes of) capping disks not intersecting $\OO_E$, see Section \ref{sec:zero_winding} and Section \ref{sec:RFH_zero_wind} for precise details.

\begin{thm}\label{thm:main}
Suppose that $(M,\om)$ satisfies condition (A1) or (A2).
\begin{enumerate}[(a)]
\item The Rabinowitz Floer homology with zero winding number
\[
\RFH_*^{\w_0}(E,\Sigma_\tau)\,,\qquad*\in\Z
\]
is defined and invariant under the change of $\tau>0$. Moreover, it admits a $\Lambda$-module structure given by iterating generators, see Remark \ref{rem:Lambda_module_str} for details.

\item There exists a long exact sequence, the Floer Gysin sequence, of $\Lambda$-modules
\[
\cdots\longrightarrow \RFH_*^{\w_0}(E,\Sigma)\longrightarrow \FH_{*}(M)\stackrel{\Psi^{c_1^E}}{\longrightarrow}\FH_{*-2}(M)\longrightarrow \RFH_{*-1}^{\w_0}(E,\Sigma)\longrightarrow\cdots
\]
where the map $\Psi^{c_1^E}$ is the Floer cap product with $-c_1^E$. Furthermore, this respects action filtrations, see Proposition \ref{prop:cap_product} for details.

\item In the case of (A1) or (A2) with $\lambda\nu\leq -\frac{1}{2}\dim M$, we have a $\Lambda$-module isomorphism
\[
\RFH^{\w_0}_*(E,\Sigma)\cong  \H_{*+\frac{\dim M}{2}}(\Sigma;\Lambda)\,,
\] 
and the Floer Gysin sequence in (b) recovers the classical Gysin sequence for the bundle $\Sigma\to M$ with coefficients in $\Lambda$. 
\item In this part, we use the notation $(E^m,\Sigma^m)$ to indicate the degree $m$ of $E^m$ and $\Sigma^m$, i.e.~$c_1^{E^m}=-m[\om]$.  
There exist natural transfer and projection homomorphisms 
\[
T:\RFH^{\w_0}_*(E^m,\Sigma^m)\to \RFH^{\w_0}_*(E^1,\Sigma^1)\,,\quad P: \RFH^{\w_0}_*(E^1,\Sigma^1)\to  \RFH^{\w_0}_*(E^m,\Sigma^m)
\] 
such that both compositions $P\circ T$ and $T\circ P$ agree with the scalar multiplication by $m$. 


\item In this part, we assume either (A1) or (A3). Let ${\Cont_0}(\Sigma,\xi)$ be the identity component of the group of contactomorphisms on $(\Sigma,\xi=\ker\alpha)$, and let $\widetilde{\Cont_0}(\Sigma,\xi)$ be its universal cover. Then the homology $\RFH^{\w_0}_*(E,\Sigma,\{\varphi_t\})$ associated with a path $\{\varphi_t\}_{t\in[0,1]}$ in $\Cont_0(\Sigma,\xi)$ with $\varphi_0=\mathrm{id}$ is defined, and there is a $\Z$-module isomorphism
\[
\RFH^{\w_0}_*(E,\Sigma,\{\varphi_t\}) \cong \RFH^{\w_0}_*(E,\Sigma)\,.
\]
Moreover if $\RFH^{\w_0}(E,\Sigma)\neq0$, then $\widetilde{\Cont_0}(\Sigma,\xi)$ is orderable in the sense of \cite{EP00} and every $\varphi\in {\Cont_0}(\Sigma,\xi)$ has a translated point with respect to $\alpha$ in the sense of \cite{San12}.
\end{enumerate}
\end{thm}

We point out throughout the text where Theorem \ref{thm:main} and the other results from the introduction are proved.

The Floer Gysin sequence in Theorem \ref{thm:main}.(b) yields that vanishing of $\RFH^{\w_0}_*(E,\Sigma)$ is equivalent to the map $\Psi^{c_1^E}$ being an isomorphism. Since we use integer coefficients, we have the following corollary where we use again the notation $(E^m,\Sigma^m)$ as in Theorem \ref{thm:main}.(d).
\begin{cor}\label{cor:orderability}
	The map $\Psi^{c_1^{E^m}}$ is not an isomorphism and thus $\RFH^{\w_0}_*(E^m,\Sigma^m)$ does not vanish for all $m\geq2$. This is also the case for $m=1$ if $c_1^{E^1}$ is not a primitive class in $\H^2(M;\Z)$.
	In particular if $\widetilde{\Cont_0}(\Sigma^m,\xi)$ is not orderable, then $m=1$.
\end{cor}

We refer to Remark \ref{rem:RFH_wind0_intro}.(e) for known examples of orderable contact manifolds. 
Theorem \ref{thm:main}.(d) implies the following corollary.

\begin{cor}\label{cor:transfer}
	If $\RFH_{\kappa}^{\w_0}(E^1,\Sigma^1)=0$ for a fixed $\kappa\in\Z$, then $\RFH_{\kappa}^{\w_0}(E^m,\Sigma^m)$ only contains torsion classes of order $m$, for all $m\in\N$. Moreover, $\RFH_{\kappa}^{\w_0}(E^m,\Sigma^m)$ is torsion for some $m\in\N$ if and only if it is torsion for every $m\in\N$. In particular, if we use coefficients in a field (instead of integers), then $\RFH_{\kappa}^{\w_0}(E^m,\Sigma^m)=0$ for some $m\in\N$ if and only if $\RFH_{\kappa}^{\w_0}(E^m,\Sigma^m)=0$ for all  $m\in\N$.	
\end{cor}

The Floer Gysin sequence also makes $\RFH^{\w_0}(E,\Sigma)$ sometimes explicitly computable. The following is one instance of this, and we refer to Remark \ref{rem:RFH_wind0_intro}.(d) for another example. 
Suppose that $\Psi^{c_1^E}$ is injective. Then we have
\[	 
\RFH_{*-1}^{\w_0}(E,\Sigma)\cong \FH_{*-2}(M) \,\big /\, \mathrm{im} \Big( \FH_{*}(M)\stackrel{\Psi^{c_1^E}}{\to}\FH_{*-2}(M)\Big)\,.
\]
This applies to the complex line bundle $\OO_{\mathbb{CP}^n}(-m)\to \CP^n$. That is, $c_1^{\OO_{\mathbb{CP}^n}(-m)}=-m[\om_\mathrm{FS}]$, where $\om_\mathrm{FS}$ denotes the Fubini-Study form on $\CP^n$ normalized so that the integral of $\om_\mathrm{FS}$ over a complex line equals 1. In this case, the circle bundle $\Sigma$ is diffeomorphic to the lens space $L(m,1)=S^{2n+1}/\Z_m$ where $\Z_m:=\Z/m\Z$ and, in particular, $\Z_1=0$.

\begin{cor}\label{eq:RFH_w_0_O}
For any $n,m\in\N$, we have 
\[
\RFH^{\w_0}_*\big(\OO_{\CP^n}(-m),L(m,1)\big) \cong \left\{
\begin{aligned} 
&\Z_m  && \quad *\in2\Z+1\,, \\[1ex]
&\;0  && \quad *\in 2\Z\,.
\end{aligned}
\right.
\]
\end{cor}

\begin{rem}\label{rem:RFH_wind0_intro}
$ $
\begin{enumerate}[(a)]
	\item Under a certain assumption, it is possible to define the Rabinowitz Floer homology on the symplectization $\R\x\Sigma$ using SFT-like almost complex structures and to construct an exact sequence isomorphic to the one in Theorem \ref{thm:main}.(b), see \cite{BKK22}.	
	\item SFT-like almost complex structures are useful to control bubbling-off of holomorphic planes at the negative end of $\R\x\Sigma$ and thus to obtain necessary compactness results. However SFT-like almost complex structures on $E\setminus\OO_E$ do not extend over $\OO_E$. In this article, we use almost complex structures making fibers and the zero section of $E\to M$ complex submanifolds. 
	
	\item  A filtration using intersection numbers with the zero section was also used in \cite{Ono95} for Lagrangian Floer homology and in \cite{GS18} for positive symplectic homology. 
		
		A Gysin sequence for Lagrangian Floer homology was constructed in \cite{BK13}, see also \cite{Per08} for a related result using pseudo-holomorphic quilts. We refer to \cite{BO09a,BO13a} for a Gysin sequence for $S^1$-equivariant symplectic homology. We also point out that results in \cite{DL19b} lead to a Gysin sequence for positive symplectic homology, see \cite[Remark 9.8]{DL19b}. These Gysin-type exact sequences were constructed in the presence of suitable Liouville fillings.
	 
	\item Another instance that simplifies the Floer Gysin sequence in Theorem \ref{thm:main}.(b) is when the singular homology of $M$ vanishes for every odd degree.  This is the case for example if there is a Morse function on $M$ without critical points of consecutive Morse indices, e.g.~if $(M,\om)$ is toric. Then the Floer Gysin sequence splits into 
	\[
0\longrightarrow \RFH_*^{\w_0}(E,\Sigma)\longrightarrow \FH_{*}(M)\stackrel{\Psi^{c_1^E}}{\longrightarrow}\FH_{*-2}(M)\longrightarrow \RFH_{*-1}^{\w_0}(E,\Sigma)\longrightarrow 0
	\]
	for every $*\in 2\Z+\frac{\dim M}{2}$. This can be used to compute $\RFH_*^{\w_0}(E,\Sigma)$, see \cite{BKK22} for some examples.
	
	\item Eliashberg and Polterovich  \cite{EP00} introduced the concept of orderability of contact manifolds and studied contact rigidity phenomena, see also \cite{EKP06}. Examples of orderable contact manifolds include $\mathbb{RP}^{2n+1}$ in \cite{EP00}, $\R^{2n+1}$ in \cite{Bhu01}, lens spaces in \cite{Mil08,San11,GKPS21}, certain contact manifolds obtained as contact reduction of $\mathbb{RP}^{2n+1}$ in \cite{BZ15,Zap20}, 1-jet bundles in \cite{CN10a,CFP17}, and cosphere bundles in \cite{EKP06,CN10,AF12,CFP17}. Some of these results rely on Givental's non-linear Maslov index on $\mathbb{RP}^{2n+1}$ appeared in \cite{Giv90} and constructed a quasi-morphism on $\widetilde{\Cont_0}$ in respective settings. 
\end{enumerate}	
\end{rem}

\subsection*{Full Rabinowitz Floer homology}

Now we consider full Rabinowitz Floer homology, i.e.~without restricting winding numbers, and present an analog of the Floer Gysin sequence, which splits into a short exact sequence in this situation, and some computational results. 

\begin{thm}\label{thm:full_rfh}
Suppose that $(M,\om)$ satisfies either (A1) or (A2).
\begin{enumerate}[(a)]
	\item The Rabinowitz Floer homology $\RFH_{*}(E,\Sigma_\tau)$ is defined.	
\item There exist a short exact sequence 
\[
	0 \longrightarrow   \mathop{\widetilde{\bigoplus}}_{k\in\Z} \FH_{*+2k}(M) \stackrel{\delta}{\longrightarrow} \mathop{\widetilde{\bigoplus}}_{k\in\Z}\FH_{*+2k}(M)\longrightarrow \RFH_{*-1}(E,\Sigma_\tau) \longrightarrow  0\,,
\]
	where $\delta=\mathrm{id}+\Psi^{c_1^E}$ and
\[
\mathop{\widetilde{\bigoplus}}_{k\in\Z}\FH_{*+2k}(M) := \left\{
\begin{aligned} 
&\bigg\{\sum_{k\leq k_0} Z_k \mid Z_k\in \FH{}_{*+2k}(M),\; k_0\in\Z \bigg\}\qquad & \tau(\lambda-m)<m\,
\\[.5ex]
&\bigg\{\sum_{|k|\leq k_0} Z_k \mid Z_k\in \FH{}_{*+2k}(M),\; k_0\in\Z \bigg\}	  & \tau(\lambda-m)=m\,  \\[.5ex]
& \bigg\{\sum_{k\geq k_0} Z_k \mid Z_k\in \FH{}_{*+2k}(M),\; k_0\in\Z \bigg\}  &  \tau(\lambda-m)>m\,
\end{aligned}
\right.
\]  
in the case of (A2). If we assume (A1), the above short exact sequence remains with $\mathop{\widetilde{\bigoplus}}_{k\in\Z}$ replaced by ${\bigoplus}_{k\in\Z}$ for all $\tau>0$.  
In particular, $\RFH_{*}(E,\Sigma_\tau)$ is 2-periodic in degree, i.e.
\[
\RFH_{*}(E,\Sigma_\tau)\cong \RFH_{*+2}(E,\Sigma_\tau)\,.
\]

\item Assume (A1) or (A2) with $\lambda \leq m$, then 
\[
 \RFH_*(E,\Sigma_\tau)=0 \qquad  \forall\tau>0\,.
\] 
Assume (A2) with $1\leq m\leq\lambda-1$, then
\[
 \RFH_*(E,\Sigma_\tau)=0 \qquad  \forall\tau\in\big(0,\tfrac{m}{\lambda-m}\big)\,.
\] 

\item We assume (A2) with $1\leq m\leq\lambda-1$. If either the $n$-th power $(\Psi^{c_1^E})^n$ vanishes for some $n\in\N$ or $\Psi^{c_1^E}$ is an isomorphism, we have	
	\[
	\RFH_*(E,\Sigma_\tau)=0 \qquad  \forall\tau >\tfrac{m}{\lambda-m}\,.
	\]
	The same conclusion holds for $\tau= \tfrac{m}{\lambda-m}$ in the case that $(\Psi^{c_1^E})^n=0$ for some $n\in\N$.
\item We assume (A2) with $1\leq m\leq\lambda-1$. If we take coefficients in a field $\F$, then
	\[
	\RFH_*(E,\Sigma_\tau) \cong  \begin{cases}
		\FH_{*-1}(M)/\ker(\Psi^{c_1^E})^n \;\; &\tau =\tfrac{m}{\lambda-m}  \\[1ex]
		0 & \tau >\tfrac{m}{\lambda-m} \\
	\end{cases}
	\]
	for any $\displaystyle n\geq\sum_{*=0}^{\dim M}\dim_{\F}\H_*(M;\F)$.
\end{enumerate}	

\end{thm}

Next we compute the full Rabinowitz Floer homology of $(\OO_{\CP^n}(-m),\Sigma_\tau)$. We also provide illustrations of the associated chain complex in Section \ref{sec:example}.

\begin{cor}\label{cor:full_rfh_o_intro}
Let $n,m\in\N$ be arbitrary. For $*\in2\Z$, we have 
\[
\RFH_*(\OO_{\CP^n}(-m),\Sigma_\tau)=0 \qquad \forall \tau>0\,.
\]
For $*\in2\Z+1$,  we have
\[
\RFH_*(\OO_{\CP^n}(-m),\Sigma_\tau)\cong \left\{
\begin{aligned} 
&\;0 \quad  && \quad \tau(n+1-m)<1\,, \\[.5ex]
&\;\Z &&\quad \tau(n+1-m)=1\,,\; m=1\,,\\[.5ex]
&\;\widetilde{\Q}_m \quad  && \quad \tau(n+1-m)=1\,,\;m\neq1\,, \\[.5ex]
&\;\Q_m \quad  && \quad \tau(n+1-m)>1\,,
\end{aligned}
\right.
\]
where the $\Z$-modules $\widetilde{\Q}_m$ and ${\Q}_m$ are defined in \eqref{eq:Q_m}. If we take coefficients in a field $\F$,  the above result holds with $\Z$, $\widetilde{\Q}_m$, and $\Q_m$ replaced by $\F$, $0$, and $0$ respectively.
\end{cor}

\begin{rem}\label{rem:RFH_full_intro}
$ $
\begin{enumerate}[(a)]
	\item Theorem \ref{thm:full_rfh}.(c) was proved in \cite{AK17}. The assumption $\tau\in(0,\frac{m}{\lambda-m})$ in the case of $1\leq m\leq\lambda-1$ on the circle bundle was not made in the statement of \cite[Theorem 1.2]{AK17} although this  was crucially used in the proof, see \cite[Lemma 3.5]{AK17}. In \cite{AK17}, we had the wrong expectation that $\RFH_*(E,\Sigma_\tau)$ is invariant under the change of $\tau$. As pointed out to us by Sara Venkatesh, and as the computations in \cite{Ven18} and the above results show, this invariance property is indeed not always true. 
	\item Oancea \cite{Oan08} proved that negative line bundles which are symplectically aspherical have vanishing symplectic homology.  Ritter \cite{Rit14,Rit16} computed the symplectic homology of negative line bundles $E$ in terms of the quantum homology of $E$ and $c_1^E$ using a generalization of the Seidel representation. Venkatesh \cite{Ven18,Ven21} extended Ritter's work to different versions of symplectic homology of $E$, called completed/reduced symplectic homology, and also computed the Rabinowitz Floer homology of $E$ in the toric case. However her definition of Rabinowitz Floer homology is the homology of the mapping cone of a chain level continuation map between completed symplectic homology and completed symplectic cohomology. It is interesting to see whether this coincides with our definition of Rabinowitz Floer homology. In  \cite{Rit14,Rit16,Ven18,Ven21}, it is assumed that the total space $(E,\Omega)$ is monotone or weakly monotone. It is worth pointing out that, under any of our standing assumptions (A1), (A2), and (A3), the total space $(E,\Omega)$ need not be of this type.
		 
\end{enumerate}
\end{rem}

\paragraph{Acknowledgments} 
We are very grateful to Urs Frauenfelder for years of enlightening discussion. 
P.~Albers acknowledges funding by the Deutsche Forschungsgemeinschaft (DFG, German Research Foundation) through Germany's Excellence Strategy EXC-2181/1 - 390900948 (the Heidelberg STRUCTURES Excellence Cluster), the Transregional Colloborative Research Center CRC/TRR 191 (281071066). 
The research of J.~Kang was supported by National Research Foundation of Korea grant NRF-2020R1A5A1016126 and an Alexander von Humboldt research fellowship. J.~Kang acknowledges the support of Universit\"at Heidelberg for his visit and is grateful to P.~Albers for the warm hospitality.

\section{Floer Gysin sequence}\label{sec:Quantum Gysin sequence}
Let $(M,\om)$ be a closed integral symplectic manifold. As in the introduction, see also Section \ref{sec:line} below, let $\wp:E\to M$ be a complex line bundle with first Chern class
\[
c_1^E=-m[\om]\in \H^2(M;\Z),\qquad m\in\N=\{1,2,\dots\}\,,
\] 
and let $\wp:\Sigma\to M$ be the associated principal $S^1$-bundle. Thus $\wp:\Sigma\to M$ has Euler class equal to $c_1^E$. 
The singular homologies of $M$ and of $\Sigma$ fit into the Gysin exact sequence
\begin{equation}\label{eq:gysin_seq}
\cdots\longrightarrow \H_{*}(\Sigma;\Z)\longrightarrow\H_{*}(M;\Z)\stackrel{\cap\, c_1^E}{\longrightarrow}\H_{*-2}(M;\Z)\longrightarrow \H_{*-1}(\Sigma;\Z)\longrightarrow\cdots
\end{equation}
where $\cap \,c_1^E$ denotes the cap product with $c_1^E$. A goal of this section is to construct a Gysin sequence involving the Floer homology of $M$ and the Floer cap product with $c_1^E$, instead of $\H_{*}(M;\Z)$ and $\cap\, c_1^E$. This will be revisited in the context of Rabinowitz Floer homology theory in Section \ref{sec:gysin_revisit}. 

In this section, we assume that $(M,\om)$ satisfies the following implication
\begin{equation}\label{eq:assumption_M}
	\om(A)>0\,,\;\;A\in\pi_2(M)\quad \Longrightarrow \quad c_1^{TM}(A)\geq 0\;\;\text{or}\;\; c_1^{TM}(A) \leq -\tfrac{1}{2}\dim M\,.
\end{equation}
This is weaker than conditions (A1), (A2), and (A3) from the introduction. 

\subsection{Hamiltonian Floer homology}\label{sec:Ham_Floer}
 We briefly recall the construction of the Hamiltonian Floer homology in the case of $C^2$-small smooth Morse functions $f:M\to\R$. This is sufficient for our purpose. Let $\mathscr L(M)$ be the space of contractible 1-periodic smooth loops in $M$, and let $\widetilde{\mathscr L}(M)$ be the covering space of $\mathscr L(M)$ with  deck transformation group 
\[
\Gamma_M:=\frac{\pi_2(M)}{(\ker\om\cap\ker c_1^{TM})}\,.
\] 
More explicitly, we write elements in $\widetilde{\mathscr L}(M)$ as equivalence classes $\mathfrak{q}=[q,\bar q]$ where 
\[
q\in\mathscr L(M)\,,\qquad \bar q: D^2\to M \;\text{ with }\; \bar q(e^{2\pi it })=q(t)\,,\quad D^2:=\{z\in\mathbb{C}\mid |z|\leq 1\},
\] 
and $(q,\bar q)\sim(q',\bar {q}')$ if and only if $q=q'$ and $[\bar q^\mathrm{rev} \# \bar{q}']=0$ in $\Gamma_M$. Here $\bar q^\mathrm{rev}$ refers to $\bar q$ with opposite orientation.
The Hamiltonian action functional for $f:M\to\R$ is defined by
\begin{equation}\label{eq:classical_action_functional}
\begin{split}
&\mathfrak a_f:\widetilde{\mathscr L}(M)\longrightarrow\R\\
&\mathfrak a_f(\q):=-\int_{D^2}\bar q^*\om-\int_0^1f(q)dt\,.
\end{split}	
\end{equation}
In general, every critical point $\q=[q,\bar q]\in\Crit\mathfrak a_f$ consists of a 1-periodic orbit $q$ of the Hamiltonian vector field $X_f$ defined by $df=\om(X_f,\cdot)$ and any $\bar q$. 
We assume that $f$ is sufficiently $C^2$-small so that $q$ is a constant loop mapping to a critical point of $f$. Abusing notation, we often equate $q$ with its image, i.e.~a point in $M$, and $\bar q$ with a map $\bar q:S^2\to M$ passing through $q$. Hence, we have the identification
\begin{equation}\label{eq:crit_a_f}
	\Crit\mathfrak a_f =\Crit f\x \Gamma_M\,.
\end{equation}
We point out that all  critical values of $\a_f$ are close to integers since $\om$ is integral and $f$ is $C^2$-small. 
The index of $\q\in\Crit\a_f$ is defined to be 
\begin{equation}\label{eq:ind_M}
\begin{split}
	\mu_\FH(\q):=-\mu_\CZ([q,\bar q])&=\mu_{-f}(q)-\tfrac{1}{2}\dim M-2c_1^{TM}([\bar q])\\
	&=\tfrac{1}{2}\dim M-\mu_f(q)-2c_1^{TM}([\bar q])\,,
\end{split}
\end{equation}
where $\mu_{\pm f}$ stands for the Morse index of $\pm f$ and $\mu_\CZ([q,\bar q])$ denotes the Conley-Zehnder index of the linearized flow of $X_f$ along $q$ with respect to the trivialization of $q^*TM$ given by $\bar q$, see Section \ref{sec:index} below.

Let $-\infty\leq a<b\leq+\infty$. If $a$ and $b$ are finite, the chain module for the Floer homology of $\a_f$ is by definition the free $\Z$-module 
\[
\FC_*^{(a,b)}(f) = \bigoplus_{\q}\Z\langle \q \rangle 
\]
generated by $\q\in\Crit\mathfrak a_f$ with $\mu_\FH(\q)=*$ and $\mathfrak a_f(\q)\in(a,b)$. For $a=-\infty$ and $b=+\infty$, we define 
\begin{equation}\label{eq:FC}
\FC_*(f)=\FC_*^{(-\infty,+\infty)}(f) =\varinjlim_{b\uparrow+\infty}\varprojlim_{a\downarrow-\infty}\FC_*^{(a,b)}(f) \,,
\end{equation}
where the limits are taken with respect to the canonical inclusions and projections
\begin{equation}\label{eq:chain_filt}
\FC_*^{(a,b)}(f)\hookrightarrow \FC_*^{(a,b')}(f)\,,\qquad \FC_*^{(a',b)}(f) \twoheadrightarrow \FC_*^{(a,b)}(f)
\end{equation}
for $a'<a<b<b'$. An equivalent formulation is that $\FC_*^{(-\infty,+\infty)}(f)$ is the $\Z$-module composed of all formal linear combinations 
\begin{equation}\label{eq:formal}
\sum_{\q\in\Crit\a_f} a_\q \q\,,\qquad a_\q\in\Z\,,\;\; \mu_\FH(\q)=*
\end{equation}
subject to the Novikov condition 
\begin{equation}\label{eq:nov_cond}
\forall \kappa\in\R\;:\; \#\{\q \mid a_\q\neq 0\,,\;\kappa\leq \a_f(\q)\}<\infty\,.	
\end{equation}
In particular, the order of the two limits in \eqref{eq:FC} does not matter, see also Remark \ref{rem:ML}. 
The chain modules $\FC_*^{(-\infty,b)}(f)$ and $\FC_*^{(a,+\infty)}(f)$ are defined by taking only one limit.

We denote by
\[
\j=\mathfrak{j}(M,\om) \subset \Gamma(S^1\x M,\mathrm{Aut}(TM)) 
\]
the space of smooth time-dependent $\om$-compatible almost complex structures $j$, i.e.~$j_t:=j(t,\cdot)$, $t\in S^1$, is an almost complex structure on $M$ compatible with $\om$, that is $g_t=\om(\cdot,j_t\cdot)$ is a Riemannian metric for each time $t\in S^1$. 
For $\q_\pm=[q_\pm,\bar q_\pm]\in\Crit\a_f$ and $j\in\j$, let 
\[
	\widehat\NN(\q_-,\q_+,\a_f,j)
\]
be the moduli space of smooth solutions $q:\R\x S^1\to M$ of
\begin{equation}\label{eq:Floer_eq_M}
\p_sq+j_t(q)(\p_tq-X_f(q))=0
\end{equation}
with asymptotic condition 
\begin{equation}\label{eq:asympt_condition}
\lim_{s\to-\infty}q(s,\cdot)=q_-\,,\qquad \lim_{s\to+\infty}q(s,\cdot)=q_+\,,\qquad [\bar q_-\#q\#\bar q_+^\mathrm{rev}]=0\;\text{ in }\;\Gamma_M\,.
\end{equation}
Since $q_-$ and $q_+$ are constant loops, every solution $q$ in the moduli space can be viewed as a continuous map $q:S^2\to M$ passing through $q_-$ and $q_+$.

There exists a residual subset
\begin{equation}\label{eq:j_reg}
\j_\mathrm{reg}=\mathfrak{j}_\textrm{reg}(f)\subset\mathfrak{j}\,,
\end{equation}
meaning that it is a countable intersection of open dense subsets,
such that the following properties (i)-(viii) hold for all $j\in \mathfrak{j}_\textrm{reg}(f)$ and all pairs $\q_-,\q_+\in\Crit\a_f$. 
\begin{enumerate}[(i)]
\item The moduli space $\widehat\NN(\q_-,\q_+,\a_f,j)$ is cut out transversely. More precisely, the operator 
	\begin{equation}\label{eq:d_q}
		\mathfrak{d}_q:W^{1,p}(\R\x S^1, q^*TM)\longrightarrow L^p(\R\x S^1, q^*TM)
	\end{equation}
	obtained by linearizing the Floer equation \eqref{eq:Floer_eq_M} at any $q\in \widehat\NN(\q_-,\q_+,\a_f,j)$ is surjective, where $p>2$. An explicit formula for $\mathfrak{d}_q$ is given in  \eqref{eq:horizontal_diff}. 
\item For $A\in\pi_2(M)$, the moduli space  
\[
\NN(A,j):=\{(t,v)\in S^1\x C^\infty(S^2,M)\mid\text{$v$~simple $j_t$-holomorphic with $[v]=A$}\big\}
\]
is cut out transversely. 
\item Let $v:S^2\to M$ be a nonconstant $j_{t}$-holomorphic sphere with $c_1^{TM}([v])\in\{0,1\}$ for some $t\in S^1$. Then $v$ misses all critical points of $f$, i.e.~$v(S^2)\cap \Crit f =\emptyset$.
\item If $\mu_\FH(\q_-)-\mu_\FH(\q_+)\leq 2$, then for every $q\in \widehat\NN(\q_-,\q_+,\mathfrak{a}_f,j)$ and $t\in S^1$, the line $q(\R,t)$ does not intersect any nonconstant $j_{t}$-holomorphic sphere $v:S^2\to M$ with $c_1^{TM}([v])=0$. 
\end{enumerate}
In addition we require properties (v)-(viii) to hold for $j\in\j_\mathrm{reg}$, which are listed in Section \ref{sec:Floer_cap}, but which we do not need in this section for the construction of Floer homology. 

From now on we take $j\in\j_\mathrm{reg}$. 
Condition (i) implies that the moduli space $\widehat\NN(\q_-,\q_+,\mathfrak{a}_f,j)$ is a smooth manifold whose dimension can be computed using \eqref{eq:ind_M} and \eqref{eq:asympt_condition} as 
\begin{equation}\label{eq:dim_M}
\begin{split}
\dim\widehat\NN(\q_-,\q_+,\mathfrak{a}_f,j) &= 	\mu_\FH(\q_-)- \mu_\FH(\q_+)\\
& =\mu_{-f}(q_-)-\mu_{-f}(q_+) + 2c_1^{TM}([q])	
\end{split}
\end{equation}
where $q$ is any element belonging to the moduli space. Moreover condition (ii) implies that the moduli space $\NN(A,j)$ is a smooth manifold of dimension 
\[
\dim \NN(A,j)=\dim M+2c_1^{TM}(A)+1.
\]
The fact that properties (i)-(iv) give rise to a residual set follows as in \cite{HS95}.  

In order to work with integer coefficients, we additionally fix orientations on all moduli spaces $\widehat\NN(\q_-,\q_+,\mathfrak{a}_f,j)$ which are coherent under the gluing operation in Floer theory, see \cite{FH93}.  
 If $\q_-\neq\q_+$, there is a free $\R$-action on $\widehat\NN(\q_-,\q_+,\mathfrak{a}_f,j)$ given by translation of solutions in the $s$-direction, i.e.~$q(\cdot,\cdot)\mapsto q(\cdot+s,\cdot)$ for $s\in\R$, and we denote the quotient space by 
\[
\NN(\q_-,\q_+,\a_f,j):=\widehat\NN(\q_-,\q_+,\a_f,j)/\R\,.
\]
Suppose $\mu_\FH(\q_-)-\mu_\FH(\q_+)=1$ so that the quotient space is a zero-dimensional manifold. This is also a finite set as we explain below. We denote by
\begin{equation}\label{eq:signed_count}
\#\NN(\q_-,\q_+,\a_f,j)\in\Z	
\end{equation}
the signed count of elements where $[q]\in\NN(\q_-,\q_+,\a_f,j)$ contributes $+1$ if the orientation at $q\in\widehat\NN(\q_-,\q_+,\mathfrak{a}_f,j)$ agrees with the one given by $\R$-action and $-1$ otherwise.

We define the boundary operators as the $\Z$-linear extension of
\begin{equation}\label{eq:floer_boundary}
\begin{split}
\p_j=\p:\FC_*^{(a,b)}(f) &\longrightarrow \FC_{*-1}^{(a,b)}(f)\\
 \q_- &\longmapsto  \sum_{\q_+} \#\NN(\q_-,\q_+,\mathfrak{a}_f,j) \q_+\,,
 \end{split}
\end{equation}
where the sum ranges over $\q_+\in\Crit\mathfrak a_f$ with $\mu_\FH(q_+)=*-1$ and $\mathfrak a_f(\q_+)\in (a,b)$. Properties (i)-(iv) together with the hypothesis \eqref{eq:assumption_M} yield that the phenomenon of bubbling-off of pseudo-holomorphic spheres does not occur  for sequences in $\NN(\q_-,\q_+,\a_f,j)$ with $\mu_\FH(\q_-)-\mu_\FH(\q_+)\leq 2$ as in \cite{HS95}. Hence  the signed count in \eqref{eq:signed_count} indeed takes values in $\Z$ and $\p\circ \p=0$ holds. We define the Floer homology 
\begin{equation}\label{eq:Floer_homology}
	\FH_*^{(a,b)}(f,j):=\H_*\big(\FC^{(a,b)}(f),\p_j\big)\,.
\end{equation}
We also write $\FH_*(f,j)=\FH_*^{(-\infty,\infty)}(f,j)$. 
The standard argument of continuation homomorphisms yields that different choices of $j$ produce  isomorphic homologies, so we often omit $j$ from the notation. Similarly, if $a,b\in\Z+\frac{1}{2}$, in particular $a$ and $b$  are not critical values of $\a_f$, and $f':M\to\R$ is another small Morse function, a continuation homomorphism induces an isomorphism 
\begin{equation}\label{eq:continuation}
\FH_*^{(a,b)}(f)\cong \FH_*^{(a,b)}({f'})\,.	
\end{equation}
The maps in \eqref{eq:chain_filt} for $a'<a<b<b'$ are chain homomorphisms and induce homomorphisms 
\begin{equation}\label{eq:direct_system}
	\FH_*^{(a,b)}(f)\to \FH_*^{(a,b')}(f)\,,\qquad \FH_*^{(a',b)}(f) \to \FH_*^{(a,b)}(f)\,,
\end{equation}
which we call action filtration homomorphisms.

We consider the Novikov ring associated with $\Gamma_M$,
\begin{equation}\label{eq:nov}
\Lambda:=\bigoplus_{{\scriptscriptstyle\heartsuit}\in\Z}\Lambda_{\scriptscriptstyle\heartsuit}\,,
\end{equation}
where $\Lambda_{\scriptscriptstyle\heartsuit}$ is defined by
\[
\left\{\sum_{A\in\Gamma_M}a_AT^A\,\bigg|\, a_A\in\Z,\; -2c_1^{TM}(A)=\heartsuit,\;\forall\kappa\in\R:\#\big\{A\mid a_A\neq0,\;-\om(A)\leq\kappa\big\}<\infty\right\}	\,.
\]
Here $T$ is a formal parameter. If we denote by $\Lambda_{\scriptscriptstyle\heartsuit}^{(a,b)}$ the subset of formal sums over $A\in\Gamma_M$ with $-2c_1^{TM}(A)=\heartsuit$ and $-\om(A)\in (a,b)$, then we have 
\[
\Lambda_{\scriptscriptstyle\heartsuit} = \Lambda_{\scriptscriptstyle\heartsuit}^{(-\infty,+\infty)}=\varinjlim_{b\uparrow +\infty}\varprojlim_{a\downarrow -\infty}\Lambda_{\scriptscriptstyle\heartsuit}^{(a,b)}\,.
\]
There is a natural action given by the $\Z$-linear extension of
\begin{equation}\label{eq:Lambda-action}
\Lambda^{(a',b')}_{\scriptscriptstyle\heartsuit} \x \FC_*^{(a,b)}(f) \longrightarrow \FC_{*+{\scriptscriptstyle\heartsuit}}^{(a+a',b+b')}(f)\,,\qquad (T^{A},[q,\bar q])\mapsto [q,\bar q\#s]\,,	
\end{equation}
where  $s:S^2\to M$ is  any continuous map with $[s]=A$. This gives rise to a $\Lambda$-module structure on $(\FC(f),\partial)$ and hence  on $\bigoplus_{*\in\Z}\FH_*(f)$. 
As established in \cite{HS95,Ono95,PSS96}, and as we will recall in \eqref{eq:isom2},  $\bigoplus_{*\in\Z}\FH_*(f)$ is isomorphic, as $\Lambda$-modules, to the singular homology of $M$ with coefficients in $\Lambda$. We sometimes denote $\FH_*(M)=\FH_*(f)$. 

\begin{rem}\label{rem:ML}
The action filtration homomorphisms in \eqref{eq:direct_system} form a bidirect system. Taking limits in the following order we obtain
\begin{equation}\label{eq:two_fh}
\varinjlim_{b\uparrow+\infty}\varprojlim_{a\downarrow-\infty}\FH_*^{(a,b)}(f) \cong \FH_*(f)\,.
\end{equation}
 To see this, it suffices to show that the inverse limit is an exact functor, which is true if the direct system given by the latter maps in \eqref{eq:direct_system},
\begin{equation}\label{eq:inverse}
\pi_{a_2,a_1}^b:\FH^{(a_2,b)}_*(f)\longrightarrow \FH^{(a_1,b)}_*(f)	\,, \qquad a_2<a_1<b
\end{equation}
satisfies the Mittag-Leffler condition. That is,  for any $a,b\in\R$, there exists $a_1<a$ such that for any $a_2<a_1$, the image of $\pi_{a_2,a}^b$ coincides with that of $\pi_{a_1,a}^b$. 
This indeed holds since due to the invariance property we may assume that $(f,j)$ is chosen such that the Floer complex for $(f,j)$ equals the Morse complex for the same pair, see Section \ref{sec:time_indep_j}. The Mittag-Leffler condition for the direct system corresponding to \eqref{eq:inverse} in  Morse homology is trivially satisfied with any number $a_1\in\Z+\frac{1}{2}$ less than $a$. We note that taking direct limits for modules always preserves exactness. 

We refer to \cite[Section 3.5]{Wei94} for an account on the Mittag-Leffler condition and to \cite{CF11} for a nice  discussion on the above issue of completing the action-window in greater generality with field coefficients. 
\end{rem}

\subsection{Floer cap product}\label{sec:Floer_cap}
We now recall the definition of the cap product with $-c_1^E$ in Floer homology. In addition to the data from above, we orient $M$ and choose a codimension two closed oriented smooth submanifold $N$ of $M$ with homology class $[N]=\PD(-c_1^E)\in \H_{\dim M-2}(M;\Z)$, where $\PD$ stands for the Poincar\'e duality. Perturbing $N$, we may assume that $N$ does not contain any critical point of $f$. Following \cite{GG19}, we have the minus sign in $-c_1^E$ to make the sign convention below agree with the corresponding sign convention in computing the Morse homology of $\Sigma$.  
  For $\q_-,\q_+\in\Crit\mathfrak a_f$ and $j\in\mathfrak{j}_\mathrm{reg}(f)$, we define the subspace  
\[
\widehat\NN(\q_-,\q_+,\mathfrak{a}_f,j,N):=\big\{q\in \widehat\NN(\q_-,\q_+,\mathfrak{a}_f,j) \mid q(0,0)\in N\big\}=\ev_{(0,0)}^{-1}(N)\,,
\]
where
\begin{equation}\label{eq:evaluation_N}
\ev_{(0,0)}:\widehat\NN(\q_-,\q_+,\mathfrak{a}_f,j)\to M \,,\qquad q\mapsto q(0,0)\,.		
\end{equation}

In addition to (i)-(iv) listed after \eqref{eq:j_reg}, we require $j\in\j_\mathrm{reg}(f)$ to satisfy the following properties for every pair $\q_-,\q_+\in\Crit\mathfrak{a}_f$:
\begin{enumerate}[(i)]
	\item[(v)] The map $\ev_{(0,0)}$ is transverse to $N$. 
	\item[(vi)] Let $\mu_\FH(\q_-)-\mu_\FH(\q_+)\leq 3$. For every $q\in\widehat\NN(\q_-,\q_+,\mathfrak{a}_f,j,N)$ and $t\in S^1$, there is no nonconstant $j_t$-holomorphic sphere $v:S^2\to M$ with $c_1^{TM}([v])=0$ intersecting the line $q(\R,t)\subset M$.
	\item[(vii)] Let $\mu_\FH(\q_-)-\mu_\FH(\q_+)\leq 3$. For every $q\in\widehat\NN(\q_-,\q_+,\mathfrak{a}_f,j)$, there is no nonconstant $j_0$-holomorphic sphere $v:S^2\to M$ with $c_1^{TM}([v])=0$ intersecting $q(0,0)$ and $N$ simultaneously.
	\item[(viii)] Let $\mu_\FH(\q_-)-\mu_\FH(\q_+)\leq 1$. For every $q\in\widehat\NN(\q_-,\q_+,\mathfrak{a}_f,j)$, there is no nonconstant $j_0$-holomorphic sphere $v:S^2\to M$ with $c_1^{TM}([v])=1$ intersecting $q(0,0)$ and $N$ simultaneously.
\end{enumerate} 
Property (v) implies that the space $\widehat\NN(\q_-,\q_+,\mathfrak{a}_f,j,N)$ is a smooth manifold of dimension 
\[
\dim \widehat\NN(\q_-,\q_+,\mathfrak{a}_f,j,N)=\mu_\FH(\q_-)-\mu_\FH(\q_+)-2\,.
\]
An adaptation of arguments in \cite{HS95} shows that $\j_\mathrm{reg}=\j_\mathrm{reg}(f)$ is indeed a residual set. For example, we consider the map
\[
\begin{split}
\mathrm{EV}:\widehat\NN(\q_-,\q_+,\a_f,j)\x \NN(A,j_0)\x_{G}S^2 &\longrightarrow M\x M \\ 
(q,[v,z]) &\longmapsto (q(0,0),v(z))\,,
\end{split}
\]
where $\NN(A,j_0)$ is the space of simple $j_0$-holomorphic spheres $v$ such that $[v]=A$ for some $A\in\pi_2(M)$ with $c_1^{TM}(A)=0$ for (vii) or $c_1^{TM}(A)=1$ for (viii), and $G=\mathrm{PSL}(2;\C)$ acts diagonally. Then for generic $j$ the map $\mathrm{EV}$ 
is transverse to $\Delta_M$, where $\Delta_M$ is the diagonal of $M\x M$. Hence $\mathrm{EV}^{-1}(\Delta_M)$ is empty for dimension reasons. This proves that properties (vii) and (viii) hold for generic $j$. Note that nonconstant $j_0$-holomorphic sphere $v$ always intersects $N$ since $v\cdot N=-c_1^{E}([v])=m\omega([v])>0$.

Suppose that $\mu_\FH(\q_-)-\mu_\FH(\q_+)=2$ so that $\widehat\NN(\q_-,\q_+,\mathfrak{a}_f,j,N)$ is compact and zero-dimensional due to properties (v)-(viii). We denote by
\begin{equation}\label{eq:intersection_N}
\#\widehat\NN(\q_-,\q_+,\mathfrak{a}_f,j,N)\in\Z	
\end{equation}
the intersection number of the map $\ev_{(0,0)}$ and $N$ with sign determined by the orientation on $\widehat\NN(\q_-,\q_+,\a_f,j)$ and the coorientation on $N$. 
The Floer cap product with $-c_1^E$ is the map
\begin{equation}\label{eq:floer_cap}
\Psi^{c_1^E}: \FH_*^{(a,b)}(f,j)\longrightarrow \FH_{*-2}^{(a,b)}(f,j)
\end{equation}
defined on the chain level by the $\Z$-linear extension of
\begin{equation}\label{eq:floer_cap_chain}
\psi^{c_1^E}(\q_-):= \sum_{\q_+}\# \widehat\NN(\q_-,\q_+,\a_f,j,N) \q_+\,,
\end{equation}
where the sum ranges over $\q_+\in\Crit\mathfrak{a}_f$ with $\mathfrak{a}_f(\q_+)\in(a,b)$ and $\mu_\FH(\q_+)=*-2$. Properties (v)-(viii) together with the hypothesis \eqref{eq:assumption_M} ensure that $\psi^{c_1^E}(\q_-)$ is well-defined and a chain homomorphism. We also remark that $\psi^{c_1^E}$ is $\Lambda$-linear on $\FC(f)$. Moreover, $\Psi^{c_1^E}$ is compatible with the isomorphism in \eqref{eq:continuation} when $a,b\in\Z+\frac{1}{2}$. To be precise, if $f':M\to\R$ is another $C^2$-small Morse function, the diagram 
\begin{equation}\label{eq:commute_homotopy}
	\begin{tikzcd}[row sep=1.5em,column sep=1.3em]
\FC_{*}^{(a,b)}(f) \arrow{r}{\psi^{c_1^E}} \arrow{d} & \FC_{*-2}^{(a,b)}(f)   \arrow{d}  \\
\FC_{*}^{(a,b)}(f') \arrow{r}{\psi^{c_1^E}} & \FC_{*-2}^{(a,b)}(f')\,, \end{tikzcd}
\end{equation}
	where the vertical arrows are chain level continuation homomorphisms, commutes up to homotopy. Therefore, at the homology level, the induced diagram genuinely commutes and the vertical maps become isomorphisms.
Similar compatibility holds for another choice of $j$ or $N$. We note that this is actually a specific instance of the so-called pair-of-pants product in Floer theory, see \cite{PSS96} or \cite[Chapter 12]{MS12}.

\subsection{Gysin sequence in Floer homology}\label{sec:quantum_gysin}
We recall that $\wp:\Sigma\to M$ is a principal $S^1$-bundle with Euler class $c_1^E$, see the beginning of this section. 
Since $f:M\to\R$ is a Morse function, its lift $\tilde f:=f\circ \wp:\Sigma\to\R$ is Morse-Bott. The critical manifold $\Crit \tilde f$ is the union of all fiber circles in $\Sigma$ over critical points of $f$, i.e.~$\Crit \tilde f$ is a disjoint union of circles. We fix a perfect Morse function 
\begin{equation}\label{eq:ftn_h}
h:\Crit\tilde f \longrightarrow\R
\end{equation}
and denote by $\hat q$ and $\check q$ the maximum and minimum points of $h$ over $q\in\Crit f$ respectively. As a set, $\Crit h=\Crit f\sqcup\Crit f$. For $\q=[q,\bar q]\in\Crit \a_f=\Crit f\x\Gamma_M$, we define equivalence classes $\hat\q=[\hat q,\bar q],\check\q=[\check q,\bar q]\in\Crit h\x\Gamma_M$, i.e.~$\hat\q=\hat\q'$ if and only if $\q=\q'$, and likewise for $\check \q$. The index of $\hat \q$ and $\check \q$ is defined by
\[
\mu_\FH^h(\hat\q):=\mu_\FH(\q)+1,\qquad \mu_\FH^h(\check\q):=\mu_\FH(\q)\,.
\]
Let $-\infty\leq a<b\leq+\infty$. For finite $a$ and $b$, we define the $\Z$-module
\begin{equation}\label{eq:floer_group}
\FC_*^{(a,b)}(\tilde f)  
\end{equation}
generated by $\hat\q$ and $\check\q$ for all $\q\in\Crit\a_f$ with $\mu_\FH^h(\hat \q)=\mu_\FH^h(\check \q)=*$ and $\a_f(\q)\in(a,b)$. We suppress the dependence of $h$ in the notation. As before, we also define 
\begin{equation}\label{eq:tilde_f_complete}
\FC_*(\tilde f)=\FC_*^{(-\infty,+\infty)}(\tilde f):=\varinjlim_{b\uparrow+\infty}\varprojlim_{a\downarrow-\infty}\FC_*^{(a,b)}(\tilde f)\,.
\end{equation}
Similarly, we have $\FC_*^{(-\infty,b)}(\tilde f)$ and $\FC_*^{(a,+\infty)}(\tilde f)$. As in \eqref{eq:formal} and \eqref{eq:nov_cond}, 
elements in $\FC_*(\tilde f)$ can be interpreted as formal linear combinations subject to the Novikov condition.

As before, let $j\in\j_\textrm{reg}(f)$, and let $N\subset M$ be a closed submanifold with $[N]=\PD(-c_1^E)$. We introduce the boundary operator
\begin{equation}\label{eq:bdry}
\p^\wp_j=\p^\wp:=\hat\p+\check\p+\p^{c_1^E}: \FC^{(a,b)}_*(\tilde f)\longrightarrow  \FC^{(a,b)}_{*-1}(\tilde f)\,.
\end{equation}
where each term is defined on generators as follows:
\[
	\begin{aligned}
	\hat\partial(\hat \q_-)&:=-\sum_{\q_+} \#\NN(\q_-,\q_+,\mathfrak{a}_f,j) \hat \q_+\,, \qquad\quad  & \hat\partial(\check \q_-)&:=0\,, \\
	\check\partial(\check \q_-)&:=\sum_{\q_+} \#\NN(\q_-,\q_+,\mathfrak{a}_f,j) \check \q_+ \,, & \check\partial(\hat \q_-)&:=0\,, \\
	\p^{c_1^E}(\check \q_-)&:= \sum_{\q_+} \#\widehat\NN(\q_-,\q_+,\mathfrak{a}_f,j,N) \hat \q_+\,, & \p^{c_1^E}(\hat \q_-)&:=0\,.
	\end{aligned}
\]
In the first two cases, the sums run over $\q_+\in\Crit\a_f$ with $\a_f(\q_+)\in(a,b)$ and
\[
\begin{split}
\mu_\FH(\q_-)-\mu_\FH(\q_+) =\mu_\FH^h(\hat \q_-)-\mu_\FH^h(\hat \q_+)
=\mu_\FH^h(\check \q_-)-\mu_\FH^h(\check \q_+)=1\,.	
\end{split}
\]
The sum in the definition of $\p^{c_1^E}$ runs over $\q_+\in\Crit\a_f$ with $\a_f(\q_+)\in(a,b)$ and
\[
\mu_\FH(\q_-)-\mu_\FH(\q_+)=2\,,\quad\text{or equivalently}\quad  \mu_\FH^h(\check \q_-)-\mu_\FH^h(\hat\q_+)=1\,.
\] 
Combining $\p\circ\p=0$ (see \eqref{eq:floer_boundary}) and $\psi^{c_1^E}\circ \p=\p\circ \psi^{c_1^E}$ (see \eqref{eq:floer_cap_chain}) on $M$, it follows that
\[
\p^\wp\circ\p^\wp=\hat\p\circ \hat\p + \check\p\circ\check\p + \p^{c_1^E}\circ\check\p +\hat\p\circ \p^{c_1^E}=  0
\] 
holds. Of course, this chain complex is exactly the mapping cone of $\psi^{c_1^E}$:
\begin{equation}\label{eq:cone}
\big(\FC_*(\tilde f),\p^\wp\big)=\mathrm{Cone\,}(\psi^{c_1^E})\,,	
\end{equation}
see \eqref{eq:mapping_cone} below. 
 We denote the resulting homology by
\begin{equation}\label{eq:floer_homology_h}
\FH_*^{(a,b)}(\tilde f,j):=\H_*\big(\FC_*^{(a,b)}(\tilde f),\p^\wp_j\big).	
\end{equation}
A different choice of $j\in\mathfrak{j}_\mathrm{reg}(f)$ leads to an isomorphic homology, so we sometimes omit this from the notation. Similarly, if we take $a,b\in\Z+\frac{1}{2}$, then another $C^2$-small Morse function $f':M\to\R$ defines an isomorphic homology, see Proposition \ref{prop:quantum_gysin}.(b) below. We sometimes denote $\FH_*(\Sigma)=\FH_*(\tilde f)=\FH_*^{(-\infty,+\infty)}(\tilde f)$. 
Since $(\FC(f),\partial)$ is a chain complex of $\Lambda$-modules, see \eqref{eq:Lambda-action}, and $\psi^{c_1^E}$ is $\Lambda$-linear on $\FC(f)$, the cone complex $(\FC(\tilde f),\p^\wp)$ and thus the homology $\FH(\tilde f)$ also have $\Lambda$-module structures.

\begin{rem}\label{rem:ML2}
There are action filtration homomorphisms on $\FH_*^{(a,b)}(\tilde f)$ like \eqref{eq:direct_system}. Inspired by Remark \ref{rem:ML}, one could ask whether the following two are isomorphic:
\[
\FH_*(\tilde f) \,\stackrel{?}{\cong}\, \varinjlim_{b\uparrow+\infty}\varprojlim_{a\downarrow-\infty}\FH_*^{(a,b)}(\tilde f)\,.
\]
Again the answer is affirmative if the inverse limit satisfies the Mittag-Leffler condition. If we use coefficients in a field, this is indeed the case and hence these two are isomorphic  due to results from \cite{CF11}.
\end{rem}

\begin{prop}\label{prop:quantum_gysin}
$ $
\begin{enumerate}[(a)]
	\item For $-\infty\leq a<b\leq+\infty$, there exists the long exact sequence 
\[
\cdots\longrightarrow \FH_{*}^{(a,b)}(\tilde f)\stackrel{\pi}{\longrightarrow} \FH_{*}^{(a,b)}(f)\stackrel{\Psi^{c_1^E}}{\longrightarrow}\FH_{*-2}^{(a,b)}(f)\stackrel{\iota}{\longrightarrow}\FH_{*-1}^{(a,b)}(\tilde f)\longrightarrow\cdots\,.
\]
In particular for $(a,b)=(-\infty,+\infty)$ we have the long exact sequence
\[
\cdots\longrightarrow \FH_{*}(\Sigma)\longrightarrow\FH_{*}(M)\stackrel{\Psi^{c_1^E}}{\longrightarrow}\FH_{*-2}(M)\longrightarrow \FH_{*-1}(\Sigma)\longrightarrow\cdots\,
\]
and all maps are $\Lambda$-linear. 
Moreover, at the chain level, $\pi$ maps $\hat\q$ to $0$ and $\check\q$ to $\q$, and $\iota$ sends $\q$ to $\hat\q$. 
	\item Let $a,b\in(\Z+\frac{1}{2})\cup\{-\infty,+\infty\}$, and let $f':M\to\R$ be another $C^2$-small Morse function. The exact sequence in (a) for $f$ and that for $f'$ are isomorphic. To be precise, there is a commutative diagram of the exact sequences,
\[
	\begin{tikzcd}[row sep=1.5em,column sep=1.3em]
\cdots \arrow{r} & \FH_{*}^{(a,b)}(\tilde f) \arrow{r} \arrow{d}{\cong} & \FH_{*}^{(a,b)}(f)  \arrow{r} \arrow{d}{\cong} & \FH_{*-2}^{(a,b)}(f) \arrow{r} \arrow{d}{\cong} &  \FH_{*-1}^{(a,b)}(\tilde f) \arrow{r}\arrow{d}{\cong} & \cdots \\
\cdots \arrow{r} & \FH_{*}^{(a,b)}(\tilde f') \arrow{r} & \FH_{*}^{(a,b)}(f') \arrow{r} & \FH_{*-2}^{(a,b)}(f')  \arrow{r} & \FH_{*-1}^{(a,b)}(\tilde f') \arrow{r} & \cdots
\end{tikzcd}
\]
where all vertical maps are isomorphisms. 
A corresponding statement holds for  changes of $h$ or $j$.\end{enumerate} 
\end{prop}

\begin{proof}
Statement (a) follows immediately from properties of the mapping cone, see \eqref{eq:cone} and \eqref{eq:mapping_cone}. To see (b), we recall the diagram in \eqref{eq:commute_homotopy} which commutes up to chain homotopy. This induces a chain homomorphism $\FC_*^{(a,b)}(\tilde f)\to \FC_*^{(a,b)}(\tilde f')$, and the commutative diagram in the statement follows, see e.g.~\cite[Section 4.1]{CO18}. Moreover continuation homomorphisms are isomorphisms at the homology level, so all  vertical maps in the diagram are isomorphisms. The rest can be shown analogously.
\end{proof}

The sole role of the function $h$ in this section is basically to double the set $\Crit f$. Its actual usage will become apparent in the next section and in Section \ref{sec:gysin_revisit}.

\subsection{Time-independent $j$}\label{sec:time_indep_j}
For computational purposes, we want to establish Proposition \ref{prop:quantum_gysin} for time-independent almost complex structures. Let $j$ be an $\om$-compatible time-independent almost complex structure on $M$. Then the moduli space 
$\widehat\NN(\q_-,\q_+,\mathfrak{a}_f,j)$ for $\q_\pm=[q_\pm,\bar q_\pm]\in\Crit\a_f=\Crit f\x \Gamma_M$ 
consists of smooth solutions $q:\R\x S^1\to M$ of 
\begin{equation}\label{eq:Floer_eq_M2}
	\p_sq+j(q)(\p_tq-X_f(q))=0
\end{equation}
satisfying the asymptotic condition in  \eqref{eq:asympt_condition}. Equation \eqref{eq:Floer_eq_M2} is identical to \eqref{eq:Floer_eq_M} but with $j_t$ replaced by $j$. Now there is an $\R\x S^1$-action on $\widehat\NN(\q_-,\q_+,\mathfrak{a}_f,j)$, namely 
\begin{equation}\label{eq:RxS1_action}
	\begin{split}
		\R\x S^1\x \widehat\NN(\q_-,\q_+,\mathfrak{a}_f,j)&\longrightarrow \widehat\NN(\q_-,\q_+,\mathfrak{a}_f,j)\\
		(s,t,q)&\longmapsto q(s+\cdot,t+\cdot)\,.
	\end{split}
\end{equation}
We denote the subspace of $t$-independent solutions  
\[
\widehat\NN(q_-,q_+,f,g)\subset \widehat\NN(\q_-,\q_+,\mathfrak{a}_f,j)
\]
which consist of flow lines of $\nabla_gf$, the (positive) gradient of $f$ with respect to the Riemannian metric $g=\om(\cdot,j\cdot)$, connecting critical points $q_-$ and $q_+$ of $\Crit f$, that is
\[
q:\R\to M\,,\qquad \p_sq-\nabla_gf(q)=0\,,\qquad \lim_{s\to \pm\infty}q(s)=q_\pm\,.
\]
We remark that this subspace is nonempty only if $[\bar q_-]=[\bar q_+]$ in $\Gamma_M$, see \eqref{eq:asympt_condition}. 
From the energy computation 
\[
\int_{-\infty}^\infty\int_0^1\|\p_s q\|_g^2\,dtds=\mathfrak{a}_f(\q_-)-\mathfrak{a}_f(\q_+)= \om([q])-f(q_-)+f(q_+)\geq 0
\]  
and the fact that $\om$ is integral and $f$ is small, we readily see $\om([q])\geq0$. Furthermore the following equivalence holds for $q\in\widehat\NN(\q_-,\q_+,\mathfrak{a}_f,j)$, see \cite[Lemma 7.1]{HS95}:
\begin{equation}\label{eq:om-energy}
 \textrm{$q$ is $t$-dependent}  \quad  \Longleftrightarrow \quad  \om([q])> 0	\,.
\end{equation}
An element $q\in\widehat\NN(\q_-,\q_+,\mathfrak{a}_f,j)$ is said to be {\it simple} if for every $\nu\in\N$, there exists $(s,t)\in\R\x S^1$ such that $q(s,t)\neq q\left(s,t+\frac{1}{\nu}\right)$. Note that $t$-independent solutions are obviously not simple. We denote by
\[
\widehat\NN_s(\q_-,\q_+,\mathfrak{a}_f,j)\subset \widehat\NN(\q_-,\q_+,\mathfrak{a}_f,j)
\]
the subspace of simple solutions. The $\R\x S^1$-action in \eqref{eq:RxS1_action} is free on this space.

\begin{prop}\label{prop:j_HS}
	For a generic $C^2$-small Morse function $f\in C^\infty(M)$, there exists a set $\j_\mathrm{HS}=\mathfrak j_\mathrm{HS}(f)$ which is residual in some open subset of the space of $\om$-compatible time-independent almost complex structures such that the following properties hold for all $j\in\mathfrak{j}_{\mathrm{HS}}(f)$ and $\q_-,\q_+\in\Crit\a_f$:
\begin{enumerate}[(a)]
\item The operator in \eqref{eq:d_q} obtained by linearizing \eqref{eq:Floer_eq_M2} at any $q\in \widehat\NN_s(\q_-,\q_+,f,j)$ is surjective. Hence, 
$\widehat\NN_s(\q_-,\q_+,f,j)$ is a smooth manifold of dimension 
\[
\mu_\FH(\q_-)-\mu_\FH(\q_+)=\mu_{-f}(q_-)-\mu_{-f}(q_+)+2c_1^{TM}([q])
\]
where $q$ is any element in $\widehat\NN_s(\q_-,\q_+,f,j)$. 
\item The moduli space $\NN(A,j)$ of all simple $j$-holomorphic spheres in $M$ representing $A\in\pi_2(M)$  is cut out transversely and thus a smooth manifold of dimension $\dim M+2c_1(A)$.
\item The pair $(f,g)$ with $g=\om(\cdot,j\cdot)$ satisfies the Morse-Smale condition. Moreover the operator in \eqref{eq:d_q} obtained by linearizing \eqref{eq:Floer_eq_M2} at any  $q\in\widehat\NN(q_-,q_+,f,g)$ is surjective.  
\item If $\mu_\FH(\q_-)-\mu_\FH(\q_+)\leq 1$, 
	then every $q\in \widehat\NN(\q_-,\q_+,\mathfrak{a}_{f},j)$ is $t$-independent, i.e.
	\[
	\widehat\NN(\q_-,\q_+,\mathfrak{a}_{ f},j) = \left\{\begin{aligned} &  \widehat\NN(q_-,q_+,f,g)  \quad & [\bar q_-]=[\bar q_+]\,, \\[.5ex]
	& \;\emptyset & [\bar q_-]\neq[\bar q_+]	\,.
	\end{aligned}\right.
	\]
\item If $\mu_\FH(\q_-)-\mu_\FH(\q_+)\leq 1$, then there is no $q\in\widehat\NN(\q_-,\q_+,\a_{f},j)$ intersecting a nonconstant $j$-holomorphic sphere $v:S^2\to M$ with $c_1^{TM}([v])\in\{0,1\}$. 
\item If $\mu_\FH(\q_-)-\mu_\FH(\q_+)\leq 3$, then there is no $q\in\widehat\NN_s(\q_-,\q_+,\mathfrak{a}_{f},j)$ intersecting a nonconstant $j$-holomorphic sphere $v:S^2\to M$ with $c_1^{TM}([v])=0$.
\item Let $N$ be a closed submanifold of $M$ such that $[N]=\PD(-c_1^E)$ as before. The evaluation map $\ev_{(0,0)}$ given in \eqref{eq:evaluation_N} restricted to $\widehat\NN_s(\q_-,\q_+,\a_f,j)$ or to $\widehat\NN(q_-,q_+,f,g)$ is transverse to $N$. Hence $\ev^{-1}_{(0,0)}(N)\cap \widehat\NN_s(\q_-,\q_+,\a_f,j)$ and $\ev^{-1}_{(0,0)}(N)\cap \widehat\NN(q_-,q_+,f,g)$ are smooth manifolds of dimension $\mu_\FH(\q_-)-\mu_\FH(\q_+)-2$ and $\mu_{-f}(q_-)-\mu_{-f}(q_+)-2$ respectively.
\item If $\mu_\FH(\q_-)-\mu_\FH(\q_+) \leq 3$,	then every $q\in \widehat\NN(\q_-,\q_+,\mathfrak{a}_{f},j,N)=\ev_{(0,0)}^{-1}(N)$ with $c_1^{TM}([q])\neq 0$ is simple. 
\end{enumerate}
\end{prop}

\begin{proof}

	We refer to \cite[Section 7]{FHS95} and \cite[Section 7]{HS95} for proofs of (a)-(d) and an account on genericity.
	
	Statements (e) and (f) correspond to properties (iii) and (iv) for $\mathfrak{j}_\mathrm{reg}(f)$ in Section \ref{sec:Ham_Floer}, and can be proved by adapting arguments in \cite{HS95}. Since the index assumptions in (e) and (f) are larger by 1 than those in (iii) and (iv), we sketch a proof. 
	In both cases (e) and (f), we assume that a $j$-holomorphic sphere $v$ with the appropriate properties exists. If $c_1^{TM}([v])=1$, then $v$ is necessarily simple. If $c_1^{TM}([v])=0$, then the underlying simple curve of $v$ has also first Chern number 0, and thus, in the argument below, we may assume that $v$ is simple. We first treat case (e). By (c) and (d), $\widehat\NN(\q_-,\q_+,\a_{f},j)=\widehat\NN(q_-,q_+,f,g)$ is a smooth manifold. We then consider
	\[
	\begin{split}
		\mathrm{EV}&:\left(\widehat\NN(q_-,q_+,f,g)\x_{\R }\R\right)  \x \Big(\NN(A,j)\x_G S^2 \Big)\longrightarrow M\x M\\
		\mathrm{EV}&\big([q,s],[v,z]\big):=\big(q(s),v(z)\big)\,.
	\end{split}
	\]
	where $G=\mathrm{PSL}(2,\C)$ acts diagonally. 
	Using a standard transversality argument, we can show that, for a generic choice of $j$, $\mathrm{EV}$ is transverse to the diagonal $\Delta_M$ in $M\x M$. Since the domain of $\mathrm{EV}$ has dimension at most $\dim M-1$, the preimage $\mathrm{EV}^{-1}(\Delta_M)$ is empty, a contradiction. For case (f), we consider
	\[
	\begin{split}
	\mathrm{EV}&:\Big(\widehat\NN_s(\q_-,\q_+,\mathfrak{a}_{f},j)\x_{(\R\x S^1)}\R\x S^1\Big) \x \Big(\NN(A,j)\x_G S^2\Big) \longrightarrow M\x M\\
	\mathrm{EV}&\big([q,s,t],[v,z]\big):=\big(q(s,t),v(z)\big)\,.		
	\end{split}
	\]
	Here it is crucial that $\widehat\NN_s(\q_-,\q_+,\mathfrak{a}_{f},j)$ carries the $\R\x S^1$-action in \eqref{eq:RxS1_action}. 
	The domain of $\mathrm{EV}$ has dimension at most $\dim M-1$, thus $\mathrm{EV}^{-1}(\Delta_M)$ is empty for a generic $j$. This completes the proofs of (e) and (f).
	
	A combination of (a), (c), and standard transversality arguments proves (g). 
	
	The proof of (h) goes along the same lines as in the proof of (d) but we include it for completeness. Let $q \in \widehat\NN(\q_-,\q_+,\mathfrak{a}_{f},j,N)$ with $c_1^{TM}([q])\neq0$ be not simple, i.e.~there exists an integer $\nu\geq 2$ such that $q\left(s,t+\frac{1}{\nu}\right)=q(s,t)$ for all $(s,t)\in\R\x S^1$. Let $\nu$ be the largest integer satisfying this, which exists since $q$ depends on $t$ by $c_1^{TM}([q])\neq0$. Then 
	\begin{equation}\label{eq:simple}
	q_\mathrm{sim}:\R\x S^1\to M\,,\qquad q_\mathrm{sim}(s,t):=q\left(\frac{s}{\nu},\frac{t}{\nu}\right)
	\end{equation}
	belongs to $\widehat\NN_s([q_-,0],[q_+,q_\mathrm{sim}],\mathfrak{a}_{\frac{1}{\nu}f},j, N)$. 
	If $c_1^{TM}([q])>0$, then this is absurd since this moduli space has dimension 
	\begin{equation}\label{eq:dim_est}
	\mu_{-f}(q_-)-\mu_{-f}(q_+)+2c_1^{TM}([q_\mathrm{sim}])-2=\mu_{-f}(q_-)-\mu_{-f}(q_+)+\frac{2}{\nu}c_1^{TM}([q])-2\leq -1
	\end{equation}
	where the last inequality follows from $\mu_{-f}(q_-)-\mu_{-f}(q_+)+2c_1^{TM}([q])\leq3$ by the hypothesis. 
	If $c_1^{TM}([q])<0$, or equivalently $c_1^{TM}([q_\mathrm{sim}])<0$, then by the assumption in \eqref{eq:assumption_M}, $c_1^{TM}([q_\mathrm{sim}])\leq-\frac{1}{2}\dim M$. In this case, the left-hand side of \eqref{eq:dim_est} is at most $-2$. This contradiction completes the proof of (h). 
%
\end{proof}

\begin{rem}\label{rem:Ono95}
 Arguing as in \cite{Ono95}, Proposition \ref{prop:j_HS} should hold under a slightly weaker assumption than \eqref{eq:assumption_M}, namely also in the case $c_1^{TM}(A)=-\frac12\dim M+1$, at the cost of working with $\frac{1}{\ell}f$ for some $\ell\in\N$ depending on the size of $\a_f(\q_-)-\a_f(\q_+)$.
\end{rem}

Let $f$ and $j$ be as in Proposition \ref{prop:j_HS}. Let $a,b\in(\Z+\frac{1}{2})\cup\{-\infty,+\infty\}$.  Proposition \ref{prop:j_HS} ensures that the Floer homology $\FH^{(a,b)}(f,j)$ is defined as in \eqref{eq:Floer_homology}, have action filtration homomorphisms as in \eqref{eq:direct_system}, and the Floer complex of $(\a_f,j)$ is identical to the Morse complex of $(-f,g)$. This yields 
	\begin{equation}\label{eq:morse_isom}
	\FH^{(a,b)}_*(f,j) \cong \bigoplus_{\diamond+{\scriptscriptstyle\heartsuit}=*+\frac{\dim M}{2}}\H_\diamond(M;\Z)\otimes\Lambda_{\scriptscriptstyle\heartsuit}^{(a,b)}\,.
	\end{equation}
	where $\Lambda_{\scriptscriptstyle\heartsuit}^{(a,b)}$ is defined after \eqref{eq:nov}. Here and below, we use a coherent orientation such that \eqref{eq:morse_isom} holds, as opposed to an isomorphism to a homology of $M$ with twisted coefficients, see \cite[Chapter 4]{Sch93}.
 For $j'\in\mathfrak{j}_\mathrm{reg}(f)$, Floer's continuation argument using Proposition \ref{prop:j_HS}.(c) yields an isomorphism
	\begin{equation}\label{eq:isom}
		\FH^{(a,b)}_*(f,j')   \stackrel{\cong}{\longrightarrow}   \FH^{(a,b)}_*(f,j) \,,
	\end{equation}
	and therefore we conclude
	\begin{equation}\label{eq:isom2}
	\begin{split}
	\FH_*(M)=\FH_*^{(-\infty,+\infty)}(f,j') &\cong \FH_*^{(-\infty,+\infty)}(f,j) \\[.5ex]
	& \cong \bigoplus_{\diamond+{\scriptscriptstyle\heartsuit}=*+\frac{\dim M}{2}} \H_{\diamond}(M;\Z)\otimes\Lambda_{\scriptscriptstyle\heartsuit} \cong \H_{*+\frac{\dim M}{2}}(M;\Lambda)\,.
	\end{split}
	\end{equation}
	 We refer to \cite[Theorem 6.1]{HS95} or \cite[Theorem 4.7]{Ono95} for details.

In order to define the Floer cap product with $j\in\mathfrak{j}_\mathrm{HS}(f)$, we assume in addition condition (H), namely:
\begin{equation}\label{eq:H}
\centering	\textrm{(H)\qquad \quad If $A\in\pi_2(M)$ has $\om(A)\neq0$, then $c_1^{TM}(A)\neq0$.\qquad }
\end{equation}
This condition is implied by any of the standing hypotheses (A1), (A2), and (A3) of this paper.	We remark that a condition similar to (H) is  needed to define the Gromov-Witten invariants with domain independent almost complex structure, see  \cite[Chapter 1]{MS12}.

\begin{rem}\label{rem:H}
We list some consequences of condition (H) in \eqref{eq:H}. 
\begin{enumerate}[(i)]
\item There is $\lambda\in\R$ such that 
\begin{equation}\label{eq:monotonicity}
\om(A) =\lambda c_1^{TM}(A) \qquad \forall A\in\pi_2(M)\,.
\end{equation}
Actually, our assumption that $\om$ is integral implies $\lambda\in\Q$.
We argue as in \cite[Lemma 1.1]{HS95}. If $\om$ vanishes on $\pi_2(M)$, then this holds with $\lambda=0$. Suppose that $\om$ is nonzero on $\pi_2(M)$. We first verify \eqref{eq:monotonicity} for all $A\in\pi_2(M)$ with $\om(A)\neq0$. Let $B,C\in\pi_2(M)$ have $\om(B)\neq0$ and $\om(C)\neq0$. If $\frac{c_1^{TM}(B)}{\om(B)}\neq \frac{c_1^{TM}(C)}{\om(C)}$, then $A:=c_1^{TM}(C)B-c_1^{TM}(B)C\neq0$ satisfies $c_1^{TM}(A)=0$ and $\om(A)\neq0$. This contradiction shows \eqref{eq:monotonicity} for all $A\in\pi_2(M)$ with $\om(A)\neq0$ for some $\lambda\neq0$. Suppose now $A\in\pi_2(M)$ satisfies $\om(A)=0$. Then we pick any $B\in\pi_2(M)$ with $\om(B)\neq 0$. Since $\om(B)=\lambda c_1^{TM}(B)$ and $\om(A+B)=\lambda c_1^{TM}(A+B)$, we deduce $c_1^{TM}(A)=0$. Hence \eqref{eq:monotonicity} holds also in the case of $\om(A)=0$. 

We note that $\ker\om\cap\ker c_1^{TM}=\ker c_1^{TM}$ by \eqref{eq:monotonicity} and thus $\Gamma_M= \pi_2(M)/\ker c_1^{TM}$.

\item At most finitely  many critical points of $\mathfrak{a}_f$ can have the same $\mu_\FH$-index, see \eqref{eq:ind_M}. Thus, for every $*\in\Z$, the modules $\FC_*^{(-\infty,+\infty)}(f)$ and $\FC_*^{(-\infty,+\infty)}(\tilde f)$ are finitely generated and composed of finite sums. The two homologies in Remark \ref{rem:ML2} are obviously isomorphic.
	Furthermore, the Novikov condition in the definition of the Novikov ring $\Lambda$ is automatically fulfilled, see \eqref{eq:nov} and below. 
	
	Suppose that $c_1^{TM}$ does not vanish on $\pi_2(M)$, which is equivalent to saying that $\Gamma_M= \pi_2(M)/\ker c_1^{TM}$ is nontrivial. Let $c_M>0$ denote the minimal Chern number of $M$, i.e.~$c_1^{TM}(\pi_2(M))=c_M\Z$. 
	Then we have a graded ring isomorphism between $\Lambda$ and the Laurent polynomial ring $\Z[t,t^{-1}]$ with $\deg t=-2c_M$ given by $T^A\mapsto t^{c_1^{TM}(A)/c_M}$ for $A\in\Gamma_M$. If $c_1^{TM}$ vanishes on $\pi_2(M)$, we have $\Lambda\cong\Z$.
\item A bound on index difference implies a bound on action difference. More precisely, for any $\mu>0$, there exists $C_\mu>0$ such that if $|\mu_\FH(\q_-)-\mu_\FH(\q_+)|\leq \mu$ for  $\q_\pm\in\Crit\a_f$, then  $|\a_f(\q_-)-\a_f(\q_+)|\leq C_\mu$. This follows from that, for $\q=[q,\bar q]\in\Crit\a_f$, $\a_f(\q)=-\om([\bar q])-f(q)$ and  $-\frac{1}{4}\dim M\leq c_1^{TM}([\bar q])+\frac{1}{2}\mu_\FH(\q)\leq\frac{1}{4}\dim M$ by \eqref{eq:ind_M}.
\end{enumerate}
\end{rem}

By condition (H), \eqref{eq:om-energy}, and Proposition \ref{prop:j_HS}.(h), every element in $\widehat\NN(\q_-,\q_+,\mathfrak{a}_{f},j,N)$ with $\mu_\FH(\q_-)-\mu_\FH(\q_+)\leq 3$ is either $t$-independent or simple. Thus we can define the cap product with $-c_1^E$,
\begin{equation}\label{eq:cap_independent}
\Psi^{c_1^E}:\FH_*^{(a,b)}(f,j)\longrightarrow \FH_{*-2}^{(a,b)}(f,j)\,,
\end{equation}
using the chain map $\psi^{c_1^E}$ defined exactly in the same manner as in \eqref{eq:floer_cap_chain}. Then $\Psi^{c_1^E}$ coincides with the one in \eqref{eq:floer_cap} up to  continuation homomorphisms in \eqref{eq:isom}. Counting only $t$-independent elements also defines a homomorphism. To be precise, using
\[
\widehat\NN(q_-,q_+,f,g,N):=\big\{q\in \widehat\NN(q_-,q_+,f,g) \mid q(0)\in N\big\}\,,
\]
we define the homomorphism
\[
\psi_0^{c_1^E}(\q_-) := \sum_{\q_+}\# \widehat\NN(q_-,q_+,f,g,N)\q_+\,,
\]
where the sum rums over all $\q_+\in\Crit\a_{f}$ with $[\bar q_-]=[\bar q_+]$, $\mu_\FH(\q_-)-\mu_\FH(\q_+)=2$, and $\a_{f}(\q_+)\in(a,b)$. Then the induced map 
\[
\Psi_0^{c_1^E}: \FH_*^{(a,b)}(f,j)\longrightarrow \FH_{*-2}^{(a,b)}(f,j)
\]
corresponds to the ordinary cap product with $-c_1^E$ in singular homology via the isomorphism \eqref{eq:morse_isom}.

Let  $j\in\j_\mathrm{HS}(f)$ and $j'\in \mathfrak{j}_\mathrm{reg}(f)$. Again, due to Proposition \ref{prop:j_HS}, we can define the Floer homology $\FH^{(a,b)}_*(\tilde f,j)$ for the pair $(\tilde f,j)$ as in Section \ref{sec:quantum_gysin} and have an isomorphism
\[
\FH^{(a,b)}_*(\tilde f,j) \cong \FH^{(a,b)}_*(\tilde f,j')	\,,
\]
see Proposition \ref{prop:quantum_gysin_simple}.(a). This isomorphism is even as $\Lambda$-modules when $(a,b)=(-\infty,+\infty)$.

\begin{prop}\label{prop:quantum_gysin_simple}
We assume conditions \eqref{eq:assumption_M} and (H) from \eqref{eq:H}. Let $j\in\j_\mathrm{HS}(f)$. 
\begin{enumerate}[(a)]
\item  
For $-\infty\leq a<b\leq+\infty$, there exists an exact sequence 
\[
\cdots\longrightarrow \FH_{*}^{(a,b)}(\tilde f,j)\longrightarrow\FH_{*}^{(a,b)}(f,j)\stackrel{\Psi^{c_1^E}}{\longrightarrow}\FH_{*-2}^{(a,b)}(f,j)\to \FH_{*-1}^{(a,b)}(\tilde f,j)\longrightarrow\cdots\,,
\]
where $\Psi^{c_1^E}$ is defined in \eqref{eq:cap_independent}. In the case $(a,b)=(-\infty,+\infty)$, all maps are $\Lambda$-linear. Moreover this is isomorphic to the exact sequence in Proposition \ref{prop:quantum_gysin}.(a). In particular, we have an isomorphism $\FH^{(a,b)}_*(\tilde f,j) \cong \FH^{(a,b)}_*(\tilde f,j')$ for $j'\in \mathfrak{j}_\mathrm{reg}(f)$.
\item The connecting homomorphism $\Psi^{c_1^E}$ in (a) is given at the chain level by
\[
\psi^{c_1^E}(\q_-)= \psi_0^{c_1^E}(\q_-) - \sum_{\q_+} \sum_q \epsilon(q) c_1^E([q])\q_+\,.
\]
The first sum ranges over  $\q_+\in\Crit\mathfrak{a}_{f}$ with $\mu_\FH(\q_-)-\mu_\FH(\q_+)=2$ and $\mathfrak{a}_{f}(\q_+)\in(a,b)$. The second one is given by choosing one $q$ in each connected component of $\widehat\NN(\q_-,\q_+,\mathfrak{a}_{f},j,N)\setminus \widehat\NN(q_-,q_+,f,g,N)$ and determining the sign $\epsilon(q)\in\{-1,+1\}$ by requiring the orientation of $q$ to be equal to $\epsilon(q)\p_sq\wedge \p_tq$.
\item Suppose that $\psi^{c_1^E}=\psi_0^{c_1^E}$ (see Remark \ref{rem:only_Morse}). Then there exists a $\Lambda$-module isomorphism 
	\[
	\FH_*(\tilde f,j)  \cong  \bigoplus_{\diamond+\scriptscriptstyle\heartsuit=*+\frac{1}{2}\dim M} \H_{\diamond}(\Sigma;\Z)\otimes\Lambda_{\scriptscriptstyle\heartsuit}\,,
	\] 
	and the exact sequence in (a) with $(a,b)=(-\infty,+\infty)$ reduces to the classical Gysin sequence \eqref{eq:gysin_seq} for singular homology with $\Lambda$-coefficients.
\end{enumerate}
\end{prop}

\begin{proof}[Proof of (a) and (b)]
The long exact sequence follows from the fact that $\FH^{(a,b)}(\tilde f,j)$ is the homology of the mapping cone of $\psi^{c_1^E}$ as in Proposition \ref{prop:quantum_gysin}.(a). This exact sequence is isomorphic to the one in Proposition \ref{prop:quantum_gysin}.(a) since $\psi^{c_1^E}$ for $j\in\j_\mathrm{HS}(f)$ corresponds to that for $j'\in\j_\mathrm{reg}(f)$ via chain level continuation homomorphisms up to chain homotopy, cf.~the proof of Proposition \ref{prop:quantum_gysin}.(b). This establishes (a).

By definition  $\psi^{c_1^E}$ is given by the signed count of solutions in $\widehat\NN(\q_-,\q_+,\mathfrak{a}_{f},j,N)$ with $\mu_\FH(\q_-)-\mu_\FH(\q_+)=2$, and  $\psi_0^{c_1^E}$ counts only $t$-independent solutions. Thus it suffices to show that the contribution of $t$-dependent solutions $q\in\widehat\NN(\q_-,\q_+,\mathfrak{a}_{f},j,N)$ is exactly the second term on the right in the identity in (b). To see this, let $q$ be a $t$-dependent solution in $\widehat\NN(\q_-,\q_+,\mathfrak{a}_{f},j,N)$, and let  $\widehat\NN_q(\q_-,\q_+,\mathfrak{a}_{f},j)$ denote the connected component of $\widehat\NN(\q_-,\q_+,\mathfrak{a}_{f},j)$ containing $q$. Then due to the $\mathbb{R}\times S^1$-action in \eqref{eq:RxS1_action} together with (a) and (h) in Proposition \ref{prop:j_HS}, we have 
  \begin{equation}\label{eq:moduli_q}
    \widehat\NN_q(\q_-,\q_+,\mathfrak{a}_{f},j)=\big\{q(\cdot+s,\cdot+t) \mid  (s,t)\in\R\x S^1\big\}
  \end{equation}
  and it is a smooth manifold diffeomorphic to $\R\x S^1$. 
  Therefore, if we set  
  \[
  \widehat\NN_q(\q_-,\q_+,\mathfrak{a}_{f},j,N)=\widehat\NN_q(\q_-,\q_+,\mathfrak{a}_{f},j) \cap \widehat\NN(\q_-,\q_+,\mathfrak{a}_{f},j,N)\,,
  \]
 then
\begin{equation}\label{eq:NN=c_1^E}
	 \#\widehat\NN_q(\q_-,\q_+,\mathfrak{a}_{f},j,N)=\epsilon(q) (q\cdot N)= -\epsilon(q) c_1^E([q])
\end{equation}
where $q\cdot N\in\Z$ denotes the intersection number between $q$ and $N$ that equals $-c_1^E([q])$ since $[N]=\mathrm{PD}(-c_1^E)$. This completes the proof of (b).
\end{proof}

\begin{rem}\label{rem:only_Morse}
The assumption $\psi^{c_1^E}=\psi_0^{c_1^E}$ in Proposition \ref{prop:quantum_gysin_simple}.(c) is met for example if any of the following holds.
\begin{enumerate}[(i)]
	\item (A1) is assumed, i.e.~$\om$ vanishes over $\pi_2(M)$, see \eqref{eq:om-energy}.
	\item (A2) with $\lambda\nu\leq-\frac{1}{2}\dim M$ is satisfied, i.e.~$c_1^{TM}(A)\leq -\frac{1}{2}\dim M$ for every $A\in\pi_2(M)$ with $\omega(A)>0$, see Proposition \ref{prop:j_HS}.(a).
\end{enumerate}
\end{rem}

To prove part (c) of Proposition \ref{prop:quantum_gysin_simple}, we recall the construction of  Morse-Bott homology, see \cite[Appendix A]{Fra04} for details. We refer to Section \ref{sec:gradient} for the Morse-Bott setting in Rabinowitz Floer homology. As before, $h$ is a perfect Morse function on $\Crit \tilde f$ defined in \eqref{eq:ftn_h}, and $\check q$ and $\hat q$ denote the minimum and maximum points of $h$ on $S_q:=\wp^{-1}(q)$ for $q\in\Crit f$ respectively. If we do not want to specify whether it is a minimum or maximum point, we write $\bar q\in \Crit h$.
Using a connection one-form $\alpha$ of $\wp:\Sigma\to M$, we lift the metric $g=\om(\cdot,j\cdot)$ on $M$ to a metric $\tilde g$ on $\Sigma$. Let $\phi_{-\nabla h}^t$ be the negative gradient flow of $h$, and let $W^s(\bar q)$ and $W^u(\bar q)$ be the stable and unstable manifolds of $\phi_{-\nabla h}^t$ at $\bar q\in\Crit h$ respectively. We orient  $S_q$ and all one-dimensional stable and unstable manifolds by the fundamental vector field $R$ generated by the $S^1$-action on $\Sigma$. For $q_\pm\in\Crit f$, we denote by 
\begin{equation}\label{eq:moduli_N_lift}
\widehat\NN(S_{q_-},S_{q_+},\tilde f,\tilde g)	
\end{equation}
the moduli space of (positive) gradient flow lines $\tilde q:\R\to\Sigma$ of $\tilde f$ with respect to $\tilde g$ such that the evaluation maps at asymptotic ends satisfy
\[
\ev_\pm(\tilde q):= \lim_{s\to\pm\infty}\tilde q(s)\in S_{q_\pm}\,.
\]
For $q_-\neq q_+$, the $S^1$-action on $\Sigma$ induces a free $S^1$-action on this moduli space, and the quotient space is diffeomorphic to $\widehat\NN(q_-,q_+,f,g)$ through projection and horizontal lift. This gives an isomorphism 
\[
T_{\tilde q}\widehat\NN(S_{q_-},S_{q_+},\tilde f,\tilde g)	= \R\cdot R(\tilde q)\oplus T_q \widehat\NN(q_-,q_+,f,g) \,,
\]
and we orient $\widehat\NN(S_{q_-},S_{q_+},\tilde f,\tilde g)$ using $R$ and the orientation on $\widehat\NN(q_-,q_+,f,g)$ (in the order that the isomorphism is written). We then define $\widehat\NN^n(\bar q_-,\bar q_+,\tilde f,\tilde g)$ the space of flow lines with $n$ cascades from $\bar q_-\in\Crit h$ to $\bar q_+\in\Crit h$ for $n\in\N=\{1,2,\dots\}$, namely $n$-tuple $\tilde{\mathbf q}=(\tilde q^1,\dots,\tilde q^n)$ such that $\tilde q^i$ are nontrivial positive gradient flow lines of $\tilde f$ satisfying
\[
\ev_-(\tilde q^1)\in W^u(\bar q_-)\,,\quad \ev_+(\tilde q^n)\in W^s(\bar q_+)\,,\quad \phi_{-\nabla h}^{t_i}(\ev_+(\tilde q^i))=\ev_-(\tilde q^{i+1})
\] 
for some $t_i\geq 0$. We orient $\widehat\NN^n(\bar q_-,\bar q_+,\tilde f,\tilde g)$ according to the fibered sum rule explained in Remark \ref{rem:ori_rule} below. For a generic $h$, this space is a smooth manifold of dimension 
$\mu_{(-f,h)}(\bar q_-)-\mu_{(-f,h)}(\bar q_+)+n-1$ 
where the index for $\bar q\in\Crit h$ is defined by 
\[
\mu_{(-f,h)}(\bar q):=\mu_{-f}(q)+\mu_h(\bar q)\,.
\] 
There is a free $\R^n$-action shifting each $\tilde q^i$ in the $s$-direction, and we denote the quotient space by 
\[
\NN^n(\bar q_-,\bar q_+,\tilde f,\tilde g)=\widehat\NN^n(\bar q_-,\bar q_+,\tilde f,\tilde g)/\R^n\,.
\]
Let $\mu_{(-f,h)}(\bar q_-)-\mu_{(-f,h)}(\bar q_+)=1$. Then the above space is a finite set and we assign a sign $\epsilon(\tilde{\mathbf{q}})=\{-1,+1\}$ so that $\epsilon(\tilde{\mathbf{q}})(\p_s\tilde q^1,\dots,\p_s\tilde q^n)$ coincides with the orientation on $\widehat\NN^n(\bar q_-,\bar q_+,\tilde f,\tilde g)$.  
We set 
\[
\NN(\bar q_-,\bar q_+,\tilde f,\tilde g) :=\bigcup_{n\in\N} \NN^n(\bar q_-,\bar q_+,\tilde f,\tilde g)
\] 
and denote the signed cardinality
\begin{equation}\label{eq:MB_count}
\#\NN(\bar q_-,\bar q_+,\tilde f,\tilde g)\in\Z\,.	
\end{equation}
In general, one also has to take the case of zero cascade into account but in our situation this corresponds to flow lines of $-\nabla h$ on $\Crit \tilde f$ whose contributions cancel out since $h$ is perfect.

By definition, the chain module for the Morse-Bott homology of $(-\tilde f,h)$ is a $\Lambda$-module generated by critical points $\bar q\in\Crit h$ and graded by $\mu_{(-f,h)}(\bar q)$. The boundary operator is given by the signed count in \eqref{eq:MB_count} for $\bar q_-$ and $\bar q_+$ with $\mu_{(-f,h)}$-index difference one. There is a degree-preserving isomorphism between the Morse-Bott homology and the singular homology $\H(\Sigma)\otimes\Lambda$, see e.g.~\cite{Fra04,BH13}.

The following lemma will be used in the proof of Proposition \ref{prop:quantum_gysin_simple}.(c).

\begin{lem}\label{lem:moduli_1-1}
Let $q_\pm\in\Crit f$ with $\mu_{-f}(q_-)-\mu_{-f}(q_+)=1$. Then we have 
\[
\#\NN(q_-,q_+,f,g)=\#\NN(\check q_-,\check q_+,\tilde f,\tilde g)=-\#\NN(\hat q_-,\hat q_+,\tilde f,\tilde g)\,.
\]
\end{lem}
\begin{proof}
	We only prove the first equality. The second follows analogously, where the minus sign is due to the orientation rule explained in Remark \ref{rem:ori_rule}. There is a bijection from $\widehat\NN(q_-,q_+,f,g)$ to $\widehat\NN(\check q_-,\check q_+,\tilde f,\tilde g)$ mapping $q$ to its horizontal lift $\tilde q$ with $\ev_-(\tilde q)=\check q_-$. We denote the natural inclusions $W^s(\check q_+)\into  S_{q_+}$ and $W^u(\check q_-)\into S_{q_-}$ by $\iota$.  Viewing $\tilde q=(\check q_-,\tilde q,\ev_+(\tilde q))$ as an element in 
\[
\widehat\NN(\check q_-,\check q_+)=W^u(\check q_-) _{\,\,\,\iota}  \!\x_{\ev_-} \widehat\NN(S_{q_-},S_{q_+}) _{\,\,\,\ev_+}  \!\!\x_{\,\iota} W^s(\check q_+)\,,
\]
we compute its sign. We omit $\tilde f$ and $\tilde g$ from the notation here and below. We first study the second fiber product of which the tangent space is 
\[ 
T_{(\tilde q,\ev_+(\tilde q))}\big(\widehat\NN(S_{q_-},S_{q_+}) _{\,\,\,\ev_+}  \!\!\x_{\,\iota} W^s(\check q_+)\big)=\ker (d_{\tilde q}\ev_+- d_{\ev_+(\tilde q)}\iota)\,.
\]
We write $\mathrm{id}$ for $d_{\ev_+(\tilde q)}\iota$, the identity map on $T_{\ev_+(\tilde q)}S_{q_+}=T_{\ev_+(\tilde q)}W^s(\check q_+)$. We orient $\ker (d_{\tilde q}\ev_+- \mathrm{id})$ by requiring that the isomorphism 
\begin{equation}\label{eq:sign_isom}
T_{\tilde q}\widehat\NN(S_{q_-},S_{q_+})\cong \ker (d_{\tilde q}\ev_+- \mathrm{id})\,,\qquad \xi\mapsto (\xi, d_{\tilde q}\ev_+\xi)
\end{equation}
is orientation preserving. Recall that $T_{\tilde q}\widehat\NN(S_{q_-},S_{q_+})$ is oriented by $\epsilon(q)R(\tilde q)\wedge \p_s\tilde q$. One can readily see that this obeys the fibered sum rule in Remark \ref{rem:ori_rule}, namely the isomorphism  
\[
\ker(d_{\tilde q}\ev_+- \mathrm{id}) \oplus \left(\frac{T_{(\tilde q,\ev_+(\tilde q))}\big(\widehat\NN(S_{q_-},S_{q_+})\x W^s(\check q_+)\big)}{\ker(d_{\tilde q}\ev_+- \mathrm{id})}\right) \cong T_{(\tilde q,\ev_+(\tilde q))}\big(\widehat\NN(S_{q_-},S_{q_+})\x W^s(\check q_+)\big)
\]
defines an orientation on the quotient space for which the isomorphism 
\[
\frac{T_{(\tilde q,\ev_+(\tilde q))}\big(\widehat\NN(S_{q_-},S_{q_+})\x W^s(\check q_+)\big)}{\ker(d_{\tilde q}\ev_+- \mathrm{id})} 
\cong T_{\ev_+(\tilde q)}S_{q_+}
\]
given by $d_{\tilde q}\ev_+- \mathrm{id}$ is orientation reversing, as we wished since $\dim W^s(\check q_+)=\dim S_{q_+}=1$.

Next we consider the whole fiber product, 
\[
\begin{split}
T_{\tilde q}\widehat\NN(\check q_-,\check q_+)&=\ker\Big(d_{\check q_-}\iota-d_{\tilde q}\ev_-|_{\ker(d_{\tilde q}\ev_+- \mathrm{id})}:T_{\check q_-}W^u(\check q_-)\x \ker(d_{\tilde q}\ev_+- \mathrm{id}) \to T_{\check q_-}S_{q_-}\Big)\\
&=\ker\Big(-d_{\tilde q}\ev_-|_{\ker(d_{\tilde q}\ev_+- \mathrm{id})}:\ker(d_{\tilde q}\ev_+- \mathrm{id}) \to T_{\check q_-}S_{q_-}\Big)\\
&=\ker \Big(-d_{\tilde q}\ev_-:T_{\tilde q}\widehat\NN(S_{q_-},S_{q_+})\to T_{\check q_-}S_{q_-}\Big)\,.
\end{split}
\]
where the second equality follows from $T_{\check q_-}W^u(\check q_-)=\{0\}$, and the last equality from the orientation preserving isomorphism in \eqref{eq:sign_isom}. The lemma follows if we show that this 1-dimensional space is oriented by $\epsilon(q)\p_s\tilde q$ according to the fibered sum rule. To see this, we assume that it is indeed oriented by $\epsilon(q)\p_s\tilde q$. Then the isomorphism 
\[
T_{\tilde q}\widehat\NN(\check q_-,\check q_+) \oplus \left(\frac{T_{\tilde q}\widehat\NN(S_{q_-},S_{q_+})}{T_{\tilde q}\widehat\NN(\check q_-,\check q_+)}\right) \cong T_{\tilde q}\widehat\NN(S_{q_-},S_{q_+})
\]
induces an orientation $-R(\tilde q)+T_{\tilde q}\widehat\NN(\check q_-,\check q_+)$ on $\left(\frac{T_{\tilde q}\widehat\NN(S_{q_-},S_{q_+})}{T_{\tilde q}\widehat\NN(\check q_-,\check q_+)}\right)$. Hence the isomorphism 
\[
\left(\frac{T_{\tilde q}\widehat\NN(S_{q_-},S_{q_+})}{T_{\tilde q}\widehat\NN(\check q_-,\check q_+)}\right)\cong T_{\check q_-}S_{q_-}
\]
induced by $-d_{\tilde q}\ev_-$ is orientation preserving. Since this is precisely the way stipulated by the fibered sum rule as $\dim \widehat\NN(S_{q_-},S_{q_+})=2$ and $\dim S_{q_-}=1$, this finishes the proof.
\end{proof}

\begin{proof}[Proof of Proposition \ref{prop:quantum_gysin_simple}.(c)]
Following \cite[Lemma 2.5]{GG19}, we prove that in this case $\FH(\tilde f,j)$ is isomorphic to the Morse-Bott homology of $(-\tilde f,h)$ with coefficients in $\Lambda$, which in turn is isomorphic to $\H(\Sigma)\otimes\Lambda$. To this end, we choose specific $h$ and $\tilde g$ as follows.

Let $U\subset M$ be a sufficiently small tubular neighborhood of our generic codimension 2 submanifold $N$ such that it does not intersect critical points of $f$ and gradient flow lines of $f$ connecting  critical points of index difference 1.  The bundle $\wp:\Sigma\to M$ restricted to $M\setminus U$ is trivial, i.e.~$\Sigma|_{M\setminus U}\cong (M\setminus U)\x S^1$. We choose a perfect Morse function $\bar h:S^1\to\R$ and set $h:(M\setminus U)\x S^1\to\R$ by $h(\cdot,\theta)=\bar h(\theta)$, where $\theta$ is the coordinate on $S^1$. We pullback this function to $\Sigma|_{M\setminus U}$ and restrict to $\Crit\tilde f$. This gives a perfect Morse function on $\Crit\tilde f$, denoted again by $h:\Crit\tilde f\to\R$. We choose a connection 1-form $\alpha$ on $\Sigma\to M$ which is $d\theta$ on $\Sigma|_{M\setminus U}\cong (M\setminus U)\x S^1$. In particular, the curvature form  $\kappa$ induced by $\alpha$ is supported on $U$. Using $\alpha$, we lift a metric $g$ on $M$ to a metric $\tilde g$ on $\Sigma$. 

The chain module $\FC(\tilde f,j)$ agrees with the chain module for the Morse-Bott homology of $(-\tilde f,h)$, see Remark \ref{rem:H}.(ii) and also the grading convention in \eqref{eq:ind_M}. We will show that the respective boundary operators are also identical. We choose $\bar q_\pm\in\Crit h$ with $\mu_{(-f,h)}(\bar q_-)-\mu_{(-f,h)}(\bar q_+)=1$. Then $\mu_{-f}(q_-)-\mu_{-f}(q_+)\leq 2$, and $\widehat\NN^m(\bar q_-,\bar q_+,\tilde f,\tilde g)$ for $m\geq 3$ is empty since cascades project to nontrivial gradient flow lines of $f$. As mentioned above, the contributions of the case $\mu_{-f}(q_-)-\mu_{-f}(q_+)=0$ add up to zero.  Hence the following three cases remain: 
\begin{enumerate}[(i)]
\item $(\bar q_-,\bar q_+)=(\check q_-,\check q_+)$ and $\mu_{-f}(q_-)-\mu_{-f}(q_+)=1$,
\item $(\bar q_-,\bar q_+)=(\hat q_-,\hat q_+)$ and $\mu_{-f}(q_-)-\mu_{-f}(q_+)=1$,
\item $(\bar q_-,\bar q_+)=(\check q_-,\hat q_+)$ and $\mu_{-f}(q_-)-\mu_{-f}(q_+)=2$.
\end{enumerate}
By Proposition \ref{prop:j_HS}.(d) and the hypothesis, the boundary operator $\p^\wp=\check \p+\hat \p+\p^{c_1^E}$ for $\FH(\Sigma)$ counts only gradient flow lines of $f$. We claim that (i), (ii), and (iii) correspond to  $\check \p$, $\hat \p$, and $\p^{c_1^E}$ respectively. This is indeed true for the first two cases by Lemma \ref{lem:moduli_1-1}.  

We now study case (iii). We first observe that $\widehat\NN^2(\check q_-,\hat q_+,\tilde f,\tilde g)$ is empty. Indeed if it has an element $(\tilde q^1,\tilde q^2)$, then $\ev_-(\tilde q^1)=\check q_-$ and our choice of $\tilde g$ and $h$ yield that $\ev_+(\tilde q^1)$, $\ev_-(\tilde q^2)$, and $\ev_+(\tilde q^2)$ are minimum points of $h$. This contradicts $\ev_+(\tilde q^2)=\hat q_+$. Therefore, it remains to show 
\begin{equation}\label{eq:equality}
\#\widehat\NN(q_-,q_+,f,g,N)=\#\widehat\NN^1(\check q_-,\hat q_+,\tilde f,\tilde g)	\,.
\end{equation} 
We denote by $\widehat\NN_q(q_-,q_+,f,g)$ the connected component of $\widehat\NN(q_-,q_+,f,g)$ containing $q$ and will show \eqref{eq:equality} for each connected component, namely the intersection number of $\ev_{(0,0)}:\widehat\NN_q(q_-,q_+,f,g)\to M$ and $N$, which is equal to $\int_{\widehat\NN_q(q_-,q_+,f,g)}\ev_{(0,0)}^*\kappa$, coincides with the signed count of all $\tilde q'\in\widehat\NN^1(\check q_-,\hat q_+,\tilde f,\tilde g)$ satisfying $\wp\circ\tilde q'\in \widehat\NN_q(q_-,q_+,f,g)$. 

The space $\NN_q(q_-,q_+,f,g):=\widehat\NN_q(q_-,q_+,f,g)/\R$ is diffeomorphic either to $S^1$ or to the open interval $(0,1)$, and we parametrize elements as $[q_r]$ for $r\in S^1$ or $r\in(0,1)$. In particular $q_r$ is a gradient flow line in $M$ from $q_-$ to $q_+$. We first consider the case $S^1$. For every $q_r$, we consider its horizontal lift $\tilde q_r$ with $\ev_-(\tilde q_r)=\check q_-$. Therefore $\tilde q'\in\widehat\NN^1(\check q_-,\hat q_+,\tilde f,\tilde g)$ satisfies $\wp\circ\tilde q'\in \widehat\NN_q(q_-,q_+,f,g)$ if and only if it equals some $\tilde q_r$ with $\ev_+(\tilde q_r)=\hat q_+$. From this, together with a sign computation for elements of $\widehat\NN^1(\check q_-,\hat q_+,\tilde f,\tilde g)$ as in the proof of Lemma \ref{lem:moduli_1-1}, we conclude that the signed count of all $\tilde q'\in\widehat\NN^1(\check q_-,\hat q_+,\tilde f,\tilde g)$ satisfying $\wp\circ\tilde q'\in \widehat\NN_q(q_-,q_+,f,g)$ equals the degree of the map 
\[
S^1\to S_{q_+}\cong S^1\,,\qquad r\mapsto \ev_+(\tilde q_r)\,.	
\]
To determine this degree we form loops $\ell_r$ in $M$ by concatenating $q_0$ and $q_r$ at $q_\pm$, and orient $\ell_r$ using the fixed orientation on  $\widehat{\NN}_q(q_-,q_+,f,g)$ and the evaluation map $\ev_{(0,0)}$. We denote by $\theta_r\in S^1$ the holonomy of $(\Sigma,\alpha)$ along $\ell_r$, i.e.~$\theta_r=-\int_{\ell_r}\alpha$ in $S^1=\R/\Z$. Then the degree of $r\mapsto \ev_+(\tilde q_r)$ equals the degree of $r\mapsto -\theta_r$. Finally the latter degree agrees with $\int_{\widehat\NN_q(q_-,q_+,f,g)}\ev_{(0,0)}^*\kappa$ since $\kappa$ is the curvature form of $(\Sigma,\alpha)$. This proves the assertion in the case of $\NN_q(q_-,q_+,f,g)\cong S^1$. 

In the case that $\NN_q(q_-,q_+,f,g)$ is diffeomorphic to $(0,1)$,  the map 
\[
(0,1) \to S_{q_+}\cong S^1\,,\qquad r\mapsto \ev_+(\tilde q_r)
\] 
again forms a loop since $\ev_+(\tilde q_r)=\check q_+$ for $r$ close to $0$ or $1$ due to the support of $\kappa$. Now the assertion follows as in the previous case. This proves \eqref{eq:equality}, and the proof of (c) is complete. 
\end{proof}

\begin{rem}\label{rem:ori_rule}
	In this paper, we use the following rules of defining orientations on quotient spaces and fibered sum spaces as in \cite[Section 4.4]{BO09} and \cite[Section 7]{DL19}.
	
	Let $V\subset W$ be oriented vector spaces. Then we orient the quotient space $W/V$ so that the isomorphism 
	\[
	V\oplus (W/V)\stackrel{\cong}{\longrightarrow} W
	\]
	given by the short exact sequence $0\to V\to W \to W/V \to 0$ is orientation preserving, see also \cite[Lemma 18]{FH93}. Let $f_i:W_i\to W$, $i=1,2$ be linear maps between oriented vector spaces such that the map $f_1-f_2:W_1\oplus W_2\to W$ is surjective. We orient ${W_1}_{\,f_1}\!\oplus_{f_2}W_2:=\ker (f_1-f_2)$ by the convention that the isomorphism
	\[
	\frac{W_1\oplus W_2}{{W_1}_{\,f_1}\!\!\oplus_{f_2}W_2}\stackrel{\cong}{\longrightarrow} W
	\]
	given by $f_1-f_2$ 
	becomes orientation preserving if $(-1)^{\dim W_2\cdot \dim W}=1$, and orientation reversing otherwise.
\end{rem}

\subsection{Nonvanishing result}\label{sec:nonvanishing}

In this section, let $j\in\mathfrak{j}_\mathrm{reg}(f)$ for a $C^2$-small Morse function $f$ on $M$. A nonzero cohomology class $\mu\in\H^*(M;\Z)$ is said to be primitive if there does not exist $\mu_0\in\H^*(M;\Z)$ satisfying $m\mu_0=\mu$ for some integer $m\neq \pm 1$. It is crucial for results in this section that we use Floer homology with integer coefficients, as opposed to field coefficients.
\begin{prop}\label{prop:nonvanishing_cap}
Suppose that $c_1^E\in\H^2(M;\Z)$ is not a primitive class. Let $\kappa\in\Z$ be such that $\H_{\kappa+\frac{1}{2}\dim M}(M;\Z)$ has nonzero rank (e.g.~$\kappa=\pm\frac{1}{2}\dim M$). Then for any $a,b\in(\Z+\frac{1}{2})\cup\{-\infty,+\infty\}$ with $a<0<b$, the Floer cap product with $-c_1^E$,
\[
\Psi^{c_1^E}:\FH_{\kappa+2}^{(a,b)}(f,j)  \longrightarrow \FH_{\kappa}^{(a,b)}(f,j)\,,
\] 
 is not surjective.
\end{prop}
\begin{proof}
The hypothesis implies that there exist an integer $m\geq2$ and $\mu\in \H^2(M;\Z)$ such that $c_1^E=m \mu$. Then the map $\Psi^{c_1^E}$ factors as follows:
\[
\begin{tikzcd}[row sep=2.5em]
\FH_{\kappa+2}^{(a,b)}(f,j) \arrow{dr}{\Psi^{\mu}} \arrow{rr}{\Psi^{c_1^E} } && \FH_{\kappa}^{(a,b)}(f,j)\\
 & \FH_{\kappa}^{(a,b)}(f,j) \arrow{ur}{\x m} 
\end{tikzcd}
\]
where $\Psi^{\mu}$ denotes the Floer cap product with $-\mu$, and $\x m$ refers to the scalar multiplication by $m$ with respect to the $\Z$-module structure on $\FH^{(a,b)}_\kappa(f,j)$. By \eqref{eq:morse_isom} and \eqref{eq:isom}, we have 
\[
	\FH_{\kappa}^{(a,b)}(f,j)\cong  \bigoplus_{\diamond+\star=\kappa+\frac{1}{2}\dim M}\H_\diamond(M;\Z)\otimes\Lambda_\star^{(a,b)}\,.
\]
Since $0\in(a,b)$, $\Lambda^{(a,b)}_0$ contains $\{aT^0\mid a\in\Z,\,0\in\Gamma_M\}\cong\Z$. Thus  $\FH_{\kappa}^{(a,b)}(f,j)$ has a torsion-free factor by the hypothesis on $\kappa\in\Z$. Therefore the multiplication $\x m$ on this is not surjective. 
This proves that $\Psi^{c_1^E}$ is not surjective as claimed.
\end{proof}

Let as before $\Sigma\to M$ be a principal $S^1$-bundle with Euler class $c_1^E$. If $c_1^E=m\mu$, as in the previous proposition, then there exists a principal $S^1$-bundle $\widetilde\Sigma\to M$	with Euler class $\mu$ and $\widetilde\Sigma$ is an $m$-fold covering of $\Sigma$, see Section \ref{sec:cyclic}.

\begin{prop}\label{prop:nonzero_class}
Suppose that the Floer cap product $\Psi^{c_1^E}$ on $\FH_*(M)$ is not an isomorphism. Then $\FH_*(\Sigma)$ is nonzero for some $*\in\Z$.
\end{prop}
\begin{proof}
	If $\FH_{*}(\Sigma)=0$ for all $*\in\Z$, then the long exact sequence in Proposition \ref{prop:quantum_gysin}.(a) yields that $\Psi^{c_1^E}$ is an isomorphism. 
\end{proof}

The preceding propositions imply that if $c_1^E$ is not primitive, then $\FH_*(\Sigma)$ is nonzero. 
In fact, it is possible to see which generators represent nonzero homology classes in $\FH_*(\Sigma)$.
In the following lemma, for simplicity, we choose $f$ to have unique local maximum and minimum points. 
\begin{lem}\label{lem:spec_finite}
Let $f$ have a unique local maximum point $q_\mathrm{max}$ and a unique local minimum point $q_\mathrm{min}$. Assume that $c_1^E$ is not a primitive class. Let $a,b\in(\Z+\frac{1}{2})\cup\{-\infty,+\infty\}$. Then for any $A\in\Gamma_M$ with $\om(A)\in(a,b)$, 
\[
[(\hat q_\mathrm{max},A)] \in \FH_{\frac{1}{2}\dim M+1-2c_1^{TM}(A)}^{(a,b)}(\tilde f,j) \,,\qquad [(\hat q_\mathrm{min},A)] \in \FH_{-\frac{1}{2}\dim M+1-2c_1^{TM}(A)}^{(a,b)}(\tilde f,j)
\]
are nonzero homology classes.

\end{lem}
\begin{proof}
	We recall from \eqref{eq:bdry} that $\p^\wp=\hat \p + \check\p +\p^{c_1^E}$. We claim:
	\[
	(\hat \p + \check\p)(\hat q_\mathrm{max},A)=0\,,\quad (\hat \p + \check\p)(\hat q_\mathrm{min},A)=0\,,\quad \p^{c_1^E}(\hat q_\mathrm{max},A)=0\,,\quad \p^{c_1^E}(\hat q_\mathrm{min},A)=0\,.
	\]
	Indeed the first two equalities are due to the fact that $q_\mathrm{max}$ and $q_\mathrm{min}$ are unique local minimum and maximum points of $f$ respectively and the latter two follow from the definition of $\p^{c_1^E}$. Therefore $\p^\wp(\hat
	q_\mathrm{max},A)=0$ and $\p^\wp(\hat
	q_\mathrm{min},A)=0$. 
	
	Since $\hat q_\mathrm{max}$ is a maximum point of $h$, the image of $\hat \p + \check\p$ does not have any term of $(\hat q_\mathrm{max},A)$. Since $c_1^E$ is not primitive, the image of $\p^{c_1^E}$ can only have a term of $(\hat q_\mathrm{max},A)$ multiplied by a scalar in $\Z\setminus\{-1,+1\}$, as in the proof of Proposition  \ref{prop:nonvanishing_cap}. This shows that $(\hat q_\mathrm{max},A)$ generates a nonzero class in $\FH^{(a,b)}(\tilde f,j)$. 
		The case for $\hat q_{\min}$ follows similarly.
\end{proof}

\begin{prop}\label{prop:CPn}
	Let $M=\CP^n$ with the Fubini-Study form $\om_\mathrm{FS}$. Let $\Sigma\to M$ be a principal $S^1$-bundle with Euler class $-m[\om_\mathrm{FS}]\in\H^2(\CP^n;\Z)$ for $m\in\N$. Then we have 
\[
\FH_*(\Sigma) \cong \left\{
\begin{aligned} 
\Z_m  & \quad *\in2\Z+1\,, \\[1ex]
0\quad & \quad *\in 2\Z\,.
\end{aligned}
\right.
\]
\end{prop}
\begin{proof}
	This follows immediately from the exact sequence in Proposition \ref{prop:quantum_gysin}.(a) and the well-known computation 
\[
\FH_*(\CP^n) \cong \left\{
\begin{aligned} 
0  & \quad *\in2\Z+1\,, \\[1ex]
\Z  & \quad *\in 2\Z\,,
\end{aligned}
\right.
\qquad 
\Psi^{c_1^E}=\x (\pm m):\FH_*(\CP^n) \to \FH_{*-2}(\CP^n)\;\; \forall *\in \Z\,.
\]
The latter means that $\Psi^{c_1^E}$ equals the scalar multiplication by $\pm m$, where the sign is determined by choice of isomorphisms $\FH_*(\CP^n)\cong\Z$ for $*\in 2\Z$. This is due to the fact that the Floer cap product with $[\om_\mathrm{FS}]$ is an isomorphism.
\end{proof}

\subsection{Transfer homomorphism}\label{sec:transfer}
In this section, we write $\wp:\Sigma^m\to M$ for a principal $S^1$-bundle with Euler class $-m[\om]$ and $\tilde f^m:=f\circ\wp:\Sigma^m\to\R$ in order to indicate the degree of the bundle. 
As mentioned in \eqref{eq:cone}, the chain complex $(\FC(\tilde f^m),\partial^\wp)$ is identical to the mapping cone of $\psi^{-m[\om]}$, i.e. 
\begin{equation}\label{eq:mapping_cone}
\FC(\tilde f^m)= \FC(f)[-1]\oplus\FC(f)\,,\qquad \partial^\wp = \begin{pmatrix}
 -\p &  \psi^{-m[\om]}\\
0 & \p
\end{pmatrix} 	
\end{equation}
where generators $\check\q$ and $\hat\q$ of $\FC(\tilde f^m)$ correspond to generators $\q$ in $\FC(f)$ and $\FC(f)[-1]$ respectively. Here $\FC(f)[-1]$ is $\FC(f)$ with grading shifted by $-1$. 
Through the above identification, we define the homomorphisms
\[
\begin{split}
&T:\FC(\tilde f^m) \to \FC(\tilde f^1)\,,\qquad  (\q_1,\q_2)\mapsto  (\q_1,m\q_2)\,,	\\[0.5ex]
&P:\FC(\tilde f^1) \to \FC(\tilde f^m)\,,\qquad (\q_1,\q_2)\mapsto  (m\q_1,\q_2)\,.
\end{split}
\]	
These are chain maps thanks to $\psi^{-m[\om]}=m\psi^{-[\om]}$, where the latter  is the composition of $\psi^{-[\om]}$ with the scalar multiplication by $m$. 
Thus the compositions $P\circ T$ and $T\circ P$ are the scalar multiplication by $m$. At the homology level, we obtain
\begin{equation}\label{eq:TP=PT=m}
\begin{split}
&\FH_*(\Sigma^m)\stackrel{T}{\longrightarrow}\FH_*(\Sigma^1)\stackrel{P}{\longrightarrow} \FH_*(\Sigma^m)\,,\qquad (P\circ T)(Z)= mZ\,, \\
&\FH_*(\Sigma^1)\,\stackrel{P}{\longrightarrow}\FH_*(\Sigma^m)\stackrel{T}{\longrightarrow}\, \FH_*(\Sigma^1)\,,\qquad (T\circ P)(Z)= mZ\,.
\end{split}	
\end{equation}
This immediately implies the following proposition.

\begin{prop}\label{prop:transfer}
	We fix $\kappa\in\Z$ and assume that $\FH_{\kappa}(\Sigma^1)=0$. Then, for any $m\in\N$, $\FH_{\kappa}(\Sigma^m)$ only contains torsion classes of order $m$. Moreover, $\FH_{\kappa}(\Sigma^m)$ is torsion for some $m\in\N$ if and only if it is torsion for every $m\in\N$. In particular, if we use coefficients in a field instead of integers $\Z$, then $\FH_{\kappa}^{\w_0}(\Sigma^m)=0$ for some $m\in\N$ if and only if $\FH_{\kappa}(\Sigma^m)=0$ for all  $m\in\N$.
\end{prop}
\begin{proof}
	The first assertion follows from the first line of \eqref{eq:TP=PT=m}. 
	Suppose that $\FH_{\kappa}(\Sigma^m)$ has only torsion classes for some $m\in\N$. By the second line of \eqref{eq:TP=PT=m}, $\FH_{\kappa}(\Sigma^1)$ cannot have free part. Now we apply the first line of \eqref{eq:TP=PT=m} to arbitrary $m\in\N$ and conclude that $\FH_{\kappa}(\Sigma^m)$ for any $m\in\N$ has only torsion part. 
\end{proof}
We recall from Proposition \ref{prop:nonzero_class} that $\FH_*(\Sigma^m)$ for $m\geq 2$ is nonzero for some $*\in\Z$.

We now provide a geometric interpretation of  the maps $T$ and $P$ in the setting of Proposition \ref{prop:quantum_gysin_simple}.(c), namely the situation that $(\FC(\tilde f^m),\p^\wp)$ is simply the Morse-Bott chain complex of $(-\tilde f^m,h^m)$, where $h^m:\Crit\tilde f^m\to\R$ is a perfect Morse function.  We refer to Section \ref{sec:transfer_revisited} for an account in the general case. To this end, we view $\Sigma^m$ as the quotient space of $\Sigma^1$ by the action of $\Z_m$, a cyclic subgroup of $S^1$, see Section \ref{sec:cyclic}  for further discussion. We denote the quotient map by
\[
\wp^m:\Sigma^1 \longrightarrow \Sigma^m\,.
\]
It sends $\Crit \tilde f^1$ to $\Crit\tilde f^m$. The function $\tilde h^m:=h^m\circ\wp^m:\Crit \tilde f^1\to\R$ is Morse but not perfect for $m\geq2$. Nevertheless the Morse-Bott chain complex $(\FC(\tilde f^1,\tilde h^m),\p^\wp)$ is defined in the same manner as before with the chain module 
\[
\FC(\tilde f^1,\tilde h^m)=\big(\FC(f)[-1]\oplus\FC(f)\big)^{\oplus m}
\]
  although in this case we have to include the case of zero cascades, i.e.~negative gradient flow lines of $\tilde h^m$ connecting critical points of $\tilde h^m$, as well in the definition of boundary operator. Moreover this chain complex is quasi-isomorphic to $(\FC(\tilde f^1, h^1),\p^\wp)$. Here and below, we indicate our choice of a Morse function on $\Crit\tilde f^1$ in $\FC$.  
 To see a specific quasi-isomorphism, on a connected component of $\Crit\tilde f^1$, we denote by $\mathfrak{M}_{h^1}$ and $\mathfrak{m}_{h^1}$ unique maximum and minimum points of $h^1$ respectively and by $\mathfrak{M}^1_{\tilde{h}^m},\dots, \mathfrak{M}^m_{\tilde h^m}$ and $\mathfrak{m}^1_{\tilde h^m},\dots,\mathfrak{m}^m_{\tilde h^m}$ the $m$ maximum and $m$ minimum points of $\tilde h^m$ respectively. We choose a homotopy from $h^1$ to $\tilde h^m$ on this component such that the following are all homotopy gradient flow lines:
\[
 \mathfrak m_{h^1}\mapsto \mathfrak m^1_{\tilde h^m}\,,  \qquad \mathfrak{M}_{h^1}\mapsto \mathfrak{M}^i_{\tilde h^m}\;\;\;\forall  i\in\{1,\dots, m\}\,.
\]
Here $a\mapsto b$ means that there is a homotopy gradient flow line from $a$ to $b$. This is true up to reordering $\{1,\dots,m\}$ since the continuation homomorphism induced by a homotopy is an isomorphism.  Conversely, we choose a homotopy from $\tilde h^m$ to $h^1$ on this component whose homotopy gradient flow lines are given by
\[
\mathfrak m^i_{\tilde h^m}\mapsto \mathfrak m_{h^1} \;\;\;\forall  i\in\{1,\dots, m\}\,,\qquad \mathfrak{M}^1_{\tilde h^m}\mapsto \mathfrak{M}_{h^1}\,.
\]
These homotopies between $h^1$ and $\tilde h^m$ give rise to continuation homomorphisms
\[
\begin{split}
&A:\FC(\tilde f^1,h^1) \to \FC(\tilde f^1,\tilde h^m)\,,\qquad (\q_1,\q_2)\mapsto \big((\q_1,\q_2),(\q_1,0),\dots,(\q_1,0)\big)	 \\[.5ex]
&B: \FC(\tilde f^1,\tilde h^m) \to \FC(\tilde f^1,h^1)\,,\qquad \big((\q_1,\q_2),\dots,(\q_{2m-1},\q_{2m})\big)\mapsto \Big( \q_1,\sum_{j=1}^m\q_{2j}\Big)
\end{split}
\]
that are homotopy inverse to each other. Now we observe that the  natural transfer and projection maps
\[
\begin{split}
&T':\FC(\tilde f^m,h^m)\to \FC(\tilde f^1,\tilde h^m)\,,\qquad (\q_1,\q_2)\mapsto\big((\q_1,\q_2),(\q_1,\q_2),\dots,(\q_1,\q_2)\big) \\[.5ex]
&P': \FC(\tilde f^1,\tilde h^m)\to \FC(\tilde f^m,h^m)\,,\qquad \big((\q_1,\q_2),\dots,(\q_{2m-1},\q_{2m})\big)\mapsto \Big(\sum_{j=1}^m \q_{2j-1},\sum_{j=1}^m \q_{2j}\Big)	
\end{split}
\]
respectively satisfy $T=B\circ T'$ and $P= P'\circ A$.

\section{Negative line bundles and winding number}

\subsection{Negative line bundles}\label{sec:line}
Let $E$ be a complex line bundle over an integral symplectic manifold $(M,\omega)$, 
\[
\wp:E\longrightarrow M
\] 
with first Chern class $c_1^{E}=-m[\om]$ for some $m\in\N=\{1,2,\dots\}$. We denote by $\OO_E$ the zero section of $\wp$.
We choose a Hermitian metric $\langle \cdot,\cdot\rangle$ on $E$ and denote by $r$ the induced fiberwise radial coordinate. We also pick a Hermitian connection and split the tangent bundle $TE$ of $E$ into the vertical subbundle $T^\mathrm{v}E$ and the horizontal subbundle $T^\mathrm{h}E$
\begin{equation}\label{eq:splitting_TE}
T_xE\cong T_x^\mathrm{v}E\oplus T_x^\mathrm{h}E \cong E_{\wp(x)} \oplus T_{\wp(x)}M\,,\qquad x\in E\,,	
\end{equation}
where $T_x^\mathrm{v}E$ is canonically isomorphic to $E_{\wp(x)}$ and $T_x^\mathrm{h}E$ is isomorphic to $T_{\wp(x)}M$ via $d\wp$. Then 
\begin{equation}\label{eq:alpha_E}
	\alpha_x(v):=\frac{1}{2\pi r^2(x)}\langle ix,v^\mathrm{v} \rangle\,,\qquad v=v^\mathrm{v} + v^\mathrm{h} \in  T_x^\mathrm{v}E\oplus T_x^\mathrm{h}E\,,
\end{equation}
where $i$ denotes the complex structure of $\wp:E\to M$, 
is a primitive 1-form of $\wp^*(m\omega)$ on $E\setminus\OO_E$ satisfying $\alpha_x(x)=0$ and $\alpha_x(2\pi i x)=1$ for $x,ix\in E_{\wp(x)}\cong T^\mathrm{v}_xE$, see e.g.~\cite[Section 3.3]{Oan08} or \cite[Section 7.2]{Rit14}. We denote by $R$ the fundamental vector field generated by the $U(1)$-action on $E$, i.e.
\begin{equation}\label{eq:R=2pi_i}
R(x)=2\pi i x\qquad \forall x\in E\,.
\end{equation}
The total space $E$ is endowed with the symplectic form
\[
\Omega=\wp^*\om+d(\pi r^2\alpha)=(1+m\pi r^2)\wp^*\om+2\pi r dr\wedge\alpha\,.
\]
Away from the zero section, the symplectic form is exact:
\[
\Omega=d\bigr((\tfrac{1}{m}+\pi r^2)\alpha\bigr)\quad\text{on $E\setminus\OO_E$}\,.
\]
In fact, $(E\setminus\OO_E,\Omega)$ is symplectomorphic to a part of the symplectization of a circle subbundle, see \eqref{eq:symplectization}.
  
Alternatively, instead of a $\C$-bundle $E$, we may consider a principal $S^1$-bundle $\Sigma$ over $M$ with the Euler class equal to $-m[\om]$,
\[
\wp:\Sigma\longrightarrow M\,.
\] 
Let $R$ be the fundamental vector field on $\Sigma$ generated by the $S^1$-action, namely
\[
R(x)= \frac{d}{dt}\Big|_{t=0} t\cdot x\,,\qquad x\in\Sigma\,,\; t\in S^1\,.
\]
Recalling our convention $S^1=\R/\Z$, we note that the flow of $R$ is 1-periodic. Let $\alpha$ be a connection 1-form, i.e.~$\mathcal{L}_R\alpha=0$ and $\alpha(R)=1$. In particular, $\wp^*(m\omega)=d\alpha$ up to adding the pullback of a 1-form on $M$ by $\wp$, and $\alpha$ is a contact form with $R$ being the associated Reeb vector field on $\Sigma$. Here we intentionally use the same letters $\wp$, $\alpha$, and $R$ as above because they can be equated as explained next. 

We consider the free $S^1$-action 
\[
\rho:S^1\x \Sigma\x \C \to  \Sigma\x \C\,,\qquad  \rho(t,x,z):=(t\cdot x,e^{-2\pi it} z)\,.
\]
Its fundamental vector field $X_\rho$ is $R\oplus(-2\pi\frac{\partial}{\partial\theta})$  on $\Sigma\x \C^*$ and vanishes on $\Sigma\x\{0\}$. The quotient space induced by $\rho$ is a complex line bundle over $M$,
\[
\Sigma \x_{\rho} \C \to M  \,,\qquad [x,z]\mapsto \wp(x)
\]
equipped with a canonical Hermitian metric. The first Chern class of this line bundle equals the Euler class of $\Sigma\to M$. Given a connection 1-form $\alpha$ on $\Sigma$, we consider $\mathfrak{p}_1^*\alpha+\frac{1}{2\pi} \mathfrak{p}_2^* d\theta$ on $\Sigma \x \C^*$, where $\mathfrak{p}_1$ and $\mathfrak{p}_2$ are the projections from  $\Sigma \x \C^*$ to the first and the second components respectively. This descends to $\Sigma \x_{\rho} \C^*$ since $\mathcal{L}_{X_\rho}(\mathfrak{p}_1^*\alpha+\frac{1}{2\pi} \mathfrak{p}_2^* d\theta)=0$ and $(\mathfrak{p}_1^*\alpha+\frac{1}{2\pi} \mathfrak{p}_2^* d\theta)(X_\rho)=0$. We point out that the induced 1-form on $\Sigma \x_{S^1} \C^*$ agrees with the 1-form in \eqref{eq:alpha_E}. 
We can also identify the unit circle bundle with $\Sigma$: 
\[
\Sigma\into \Sigma \x_{\rho} \C \,, \qquad  x\mapsto [x,1]\,.
\]
This inclusion map is $S^1$-equivariant, where the $S^1$-action on $\Sigma \x_{\rho}\C$ is given by restricting the multiplication with elements in $\C$ to $U(1)\cong S^1$.

\subsection{Winding numbers}

Let $u:S^1\to E\setminus\OO_E$ be a smooth loop which is contractible inside $E$, and let $\bar u:D^2\to E$ be a smooth capping disk, i.e.~$\bar u(e^{2\pi it})=u(t)$. Following \cite{Fra08}, we define the winding number of the pair $(u,\bar u)$ by
\[
\w(u,\bar u):=\int_0^1u^*\alpha-m\int_{D^2}\bar u^*\wp^*\om.
\]
We point out that the winding number is only defined for loops in $E\setminus\OO_E$ which are contractible in $E$. 
If $u$ maps $S^1$ into a single fiber of $\wp:\Sigma\to M$, then the first term counts the oriented covering number of the fiber circle. In this case we denote 
\begin{equation}\label{eq:covering_number}
\cov (u):=\int_0^1u^*\alpha\in\Z
\end{equation}
and view $\wp\circ\bar u$ as a map $\wp\circ\bar u:S^2\to M$ as in Section \ref{sec:Quantum Gysin sequence}.

In the following proposition, we list some properties of the winding number that will be used throughout. The proposition essentially follows from properties of the first Chern class and the homotopy lifting property, but we include a proof for completeness.

\begin{prop}\label{prop:winding}
Let $(u,\bar u)$ and $(v,\bar v)$ be as above. 	
\begin{enumerate}[(a)]
\item Suppose that $u$ maps $S^1$ into a fiber circle of $\wp:\Sigma\to M$. Then 
\[
\w(u,\bar u)=\cov(u) - m\,\om([\wp\circ\bar u])=\cov(u)+c_1^E([\wp\circ\bar u])\in\Z
\]
holds. 
In particular if $\bar u$ is also contained in a fiber of $\wp:E\to M$, then
\[
\w(u,\bar u)=\cov (u)\,.
\]

\item Suppose that $u$ and $v$ are freely homotopic inside $E\setminus \OO_E$, i.e.~there exists a continuous map  
\[
c:[0,1]\x S^1\to E\setminus \mathcal O_E\quad \text{with}\quad c(0,\cdot)=u,\quad c(1,\cdot)=v.
\]
Then, for any capping disk $\bar u$ of $u$,
\[
\w(u,\bar u)=\w(v,\bar u\# c)
\]
where $\bar u\# c:D^2\to E$ is a capping disk for $v$.

\item It holds 
\[
\w(u,\bar u)=\bar u\cdot \OO_E\in\Z\,,
\] 
where $\cdot$ denotes the intersection number. In particular if the image of $\bar u$ is contained in $E\setminus\mathcal O_E$, then $\w(u,\bar u)=0$.

\item Let $c:[0,1]\x S^1\to E$ be a continuous map such that
\[
c(0,\cdot)=u\,,\qquad c(1,\cdot)=v\,,\qquad c_1^E\big([\wp\circ(\bar u\# c\# \bar v^\mathrm{rev})]\big)=0\,.
\]
Then it holds that
\[
\w(u,\bar u)-\w(v,\bar v)=-c\cdot\OO_E\,.
\]
Moreover if both $u$ and $v$ are mapped into fibers of $\wp:\Sigma\to M$, then 
\[
\w(u,\bar u)-\w(v,\bar v)=\cov(u)-\cov(v)-c_1^E([\wp\circ c])\,.
\]
 Note that this subsumes (b). 

\item For a continuous map $s:S^2\to E\setminus\OO_E$, we have $\w(u,\bar u)=\w(u,\bar u\# s)$. Conversely, if $\w(u,\bar u)=\w(u,\bar u')$, then there is a continuous map $s:S^2\to E\setminus\OO_E$ such that $\bar u$ and $\bar u'\#s$ are homotopic relative to boundary.

\item The map $u:S^1\to E\setminus\OO_E$ is contractible inside $E\setminus\OO_E$ if and only if there exists a capping disk $\bar u:D^2\to E$ of $u$ with $\w(u,\bar u)=0$. Moreover if $\w(u,\bar u)=0$, then $\bar u$ is homotopic relative to boundary to another capping disk $\bar u':D^2\to E\setminus\OO_E$ of $u$.

\end{enumerate}
\end{prop}

\begin{proof}
Statement (a) is obvious, and (b) follows from $\wp^*(m\om)|_{E\setminus\OO_E}=d\alpha$ and Stokes' theorem. 
To see (c), we consider the homotopy exact sequence for $S^1\into \Sigma\to M$:
\begin{equation}\label{eq:homotopy_les}
	0\longrightarrow \pi_2(\Sigma) \longrightarrow\pi_2(M)\stackrel{\p}\longrightarrow\pi_1(S^1)\stackrel{\iota}\longrightarrow\pi_1(\Sigma)\stackrel{}{\longrightarrow}\pi_1(M)\longrightarrow\cdots \,.
\end{equation}
In particular, for $[s]\in\pi_2(M)$, we have
\[
\p([s])=  c_1^E([s]) \in \pi_1(S^1)\cong\Z
\]
and
\[
\iota(\ell)=\ell[w] \quad  \ell\in\Z\cong\pi_1(S^1) 
\]
where $w$ is the simple oriented fiber loop through the base point of $\pi_1(\Sigma)$.
Since $E\setminus\OO_E$ is homotopy equivalent to $\Sigma$, by (b) we may assume that $u$ is mapped into $\Sigma$.
Since, by assumption, $u$ is contractible in $E$, $\wp\circ u:S^1\to M$ is also contractible and thus $[u]\in\pi_1(\Sigma)$ is in the image of $\iota$. Hence, again by (b), we may assume that $u$ is mapped into a fiber of $\wp:\Sigma\to M$. Then  
\[
\w(u,\bar u)=\cov(u)+c_1^E(\wp\circ\bar u)=\cov(u)+(\bar u\cdot\OO_E-\cov(u))=\bar u\cdot\OO_E
\]
where the first equality is by (a) and the second one can be seen by applying the fact that $c_1^E([\wp\circ s])=s\cdot \OO_E$ holds for any $s:S^2\to E$ to $s=\bar u\#(\bar u_\mathrm{fib})^\mathrm{rev}$ where $\bar u_\mathrm{fib}$ is a capping disk of $u$ contained in a fiber.

The assumption in (d) implies 
\[
0=(\bar u\# c\# \bar v^\mathrm{rev})\cdot \OO_E= \bar u\cdot \OO_E +c\cdot\OO_E -\bar v\cdot\OO_E
\]
and thus the first equality in (d) follows from (c). The second equality follows from (a).

The first claim in (e) follows from (c) or from the fact that $c_1^E([\wp\circ s])=0$ for any $s:S^2\to E\setminus \OO_E$. To show the converse-part,  suppose $\w(u,\bar u)=\w(u,\bar u')$. We choose $[s]\in\pi_2(E)$  such that $\bar u$ is homotopic to $\bar u'\# s$.  It suffices to show that $[s]$ has a representative entirely contained in $E\setminus\OO_E$. Due to \eqref{eq:homotopy_les}, $\pi_2(E\setminus\OO_E)\cong\pi_2(\Sigma)\cong \ker \p$. The computation
\[
c_1^E([\wp\circ s])=\wp^*c_1^E([s])=\wp^*c_1^E([\bar u\#(\bar u')^\mathrm{rev}])=\w(u,\bar u)-\w(u,\bar u') = 0 
\]
yields $\p([\wp\circ s])=c_1^E([\wp\circ s])=0$, and therefore there is  $s':S^2\to E\setminus\OO_E$ such that $[\wp\circ s']=[\wp\circ s]$ in $\pi_2(M)$. Since $\pi_2(E)\cong\pi_2(M)$ through $\wp$, $[s']=[s]$ in $\pi_2(E)$.

Now we show (f). Suppose that $u$ is contractible in $E\setminus\OO_E$. For any capping disk  $\bar u:D^2\to E\setminus\OO_M$ of $u$, assertion (c) shows $\w(u,\bar u)=0$.  
To prove the converse, assume that there is a capping disk $\bar u$ with $\w(u,\bar u)=0$. Arguing as above, we may assume that $u$ is mapped into the  fiber of $\wp:\Sigma\to M$ through the base point of $\pi_1(\Sigma)$. Then the assumption 
\[
0=\w(u,\bar u)\stackrel{(a)}{=}\cov(u)+c_1^E([\wp\circ \bar u])
\]
together with \eqref{eq:homotopy_les} shows
\[
[u]= \iota(\cov(u))=\iota(-c_1^E([\wp\circ\bar u]))=-\iota \circ \p([\wp\circ \bar u])=0 \quad \textrm{in }\;\pi_1(\Sigma)\,.
\]
The moreover-part follows from  (e).
\end{proof}

\section{Action functionals and indices}\label{sec:action}
Let $\wp:E\to M$ be a Hermitian line bundle with $c_1^E=-m[\om]$ as in Section \ref{sec:line}. 
For $\tau>0$ we define
\[
\mu_\tau:E\to\R,\qquad \mu_\tau:=m\pi r^2-\tau
\]
and denote 
\[
\Sigma_\tau:=\mu_\tau^{-1}(0) =\left\{e \in E \mid m\pi r^2(e)=\tau \,\right\}\,.	
\] 
The  1-form $\alpha$ defined in \eqref{eq:alpha_E} is a contact form when restricted to $\Sigma_\tau$ for every $\tau>0$. 
Let $f:M\to\R$ be a smooth $C^2$-small Morse function as in Section \ref{sec:Quantum Gysin sequence}. We lift it to 
\[
F:E\to\R,\qquad F:=\big(1+m\pi r^2\big) f\circ\wp \,.
\]
We denote by $X_f^{\mathrm h}$ the horizontal lift of $X_f$ to $E$. Recalling the symplectic form
\[
\Omega=\wp^*\om+d(\pi r^2\alpha)=(1+m\pi r^2)\wp^*\om+2\pi r dr\wedge\alpha\,,
\]
we compute
\begin{equation}\label{eq:X_F}
X_{\mu_\tau}=-m R\,,\qquad X_{F}=-m(f\circ \wp) R+X_f^{\mathrm h}\,,
\end{equation}
where $R$ is the fundamental vector field of the $U(1)$-action on $E$, see \eqref{eq:R=2pi_i}, which coincides with the Reeb vector field associated with $\alpha$ when restricted to $\Sigma_\tau$. 
As in Section \ref{sec:Quantum Gysin sequence}, our convention for Hamiltonian vector fields is $d\mu_\tau=\Omega(X_{\mu_\tau},\cdot)$ and likewise for $X_F$.

\subsection{Rabinowitz action functional}
Let $\mathscr L(E)$ be the space of contractible 1-periodic smooth loops in $E$, and let $\widetilde{\mathscr L}(E)$ be its covering space with deck transformation group
\[
\Gamma_E:=\frac{\pi_2(E)}{\ker\Omega\cap\ker c_1^{TE}}\,.
\]
We write elements in $\widetilde{\mathscr L}(E)$ as equivalence classes $[u,\bar u]$ where $u\in \mathscr L(E)$ and $\bar u:D^2\to E$ is a capping disk of $u$, namely $u(t)=\bar u(e^{2\pi it})$. The equivalence relation is given by 
\[
(u,\bar u)\sim(u',\bar u') \qquad \textrm{if }\; u=u'\,\textrm{ and } \;\Omega([\bar u^\mathrm{rev}\#\bar u'])=c_1^{TE}([\bar u^\mathrm{rev}\#\bar u'])=0\,.
\] 
where $\bar u^\mathrm{rev}$ refers to $\bar u$ with opposite orientation. 
The Rabinowitz action functional is defined by
\begin{equation}\label{eq:Rabinowitz_functional}
	\begin{split}
&\AA_f^\tau:\widetilde{\mathscr L}(E)\x\R\longrightarrow\R,\\
&\AA_f^\tau\big([u,\bar u],\eta\big):= -\int_{D^2}\bar u^*\Omega-\eta\int_0^1\mu_\tau(u)\,dt-\int_0^1F(u)\,dt\,.
\end{split}
\end{equation}
A straightforward computation shows that $([u,\bar u],\eta)\in\Crit \AA_f^\tau$ if and only if
\begin{equation}\label{eqn:critical_point_equation_A_f_tau}
\left\{
\begin{aligned}
\;&\dot u=\eta X_{\mu_\tau}(u)+X_{F}(u)\\
&\int_0^1\mu_\tau(u)\,dt=0\;.
\end{aligned}
\right.
\end{equation}
Since $d\mu_\tau(X_{\mu_\tau})=d\mu_\tau(X_{F})=0$, it follows $d\mu_\tau(\dot u)=0$ and hence the second identity on the right-hand side of \eqref{eqn:critical_point_equation_A_f_tau} yields $\mu_\tau(u)=m\pi r^2(u)-\tau=0$, i.e.~$u(S^1)\subset\Sigma_\tau$. Therefore the first equality in \eqref{eqn:critical_point_equation_A_f_tau} simplifies to
\begin{equation}\label{eq:dot_u}
\dot u=-m\big(\eta+(f\circ\wp)(u)\big)R(u)+X_f^{\mathrm h}(u)\,,
\end{equation}
see \eqref{eq:X_F}. 
Since $R$ vanishes under the map $d\wp$, the smooth loop $q:=\wp\circ u:S^1\to M$ solves 
\[
\dot q=d\wp(u)[\dot u]=X_f(q)\,.
\]
Since $f$ is $C^2$-small, such solutions are constant loops with values in the set of critical points of $f$. As in Section \ref{sec:Quantum Gysin sequence}, we often equate $q$ with its image, a point in $M$. Therefore \eqref{eqn:critical_point_equation_A_f_tau} reads 
\begin{equation}\label{eqn:critical_point=>Reeb_orbit}
\left\{
\begin{aligned}
\;&\dot u=-m(\eta+f(q))R(u) \;\text{ for some } q\in\Crit f\\[0.5ex]
&\pi mr^2(u)=\tau\;.
\end{aligned}
\right.
\end{equation}
Since both $u$ and the flow of $R$ are 1-periodic, the first line in \eqref{eqn:critical_point=>Reeb_orbit} yields
\begin{equation}\label{eq:eta=cov}
m(\eta+f(q))=-\cov(u)\in\Z\,,
\end{equation}
where the covering number is defined in \eqref{eq:covering_number}. Actually, using the $U(1)$-action on $E$, we can write $u(t)=e^{2\pi i\cov(u) t}u(0)$. 
The action of $([u,\bar u],\eta)$ is computed as
\begin{equation}\label{eq:action_value_full}
\begin{split}
{\AA}_f^\tau([u,\bar u],\eta)
&=-\int_{D^2}\bar u^*\wp^*\om-\int_0^1u^*(\pi r^2\alpha)-\eta\int_0^1\mu_\tau(u)\,dt-\int_0^1F(u)\,dt\\
&=-\om([\wp\circ \bar u]) + \tau (\eta+f(q))-(1+\tau)f(q)\\
&=-\om([\wp\circ\bar u])-\frac{\tau}{m}\cov(u)-(1+\tau)f(q) \,.
\end{split}
\end{equation}
We define the winding number of a critical point $w=([u,\bar u],\eta)\in\Crit\AA^\tau_f$ by 
\[
\w(w):=\w(u,\bar u) 
\]
which is well-defined since $[u,\bar u]=[u,\bar u']$ implies 
\begin{equation}\label{eq:wind_crit}
\wp^*c_1^{E}([\bar u^\mathrm{rev}\#\bar u'])=-m\wp^*\om([\bar u^\mathrm{rev}\#\bar u'])=-m\Omega([\bar u^\mathrm{rev}\#\bar u'])=0\,.
\end{equation}
The action and the winding number are related through 
\begin{equation}\label{eq:full_action-winding}
{\AA}_f^\tau(w) - \frac{\w(w)}{m} = -\om([\wp\circ \bar u])+ \tau\eta-f(q)-\frac{1}{m}\cov(u)+\om ([\wp\circ \bar u])=(1+\tau)\eta\,.
\end{equation}
Critical points of $\AA^\tau_f$ come in $S^1$-family, namely if $w=([u,\bar u],\eta)$ is a critical point of $\AA^\tau_f$, so is $t^*w=([t^*u,t^*\bar u],\eta)$ for $t\in S^1$ where 
\begin{equation}\label{eq:capping_shift}
t^*u(t_0):=u\left(t_0+\tfrac{t}{\cov(u)}\right)\,,\qquad t^*\bar u(z)=\bar u\left(e^{2\pi i\frac{t}{\cov(u)}}z\right)\,.	
\end{equation}
Since $u$ lies in a fiber of $\Sigma_\tau\to M$ solving equation  \eqref{eqn:critical_point=>Reeb_orbit}, using the $U(1)$-action on $E$ we can write
\[
t^*u(\cdot)=e^{2\pi it}u(\cdot)\,.
\]
We denote the $S^1$-family of critical points of  $\AA^\tau_f$ containing $w=([u,\bar u],\eta)$ by
\[
S_w=S_{([u,\bar u],\eta)}:=\big\{t^*w=([t^*u,t^*\bar u],\eta)\mid t\in S^1\big\}\,.	
\]
Note that it is diffeomorphic to the fiber $\Sigma_p$ of $\Sigma\to M$ over $p=\wp\circ u\in\Crit f$ via 
\begin{equation}\label{eq:crit_diffeo}
	S_w\to \Sigma_p\,,\qquad t^*w=([t^*u,t^*\bar u],\eta) \mapsto t^*u(0)\,.
\end{equation}

\subsection{Hamiltonian action functional}
In this section we relate the Rabinowitz action functional $\AA^\tau_f$ with the Hamiltonian action functional $\mathfrak{a}_f$ in \eqref{eq:classical_action_functional}. 
To begin with, we 
observe  
\begin{equation}\label{eq:c_1}
c_1^{TE}=\wp^*c_1^{TM}+\wp^*c_1^E=\wp^*c_1^{TM}-m\wp^*[\om]\,.	
\end{equation}
Since $\wp^*\om=\Omega:\pi_2(E)\to\Z$, the map
\[
\wp_*:\Gamma_E=\frac{\pi_2(E)}{\ker \Omega\cap\ker c_1^{TE}}\;\;\stackrel{\cong}{\longrightarrow} \;\;\Gamma_M=\frac{\pi_2(M)}{\ker\om\cap\ker c_1^{TM}}\,,\qquad \wp_*[s]:=[\wp\circ s]
\]
is an isomorphism. The projection $\wp$ also induces a natural map 
\[
\begin{split}
\Pi:\widetilde{\mathcal L}(E)&\longrightarrow\widetilde{\mathcal L}(M)\\
[u,\bar u]&\longmapsto [\wp\circ u,\wp\circ\bar u]\,.	
\end{split}
\]
Due to \eqref{eq:crit_a_f} and \eqref{eqn:critical_point=>Reeb_orbit}, this gives rise to a map, denoted again by $\Pi$,
\begin{equation}\label{eq:Pi}
\begin{split}
\Pi:{\Crit\AA_f^\tau} &\longrightarrow \Crit\mathfrak{a}_f\cong\Crit f\x\Gamma_M\\
w=\big([u,\bar u],\eta\big)&\longmapsto \Pi(w)=\big[\wp\circ u,\wp\circ\bar u]\,.
\end{split}
\end{equation}

\begin{lem}\label{lem:one-to-one}
For each $k\in\Z$, the projection $\Pi$ induces a bijection
\[
\big\{S_{w} \mid w\in\Crit {\AA}^\tau_f,\; \w(w)=k\big\} \quad \longleftrightarrow \quad \Crit\mathfrak a_f.
\]
\end{lem}
\begin{proof}
We first observe that, for any $w=([u,\bar u],\eta)\in\Crit {\AA}^\tau_f$, there is $[s]\in\Gamma_E$ satisfying 
\[
[u,\bar u]=[u,\bar u_\mathrm{fib}\#s]\qquad \textrm{in }\; \widetilde{\mathcal L}(E)
\]
where $\bar u_\mathrm{fib}:D^2\to E$ is a fiber capping disk of $u$, i.e.~$\bar u_\mathrm{fib}(D^2)$ is entirely contained in a  fiber of $\wp:E\to M$. 
We choose $[q,\bar q]\in\Crit\mathfrak a_f$. Let $u_q$ be a  simple 1-periodic orbit of $R$ in $\Sigma_\tau$ over $q$, and let $s:S^2\to E$ be a continuous map satisfying $[s]=[\bar q]$ in $\Gamma_E\cong\Gamma_M$.  We set 
\[
[u,\bar u]:= [u_q^{m\om([\bar q])+k},(\bar u_q^{m\om([\bar q])+k})_\mathrm{fib}\#s]
\] 
where $u_q^{\ell}$ denotes the $\ell$-fold cover of $u_q$. This gives an $S^1$-family of critical points $S_{([u,\bar u],\eta)}$ of $\AA^\tau_f$, where $\eta$ is determined by \eqref{eq:eta=cov}, such that 
\begin{equation}\label{eq:one-to-one}
\w(u,\bar u)=k\,,\qquad \Pi([u,\bar u],\eta)=[q,\bar q].
\end{equation}
To see that this $S^1$-family is unique, suppose $([u',\bar u'],\eta')$ is another critical point of ${\AA}^\tau_f$  satisfying \eqref{eq:one-to-one}. If write $[u',\bar u']=[u',\bar u'_\mathrm{fib}\#s']$ for $[s']\in\Gamma_E$, then $[s']= [\bar q]$ in $\Gamma_E\cong\Gamma_M$ by the second condition in \eqref{eq:one-to-one}. Then the first condition in \eqref{eq:one-to-one} implies $u=u'$ up to timeshift and $\eta'=\eta$ by \eqref{eq:eta=cov}. Hence $([u',\bar u'],\eta')$ belongs to $S_{([u,\bar u],\eta)}$.
\end{proof}

For later purposes, we compute for $w=([u,\bar u],\eta)\in\Crit{\AA}^\tau_f$ with $\wp\circ u=q$, 
\begin{equation}\label{eq:action_winding_eta}
\mathfrak{a}_f(\Pi(w))= -\om([\wp\circ\bar u]) - f(q) =  \frac{1}{m}\left(\w(w)-\cov(u)\right)-f(q)=\frac{1}{m}\w(w)+\eta	
\end{equation}
by Proposition \ref{prop:winding}.(a) and \eqref{eq:eta=cov}. Equation \eqref{eq:full_action-winding} yields
\begin{equation}\label{eq:action_comparision}
\mathfrak{a}_f(\Pi(w))={\AA}_f^\tau(w)-\tau\eta\,,\qquad (1+\tau)\a_f(\Pi(w))=\frac{\tau}{m}\w(w)+\AA^\tau_f(w)\,.	
\end{equation}

\subsection{Indices and winding number}\label{sec:index}
We define the RFH-index of $w=([u,\bar u],\eta)\in\Crit\AA^\tau_f$ as follows. We choose a unitary trivialization $\Phi_w:D^2\x\R^{\dim E}\to \bar u^*TE$ of $\bar u^*TE$ and denote
\begin{equation}\label{eq:trivialization}
\Phi_w(t):=\Phi_w(e^{2\pi it},\cdot):\R^{\dim E}\to T_{u(t)}E\,.
\end{equation}
The linearized flow along $u$ in this trivialization 
\begin{equation}\label{eq:lin_triv}
\Psi_w(t):=\Phi_w(t)^{-1}\circ d\phi_{\eta X_{\mu_\tau}+X_F}^{t}(u(0))\circ \Phi_w(t)\,,\qquad t\in[0,1]\,,	
\end{equation}
where $\phi_{\eta X_{\mu_\tau}+X_F}^{t}$ is the flow of $\eta X_{\mu_\tau}+X_F$, 
defines a path of symplectic matrices, and we set
\[
\mu_\RFH(w):=-\mu_\CZ(\Psi_w)\,.
\]
Here $\mu_\CZ$ denotes the Conley-Zehnder index, see \cite{RS93}. The index $\mu_\RFH$ is well-defined since a unitary trivialization of $\bar u^*TE$ is unique up to homotopy. Moreover since $u$ is contained in a fiber circle over $q=\wp\circ u\in\Crit f$, we can give a simple expression of index. We represent a capping disk $\bar u$ of $u$ again by a fiber capping disk $\bar{u}_\mathrm{fib}$ connected sum with some $[s]\in\pi_2(E)$.
By properties of the Conley-Zehnder index, we have
\[
\mu_\RFH([u,\bar{u}_\mathrm{fib}\#s],\eta)=\mu_\RFH([u,\bar{u}_\mathrm{fib}],\eta) - 2c_1^{TE}([s])\,.
\]
Since the flows of $R$ and $X_f^{\mathrm h}$ commute, we may write the linearized flow $d\phi_{\eta X_{\mu_\tau}+X_F}^t$ along $u$ as 
\[
d\phi_{R}^{\cov(u)t}(u(0)) \oplus d\phi_{X_f}^t(q) 
\]
with respect to the  decomposition in \eqref{eq:splitting_TE}, see \eqref{eq:dot_u} and \eqref{eq:eta=cov}. Using properties of the Conley-Zehnder index, we compute
\begin{equation}\label{eq:fiber_index}
\begin{split}
-\mu_\RFH([u,\bar{u}_\mathrm{fib}],\eta) &= \mu_\CZ\big(\{e^{2\pi i t\cov(u)}\}_{t\in[0,1]}\big) + \mu_\CZ([q,0])\\
&= 2\,\cov(u) + \frac{1}{2}\dim M-\mu_{-f}(q)
\end{split}	
\end{equation}
where we used \eqref{eq:ind_M} in the second line.  Here $\{e^{2\pi i t\cov(u)}\}_{t\in[0,1]}$ is a loop of symplectic matrices on $E_q=\C$, its Conley-Zehnder index equals $2\cov(u)$. 
Putting all these together, we obtain 
\begin{equation}\label{eq:mu_RFH_cov}
\mu_\RFH(w)=\mu_\RFH\big([u,\bar u_{\mathrm{fib}}\# s],\eta\big)=-2\,\cov(u)-2c_1^{TE}([s])+\mu_{-f}(q)- \frac{1}{2}\dim M\,.	
\end{equation}
Another formulation using Proposition \ref{prop:winding}.(a), \eqref{eq:c_1}, and \eqref{eq:ind_M} is 
\begin{equation}\label{eq:indices_and_winding}
\mu_\RFH(w)=-2\w(w) +\mu_\FH(\Pi(w))\,. 
\end{equation}
\begin{rem}
	Due to the fact that we use negative gradient flow lines of $\AA^\tau_f$, as opposed to positive ones in \cite{AK17}, the sign convention here is opposite to the one in \cite{AK17}. 
\end{rem}

Since $\Crit\AA^\tau_f$ is a disjoint union of copies of $S^1$, we choose an auxiliary smooth perfect Morse function 
\begin{equation}\label{eq:perfect}
	h:\Crit\AA^\tau_f\longrightarrow\R\,.
\end{equation}
In view of \eqref{eq:crit_diffeo}, we abuse notation and use the same letter $h$ which was already occupied by a perfect Morse function in \eqref{eq:ftn_h} defined on $\Crit \tilde f$.  
We sometimes denote an element of $\Crit h\subset\Crit\AA^\tau_f$ by 
$\hat{w}$ or $\check{w}$ to indicate that it is a maximum or minimum point of $h$ respectively. 
We associate the RFH-index to elements in $\Crit h$ by
\begin{equation}\label{eq:index_h}
	\mu^h_\RFH(\hat w):=\mu_\RFH(w)+1\,,\qquad \mu^h_\RFH(\check w):=\mu_\RFH(w)\,.
\end{equation}

\begin{ex}\label{ex:CP}
We compute the action, the index, and the winding number for the complex line bundle $E=\OO_{\CP^n}(-m)$ over $M=\CP^n$ with $c_1^{\OO_{\CP^n}(-m)}=-m[\om_\mathrm{FS}]$. In this case,
\[
 c_1^{T\CP^n}=(n+1)[\om_\mathrm{FS}]\,,\qquad \pi_2(M)=\Z\langle [s]\rangle\,,\qquad \om_\mathrm{FS}([s])=1\,,
\]
where $s:S^2\to \CP^n$ maps to a complex line. 
For $w=([u,\bar u_\mathrm{fib}\#ks],\eta)\in\Crit{\AA}^\tau_f$,  we have
\[
{\AA}^\tau_f(w)= - k -\frac{\tau}{m}\cov(u)  - (1+\tau) f(q)\,,
\]
where $q=\wp\circ u$ as usual, and 
\[
 \mu_\RFH(w)=-2\,\cov(u)- 2(n+1-m)k+\mu_{-f}(q)-\tfrac{1}{2}\dim M \,,\quad \w(w)=\cov(u) - m k\,.
\]
\end{ex}

\subsection{Rabinowitz functional for zero winding number}\label{sec:zero_winding}

Since we are mainly concerned with the case of zero winding number, we sum up the action and index correspondences between $\Crit\AA^\tau_f$ and $\Crit\mathfrak{a}_f$ computed in \eqref{eq:indices_and_winding} and \eqref{eq:action_comparision} in the case of winding zero.

\begin{lem}\label{lem:one-to-one_index}
The projection $\Pi$ gives a one-to-one correspondence
\[
\begin{split}
\Pi:\big\{S_{w} \mid w=([u,\bar u],\eta)\in\Crit {\AA}^\tau_f\mid \w(w)=0\big\}  &\longrightarrow  \Crit\mathfrak a_f\,,\\
w&\longmapsto[\wp\circ u,\wp\circ\bar u]\,.	
\end{split}
\]
Moreover, the indices and the actions are related as follows:
\[
\mu_\RFH(w)=\mu_\FH(\Pi(w))\,,\qquad \mathfrak a_f(\Pi(w))=\eta=\frac{1}{1+\tau}\AA^\tau_f(w)\,.
\]
\end{lem}

One aim of the present paper is to define the Rabinowitz Floer homology of the action functional $\AA^\tau_f$ restricted to the space
\begin{equation}\label{eq:winding_zero_loops}
\big\{([u,\bar u],\eta)\in\widetilde{\mathscr L}(E)\x\R\mid u(S^1) \subset E\setminus\OO_E,\;\w(u,\bar u)=0\big\}\,.	
\end{equation}
We consider the space $\mathscr L(E\setminus\OO_E)$ of smooth 1-periodic loops in $E\setminus\OO_E$ contractible inside $E\setminus\OO_E$ and observe that there is a bijection
\[
\begin{split}
\mathscr{L}(E\setminus\OO_E)\x \frac{\pi_2(E\setminus\OO_E)}{\ker \wp^*c_1^{TM}} &\longrightarrow \big\{[u,\bar u]\in\widetilde{\mathscr L}(E)\mid u(S^1)\subset E\setminus\OO_E,\;\w(u,\bar u)=0\big\}\\
 (u,[s]) &\longmapsto ([u,\bar u'\#s])
\end{split}
\]
where $\bar u'$ is any capping disk of $u$ contained in $E\setminus\OO_E$. 
The map is well-defined since $\wp^*c_1^E([s])=0$ for a continuous map $s:S^2\to E\setminus\OO_E$ and thus $\w(u,\bar u'\#s)=\w(u,\bar u')=0$. The surjectivity of the map follows from (e) and (f) in Proposition \ref{prop:winding}. The injectivity is a consequence of the fact that $\Omega([s])=\wp^*\omega([s])=-m\wp^*c_1^E([s])=0$ holds automatically for $s:S^2\to E\setminus\OO_E$ and thus $\ker c_1^{TE}=\ker \wp^*c_1^{TM}$ on $\pi_2(E\setminus\OO_E)$ by \eqref{eq:c_1}.

Let us offer the following explanation for restricting to the case of zero winding number.
We note that under our typical assumption that $c_1^{TM}=\lambda\om$ on $\pi_2(M)$ for some $\lambda\in\R$ 
\[
\frac{\pi_2(E\setminus\OO_E)}{\ker \wp^*c_1^{TM}}=0\,,\qquad
\mathscr{L}(E\setminus\OO_E)\cong \big\{([u,\bar u])\in\widetilde{\mathscr L}(E)\mid u(S^1)\subset E\setminus\OO_E,\;\w(u,\bar u)=0\big\}\,.
\]
Moreover, we can define
\[
\begin{split}
\mathscr{A}_f^\tau:{\mathscr{L}}(E\setminus\OO_E)\x\R\longrightarrow\R\,,\qquad \mathscr{A}_f^\tau(u,\eta):=\AA_f^\tau\big([u,\bar u],\eta\big)
\end{split}
\]
where $\bar u$ is any capping disk of $u$ with $\w(u,\bar u)=0$. Indeed, this does not depend on the choice of capping disks since $\Omega([\bar u^\mathrm{rev}\# \bar u'])=0$ for another capping disk $u'$ with $\w(u,\bar u')=0$. We can even take $\bar u$ contained in $E\setminus\OO_E$ by Proposition \ref{prop:winding}.(f). Recalling that $\Omega=d\bigr((\tfrac{1}{m}+\pi r^2)\alpha\bigr)$ on $E\setminus\OO_E$, we deduce
\begin{equation}\label{eq:functional_exact}
\mathscr{A}_f^\tau(u,\eta)=-\int_0^1u^*\bigr((\tfrac{1}{m}+\pi r^2)\alpha\bigr)-\eta\int_0^1\mu_\tau(u)\,dt-\int_0^1F(u)\,dt\,.
\end{equation}
In particular, critical points of $\AA^\tau_f$ with zero winding number  correspond to  critical points of the simpler action functional $\mathscr{A}^\tau_f$ on $\mathscr{L}(E\setminus\OO_E)$. Of course, spaces of Floer cylinders in $E\setminus\OO_E$ may  have bad compactness properties. However, we will see in Section \ref{sec:RFH} that   Floer cylinders of $\AA^\tau_f$ connecting critical points with  zero winding number are actually contained in $E\setminus\OO_E$ for an appropriate choice of almost complex structures and therefore coincide with Floer cylinders of $\mathscr{A}^\tau_f$. We point out that this is no longer true in the construction of the full Rabinowitz Floer homology of $\AA^\tau_f$, see Section \ref{sec:full}.

\section{Three classes of almost complex structures}\label{sec:J}
We recall from Section \ref{sec:line} the splitting 
\begin{equation}\label{eq:decomposition}
T_xE\cong  T^\mathrm{v}_xE\oplus T^\mathrm{h}_xE\cong E_{\wp(x)}\oplus T_{\wp(x)}M
\end{equation}
and the complex structure $i$ of the complex bundle $\wp:E\to M$. We define a class of $S^1$-families of diagonal almost complex structures on $E$,
\[
\JJ_\mathrm{diag}\subset \Gamma(S^1\x E,\mathrm{Aut}(TE)) 
\]
whose elements $J$ are of the form
\begin{equation}\label{eq:diagonal_J}
J=\begin{pmatrix}
 i & 0\\
 0 & j
\end{pmatrix} \qquad \textrm {for some }\;j\in\mathfrak{j}_\mathrm{reg}(f)\cup \mathfrak{j}_\mathrm{HS}(f)
\end{equation}
with respect to the decomposition \eqref{eq:decomposition}.
If the choice of $f$ is clear from the context, we simply write $\j_\mathrm{reg}=\j_\mathrm{reg}(f)$ and $\j_\mathrm{HS}=\j_\mathrm{HS}(f)$, the spaces of almost complex structures defined in Section \ref{sec:Quantum Gysin sequence}. 
 Note that every $J\in\mathcal{J}_\mathrm{diag}$ is automatically $\Omega$-compatible, i.e.~$\Omega(\cdot,J_t\cdot)$ is a Riemannian metric on $E$ for all $t\in S^1$. 

In addition, we consider a larger class of almost complex structures. 
Let us fix a disjoint union $\mathcal U$ of open balls around each critical point of a Morse function $f:M\to\R$. 
 For $j\in\mathfrak{j}_\mathrm{reg}$, we define a subspace 
\[
\BB(j)\subset \Gamma\big(\R\x S^1\x E,L(T^\mathrm{h}E,T^\mathrm{v}E)\big)
\]
as follows. An element $B\in \BB(j)$ is a smooth section $B:\R\x S^1\x E\to L(T^\mathrm{h}E,T^\mathrm{v}E)$, i.e. 
\[
B_{(\eta,t)}(x)=B(\eta,t,x) \in L(T_x^\mathrm{h}E,T_x^\mathrm{v}E)
\]
where $L$ means the space of linear maps. Moreover we require that $B$ has support inside
\begin{equation}\label{eq:support_B}
	\R\x S^1\x \{x\in E\mid R_0<r(x)<R_1\}\quad\text{for some $0<R_0<\sqrt{\frac{\tau}{m\pi}}<R_1$}\,.
\end{equation}
and satisfies
\[
iB_{(\eta,t)}+B_{(\eta,t)}j_t=0\,,\qquad B_{(\eta,t)}(x)=0\quad\forall x\in \wp^{-1}(\mathcal U)\,.
\]
Then $j\in\j_\mathrm{reg}$ and $B\in\BB(j)$ define an $(\R\x S^1)$-family of almost complex structures
\begin{equation}\label{eq:splitting_J}
J^B=\begin{pmatrix}
 i & B\\
 0 & j
\end{pmatrix}\,.
\end{equation}
We note that $J^B$ is clearly not $\Omega$-compatible for $B\neq0$ but still $\Omega$-tame provided that $B$ is sufficiently small since this is true for $B=0$. We denote
\[
\JJ^\BB:=\big\{J^B  \mid j\in\j_\mathrm{reg}\,,\; B\in \BB(j)
\textrm{ and $J^B$ is $\Omega$-tame}\big\}\,.
\]
We finally define the space
\begin{equation}\label{eq:JJ}
\JJ\subset\Gamma\big(\R\x S^1\x E,\mathrm{Aut}(TE)\big)	
\end{equation}
consisting of $J\in\JJ$ satisfying the following two conditions:
\begin{enumerate}[(i)]
\item for each $(\eta,t)\in \R\x S^1$, $J_{(\eta,t)}:=J(\eta,t,\cdot)$ is an $\Omega$-tame almost complex structure which is, in addition, $\Omega$-compatible on $\R\x S^1\x\wp^{-1}(\mathcal U)$;
\item $J$ is diagonal as in \eqref{eq:diagonal_J}  outside the region \eqref{eq:support_B}.
\end{enumerate}
The three space of almost complex structures on $E$ we have defined satisfy
\begin{equation}\label{eq:J_incl}
\mathcal J \supset \mathcal J^{\mathcal B} \supset \mathcal J_\mathrm{diag}\,.	
\end{equation}

\begin{rem}
We briefly explain the purpose of these spaces of almost complex structures. 
\begin{itemize}
	\item [$\JJ$:] Generically, the Rabinowitz Floer homology with zero winding number is defined. An analogous homology is also defined for the action functional perturbed by a contact Hamiltonian.
	\item [$\JJ^\BB$:] The projection map $\wp:E\to M$ is $(J,j)$-holomorphic. The Rabinowitz Floer homology with zero winding number and the full Rabinowitz Floer homology both are generically defined. This type of almost complex structure is used in \cite{AK17}.
	\item [$\JJ_\mathrm{diag}$:] We define the Rabinowitz Floer homology with zero winding number for all $j\in\mathfrak{j}_\mathrm{reg}(f)\cup \mathfrak{j}_\mathrm{HS}(f)$. Here we assume that $\tau>0$ is small, which seems to be a technical condition. For $j\in\mathfrak{j}_\mathrm{HS}(f)$, we show that the boundary operator for this homology corresponds to that of the cone complex of $\psi^{c_1^E}:\FC(f)\to\FC(f)$ in Section \ref{sec:quantum_gysin}. 
\end{itemize}
The zero section $\OO_E$ is $J$-holomorphic since $J$ is diagonal near $\OO_E$. We also require $J$ to be diagonal near infinity so that a maximum principle holds. We point out that none of the almost complex structure used in this article are of SFT-type. The $\Omega$-compatibility for $J$ in $\JJ^\BB$ or $\JJ$ is used to establish the Fredholm property of the operator obtained by linearizing the Rabinowitz-Floer equation we will study.
\end{rem}

\subsection{Gradient flow equation}\label{sec:gradient}
For a given $J\in\JJ$, we define a bilinear form $\mathfrak m$ on $T(\widetilde{\mathscr L}(E)\x \R)$. For $(\hat u_1,\hat\eta_1),\,(\hat u_2,\hat\eta_2)\in T_{([u,\bar u],\eta)}\big(\widetilde{\mathscr L}(E)\x \R\big)\cong \Gamma(S^1,u^*TE)\x\R$, we set
\begin{equation}\label{eq:bilinear form}
\mathfrak m_{([u,\bar u],\eta)} \big((\hat u_1,\hat\eta_1),(\hat u_2,\hat\eta_2)\big):=-\int_0^1\Omega_{u(t)}\big(J(\eta,t,u(t))\hat u_1(t),\hat u_2(t)\big)dt+\hat\eta_1\hat\eta_2\,.
\end{equation}
This bilinear form is positive definite but symmetric only on the region where $J$ is $\Omega$-compatible. The gradient vector field $\nabla\AA_f^\tau(w)$ at $w=([u,\bar u],\eta)$ with respect to $\mathfrak m$ is defined by
\[
d\AA^\tau_f(w)\hat w=\mathfrak m\big(\nabla\AA^\tau_f(w),\hat w\big)\qquad \forall \hat w\in T_w\big(\widetilde{\mathscr L}(E)\x \R \big)
\]
and has an explicit expression 
\[
\nabla\AA^\tau_f(w)= \left(J(\eta,t,u)\big(\p_tu-\eta X_{\mu_\tau}(u)-X_F(u)\big),\,-\int_0^1\mu_\tau(u)dt\right)\,.
\]
We interpret a negative gradient flow line of $\AA^\tau_f$, i.e. 
\[
w\in C^\infty(\R,\widetilde{\mathscr L}(E)\x\R)\,,\qquad \p_s w+\nabla\AA^\tau_f(w)=0,
\]
as a smooth solution $w=(u,\eta)\in C^\infty(\R\x S^1,E)\x C^\infty(\R,\R)$ of 
\begin{equation}\label{eq:Floer_eqn_for_A}
\overline{\p}_{J,\tau,f}(w)=\overline{\p}_{J,\tau,f}(u,\eta):=\left(\begin{aligned}
&\p_su+J(\eta,t,u)\big(\p_tu-\eta X_{\mu_\tau}(u) - X_F(u)\big)\\[.5ex]
&\p_s\eta-\int_0^1\mu_\tau(u)dt\,
\end{aligned}\right)=0\,.
\end{equation}
The above equation is referred to as the Rabinowitz-Floer equation. 
Using \eqref{eq:X_F}, the first equation in \eqref{eq:Floer_eqn_for_A} can be rephrased as
\begin{equation}\label{eq:Floer_eqn_for_A_again}
\p_su+J(\eta,t,u)\bigr(\p_tu+m(\eta+(f\circ \wp)(u)) R(u)-X_f^{\mathrm h}(u)\bigr)=0\,.
\end{equation}
We defined the energy of a solution $w=(u,\eta)$ of \eqref{eq:Floer_eqn_for_A} by
\[
E(w):=\int_\R\int_{S^1}\big(|\p_su|^2+|\p_s\eta|^2\big)dtds 
\]
where the norm $|\p_su|$ is given by the positive definite form $-\Omega(J_t\cdot,\cdot)$. It is finite if and only if there exist $([u_\pm,\bar u_\pm],\eta_\pm)\in\Crit\AA^\tau_f$ such that
\[
\lim_{s\to\pm\infty}\big(u(s,\cdot),\eta(s)\big)=(u_\pm,\eta_\pm)\,.
\]
In this case, we have
\[
E(w)=\AA^\tau_f([u_-,\bar u_-],\eta_-)-\AA^\tau_f([u_+,\bar u_-\#u],\eta_+)\,.
\]
For $w_\pm=([u_\pm,\bar u_\pm],\eta_\pm)\in\Crit\AA^\tau_f$, we denote by
\begin{equation}\label{eq:rfh_moduli}
	\widehat\MM(S_{w_-},S_{w_+},\AA^\tau_f,J)
\end{equation}
the moduli space of solutions $w=(u,\eta)\in C^\infty(\R\x S^1,E)\x C^\infty(\R,\R)$ of $\overline{\p}_{J,\tau,f}(w)=0$, see \eqref{eq:Floer_eqn_for_A}, with asymptotic condition 
\[
\lim_{s\to \pm\infty}w(s) = (e^{2\pi i\theta_\pm}u_\pm,\eta_\pm)\,,\qquad  [\bar u_-\#u\#\bar u_+^\mathrm{rev}]=0 \;\textrm{ in }\;\Gamma_E\,,
\]
for some $\theta_\pm\in S^1$. The limits are in $C^\infty(S^1,E)$. The virtual dimension of this moduli space, which is by definition the Fredholm index of the associated operator $D_w$ in \eqref{eq:D_w}, is
\begin{equation}\label{eq:moduli_dim_S}
\mathrm{virdim\,} \widehat\MM(S_{w_-},S_{w_+},\AA^\tau_f,J) = \mu_\RFH(w_-)-\mu_\RFH(w_+) +1\,,
\end{equation}
see Section \ref{sec:fredholm}. 
Unless $w_-= w_+$, there is a free $\R$-action on these moduli spaces by translation in $s$, i.e.~ 
\begin{equation}\label{eq:translation}
 \sigma^*w(s)=(\sigma^*u(s,\cdot),\sigma^*\eta(s)):= (u(s-\sigma,\cdot),\eta(s-\sigma))\,,\qquad \sigma\in\R\,.
\end{equation}
We consider the evaluation maps at asymptotic ends:
\begin{equation}\label{eq:asymp_ev}
	\begin{split}
	&\ev_\pm:\widehat\MM(S_{w_-},S_{w_+},\AA^\tau_f,J)\to S_{w_\pm}\,,\qquad \ev_\pm(u,\eta):=\lim_{s\to\pm\infty} \big([u(s,\cdot),t_\pm^*\bar u_\pm ],\eta(s)\big) \,,
	\end{split}
\end{equation}
where $t_\pm\in S^1$ are uniquely chosen so that $t_\pm^*\bar u$ are capping disks for the asymptotic orbits of $u$ as $s\to\pm\infty$.
Let $h$ be a smooth perfect Morse function on $\Crit\AA^\tau_f$ as in \eqref{eq:perfect}. We denote by $W^u(w,h)$ and $W^s(w,h)$ the unstable and stable manifold of the negative gradient of $h$ at $w\in\Crit h$ with respect to some metic respectively. Let $w_-,w_+\in\Crit h$ and $w_1,w_2,w_3\in\Crit\AA^\tau_f$ be arbitrary. We observe that $W^u(w_-,h)\x W^s(w_+,h)$ is transverse to 
\begin{equation}\label{eq:ev_minus_plus}
\ev_-\x \ev_+: \widehat\MM(S_{w_1},S_{w_2},\AA^\tau_f,J)\longrightarrow S_{w_1} \x S_{w_2}
\end{equation}
and also to 
\begin{equation}\label{eq:ev_minus_plus2}
\ev_-\x \ev_+: \widehat\MM(S_{w_1},S_{w_2},\AA^\tau_f,J)\;_{\ev_+}\!\!\x_{\,\ev_-}\widehat\MM(S_{w_2},S_{w_3},\AA^\tau_f,J)\to S_{w_1} \x S_{w_3}
\end{equation}
if every pair $(w_-,w_+)$ of a minimum point $w_-$ and a maximum point $w_+$ of $h$ is a regular value of $\ev_-\x \ev_+$. Throughout this paper, we tacitly assume that $h$ has this property.

 For $w_\pm\in\Crit h$, we denote by 
\begin{equation}\label{eq:m-cascade}
\widehat\MM^n(w_-,w_+,{\AA}^\tau_f,J)\,,\qquad n\in\N=\{1,2,\dots\}	
\end{equation}
the moduli space of flow lines with $n$ cascades from $w_-$ to $w_+$, namely $\mathbf{w}=(w^1,\dots,w^n)$ for
\[
w^i\in \widehat\MM(S_{w^i_-},S_{w^i_+},\AA^\tau_f,J)\,,\quad w^i_\pm\in\Crit\AA^\tau_f\,,\; S_{w^i_+}=S_{w^{i+1}_-}\,,\; S_{w^1_-}=S_{w_-}\,,\;S_{w^n_+}=S_{w_+}
\]
such that 
\begin{equation}\label{eq:connecting}
\phi_{-\nabla h}^{t_i}\left(\ev_+(w^i)\right)=\ev_-(w^{i+1})\,,\qquad 1\leq i\leq n-1	
\end{equation}
for some $t_i\geq 0$ and
\[
\ev_-(w^1)\in W^u\big(w_-,h\big)\,,\qquad \ev_+(w^n)\in W^s\big(w_+, h\big)\,.
\]
To avoid having trivial cascades, we assume in addition $S_{w^i_-}\neq S_{w^i_+}$. 
As usual we denote $w^i_\pm=([u^i_\pm,\bar u^i_\pm],\eta^i_\pm)$ and $w^i=(u^i,\eta^i)$. The positive end of $u^i$ and the negative end of $u^{i+1}$ match up to timeshift, so we can glue them topologically and observe that in $\Gamma_E$
\[
	\begin{split}
		[\bar u_- \# u^1\#u^2\#\cdots u^n\# (\bar u_+)^\mathrm{rev}] &= [\bar u_+^1 \# u^2\#\cdots u^n\# (\bar u_+)^\mathrm{rev}] \\
		&= [\bar u_-^2 \# u^2\#\cdots u^n\# (\bar u_+)^\mathrm{rev}] \\
		&= \cdots =0\,.
	\end{split}
\]
By \eqref{eq:moduli_dim_S}, the virtual dimension of this moduli space is
\[
\mathrm{virdim\,} \widehat\MM^n(w_-,w_+,\AA^\tau_f,J) = \mu_\RFH^h(w_-)-\mu_\RFH^h(w_+)+n-1\,.
\]
\begin{rem}\label{rem:zero_cascade}
In contrast to the general case in Morse-Bott homology, we may ignore the case of flow lines without cascade, i.e.~negative gradient flow lines of $h$ connecting critical points of $h$, in our situation. Indeed there are exactly two flow lines of $-\nabla h$ from the maximum point to the minimum point on each component of $\Crit\AA^\tau_f$ having opposite signs, see \cite{BO09}, since every orbit is good, i.e.~the Conley-Zehnder indices of any simple periodic Reeb orbit and all its iterates have the same parity, as computed in Section \ref{sec:index}.
\end{rem}

There is a free $\R^n$-action given by translating each $w^i$ in the $s$-direction as defined in \eqref{eq:translation}. We denote the quotient space by 
\[
\MM^n(w_-,w_+,\AA^\tau_f,J):=\widehat\MM^n(w_-,w_+,\AA^\tau_f,J)/\R^n\,.
\]
If the moduli spaces $\widehat\MM^n(w_-,w_+,\AA^\tau_f,J)$ are cut out transversely, then 
\[
\MM(w_-,w_+,\AA^\tau_f,J) := \bigcup_{n\in\N} \MM^n(w_-,w_+,\AA^\tau_f,J)
\]
is a smooth manifold of dimension  
\[
\dim\MM(w_-,w_+,\AA^\tau_f,J)=\mu_\RFH^h(w_-)-\mu_\RFH^h(w_+)-1\,.
\]

Like other Floer theories, modulis spaces $\MM(w_-,w_+,\AA^\tau_f,J)$ enjoy transversality and compactness properties with a suitable choice of $J$ as we will establish in this section. To discuss orientations, we assume for the moment that $J$ is such. We choose orientations on $\widehat\MM(S_{w_-},S_{w_+},\AA^\tau_f,J)$, $W^u(w,h)$, $W^s(w,h)$, and $S_w\x\R$ coherent under gluing operations. These induce an orientation on $\widehat\MM^n(w_-,w_+,\AA^\tau_f,J)$ by the fibered sum rule in Remark \ref{rem:ori_rule}. We refer to Section \ref{sec:fredholm} and Section \ref{sec:orientation}, and to  \cite{BO09,DL19} for details. 
For $\mu_\RFH^h(w_-)-\mu_\RFH^h(w_+)=1$, the space $\MM(w_-,w_+,\AA^\tau_f,J)$ is a finite set. By assigning a sign $\epsilon({\mathbf{w}})=\{-1,+1\}$ to $\mathbf{w}=(w^1,\dots,w^n)\in \widehat\MM^n(w_-,w_+,\AA^\tau_f,J)$ so that $\epsilon({\mathbf{w}})(\p_s w^1,\dots,\p_s w^n)$ coincides with the orientation on $\widehat\MM^n(w_-,w_+,\AA^\tau_f,J)$, we obtain the signed count 
\begin{equation}\label{eq:sign_count_rfh}
\# \MM(w_-,w_+,\AA^\tau_f,J)\in\Z\,.	
\end{equation}

\subsubsection*{Projection of solutions of the Rabinowitz Floer equation}
For $J=\begin{pmatrix}
 i & B\\
 0 & j
\end{pmatrix}\in\JJ^\BB$ the projection $\wp:E\to M$ is $(J,j)$-holomorphic, i.e.~$j\circ d\wp=d\wp\circ J$. In particular, every solution $w=(u,\eta)$ of the Rabinowitz Floer equation \eqref{eq:Floer_eqn_for_A} with respect to $J$ projects to a solution $\Pi(w)=\wp\circ u:\R\x S^1\to M$ of the Floer equation \eqref{eq:Floer_eq_M} with respect to $j$. That is, we have 
\begin{equation}\label{eq:proj_cylinder}
	\begin{split}
\Pi: \widehat\MM(S_{w_-},S_{w_+},\AA^\tau_f,J)&\longrightarrow \widehat\NN(\Pi(w_-),\Pi(w_+),\a_f,j)\\
w=(u,\eta)\;&\longmapsto \;\Pi(w)=\wp\circ u\,.
\end{split}
\end{equation}
Let us assume that $\mu_\RFH^h(w_-)-\mu_\RFH^h(w_+)=1$ for $w_\pm\in\Crit h$ with $\w(w_-)=\w(w_+)$. We point out that every cascade in $\widehat\MM^n(w_-,w_+,\AA^\tau_f,J)$ projects to a nontrivial solution of \eqref{eq:Floer_eq_M} in $M$ since a cascade lying inside a fiber of $E\to M$ would intersect the zero section, contradicting $\w(w_-)=\w(w_+)$, see Proposition \ref{prop:positivity_of_intersection} below.  By \eqref{eq:index_h} and Lemma \ref{lem:one-to-one_index}, we have $\mu_\FH(\Pi(w_-))-\mu_\FH(\Pi(w_+))\leq 2$. This implies that  $\widehat\MM^n(w_-,w_+,\AA^\tau_f,J)=\emptyset$ for $n\geq3$, i.e. 
\begin{equation}\label{eq:no_3}
\widehat\MM(w_-,w_+,\AA^\tau_f,J)=\widehat\MM^1(w_-,w_+,\AA^\tau_f,J)\sqcup \widehat\MM^2(w_-,w_+,\AA^\tau_f,J)	\,,
\end{equation}
and we even have $\widehat\MM^2(w_-,w_+,\AA^\tau_f,J)=\emptyset$ if $\mu_\FH(\Pi(w_-))-\mu_\FH(\Pi(w_+))\leq 1$. In fact, this is true without the assumption $\w(w_-)=\w(w_+)$, see Section \ref{sec:two_filtrations}.

A notable distinction between $\JJ_\mathrm{diag}$ and $\JJ^\BB$ is the presence of a free $S^1$-action on $\widehat\MM(S_{w_-},S_{w_+},\AA^\tau_f,J)$ for $J\in\JJ_\mathrm{diag}$ induced by the $U(1)$-action on $E$, namely 
\begin{equation}\label{eq:rotating_cylinder}
\widehat\MM(S_{w_-},S_{w_+},\AA^\tau_f,J) \to \widehat\MM(S_{w_-},S_{w_+},\AA^\tau_f,J)\,,\quad (u,\eta)\mapsto (e^{2\pi i\theta}u,\eta),\quad\theta\in S^1	\,.
\end{equation}
We will see in Theorem \ref{thm:bijection} that the projection \eqref{eq:proj_cylinder} is even bijective modulo this $S^1$-action provided that $\w(w_-)=\w(w_+)$ and $\tau>0$ is small enough.

\subsection{Parallel transport of Rabinowitz-Floer cylinders}
We review an idea  \cite[Section 5]{Fra08} of parallel-transporting the cylinder component of a solution of the Rabinowitz-Floer equation \eqref{eq:Floer_eqn_for_A} into a single fiber of $\wp:E\to M$.

Given a path $y:\R\to M$, the connection 1-form $\alpha$ induces the parallel transport isomorphisms
\begin{equation}\label{eq:parallel}
P_y^t:E_{y(0)}\longrightarrow E_{y(t)}\,,\qquad  t\in\R\,.	
\end{equation}
Let $(u,\eta)$ be a finite energy solution of \eqref{eq:Floer_eqn_for_A}.
We parallel transport $u:\R\x S^1\to E$ along $q=\wp\circ u$ to a map $u^0:\R\x S^1\longrightarrow E_{q(0,0)}$. 
More precisely, we set
\begin{equation}\label{eq:u^0}
\begin{split}
u^0&:\R\x S^1\longrightarrow E_{q(0,0)}\\
 u^0&(s,t):=\exp{\left(-2\pi i t\int_{D^2}\overline{q}_s^*(m\om)\right)}(P^{s}_{q(\cdot,0)})^{-1}(P^{t}_{q(s,\cdot)})^{-1}u(s,t)\,,	
\end{split}
\end{equation}
where $\bar q_s:D^2\to M$ is a capping disk of $q(s,\cdot):S^1\to M$ formed by adding the negative asymptotic end of $q$, which is a critical point of $f$, to the half-infinite cylinder $q((-\infty,s],S^1)$, i.e.
\[
\bar q_s(e^{\sigma+2\pi it}):=q(\sigma+s,t)\,,\qquad \bar q_s(0)=\lim_{\sigma\to-\infty}q(\sigma,\cdot)\in\Crit f\,. 
\]
We note that $u^0$ is indeed $1$-periodic in $t$ since $-m[\om]=c_1^E$.
We define two smooth functions $\chi_1,\chi_2:\R\x S^1\to\R$ by
\begin{equation}\label{eq:chi}
	\begin{split}
\chi_1(s,t)&:=\int_0^tm\om(\p_sq,\p_tq)\,dt-t\frac{d}{ds}\int_{D^2}\overline{q}_s^*(m\om)\\[1ex]
\chi_2(s,t)&:=-\int_{D^2}\overline{q}_s^*(m\om)\,.
\end{split}
\end{equation}
As in \eqref{eq:splitting_TE} and \eqref{eq:R=2pi_i} we identify $T^\mathrm{v}_xE=E_{\wp(x)}$ and $
R(x)=2\pi i x$ for $x\in E$. 
A straightforward computation yields 
\begin{equation}\label{eq:derivatives_u^0}
	\begin{split}
	\p_su^0&=\exp{\left(-2\pi it\int_{D^2}\overline{q}_s^*(m\om)\right)}(P^{s}_{q(\cdot,0)})^{-1}(P^{t}_{q(s,\cdot)})^{-1}\partial^\mathrm{v}_su+\chi_1R(u^0)\,,\\
	\p_tu^0&=\exp{\left(-2\pi it\int_{D^2}\overline{q}_s^*(m\om)\right)}(P^{s}_{q(\cdot,0)})^{-1}(P^{t}_{q(s,\cdot)})^{-1}\partial^\mathrm{v}_tu+\chi_2R(u^0)\,,\\
	\end{split}
\end{equation}
where $\p^\mathrm{v}$ denotes derivatives in the vertical direction.

Let $J=\begin{pmatrix}
 i & 0\\
 0 & j
\end{pmatrix}\in\JJ_\mathrm{diag}$. Then $(u,\eta)$ is a solution of the Rabinowitz-Floer equation \eqref{eq:Floer_eqn_for_A}, see also \eqref{eq:Floer_eqn_for_A_again}, if and only if $q=\wp\circ u$ is a solution of the Floer equation \eqref{eq:Floer_eq_M} and it holds that
\begin{equation}\label{eqn:Floer_eqn_fiber1}
\overline{\p}^\mathrm{v}(u,\eta):=\left(
\begin{aligned}
&\p_s^\mathrm{v}u+i(\p_t^\mathrm{v} u + m(\eta+f(q))R(u))\\[.5ex]
&\p_s\eta-\int_0^1\mu_\tau(u)\,dt
\end{aligned}\right)=0\,.
\end{equation}
Due to \eqref{eq:derivatives_u^0}, equation \eqref{eqn:Floer_eqn_fiber1} translates into
\begin{equation}\label{eqn:Floer_eqn_fiber}
\overline{\p}^0(u^0,\eta):=\left(
\begin{aligned}
&\p_su^0-\chi_1R(u^0)+i\big(\p_tu^0+(-\chi_2+m\eta+mf(q))R(u^0)\big)\\[.5ex]
&\p_s\eta-\int_0^1\mu_\tau(u^0)\,dt
\end{aligned}\right)=0\,.
\end{equation}
Note that the maps $\overline{\p}^\mathrm{v}$ and $\overline{\p}^0$ depend on $f$, $q$, and $\tau$.

It is also possible to explicitly recover $(u,\eta)$ from $q$ and $(u^0,\eta)$. If $q:\R\x S^1\to M$ is a finite energy solution of \eqref{eq:Floer_eq_M} with respect to $j$ and $(u^0,\eta)$ is a solution of \eqref{eqn:Floer_eqn_fiber}, the map 
\begin{equation}\label{eq:u_from_u0}
u(s,t)=\exp{\left(2\pi it\int_{D^2}\overline{q}_s^*(m\om)\right)}P^{t}_{q(s,\cdot)}P^{s}_{q(\cdot,0)}u^0(s,t)
\end{equation}
together with $\eta$ is a solution of \eqref{eq:Floer_eqn_for_A} with respect to $J=\begin{pmatrix}
 i & 0\\
 0 & j
\end{pmatrix}$.
Using $R(x)=2\pi i x$, we recast the first equation in \eqref{eqn:Floer_eqn_fiber} as 
\begin{equation}\label{eqn:eqn_for_u^0}
\p_su^0+i\p_tu^0+2\pi\big((\chi_2-m\eta-mf(q))\mathrm{id}-\chi_1i\big)u^0=0\,.
\end{equation}
Now let $J\in\JJ$, and let $(u,\eta)$ be solution of the Rabinowitz-Floer equation \eqref{eq:Floer_eqn_for_A} with this $J$. Since $J$ is still diagonal near $\OO_E$, the above expressions \eqref{eqn:Floer_eqn_fiber1} and \eqref{eqn:Floer_eqn_fiber} remain valid in a subset of $\R\x S^1$ which is mapped near to $\OO_E$ under $u$.

\subsection{Positivity of intersections}

\begin{prop}\label{prop:positivity_of_intersection}
Let $(u,\eta)\in\widehat\MM(S_{w_-},S_{w_+},\AA^\tau_f,J)$ with $J\in\JJ$. Then the intersection number of $u$ with the zero section $\OO_E$ is nonnegative, i.e.~$u\cdot \OO_E=\w(w_+)-\w(w_-)\geq 0$. Moreover, $u\cdot\OO_E=0$ if and only if $u$ does not intersect $\OO_E$.

In particular, if $\mathbf{w}=(w^1,\dots,w^n)\in\widehat\MM^n({w_-},{w_+},\AA^\tau_f,J)$, then 
\[
\w(w_-)\leq\w(w_-^2)\leq  \cdots \leq \w(w_-^n)\leq \w(w_+)
\]
where $w_-^i$ are critical points of $\AA^\tau_f$ involved in the definition of the moduli space, see \eqref{eq:m-cascade}. Thus, if $\w(w_-)=\w(w_+)$, then for all $w^i=(u^i,\eta^i)$ we have  $u^i\cdot\OO_E=0$.
\end{prop}

\begin{proof}
This follows from \cite[Proposition 5.1]{Fra08}.  The equality $u\cdot\OO_E=\w(w_+)-\w(w_-)$ is shown in Proposition \ref{prop:winding}.(d).

Suppose that there is $(s_0,t_0)\in\R\x S^1$ such that $u(s_0,t_0)\in\OO_E$. Since $J\in\JJ$ is of the form $J=\begin{pmatrix}
 i & 0\\
 0 & j
\end{pmatrix}$
near $\OO_E$, the map  $u^0:\R\x S^1\to E_{q(0,0)}$ defined in \eqref{eq:u^0} satisfies equation \eqref{eqn:eqn_for_u^0} on an open neighborhood of $(s_0,t_0)$ in $\R\x S^1$. 
Then the local intersection number of $u$ with $\OO_E$ at $(s_0,t_0)$ agrees with the local winding number of $u^0$ at $(s_0,t_0)$. Thanks to the Carleman similarity principle, see \cite[Section 2.3]{MS12}, we can transform a solution of \eqref{eqn:eqn_for_u^0} to a holomorphic function. This implies that the local winding number of $u^0$ is always positive. Therefore the local intersection number of $u$ with $\OO_E$ at $(s_0,t_0)$ is positive, and we have $u\cdot\OO_E>0$. 

The claim for the moduli space of flow lines with cascades follows from its definition. 
\end{proof}

\subsection{Cyclic group actions}\label{sec:cyclic}

In this section, we denote by $E^m$ a complex line bundle with $c_1^{E^m}=-m[\om]$ and by $\Sigma^m$ the corresponding principal $S^1$-bundle, to indicate the degree of bundles. We have the commutative diagrams 
\begin{equation}\label{eq:triangle_diagram}
	\begin{tikzcd}[column sep=1em]
	E^1  \arrow{dr}[left]{\wp\;} \arrow{rr}{\wp^m}  &&  E^m \arrow{dl}[right]{\;\wp} \\
	& M &
	\end{tikzcd}
	\qquad\qquad 
	\begin{tikzcd}[column sep=1em]
	\Sigma^1  \arrow{dr}[left]{\wp\;} \arrow{rr}{\wp^m}  &&  \Sigma^m \arrow{dl}[right]{\;\wp}  \\
	& M &
	\end{tikzcd}	
\end{equation}
where the horizontal map on the left is the holomorphic projection 
\begin{equation}\label{eq:hol_proj}
\wp^m:E^1\to E^m\cong (E^1)^{\otimes m}\,,\qquad x\mapsto x\otimes \cdots\otimes x\,.	
\end{equation}
This is an $m$-fold covering map away from the zero sections. Its deck transformation group is the cyclic subgroup $\Z_m$ of $U(1)$ acting on $E^1\setminus\OO_{E^1}$. We choose a Hermitian metric on $E^1$ and endow $E^m\cong(E^1)^{\otimes m}$ with the induced metric. 
Then the connection 1-forms $\alpha$ and $\tilde\alpha$ on $E^m\setminus\OO_{E^m}$ and $E^1\setminus\OO_{E^1}$ respectively, defined in \eqref{eq:alpha_E}, satisfy $(\wp^m)^*\alpha=m\tilde\alpha$ on $E^1\setminus\OO_{E^1}$.
Restricting $\wp^m:E^1\to E^m$ to the respective circle bundles, we obtain the diagram on the right-hand side in \eqref{eq:triangle_diagram}.  
In particular, $\Sigma^1/\Z_m$ viewed as a principal $S^1$-bundle over $M$, is isomorphic to $\Sigma^m$.
Conversely, we can recover \eqref{eq:hol_proj} from the projection $\wp^m:\Sigma^1\to\Sigma^1/\Z_m$. Indeed, in view of Section \ref{sec:line}, the map $\wp^m$ in \eqref{eq:hol_proj} is isomorphic to 
\[
\Sigma^1\x_{\rho}\C \to \Sigma^1/\Z_m\x_\rho \C\,,\qquad [x,z]\mapsto [\wp^m(x),z^m]
\]
as maps between complex line bundles since 
\[
\begin{split}
\Sigma^1\x_\rho\C^{\otimes m}\;\; &\stackrel{\cong}{\longrightarrow} \;\; \Sigma^1\x_{\rho_m}\C \;\;\, \stackrel{\cong}{\longrightarrow} \;\; \Sigma^1/\Z_m\x_\rho \C	\\
[x,z_1\otimes\cdots\otimes z_m]&\longmapsto [x,z_1\cdots z_m] \longmapsto [\wp^m(x),z_1\cdots z_m]\,,
\end{split}
\]
where 
\[
\rho_m:S^1\x \Sigma^1\x \C \to  \Sigma^1\x \C\,,\qquad  \rho_m(t,x,z):=(t\cdot x,e^{-2\pi m it} z)\,.
\]
The fiberwise radial coordinates $\tilde r$ and $r$ on $E^1$ and $E^m$ respectively satisfy $r\circ\wp^m=\tilde r^m$. We pullback the symplectic form $\Omega$ and the functions $\mu_\tau$ and $F$ defined on $E^m$ to $E^1$ by $\wp^m$, and denote 
\begin{equation}\label{eq:tilde_om}
\widetilde\Omega:=(\wp^m)^*\Omega=\wp^*\omega+d(m\pi\tilde r^{2m}\tilde\alpha)	
\end{equation}
and
\begin{equation}\label{eq:tilde_mu}
	\tilde{\mu}_\tau:=\mu_\tau\circ\wp^m=m\pi \tilde{r}^{2m}-\tau\,,\quad \quad\widetilde F:=F\circ\wp^m=(1+m\pi\tilde{r}^{2m})f\circ\wp\,.
\end{equation}
In particular, we still have $\widetilde\Omega(A)=\om(A)$ for $A\in\pi_2(E^1)\cong\pi_2(M)$.
We also denote by $\widetilde\AA^{\tau}_f$ the Rabinowitz action functional on $E^1$ defined with these ingredients. Like this, for distinction, we add ``tildes'' to objects for $E^1$. We will see below that critical points and gradient flow lines of $\widetilde\AA^{\tau}_f$ for $E^1$ correspond to those of $\AA^\tau_f$ for $E^m$. 
To this end, we note 
\begin{equation}\label{eq:dwp}
	X_{\tilde \mu_\tau}=-\widetilde R\,,\;\;\; X_{\mu_\tau}=-mR\,,\;\;\; d\wp^m X_{\tilde \mu_\tau}=X_{\mu_\tau}\,,\;\;\; d\wp^m X_{\widetilde F}=X_F\,.
\end{equation}
We define winding number, indices, and moduli spaces for $\widetilde\AA^\tau_f$ in the same way as those for $\AA^\tau_f$. Computations made for the action and the index of critical points of $\AA^\tau_f$ go through also for those of $\widetilde\AA^\tau_f$.

\begin{lem}\label{lem:cyclic1}
For any $k\in\Z$, the map
\[
\begin{split}
\Pi^m:\{ \tilde w\in\Crit\widetilde\AA^{\tau}_f \mid  \w(\tilde w)=k\} &\longrightarrow \{ w\in \Crit\AA^\tau_f \mid \w(w)=mk\} \\[0.5ex]
\tilde w=([\tilde u,{\bar{\tilde u}}_\mathrm{fib}\#s], \eta) &\longmapsto 	w=([u:=\wp^m\circ\tilde u,\bar{u}_\mathrm{fib}\#s],\eta)\,.
\end{split}
\]
is surjective and $m$-to-1, where $[s]\in \pi_2(E^1)\cong \pi_2(E^m)\cong\pi_2(M)$. Moreover it holds
\[
\widetilde\AA^{\tau}_f(\tilde w)=\AA^\tau_f(w)\,,\qquad \mu_\RFH(\tilde w)-\mu_\RFH(w) = 2(\w(w)-\w(\tilde w))=2(m-1)k\,.
\]
\end{lem}
\begin{proof}
	A straightforward computation using Proposition \ref{prop:winding} and \eqref{eq:dwp} shows that if $\tilde w\in\Crit\widetilde\AA^{\tau}_f$ with $\w(\tilde w)=k$, then $w\in\Crit\AA^{\tau}_f$ with $\w(w)=mk$. For a given $w=([u,\bar u_\mathrm{fib}\#s],\eta)\in\Crit\AA^\tau_f$ with $\w(w)=mk$, the covering number $\cov(u)=\w(w)+m\om([s])$ is a multiple of $m$, and there exist exactly $m$ lifts $\tilde u$ of $u$ with $\cov(\tilde u)=\frac{1}{m}\cov(u)$. Moreover $\tilde w$ defined as in the statement is indeed a critical point of $\widetilde\AA^{\tau}_f$ with $\w(w)=k$. The claim on action and index follows immediately from \eqref{eq:action_comparision} and  \eqref{eq:indices_and_winding} since $\Pi(\tilde w)= \Pi(w)$.
	\end{proof}
For $w\in\Crit\AA^\tau_f$, we denote
\begin{equation}\label{eq:z_m}
(\Pi^m)^{-1}(w)=\{\tilde w^{i} \mid i\in\Z_m\}\,,\qquad j\cdot \tilde w^i= \tilde w^{i+j}\,,\quad j\in\Z_m\,.	
\end{equation}
If we view $\Pi^m$ as a map between $S^1$-families instead, i.e.~$S_{\tilde w}\mapsto S_w$, then it is bijective and compatible with the bijection in Lemma \ref{lem:one-to-one}. 
Following our convention, $\JJ$ is the space of almost complex structures on $E^m$ defined in \eqref{eq:JJ} and $\widetilde\JJ$ denotes the corresponding space defined for $(E^1,\widetilde\Omega)$. 
For a given $J\in\JJ$, there exists $\widetilde J\in\widetilde\JJ$ such that $\wp ^* J=\widetilde J$ in the following sense. We can lift $J$ to $\widetilde J$ away from $\OO_{E^1}$ since $\wp^m$ is a local diffeomorphism. Since $J$ is diagonal of the form in \eqref{eq:diagonal_J} near $\OO_{E^m}$ and $\wp^m$ is holomorphic,  $\widetilde J$ is also diagonal near $\OO_{E^1}$ and thus extends over $\OO_{E^1}$. Conversely for every $\widetilde J\in\widetilde\JJ$ which is invariant the $\Z_m$-action, there exists $J\in\JJ$ such that $\wp^*J=\widetilde J$. Therefore, if we denote by
\[
\widetilde\JJ^{\Z_m}\subset \widetilde\JJ
\] 
the subset of $\Z_m$-invariant elements, there is a bijection between $\JJ$ and $\widetilde\JJ^{\Z_m}$.

\begin{lem}\label{lem:cyclic2}
Let $J$ and $\widetilde J$ be as above, in particular $\widetilde J$ is $\Z_m$-invariant. Let $w_\pm\in\Crit{\AA}^\tau_f$ with $\w(w_-)=\w(w_+)\in m\Z$. Then for fixed $i\in\Z_m$, there is a bijection
\[
\begin{split}
	\Pi^m:\Big\{\tilde w\in\widehat\MM(S_{\tilde w_-},S_{\tilde w_+},\widetilde{\AA}^{\tau}_f,\widetilde J) \mid \ev_-&(\tilde w)=\tilde w_-^i\,,\;\ev_+(\tilde w)=\tilde w_+^j \textrm{ for some } j\in\Z_m \Big\}\\
	&\longrightarrow \left\{w\in\widehat\MM(S_{w_-},S_{w_+},{\AA}^\tau_f, J)\mid \ev_\pm(w)=w_\pm \right\} \\
	\tilde w=(\tilde u,\eta) &\longmapsto w=(u:=\wp^m\circ\tilde u,\eta)\,,
\end{split}
\]
where $\tilde w_-^i$ and $\tilde w_+^j$ are given as in \eqref{eq:z_m}. 
Moreover if both spaces are cut out transversely, then $\Pi^m$ is a diffeomorphism.
\end{lem}
\begin{proof}
Let $\tilde w=(\tilde u,\eta)$ be an element in the domain of $\Pi^m$. Since $\w(\tilde w_-)=\w(\tilde w_+)$ by the assumption $\w(w_-)=\w(w_+)$ and Lemma \ref{lem:cyclic1}, Proposition \ref{prop:positivity_of_intersection} yields that $\tilde u$ is contained in $E^1\setminus\OO_{E^1}$ on which the $\Z_m$-action is free. We set $u:=\wp^m\circ\tilde u$. 
 Applying \eqref{eq:dwp} to \eqref{eq:Floer_eqn_for_A}, one can readily verify that $\Pi^m$ is well-defined, i.e.~$(u,\eta)$ is a solution of the Rabinowitz-Floer equation for $\AA^\tau_f$. Since no nontrivial element in $\Z_m$ fixes $\tilde w^i_-$, the map $\Pi^m$ is injective. To show surjectivity, we pick $w=(u,\eta)$ from the target space of $\Pi^m$,  and observe that $u$ does not intersect $\OO_{E^m}$ again by Proposition \ref{prop:positivity_of_intersection}. We also note that $u_\#(\pi_1(\R\x S^1,z))\subset \wp^m_\#(\pi_1(E^1\setminus\OO_{E^1},\varepsilon))$ when considered as subgroups of $\pi_1(E^m\setminus\OO_{E^m},\wp^m(\varepsilon))$, where $z$ and $\varepsilon$ are base points satisfying $u(z)=\wp^m(\varepsilon)$. Therefore there exist exactly $m$ lifts of $u$, say $\tilde u^{1},\dots, \tilde u^m$, of $u$ such that  $(\tilde u^\ell,\eta)$ converges to $\tilde w^\ell_-$ at the negative end for every $\ell\in\Z_m$. Moreover  $(\tilde u^{\ell},\eta)$ for every $\ell\in\Z_m$ is a solution of the Rabinowitz-Floer equation for $\widetilde\AA^{\tau}_f$ and $\widetilde J$ since $\wp^m:E^1\setminus\OO_{E^1}\to E^m\setminus\OO_{E^m}$ is a local diffeomorphism satisfying \eqref{eq:dwp}. This proves that $\Pi^m$ is bijective. The last claim follows also from the fact that $\wp^m$ is a local diffeomorphism away from the zero section.
\end{proof}

Let $w_\pm\in\Crit \AA^\tau_f$ with $\w(w_-)=\w(w_+)$, and let $\widetilde w_\pm\in\Crit\widetilde\AA^{\tau}_f$ be lifts of $w_\pm$ as in Lemma \ref{lem:cyclic1}. Since the $\Z_m$-action on $E^1\setminus\OO_{E^1}$ is free, the transversality problem for $\widehat\MM(S_{\tilde w_-},S_{\tilde w_+},\widetilde{\AA}^{\tau}_f,\widetilde J)$ for invariant $\widetilde J$ is equivalent to that for   $\widehat\MM(S_{w_-},S_{w_+},{\AA}^\tau_f, J)$. Thus if one of these moduli spaces is cut out transversely, then so is the other. If this is the case, the map
\[
\widehat\MM(S_{\tilde w_-},S_{\tilde w_+},\widetilde{\AA}^{\tau}_f,\widetilde J)\to \widehat\MM(S_{w_-},S_{w_+},{\AA}^\tau_f, J) \,,\quad (\tilde u,\eta)\mapsto (\wp^m\circ\tilde u,\eta)
\]
is an $m$-fold covering map and a local diffeomorphism.

\subsection{Transversality with $J\in\JJ$ and $J\in\JJ^\BB$}\label{sec:J_in_JJ}
We fix a complex line bundle $E=E^m$ over $M$ with $c_1^{E}=-m[\om]$ for arbitrary $m\in\N$.
For $J\in\JJ$ and $A\in\pi_2(E)$, we consider the moduli space
\[
\MM(A,J):=\{(\eta,t,v)\in  \R\x S^1\x C^\infty(S^2,E)\mid\text{$v$ simple $J_{(\eta,t)}$-holomorphic, $[v]=A$}\big\}
\]
and the subspace 
\begin{equation}\label{eq:MMAJ}
\MM^*(A,J):=\{(\eta,t,v)\in  \MM(A,J) \mid  v(S^2) \cap (E\setminus\OO_E)\neq\emptyset \big\}\,.	
\end{equation}

\begin{lem}\label{lem:j-sphere}
Let $v:S^2\to E$ be a nonconstant $J_{(\eta,t)}$-holomorphic sphere for some $J\in\JJ$ and $(\eta,t)\in\R\x S^1$. Then the image of $v$ is either entirely contained in the zero section $\OO_E$ or intersects the region in \eqref{eq:support_B} on which $J$ is not diagonal.
\end{lem}
\begin{proof}
	We note that $v$ necessarily intersects $\OO_E$ since the symplectic form $\Omega$ is exact on $E\setminus\OO_E$. Suppose that $v$ does not pass through the region in \eqref{eq:support_B}. Then the maximum principle implies that $v$ is contained inside the zero section $\OO_E$. 
\end{proof}

\begin{prop}\label{prop:transversality_sphere}
	There exists a residual subset $\JJ_\mathrm{r}\subset\JJ$ such that for every $J\in \JJ_\mathrm{r}$ the moduli space $\MM^*(A,J)$ is a smooth manifold of dimension $\dim E+2c_1^{TE}(A)+2$. 
\end{prop}
\begin{proof}
	According to Lemma \ref{lem:j-sphere}, for every $(\eta,t,v)\in \MM^*(A,J)$, the sphere $v$ passes through the region in \eqref{eq:support_B} on which we may freely perturb $J$. Therefore the claim follows from standard arguments, see e.g.~\cite[Chapter 3]{MS12}.
\end{proof}

\begin{rem}
It is not reasonable to expect a transversality result for $J_{t}$-holomorphic spheres contained in $\OO_E$ if $J_t=\begin{pmatrix}
 i & 0\\
 0 & j_t
\end{pmatrix}$ near $\OO_E$. Indeed, such spheres can also be interpreted as $j_t$-holomorphic spheres in $M$. The Fredholm index for a $j_t$-holomorphic map $s:S^2\to M$ is
\[
\dim M+2c_1^{TM}([s]) \,.
\]
 However, if we view the map as a $J_{t}$-holomorphic map $s:S^2\to \OO_E\subset E$, then its index is
\[
\dim E+2c_1^{TE}([s])=\dim M+2+2c_1^{TM}([s])-m\,\om([s]).
\]
Unless $2=m\,\om([s])$, these two indices do not agree, and therefore we cannot achieve transversality on $E$ for this class of almost complex structures. 
\end{rem}

\begin{prop}\label{prop:transversality_avoiding}
	There exists a residual subset $\JJ_\textrm{e}\subset\JJ$  such that for every $J\in\JJ_\textrm{e}$ the following hold for all $([u,\bar u],\eta)\in\Crit h\subset\Crit\AA^\tau_f$. 
	\begin{enumerate}[(i)]
	\item For every $t\in S^1$, the point $u(t)$ does not lie in the image of any $J_{(\eta,t)}$-holomorphic sphere $v:S^2\to E$ with $c_1^{TE}([v])\leq 1$.
	\item For every $(\lambda,t)\in\R\times S^1$, the image of $u$ does not intersect the image of any $J_{(\lambda,t)}$-holomorphic sphere $v:S^2\to E$ with $c_1^{TE}([v])\leq 0$.
	\end{enumerate}
\end{prop}
\begin{proof}
	The proof goes along similar lines as the one for \cite[Theorem 3.1]{HS95}. To see (i), we note that if $u(t)$ lies in the image of a $J_{(\eta,t)}$-holomorphic sphere $v$, then $v$ has to pass through the region \eqref{eq:support_B}, where we may freely perturb $J$. According to Proposition \ref{prop:transversality_sphere}, the space $\MM^*(A,J)\x_{G} S^2$,  where $G=\mathrm{PSL}(2;\C)$, has dimension  $\dim E+2c_1^{TE}(A)-2$ for $J\in\JJ_\mathrm{r}$. We consider the evaluation map
	\[
	\begin{split}
	\MM^*(A,J)\x_{G} S^2 \x \Crit h \x S^1  &\longrightarrow E\x E \x \R\x \R	\x S^1\x S^1\\
	\big([(\lambda,\mathfrak{t}, v),z],([u,\bar u],\eta),t\big) &\longmapsto (v(z),u(t),\lambda,\eta,\mathfrak{t},t)
	\end{split}
	\] 
	and observe that statement (i) follows once we prove that the preimage of the diagonal is empty. A simple computation shows that the preimage of the diagonal has virtual dimension at most $-1$ provided $c_1^{TE}(A)\leq1$. 
	
	For (ii), we look at another evaluation map  
	\[
	\begin{split}
	\MM^*(A,J)\x_{G} S^2 \x \Crit h \x S^1  &\longrightarrow E\x E \\
	\big([(\lambda,\mathfrak{t}, v),z],([u,\bar u],\eta),t\big) &\longmapsto (v(z),u(t))
	\end{split}
	\] 
	and note that the preimage of the diagonal has negative virtual dimension if $c_1^{TE}(A)\leq0$. Hence the proposition follows for $J\in\JJ_\mathrm{r}$ making these evaluation maps transverse to the respective diagonals. We denote by $\JJ_\textrm{e}$ the residual subset of these almost complex structures.
\end{proof}

\begin{rem}
If $c_1^{TM}=\lambda\omega$ on $\pi_2(M)$, we have $c_1^{TE^1} = (\lambda-1)\wp^*\om$ on $\pi_2(E)$. If $E^1$ is either sufficiently negative monotone or positive monotone, compactness in Floer theory is easily achieved. If $E^1$ is sufficiently negative monotone, then so is $E^m$. However if $E^1$ is positive monotone, then, for $m$ in a certain range, the bundle $E^m$ is neither nor. In this range, we prove compactness by lifting Rabinowitz Floer cylinders via $\wp^m:E^1\setminus\OO_{E^1}\to E^m\setminus\OO_{E^m}$.	This requires additional transversality results for moduli spaces of Rabinowitz Floer cylinders in $E^1$ without any assumption on winding numbers.
\end{rem}

\begin{prop}\label{prop:transversality_cylinder} 
	There exists a residual subset $\JJ_\mathrm{g}\subset\JJ$ such that for every $J\in \JJ_\mathrm{g}$ and $\tilde w_\pm\in\Crit\widetilde{\AA}^\tau_f$, the spaces $\widehat\MM(S_{\tilde w_-},S_{\tilde w_+},\widetilde{\AA}^\tau_f,\widetilde J)$ are smooth manifolds of dimension $\mu_\RFH(\tilde w_-)-\mu_\RFH(\tilde w_+) +1$, where $\widetilde J=(\wp^m)^*J$ etc.~are as in the preceding section. 
\end{prop}

\begin{rem}
Due to the discussion after Lemma \ref{lem:cyclic2}, this proposition in particular implies that $\widehat\MM(S_{ w_-},S_{ w_+},{\AA}^\tau_f, J)$ with $\w(w_-)=\w(w_+)$ are also smooth manifolds of dimension $\mu_\RFH( w_-)-\mu_\RFH( w_+) +1$ for $J\in\mathcal{J}_\mathrm{g}$. 
\end{rem}
\begin{proof}
		For every $(u,\eta)\in \widehat\MM(S_{\tilde w_-},S_{\tilde w_+},\widetilde{\AA}^\tau_f,\widetilde J)$, the map $u$ passes through the region in \eqref{eq:support_B} where both asymptotic limits of $u$ are located. In particular, the image of $u$ is not contained in the fixed locus of the $\Z_m$-action, i.e.~the zero section $\OO_{E^1}$. Therefore, we can perturb $\widetilde J$ keeping it $\Z_m$-equivariant so that the moduli space is cut out transversely. We refer to \cite[Section 5c]{KS02} for details of equivariant transversality in a comparable setting. Now  the claim follows from \cite[Theorem 4.11]{AbM18}, which is a suitable adaptation of standard arguments in \cite{FHS95}, with minor modifications. The dimension computation follows from  \cite{CF09}, see Section \ref{sec:fredholm} below.
\end{proof}

We set
\[
\JJ_\mathrm{reg}:=\JJ_\mathrm{r}\cap \JJ_\mathrm{e}\cap \JJ_\mathrm{g}\,.
\]
We also show a transversality result for almost complex structures in $\JJ^\BB$. 

\begin{prop}\label{prop:transversality_J_B}
For every $j\in\mathfrak{j}_\textrm{reg}$, there exists a residual subset $\BB_\textrm{reg}(j)$ of the space $\{B\in \BB(j)\mid J^B \textrm{ is $\Omega$-tame}\}$ such that for every  
$J=\begin{pmatrix}
 i & B\\
 0 & j
\end{pmatrix}$ with $B\in \BB_\mathrm{reg}(j)$ and for every $w_\pm\in\Crit\AA^\tau_f$, the spaces $\widehat\MM(S_{ w_-},S_{ w_+},{\AA}^\tau_f, J)$ are smooth manifolds of dimension $\mu_\RFH(w_-)-\mu_\RFH(w_+) +1$.
\end{prop}
\begin{proof}
This can be proved in the same manner as in \cite{AK17} with  minor adjustments of requiring $B$ to have support away from $\OO_E$ and parametrizing it by $\R\x S^1$.
\end{proof}

Statements corresponding to Proposition \ref{prop:transversality_sphere} and Proposition \ref{prop:transversality_avoiding} also hold for generic $B\in\BB(j)$ but we do not need these here. We set
\[
\JJ^\mathcal{B}_\mathrm{reg}:=\left\{ J=\begin{pmatrix}
 i & B\\
 0 & j
\end{pmatrix} \;\Bigg|\; j\in\mathfrak{j}_\textrm{reg}\,,\; B\in \BB_\textrm{reg}(j)\right\} \,.
\]
Finally, we note that for $J\in\JJ^\mathrm{diag}$ by the maximum principle all $J_t$-holomorphic spheres lie inside $\OO_E$. A statement corresponding to Proposition \ref{prop:transversality_J_B} for $J\in\JJ^\mathrm{diag}$ is established in Section \ref{sec:trans_diag} below.

\subsection{Compactness result with $J\in\JJ$}\label{sec:compactness_J}
We use the notation introduced in Section \ref{sec:cyclic}. We point out that condition (ii) in Proposition \ref{prop:compactness1} is slightly weaker than our standing assumption (A3).

\begin{prop}\label{prop:compactness1}
Let $E=E^m$, i.e.~$c_1^E=-m[\om]$. We assume  one of the following.\begin{enumerate}[(i)]
	\item $\om$ vanishes on $\pi_2(M)$.
	\item $c_1^{TM}=\lambda\om$ on $\pi_2(M)$ for some $\lambda\in\R$ satisfying either $\lambda\nu\leq -\frac{1}{2}\dim M+1$ or $(\lambda-1)\nu\geq 1$,
	where $\nu\in\N$ is defined by $\om(\pi_2(M))=\nu\Z$. 
\end{enumerate}
Let $\tau>0$ and $J\in\JJ_\mathrm{reg}$ be arbitrary. Then for every $w_\pm=([u_\pm,\bar u_\pm],\eta_\pm)\in\Crit h$ satisfying
\[
\w(w_-)=\w(w_+)\in m\Z\,,\qquad \mu^h_\RFH(w_-)-\mu^h_\RFH(w_+)\leq 2\,,
\]  
the moduli space $\widehat\MM(w_-,w_+,\AA^\tau_f,J)$ is compact in the $C^\infty_{loc}$-topology.
\end{prop}

\begin{rem}\label{rem:reason_lifting}
 If $w=(u,\eta)\in\Crit \AA^\tau_f$ has $\w(w)\notin m\Z$, then a lift $\tilde u$ of $u$ is not a closed loop but closes up modulo the $\Z_m$-action on $E^1$. It seems feasible to weaken the hypothesis in Proposition \ref{prop:compactness1} to $\w(w_-)=\w(w_+)\in \Z$ by considering a twisted loop space on $E^1$, cf.~\cite{Ba21}.
\end{rem}

\begin{proof}
Since $J\in\JJ_\mathrm{reg}$ is diagonal outside a compact region, applying \cite[Proposition 6.2]{Fra08} we conclude that there is a uniform $L^\infty$-bound on $u$ of $(u,\eta)\in\widehat\MM(S_{w_-},S_{w_+},\AA^\tau_f,J)$. It is worth pointing out that the function $\rho(s,t):=\ln r(u(s,t))$  is a subsolution of the inhomogeneous equation $\Delta \rho +m\tau\geq 0$, as computed in \eqref{eq:Kazdan-Warner} below. Therefore the upper bound of $\rho$ given in \cite[Proposition 6.2]{Fra08} in general depends on the energy of $(u,\eta)$, i.e.~on $\AA^\tau_f(w_-)-\AA^\tau_f(w_+)$, see also \cite[Section 4]{CFO10}.

It remains to show that the derivatives of cylinder components of  $\mathbf{w}\in\widehat\MM(w_-,w_+,\AA^\tau_f,J)$ are uniformly bounded on compact subsets of $\R\x S^1$, i.e.~we have to show that there is no bubbling-off of pseudo-holomorphic spheres. This is obviously true when $\om$ vanishes on $\pi_2(M)$, and therefore let us assume (ii).

Let $(\mathbf{w}_\nu)_{\nu\in\N}$ be a sequence in $\widehat\MM(w_-,w_+,\AA^\tau_f,J)$.
The hypothesis $\w(w_-)=\w(w_+)$ together with Lemma \ref{lem:cyclic2} implies that we can lift $\mathbf{w}_\nu$ to $\tilde{\mathbf{w}}_\nu\in\widehat\MM(\tilde w_-,\tilde w_+,\widetilde{\AA}^{\tau}_f,\widetilde J)$ in $E^1$. If $m=1$, we have $\tilde{\mathbf{w}}_\nu=\mathbf{w}_\nu$. It suffices to show that bubbling-off does not occur for $(\tilde{\mathbf{w}}_\nu)_{\nu\in\N}$.
We also note that
\begin{equation}\label{eq:lift_ind}
\mu^h_\RFH(\tilde w_-)-\mu^h_\RFH(\tilde w_+)=\mu^h_\RFH( w_-)-\mu^h_\RFH( w_+)\leq2	
\end{equation}
due to Lemma \ref{lem:cyclic1}. Suppose that $\tilde{\mathbf{w}}_\nu$ converges in the Gromov-Floer topology to broken flow lines with cascades and nonconstant $\widetilde{J}_{\eta(s^n),t^n}$-holomorphic spheres $\{v^n\}_{1\leq n\leq\ell}$ in $E^1$. Here at least one $v^n$ intersects some cylinder component in the limit broken flow lines with cascades, and $(s^n,t^n)\in\R\x S^1$ is the position where $v^n$ bubbles off. Since $\om([\wp\circ v^n])=\widetilde\Omega([v^n])>0$,  assumption  (ii) yields that every $v^n$ satisfies
\[
c_1^{TE^1}([v^n]) = c_1^{TM}([\wp\circ v^n]) - \om([\wp\circ v^n]) =(\lambda-1)\om([\wp\circ v^n]) \notin\left[-\tfrac{1}{2}\dim M+1, 0\right]\,.
\]
We first assume $c_1^{TE^1}([v^n])\leq -\frac{1}{2}\dim M$ and show that this case cannot occur. Note that the underlying simple curve $v^n_\mathrm{sim}$ of $v^n$ also satisfies $c_1^{TE^1}([v^n_\mathrm{sim}])\leq -\frac{1}{2}\dim M$. By Lemma \ref{lem:j-sphere}, either $v^n$ is entirely contained in $\OO_{E^1}$ or $v^n_\mathrm{sim}$ belongs to $\MM^*(A,\widetilde J)$, see \eqref{eq:MMAJ}, where $A=[v^n_\mathrm{sim}]\in\pi_2(E^1)$. The latter case is excluded by Proposition \ref{prop:transversality_sphere}. Indeed since $\wp^m:E^1\to E^m=E$ is $(\widetilde J,J)$-holomorphic, the simple curve  $\wp^m\circ v^n_\mathrm{sim}:S^2\to E^m$ is in $\MM^*(A,J)$, where we view $A$ as a class in $\pi_2(E^m)\cong\pi_2(E^1)$, and the quotient space $\MM^*(A,J)/\mathrm{PSL}(2;\C)$ has negative dimension by $c_1^{TE^m}(A)\leq c_1^{TE^1}(A)\leq -\frac{1}{2}\dim M$.
Arguing on $E^m$ instead of $E^1$ in this case avoids addressing equivariant transversality for $\widetilde J$-holomorphic spheres in $E^1$. 

If $v^n$ lies inside $\OO_{E^1}$, then $\wp\circ v^n:S^2\to M$ is $j_t$-holomorphic since $J$ is diagonal on $\OO_{E^1}$. Moreover the simple curve $\wp\circ v^n_\mathrm{sim}$ has $c_1^{TM}([\wp\circ v^n_\mathrm{sim}])\leq -\frac{1}{2}\dim M+1$ by the assumption since $[\wp\circ v^n_\mathrm{sim}]=A\in\pi_2(M)\cong\pi_2(E^1)$. Such curve does not exist for dimension reason again, see the definitions of $\j_\mathrm{reg}$ and $\j_\mathrm{HS}$ in Section \ref{sec:Quantum Gysin sequence}. 

In the other case, namely $c_1^{TE^1}([v^n])\geq1$, we use the fact that the limit of $\tilde{\mathbf{w}}_\nu$ has index at most 2 by \eqref{eq:lift_ind} and that each $v^n$ contributes at least $2c_1^{TE^1}([v^n])\geq2$ to this index. By Proposition \ref{prop:transversality_cylinder}, the only possible scenario is that $(u_-,\eta_-)=(u_+,\eta_+)$ and the limit of $\tilde{\mathbf{w}}_\nu$ consists of a single $s$-independent cylinder mapping to $u_-=u_+$ with constant $\eta=\eta_-=\eta_+$  and a single $\widetilde{J}_{(\eta,t)}$-holomorphic sphere $v$ with $c_1^{TE^1}([v])=1$ passing through $u_-(t)$ for some $t\in S^1$. In particular, the index difference in \eqref{eq:lift_ind} equals 2, and the equation $1=c_1^{TE^1}([v])=(\lambda-1)\om([\wp\circ v])$ implies  $\lambda=2$ and $\nu=1$. In the case of $m=1$, this does not happen due to Proposition \ref{prop:transversality_avoiding}.(i). If $m\geq2$, we observe that the ${J}_{(\eta,t)}$-holomorphic sphere $\wp^m\circ v:S^2\to E^m$ intersects the image of $\wp^m\circ u_-$. However, Proposition \ref{prop:transversality_avoiding}.(ii) rules out this phenomenon since $c_1^{TE^m}([\wp^m\circ v])\leq 2-m\leq0$ in view of  $[\wp^m\circ v]=[v]$ in $\pi_2(E^m)\cong\pi_2(E^1)$. This completes the proof. 
\end{proof}

\subsection{Functional setting}\label{sec:functional}
For $\delta>0$, we fix a smooth function $\beta_\delta:\R\to\R$ such that 
\begin{equation}\label{eq:beta}
\beta_\delta = e^{\delta|s|}\qquad \forall |s|>s_0	
\end{equation}
for some $s_0>0$, and define for $k\in\N\cup\{0\}$ and $p>2$
\[
W^{k,p}_\delta(\mathbb{K},\R^n):=\{f\in W_{loc}^{k,p}(\mathbb{K},\R^n) \mid \beta_\delta f\in W^{k,p}(\mathbb{K},\R^n)\}\,
\]
where $\mathbb{K}$ is either $\R$ or $\R\x S^1$, moreover we set
\[
\|f\|_{W^{k,p}_\delta}:=\|\beta_\delta f\|_{W^{k,p}}\,.
\]
The space depends on $\delta$ but not on the choice of $\beta_\delta$. 
Other weighted Sobolev spaces, in particular with values in vector bundles, are analogously defined. As before, we write $w_\pm=([u_\pm,\bar u_\pm],\eta_\pm)\in\Crit\AA^\tau_f$. We denote by $\CC(S_{w_-},S_{w_+})$  the space of maps 
\[
w=(u,\eta) \in W^{1,p}_{loc}(\R\x S^1,E)\x W^{1,p}_{loc}(\R,\R)
\]
such that 
\[
\eta-\eta_-\in W^{1,p}_\delta((-\infty,0),\R)\,,\quad \eta-\eta_+\in W^{1,p}_\delta((0,+\infty),\R)\,,
\]
there exist $s_1>0$, $\theta_\pm\in S^1$, and $\xi_\pm \in W^{1,p}_\delta(\R\x S^1,(e^{2\pi i\theta_\pm}u_\pm)^*TE)$ satisfying
\[
u(s,t)=\left\{\begin{aligned}
	&\exp_{e^{2\pi i\theta_-}u_-(t)}(\xi_-(s,t))\qquad \forall s\leq -s_1\\
	&\exp_{e^{2\pi i\theta_+}u_+(t)}(\xi_+(s,t))\qquad \forall s\geq s_1
\end{aligned}
\right.\,,
\]
and
\begin{equation}\label{eq:u_Gamma_0}
[\bar u_-\#u\#\bar u_+^\mathrm{rev}]=0 \;\textrm{ in }\;\Gamma_E\,.	
\end{equation}
 The tangent space at $w=(u,\eta)\in\CC(S_{w_-},S_{w_+})$ can be identified with
\begin{equation}\label{eq:tangent_space}
T_w\CC(S_{w_-},S_{w_+}) = T_uW^{1,p}_\delta(\R\x S^1, u^*TE)\oplus V_-\oplus V_+ \oplus W^{1,p}_\delta(\R,\R) 	
\end{equation}
where $V_-$ and $V_+$ are 1-dimensional vector spaces spanned by $(1-\kappa(s))R(u(s,t))$ and $\kappa(s)R(u(s,t))$ respectively where $\kappa:\R\to[0,1]$ is a smooth cutoff function such that $\kappa'\geq0$, $\kappa(-1)=0$, and $\kappa(1)=1$. The spaces $V_\pm$ correspond to $TS_{w_\pm}$.

 We consider the smooth Banach bundle 
\begin{equation}\label{eq:banach_bundle}
\EE\longrightarrow \CC(S_{w_-},S_{w_+})	
\end{equation}
with fiber $\EE_w$ over $w\in \CC(S_{w_-},S_{w_+})$ given by
\[
\EE_w:=L^p_\delta (\R\x S^1, u^*TE)\x L^p_\delta(\R,\R)\,.
\]
We view $\overline{\p}_{J,\tau,f}$ given in \eqref{eq:Floer_eqn_for_A} as a section of $\EE\to\CC(S_{w_-},S_{w_+})$. Due to elliptic regularity results \cite[Appendix B]{MS12}, elements in $\overline{\p}_{J,\tau,f}^{-1}(\OO_\EE)$ are smooth, where $\OO_\EE$ denotes the zero section, and thus $\overline{\p}_{J,\tau,f}^{-1}(\OO_\EE)$ coincides with the moduli space in \eqref{eq:rfh_moduli}:
\[
\overline{\p}_{J,\tau,f}^{-1}(\OO_\EE)=\widehat\MM(S_{w_-},S_{w_+},\AA^\tau_f,J)\,.
\] 
The vertical differential of $\overline{\p}_{J,\tau,f}$ at $w\in\overline{\p}_{J,\tau,f}^{-1}(\OO_\EE)$, denoted by 
\begin{equation}\label{eq:D_w}
D_w: T_w\CC(S_{w_-},S_{w_+})\longrightarrow \EE_w\,,
\end{equation}
has the explicit form
\begin{equation}\label{eq:linop}
D_w\left(\begin{aligned}
& \hat u\\[.5ex]
& \hat\eta
\end{aligned}\right)
= 
\left(\begin{aligned}
& \nabla_s\hat u+J(\eta,t,u)(\nabla_t\hat u-\eta\nabla_{\hat u}X_{\mu_\tau}(u)-\nabla_{\hat u}X_F(u))\\ & \;\;+\nabla_{\hat u}J(\eta,t,u)(\p_t u - \eta X_{\mu_\tau}(u) - X_F(u)) -\hat\eta J(\eta,t,u) X_{\mu_\tau}(u)
\\[1ex]
& \p_s \hat\eta-\int_0^1d\mu_\tau(u)\hat u\,dt
\end{aligned}\right)\,.
\end{equation}
The operator $D_w$ is Fredholm. In particular, if $D_w$ is surjective for all $w\in\overline{\p}_{J,\tau,f}^{-1}(\OO_\EE)$, i.e.~$\overline{\p}_{J,\tau,f}\pitchfork\OO_\EE$, then $\overline{\p}_{J,\tau,f}^{-1}(\OO_\EE)$ is a smooth manifold of dimension 
\[
\ind D_w=\mu_\RFH(w_-)-\mu_\RFH(w_+)+1
\] 
as computed in Section \ref{sec:fredholm} below. This is indeed the case for $J\in\JJ_\mathrm{reg}\cup\JJ^\BB_\mathrm{reg}$ according to Proposition \ref{prop:transversality_cylinder} and Proposition \ref{prop:transversality_J_B}. A goal of the following sections is to show that this transversality result  holds also for all $J\in\JJ_\mathrm{diag}$.

Let $J=\begin{pmatrix}
 i & 0\\
 0 & j
\end{pmatrix}\in\JJ_\mathrm{diag}$. Using the splitting $T_xE\cong E_{\wp(x)}\oplus T_{\wp(x)}M$ in  \eqref{eq:splitting_TE}, we write
\[
T_w\CC(S_{w_-},S_{w_+}) = T_w^\mathrm{v}\CC(S_{w_-},S_{w_+}) \oplus T_w^\mathrm{h}\CC(S_{w_-},S_{w_+})\,,\qquad \EE_w=\EE^\mathrm{v}_w \oplus \EE^\mathrm{h}_w
\]
where 
\[
\begin{split}
T_w^\mathrm{v}\CC(S_{w_-},S_{w_+})&:= W^{1,p}_\delta(\R\x S^1,q^*E)\oplus V_-\oplus V_+\oplus W^{1,p}_\delta(\R,\R)\,,\\[0.5ex]
T_w^\mathrm{h}\CC(S_{w_-},S_{w_+})&:= W^{1,p}_\delta(\R\x S^1,q^*TM)\,,\\[0.5ex]
\EE^\mathrm{v}_w &:= L^{p}_\delta(\R\x S^1,q^*E)\x L_\delta^p(\R,\R)\,,\\[0.5ex]
\EE^\mathrm{h}_w &:= L^{p}_\delta(\R\x S^1,q^*TM)\,.
\end{split}
\]
Here we denote $q=\Pi(w)=\wp\circ u$ as usual. With respect to this splitting, we write 
\[
(\hat u,\hat\eta)=\big( (\hat e,\hat\eta) , \hat q \big)\in T_w\CC(S_{w_-},S_{w_+})
\]
where 
\[
\hat e\in W^{1,p}_\delta(\R\x S^1,q^*E)\oplus V_-\oplus V_+\,,\quad \hat\eta\in  W^{1,p}_\delta(\R,\R)\,,\quad \hat q\in W^{1,p}_\delta(\R\x S^1,q^*TM)\,.
\]
To decompose $D_w$ into vertical and horizontal part, we use the Levi-Civita connection $\nabla$ associated with a Riemannian metric on $E$ of product form with respect to the splitting $T_xE\cong E_{\wp(x)}\oplus T_{\wp(x)}M$. Keeping the notation $\nabla$ when restricted to the vertical or horizontal subspace and using $X_{\mu_\tau}(u)=-mR(u)=-2m\pi iu$, see  \eqref{eq:R=2pi_i}, we may express $D_w$ as 
\begin{equation}\label{eq:lin_block}
D_w \left(\begin{aligned}
\hat u\\
\hat\eta
\end{aligned}\right)
=
\left(\begin{aligned}
 D_w^\mathrm{v} && D_w^\mathrm{hv}\\
 0 && D_w^\mathrm{h}
\end{aligned}\right)
\left(\begin{aligned}
& (\hat e,\hat\eta) \\
& \quad\hat q
\end{aligned}\right)
\end{equation}
where 
\begin{equation}\label{eq:horizontal_diff}
\begin{split}
&D_w^\mathrm{v}:T_w^\mathrm{v}\CC(S_{w_-},S_{w_+})\longrightarrow \EE^\mathrm{v}_w	\\[0.5ex]
& D_w^\mathrm{v}\left(\begin{aligned}
& \hat e\\[.5ex]
& \hat\eta
\end{aligned}\right)
:= 
\left(\begin{aligned}
& \nabla_s \hat e + i\nabla_t\hat e -2m\pi ( \eta + f(q))\hat e - 2m\pi  \hat\eta u
\\[0.7ex]
& \p_s \hat\eta-\int_0^1d\mu_\tau(u)\hat e\,dt
\end{aligned}\right)\,,\\[3ex]
&D_w^\mathrm{h}:T_w^\mathrm{h}\CC(S_{w_-},S_{w_+})\longrightarrow \EE^\mathrm{h}_w 	\\[0.5ex]
&D_w^\mathrm{h}\hat q := \nabla_s \hat q + j(t,q)(\nabla_t\hat q-\nabla_{\hat q}X_f(q))+\nabla_{\hat q}j(t,q)(\p_t q-X_f(q))\,,\\[3ex]
&D_w^\mathrm{hv}:T_w^\mathrm{h}\CC(S_{w_-},S_{w_+})\longrightarrow \EE^\mathrm{v}_w	\\[0.5ex]
&D_w^\mathrm{hv}\hat q := m df_{q}(\hat q) R(u)\,.
\end{split}	
\end{equation}
The horizontal part $D_w^\mathrm{h}$ agrees with the operator $\mathfrak{d}_q$ in \eqref{eq:d_q} obtained by linearizing the Floer equation \eqref{eq:Floer_eq_M} on $M$. Note that in the horizontal part the weight $\delta$ does not play any role since all critical points of $\a_f$ are nondegenerate. We note that $D_w^\mathrm{hv}$ vanishes asymptotically, and $\left(\begin{aligned}
 D_w^\mathrm{v} && \varepsilon D_w^\mathrm{hv}\\
 0 && D_w^\mathrm{h}
\end{aligned}\right)
$ is a Fredholm operator for all $\varepsilon\in[0,1]$. In particular, it holds that 
\[
\ind D_w = \ind D_w^\mathrm{v} + \ind D_w^h\,.
\]
We also point out that $D_w$ is surjective if both $D_w^\mathrm{v}$ and $D_w^\mathrm{h}$ are surjective.

\begin{lem}\label{lem:index_vert}
	For every $w=(u,\eta)\in\widehat\MM(S_{w_-},S_{w_+},\AA^\tau_f,J)$, we have $(R(u),0)\in\ker D_w^\mathrm{v}$. Moreover if $\w(w_-)=\w(w_+)$, then $\ind  D_w^\mathrm{v}=1$.
\end{lem}

\begin{proof}
The first claim follows from the fact that that $\overline{\p}_{\tau,J,f}^{-1}(\OO_\EE)=\widehat\MM(S_{w_-},S_{w_+},\AA^\tau_f,J)$ is invariant under the $S^1$-action in \eqref{eq:rotating_cylinder}. We can also see this from the following direct computation. The identification  $R(u)=2\pi i u$ in \eqref{eq:R=2pi_i} yields 
\[
\nabla_s R(u)=2\pi i\p_s^\mathrm{v} u \,,\qquad  \nabla_t R(u)=2\pi i\p_t^\mathrm{v} u\,,
\]
where $\p^\mathrm{v}$ denotes the partial derivatives in the vertical direction. Using this and \eqref{eqn:Floer_eqn_fiber1}, we deduce
\[
D_w^\mathrm{v}\left(\begin{aligned}
& R(u)\\[.5ex]
& \;\;\; 0
\end{aligned}\right)
= 
\left(\begin{aligned}
& 2\pi i(\p_s^\mathrm{v} u + i\p_t^\mathrm{v} u) -2m\pi\big( \eta + f(q)\big)R(u)
\\[1ex]
& -\int_0^1d\mu_\tau(u) R(u)\,dt
\end{aligned}\right)
= 0\,.
\]
The index computation follows immediately from 
\[
\mu_\RFH(w_-)-\mu_\RFH(w_+) +1 = \ind D_w= \ind D_w^\mathrm{v} + \ind D_w^\mathrm{h} 
\]
and 
\[
\ind D_w^\mathrm{h}=\ind\mathfrak{d}_q=\mu_\FH(\Pi(w_-))-\mu_\FH(\Pi(w_+)) =\mu_\RFH(w_-)-\mu_\RFH(w_+)\,,
\]
where we used the assumption $\w(w_-)=\w(w_+)$ and Lemma \ref{lem:one-to-one_index}.
\end{proof}

\subsection{Index and orientation}\label{sec:fredholm}
 	In this section, we study the linearized operator $D_w$ in a trivialization and discuss its index and orientation issues.
 	
\subsubsection*{Homotopy along Fredholm operators}

As in \cite[p.287]{CF09}, it is convenient to perturb $\mu_\tau$, which satisfies $r\p_r^2\mu_\tau-\p_r\mu_\tau=0$. We homotope $\mu_\tau$ to another function $\mu_\tau':E\to\R$ which again depends only on the radial coordinate $r$ and satisfies $r\p_r^2\mu_\tau'-\p_r\mu_\tau' >0$, or equivalently $\p_{r^2}^2\mu_\tau'>0$, near  $\Sigma_\tau$. Then the linearized return map $d\phi^1_{\eta X_{\mu_\tau} +X_F}(u(0))$ for $([u,\bar u],\eta)\in\Crit\AA^\tau_f$ has exactly one eigenvalue equal to 1.  We choose a homotopy $\{\mu_\tau^\theta\}_{\theta\in[0,1]}$ between $\mu_\tau$ and $\mu_\tau'$ in such a way that the action functional $\AA^\tau_f$ with $\mu_\tau$ replaced by $\mu_\tau^\theta$ is Morse-Bott for each $\theta$. This gives rise to a homotopy between two Fredholm operators obtained by linearizing the Rabinowitz-Floer equations for $\mu_\tau$ and $\mu_\tau'$.

Given $w=([u,\bar u],\eta)\in\Crit\AA^\tau_f$, let $\Phi_w(t)$ be the trivialization in \eqref{eq:trivialization}, and let $\Psi_w(t)$ be the linearized flow in this trivialization in \eqref{eq:lin_triv}. We denote by $A_w(t)$ the path of symmetric matrices associated with $\Psi_{w}(t)$ via 
\[
A_w(t)=-J_0\dot\Psi_w(t)\Psi_w^{-1}(t),\qquad \Psi_w(0)=I	
\]
where $I$ is the identity matrix and $J_0$ is the standard complex structure on $\R^{\dim E}$. The path $A_w(t)$ is 1-periodic since $\Psi_w(t+1)=\Psi_w(t)\Psi_w(1)$.
We set $h_w(t):=\Phi_w(t)^{-1}(-\nabla \mu_\tau'(u(t)))$. Then $A_wh_w=c_w h_w$ for some $c_w\in\R$ which is nonzero due to our perturbation $\mu_\tau'$ of $\mu_\tau$ above. Moreover it has sign  
\begin{equation}\label{eq:sign_tau_eta}
\sign c_w = \sign \eta \,,	
\end{equation}
see \cite[Lemma C.6]{CF09} and \cite[Section 2.4]{MP11}. 

Let $w=(u,\eta)\in\widehat\MM(S_{w_-},S_{w_+},\AA^\tau_f,J)$. We trivialize $u^*TE$ matching the trivializations $\Phi_{w_-}$ and $\Phi_{w_+}$ at the asymptotic ends. The linearized operator for the Rabinowitz-Floer equation in \eqref{eq:D_w} in this trivialization is of the form
\[
L:\big(W^{1,p}_\delta(\R\x S^1,\R^{\dim E})\oplus \bar{V}_-\oplus \bar{V}_+\big) \x W^{1,p}_\delta(\R,\R) \to L^p_\delta(\R\x S^1,\R^{\dim E})\x L^p_\delta(\R,\R)
\]
\begin{equation}\label{eq:L}
L\left(\begin{aligned}
\xi \\[1ex]	
\zeta
\end{aligned}\right)
=
\left(\begin{aligned}
&\p_s\xi+J_0\p_t\xi +A\xi+\zeta h \\[1.5ex]
& \p_s\zeta+\int_0^1\langle h,\xi\rangle dt
\end{aligned}\right) 
= 
\p_s\left(\begin{aligned}
 \xi \\[2ex]
\zeta
\end{aligned}\right)
  + 
\left(\begin{aligned}
 & J_0\p_t + A  & h \\
 & \int_0^1 \langle h,\cdot\rangle dt\; & 0
\end{aligned}\right)
\left(\begin{aligned}
 \xi \\[2ex]
\zeta
\end{aligned}\right)	
\end{equation}
where $A:\R\x S^1\to \mathrm{End}(\R^{\dim E})$ and $h:\R\x S^1\to \R^{\dim E}$ satisfy  
\begin{equation}\label{eq:L_end}
\lim_{s\to\pm\infty}A(s,t)=A_{w_\pm}(t+t_\pm)\,,\qquad \lim_{s\to\pm\infty}h(s,t)=h_{w_\pm}(t+t_\pm)	
\end{equation}
for some $t_\pm\in S^1$. Here $\bar{V}_-$ and $\bar{V}_+$ are 1-dimensional vector spaces corresponding to $V_-$ and $V_+$ in \eqref{eq:tangent_space} respectively. 
To deform the operator $L$ to a product one, we consider a path of linear operators
\[
L^r=
\p_s
  + 
\left(\begin{aligned}
 & \qquad J_0\p_t + A  & (1-r)h \\
 & (1-r)\int_0^1 \langle h,\cdot\rangle dt\; & c^r\;\;\;
\end{aligned}\right)
\]
where $c^r\in C^\infty(\R,\R)$ for each $r\in[0,1]$ satisfying $c^0=0$ and 
$\displaystyle \lim_{s\to\pm\infty}c^r(s)=-rc_\pm$. Then $L^0=L$ and we write $L^1$ as 
\[
L^1
= \begin{pmatrix}
 L^1_a & 0\\
 0 & L^1_b
\end{pmatrix}
= \begin{pmatrix}
 \p_s+J_0\p_t + A & 0\\
 0 & \p_s+c^1
\end{pmatrix}
\]
It follows from \cite[Theorem C.5]{CF09} that if $\delta>0$ is sufficiently small, $L^r$ is a Fredholm operator for every $r\in[0,1]$, see also \cite[Section 4.4]{AbM18}. We remark that only the signs of the asymptotic ends of $c^r$, rather than their exact values, are crucial to have the Fredholm property. 

\subsubsection*{Fredholm Index}

A straightforward computation shows that the Fredholm index of $L^1_b$ equals
\[
\ind L^1_b = \frac{1}{2}(\sign c_- - \sign c_+)=\frac{1}{2}(\sign \eta_- - \sign \eta_+)\,,
\]
where the last equality is due to \eqref{eq:sign_tau_eta}. 
To compute the Fredholm index of $L^1_a$, we consider the Cauchy problem 
\[
\dot\Psi^\sigma(t)=J_0 (A(t)-\sigma I)\Psi^\sigma(t),\qquad \Psi^{\sigma}(0)=I	
\]
for $\sigma\in\R$ sufficiently close to zero. 
We denote by $\Psi^\sigma_w$ the unique solution associated with $A=A_w$ for $w\in\Crit\AA^\tau_f$. We set
\[
\mu_\CZ^\sigma(\Psi_{w}):=\mu_\CZ(\Psi_{w}^\sigma)\,.
\]
Due to the perturbation, the linearized flow $d\phi_{X_{\mu_\tau'}}^t$ is not equal to $d\phi_{X_{\mu_\tau}}^t=d\phi_{R}^{m t}$ anymore, and $\mu_\CZ\big(\{e^{2\pi i t\cov(u)}\}_{t\in[0,1]}\big)$ in \eqref{eq:fiber_index} needs to be replaced by the Conley-Zehnder index of $d\phi_{\eta X_{\mu_\tau'}+X_F}^t(u(0))$ restricted to $E_q=\C$, where $q=\wp(u)$. As computed in \cite[Lemma 4.3]{CF09}, in the frame $(\frac{\p}{\p r},R)$, the $\mu^\sigma_\CZ$-index of $d\phi_{\eta X_{\mu_\tau'}+X_F}^t(u(0))|_{E_q}$ equals $\frac{1}{2}(\sign \eta - \sign \sigma)$. Therefore in the standard cartesian frame of $E_q=\C$, the $\mu^\sigma_\CZ$-index is 
\[
2\cov(u)+\frac{1}{2}(\sign \eta - \sign \sigma)\,.
\]
We conjugate $L^1_a$ with the isomorphisms $W^{1,p}_\delta\cong W^{1,p}$ and $L^p_\delta\cong L^p$ given by $f\mapsto \beta_\delta f$ to have an operator from $W^{1,p}$ to $L^p$, see \eqref{eq:beta} for $\beta_\delta$. Then the asymptotic formula of the conjugated operator is $J_0\p_t+A + \delta I$ at the negative end and $J_0\p_t+A - \delta I$ at the positive end. Therefore the index of $L^1_a$ is 
\[
\ind L^1_a=  \mu_\CZ^{\delta}(\Psi_{w_+}) - \mu_\CZ^{-\delta}(\Psi_{w_-})+2\,,
\]
see e.g.~\cite[Theorem 4.2]{CF09}, where $+2$ is due to $\overline{V}_-\oplus \overline{V}_+$. We take $[s_\pm]\in\pi_2(E)$ such that $[\bar{u}_\pm]=[(\bar{u}_\pm)_\mathrm{fib}\#s_\pm]$ in $\pi_2(E)$ as in Section \ref{sec:index}. Finally the index of $L$ is computed as
\[
\begin{split}
	\ind L=\ind L^1  &=\ind L^1_a +\ind L^1_b  \\
	&=\mu_\CZ^{\delta}(\Psi_{w_+}) - \mu_\CZ^{-\delta}(\Psi_{w_-}) + 2 + \frac{1}{2}(\sign \eta_- - \sign \eta_+)\\
	&= 2\cov(u_+)+\frac{1}{2}(\sign \eta_+ -1)+2c_1^{TE}([s_+]) - 2\cov(u_-)-\frac{1}{2}(\sign \eta_- +1)\\
	&\quad  -2c_1^{TE}([s_-])   - \mu_{-f}(q_+)  +\mu_{-f}(q_-) +2 + \frac{1}{2}(\sign \eta_- - \sign \eta_+)\\
	&=\mu_\RFH(w_-)-\mu_\RFH(w_+) +1 \,.
\end{split}
\]

\subsubsection*{Coherent orientation}

The space of operators of the form \eqref{eq:L} for some $A$ and $h$ satisfying \eqref{eq:L_end} for given $(A_{w_\pm},h_{w_\pm})$ has orientable determinant line bundle, see \cite[Lemma 4.30]{BO09}. It is crucial here that in our situation every periodic Reeb orbit is good, i.e.~the Conley-Zehnder indices of a simple Reeb orbit and all its multiple covers have the same parity, as computed in Section \ref{sec:index}. We can associate coherent orientations on moduli spaces $\widehat\MM(w_-,w_+,\AA^\tau_f,J)$ systematically as in \cite{FH93} which are compatible with gluing operation. We refer to \cite{BO09,DL19,Sch16} for a detailed account of this in the Morse-Bott setting.

There is another way to produce coherent orientations using the fact that the determinant bundle of a complex linear operator (e.g.~the Cauchy-Riemann operator) on $\C P^1$ has a canonical orientation, see \cite{BM04,BO09}. We also remark that relying on this fact one can alternatively work with chain modules generated by orientation lines over critical points, see e.g.~\cite{Abo15}. This idea can be adapted to our setting as follows. The homotopy $L^r$ produces an isomorphism 
\[
\Det (L)\cong \Det(L^1)=\Det(L^1_a)\otimes \Det (L^1_b)\,,
\]
where $\Det$ denotes the determinant of a Fredholm operator, see \cite[Appendix A]{MS12}. 
We orient each of $\Det(L^1_a)$ and $\Det (L^1_b)$, where the latter one is oriented as in \cite{BM04,BO09}.
For $w=([u,\bar u],\eta)\in\Crit\AA^\tau_f$, let $A_w(t)$ be a 1-periodic loop of symmetric matrices defined above. We choose any matrices $A^+_w(z)$ parametrized by $z\in \C$ and $A^+_w(e^{s+it})=A_w(t)$ for $s\gg0$. Similarly we choose $A^-_w(z)$ parametrized by $z\in \C$ and $A^-_w(e^{-s-it})=A_w(t)$ for $s\ll0$. Consider the operators 
\[
D_w^\pm:W^{1,p}(\C,\R^{\dim E})\to L^p(\C,\R^{\dim E})\,,\qquad D_w\xi:=\p_s\xi + J_0(\p_t\xi - A_w^\pm\xi)\,.
\]
We deform the glued operator 
\[
D_w^+\# D_w^-:W^{1,p}(\C P^1,\R^{\dim E})\to L^p(\C P^1,\R^{\dim E})
\]
through Fredholm operators to a complex linear operator whose determinant has a natural orientation since its kernel and cokernel are complex spaces. 
This gives $\Det(D_w^+\# D_w^-)$ an orientation. Choosing an orientation for the determinant of $D_w^+$ arbitrarily then induces an orientation for the determinant of $D_w^-$. Now let $L^1_a$ be the operator associated to $w\in\widehat\MM(S_{w_-},S_{w_+},\AA^\tau_f,J)$. Then we orient $\Det(L^1_a)$ in such a way that the induced orientation on the determinant of the glued operator 
\[
D_w^+\# L^1_a \# D_w^-:W^{1,p}(\C P^1,\R^{\dim E})\to L^p(\C P^1,\R^{\dim E})
\] 
agrees with the natural orientation on the determinant of its deformation to a complex linear operator. 

We can orient $L^1_b$ similarly. We choose the trivial orientation on the determinant of the operator $\mathrm{d}_1:\zeta\mapsto \p_s\zeta+\zeta$. For $c_-\neq 0$, we take any $c\in C^\infty(\R,\R)$ with $\displaystyle \lim_{s\to-\infty}c(s)=1$ and $\displaystyle\lim_{s\to+\infty}c(s)=c_-$ and define the operator $\mathrm{d}_{c_-}^-=\p_s +c$. We also define an operator $\mathrm{d}_{c_+}^+=\p_s+c$ for  $c_+\neq 0$ by choosing any $c\in C^\infty(\R,\R)$ with $\displaystyle \lim_{s\to-\infty}c(s)=c_+$ and $\displaystyle \lim_{s\to+\infty}c(s)=1$. We endow $\Det(\mathrm{d}^-_{c_-})$ with arbitrary orientation. We then orient $\Det(\mathrm{d}^+_{c_+})$ in such a way that the induced orientation on the determinant of the glued operator $\mathrm{d}_{c_+}^+\#\mathrm{d}_{c_-}^-$, when $c_-=c_+$, agrees with the the trivial one on $\Det(\mathrm{d}_1)$ through a deformation of Fredholm operators given by changing $c$ fixing the asymptotic ends. Finally we orient $L^1_b$ so that the induced orientation on the determinant of the glued operator $\mathrm{d}_{c_+}^+\# L^1_b\#\mathrm{d}_{c_-}^-$ again agrees with the trivial one on $\Det(\mathrm{d}_1)$ through a deformation.

\subsection{Integro Kazdan-Warner equation}\label{sec:KW_eq}
In this section we review and discuss  results from \cite[Section 5]{Fra04}. 
We continue to assume $J\in\JJ_\mathrm{diag}$. Let $(u,\eta)\in\widehat\MM(S_{w_-},S_{w_+},\AA^\tau_f,J)$ with $\w(w_-)=\w(w_+)$ and let $u^0:\R\x S^1\to E_{q(0,0)}$ denote the parallel transport of $u$ defined in \eqref{eq:u^0}. As before we write $q=\wp\circ u$. By Proposition \ref{prop:positivity_of_intersection}, $u$ does not intersect $\OO_E$ and this is equivalent to saying that $u^0$ does not pass through  zero in $E_{q(0,0)}$. We express $u^0$ as
\begin{equation}\label{eq:u^0_polar}
u^0(s,t)=e^{2\pi (\rho(s,t)+i\psi(s,t))}u(0,0)
\end{equation}
for smooth functions $\rho:\R\x S^1\to\R$ and $\psi:\R\x S^1\to S^1$. Then \eqref{eqn:eqn_for_u^0} is equivalent to 
\begin{equation}\label{eq:rho_psi}
\left\{
\begin{aligned}
&\p_s\rho-\p_t\psi=m(\eta+f(q))-\chi_2\\[.5ex]
&\p_t\rho+\p_s\psi=\chi_1\,.
\end{aligned}\right.
\end{equation}
We point out that if $q$ is $t$-independent, then $\chi_1=\chi_2=0$, see \eqref{eq:chi}. 
Summing up $\partial_s$ of the first equation and $\partial_t$ of the second equation, we obtain an integro version of the Kazdan-Warner equation, namely
\begin{equation}\label{eq:Kazdan-Warner}
\begin{split}
\Delta \rho&=m\p_s\eta+mdf_q(\p_sq)-\p_s\chi_2+\p_t\chi_1\\
&=m\int_0^1\mu_\tau(u^0)dt+m\om(X_f(q),\p_sq)+m\om(\p_sq,\p_tq)\\
&=m^2\pi\int_0^1e^{4\pi\rho}dt-m\tau+ m\om(\p_sq,j(t,q)\p_sq)\,.
\end{split}
\end{equation}
We conclude that, if $(u,\eta)\in\widehat\MM(S_{w_-},S_{w_+},\AA^\tau_f,J)$, then the corresponding $\rho$ satisfies \eqref{eq:Kazdan-Warner}. The converse is true in the following sense. Let us consider equations \eqref{eq:rho_psi} as equations determining $\psi$. Integrating these equations independently leads to a well-defined function $\psi$ if and only if $\rho$ satisfies \eqref{eq:Kazdan-Warner}. See the proof of Theorem \ref{thm:bijection} for a detailed discussion.

This motivates us to study for a nonnegative function $\kappa\in C^\infty(\R\x S^1,\R)\cap L^1(\R\x S^1,\R)$ the moduli space 
\begin{equation}\label{eq:KW_moduli}
\MM_\mathrm{KW}(\tau,\kappa)
\end{equation}
which consists of $\rho\in C^\infty(\R\x S^1,\R)$ such that 
\begin{equation}\label{eq:KW}
\Delta\rho =m^2\pi\int_0^1e^{4\pi\rho}dt-m\tau +\kappa(s,t)\,,\qquad  \rho-\frac{1}{4\pi}\ln\frac{\tau}{m\pi} \in  W^{2,2}(\R\x S^1,\R)\,.
\end{equation}
Linearizing the equation \eqref{eq:KW}, we obtain the operator
\begin{equation}\label{eq:linearization_of_KW}
L_\rho:W^{2,2}(\R\x S^1,\R)\to L^2(\R\x S^1,\R)\,,\qquad \hat\rho\mapsto\Delta\hat\rho-4m^2\pi^2\int_0^1e^{4\pi\rho}\hat\rho\,dt
\end{equation}
for $\rho\in C^\infty(\R\x S^1,\R)$ with $\rho-\frac{1}{4\pi}\ln\frac{\tau}{m\pi} \in  W^{2,2}(\R\x S^1,\R)$.

\begin{lem}\label{lem:rho_t_independent}
The operator $L_\rho$ is injective if $\rho$ is independent of $t$. 	
\end{lem}
\begin{proof}
		 Suppose that $\hat\rho\in\ker L_\rho$, i.e.~$\Delta\hat\rho= 4m^2\pi^2 e^{4\pi\rho}\int_0^1\hat\rho\,dt$. Multiplying $\hat\rho$ both sides and integrating over $\R\x S^1$, we deduce 
	 \[
	 -\int_{-\infty}^{+\infty}\int_0^1 \|\p_t\hat\rho\|^2+\|\p_s\hat\rho\|^2\,dtds = \int_{-\infty}^{+\infty}\int_0^1\hat \rho \Delta\hat\rho \,dtds = 4m^2\pi^2 \int_{-\infty}^{+\infty}e^{4\pi\rho}\left(\int_0^1\hat\rho\,dt\right)^2ds\,.
	 \]
	 This proves that $\hat\rho=0$ and $\ker L_\rho$ is trivial.
\end{proof}

\begin{lem}\label{lem:L_rho_fredholm}
The operator $L_\rho$ is Fredholm of index zero.
\end{lem}
\begin{proof}
We first prove the claim when $\rho$ is constant, i.e.~$\rho=\rho_0:=\frac{1}{4\pi}\ln\frac{\tau}{m\pi}$. By Lemma \ref{lem:rho_t_independent}, $L_{\rho_0}$ is injective. To prove surjectivity, we write $\xi\in W^{2,2}(\R\times S^1,\R)$ and $\zeta\in L^2(\R\times S^1,\R)$ in Fourier series
\[
\xi(s,t)= \sum_{n\in\Z}a_n(s)\cos 2\pi nt +b_n(s)\sin 2\pi nt\,,\quad \zeta(s,t)= \sum_{n\in\Z}c_n(s)\cos 2\pi nt +d_n(s)\sin 2\pi nt\,.
\]
Since
\[
L_{\rho_0}\xi= \sum_{n\in\Z}\Big(\big(\p_s^2a_n(s)-4\pi^2n^2 a_n(s)\big)\cos 2\pi n t+\big(\p_s^2b_n(s)-4\pi^2n^2b_n(s)\big)\sin 2\pi n t\Big) -4m\pi\tau a_0(s),
\]
the equation $L_{\rho_0}\xi=\zeta$ translates to
\[
\begin{cases}
	\p_s^2a_0(s)-4m\pi\tau a_0(s)=c_0(s)& \\[0.5ex]
	\p_s^2a_n(s)-4\pi^2n^2 a_n(s)=c_n(s) & n\in \Z \setminus\{0\}\\[0.5ex]
	\p_s^2b_n(s)-4\pi^2n^2 b_n(s) = d_n(s) & n\in \Z \setminus\{0\}\,. 
\end{cases}
\]
As shown in \cite[Step 2 in Proposition 3.6]{Fra22}, for given $c_n$ and $d_n$ we can find solutions $a_n$ and $b_n$ to the above equations using the Fourier-Plancherel transformation. This proves that $L_{\rho_0}$ is surjective. Therefore $L_{\rho_0}$ is an isomorphism and in particular a Fredholm operator of index zero. 

Now we consider general $\rho$ as above. For $T>0$, we choose a smooth function $\beta:\R\to [0,1]$ such that $\beta(s)=1$ for $|s|>T$ and $\beta(s)=0$ for $|s|<T-1$. Since $\rho-\rho_0\in  W^{2,2}(\R\x S^1,\R)$, we can make the operator norm $\|L_{\beta\rho}-L_{\rho_0}\|$ arbitrarily small by taking large $T>0$. Since $L_{\rho_0}$ is Fredholm of index zero, so is  $L_{\beta\rho}$ provided that $T>0$ is sufficiently large. Finally since $L_\rho-L_{\beta\rho}$ is a compact operator, $L_\rho$ is also a Fredholm operator of index zero.
\end{proof}

\begin{prop}\label{prop:uniqueness_rho}\cite[Theorem 5.3]{Fra08}
There exist $c_1,c_2>0$ such that if $\tau>0$ satisfies
\[
\tau\leq c_1,\qquad  \tau^3 \|\kappa\|_{L_1}^2\leq c_2\,,
\]
then $\MM_\mathrm{KW}(\tau,\kappa)$ consists of a unique solution. 
\end{prop}
\begin{proof}
	We outline the proof given in \cite[Theorem 5.3]{Fra08}. For small $\tau>0$, the operator $L_\rho$ has trivial kernel for every  $\rho\in\MM_\mathrm{KW}(\tau,\kappa)$, see \cite[Proposition 5.8]{Fra08}. The constants $c_1$ and $c_2$ are given by this step. Due to Lemma \ref{lem:L_rho_fredholm}, $L_\rho$ is in fact an isomorphism. This is then also true for every $\rho\in\MM_\mathrm{KW}(\tau,\varepsilon\kappa)$ with $\varepsilon\in[0,1]$ as $\|\varepsilon\kappa\|_{L^1}\leq \|\kappa\|_{L^1}$. Since $\MM_\mathrm{KW}(\tau,0)$ consists of a single element $\rho_0=\frac{1}{4\pi}\ln\frac{\tau}{m\pi}$ as observed in \cite[p.48]{Fra08}, the implicit function theorem implies that the cardinality of $\MM_\mathrm{KW}(\tau, \varepsilon\kappa)$ is one for all $\varepsilon\in[0,1]$.
\end{proof}

\begin{thm}\label{thm:bijection}\cite[Theorem A]{Fra08}
Let $J=\begin{pmatrix}
 i & 0\\
 0 & j
\end{pmatrix}\in\JJ_\mathrm{diag}$. For a given $K>0$, there exists $\tau_0>0$ depending on $K$ such that the following property holds for every $\tau\in(0,\tau_0)$. Suppose that $w_\pm\in\Crit{\AA}^{\tau}_f$ satisfy
\[
\w(w_-)=\w(w_+),\qquad \mathfrak a_f(\Pi(w_-))-\mathfrak a_f(\Pi(w_+))\leq K\,.
\] 
Then the projection 
\[
\Pi:\widehat\MM(S_{w_-},S_{w_+},\AA^\tau_f,J)\longrightarrow \widehat\NN(\Pi(w_-),\Pi(w_+),\mathfrak{a}_f,j)
\]
defined in \eqref{eq:proj_cylinder} is bijective modulo the $S^1$-action given in \eqref{eq:rotating_cylinder}. 
\end{thm}
\begin{proof}
 We take any $q\in \widehat\NN(\Pi(w_-),\Pi(w_+),\mathfrak{a}_f,j)$. Due to the assumption, we can apply Proposition \ref{prop:uniqueness_rho} for $\kappa=m\om(\p_sq,j_t\p_sq)$. Note that $\frac{1}{m}\|\kappa\|_{L^1}= \a_f(\Pi(w_-))-\a_f(\Pi(w_+))$. This determines $\tau_0>0$ and gives a unique solution $\rho\in \MM_\mathrm{KW}(\tau,\kappa)$ for every $\tau\in(0,\tau_0)$. Then $\rho\in \MM_\mathrm{KW}(\tau,\kappa)$  uniquely determines $\eta$ by $\p_s\eta=m\pi\int_0^1e^{4\pi\rho}dt-\tau$ together with the asymptotic conditions $w_\pm$. Integrating either equation in \eqref{eq:rho_psi} determines $\psi$ up to an additive constant since $\rho$ satisfies \eqref{eq:Kazdan-Warner}. This additive constant  precisely corresponds to the $S^1$-action. Such $\rho$, $\psi$, and $\eta$ induce $w=(u,\eta)\in \widehat\MM(S_{w_-},S_{w_+},\AA^\tau_f,J)$ satisfying $\Pi(w)=q$ which is unique up to the $S^1$-action. This proves the theorem. 
\end{proof}

While we will define Rabinowitz Floer homology for arbitrary $\tau>0$ with $J\in\JJ^\BB_\mathrm{reg}$, we work only with small $\tau$ when $J\in\JJ_\mathrm{diag}$ is used due to Theorem \ref{thm:bijection}. However this smallness condition on $\tau$, which seems to be a technical condition, does not impose any additional restriction on the main results of this paper because we will use $J\in\JJ_\mathrm{diag}$ mainly for computational purposes. 
In what follows, we observe that when $j\in\j_\mathrm{HS}$ and $q$ is independent of $t$, i.e.~$q$ is a Morse gradient flow line, the smallness condition on $\tau$ can be easily removed.
\begin{lem}\label{lem:kappa_t_indep}
	Let  $\tau>0$ be arbitrary. Suppose that $\kappa$ is independent of $t$. Then every $\rho\in\MM_\mathrm{KW}(\tau,\kappa)$ is independent of $t$ and the operator $L_\rho$ in \eqref{eq:linearization_of_KW} is an isomorphism.
\end{lem}
\begin{proof}
	 Since the right-hand side of the equation in \eqref{eq:KW} is $t$-independent by the hypothesis, we have $\Delta\p_t\rho=\p_t\Delta\rho =0$ for every $\rho \in\MM_\mathrm{KW}(\tau,\kappa)$. Since $\p_t\rho$ vanishes at both asymptotic ends by the asymptotic condition in \eqref{eq:KW}, the maximum principle yields that $\p_t\rho$ vanishes everywhere. This proves that $\rho$ is independent of $t$. Then due to Lemma \ref{lem:rho_t_independent}, $L_\rho$ is injective. Since $L_\rho$ has index zero by Lemma \ref{lem:L_rho_fredholm}, it is an isomorphism.
	 \end{proof}
We now consider the subspace 
\begin{equation}\label{eq:om=0_moduli}
	\widehat\MM_{\om=0}(S_{w_-},S_{w_+},\AA^\tau_f,J)\subset \widehat\MM(S_{w_-},S_{w_+},\AA^\tau_f,J)
\end{equation}
consisting of $w$ with $t$-independent $\Pi(w)$, i.e.~$\Pi(w)\in \widehat{\NN}(\Pi(w_-),\Pi(w_+),f,g)$.
\begin{prop}\label{prop:bijection}
	Let $J\in\JJ_\mathrm{diag}$ with $j\in\j_\mathrm{HS}$. Let $\tau>0$ be arbitrary. For every $w_\pm\in\Crit{\AA}^{\tau}_f$ with $\w(w_-)=\w(w_+)$, the projection 
\[
\Pi:\widehat\MM_{\om=0}(S_{w_-},S_{w_+},\AA^\tau_f,J) \longrightarrow \widehat\NN(\Pi(w_-),\Pi(w_+),f,g)
\]
is bijective modulo the $S^1$-action given in \eqref{eq:rotating_cylinder}. 
\end{prop}
\begin{proof}
	Lemma \ref{lem:kappa_t_indep} and the proof of Proposition \ref{prop:uniqueness_rho} show that $\MM_\mathrm{KW}(\tau,\kappa)$ consists of a single element for every $\tau>0$ provided that $\kappa$ is $t$-independent. Hence the proof of Theorem \ref{thm:bijection} shows the thesis.
\end{proof}

\begin{cor}\label{cor:rfh_lift}
Let $w=(u,\eta)\in \widehat\MM_{\om=0}(S_{w_-},S_{w_+},\AA^\tau_f,J)$ with $\w(w_-)=\w(w_+)$. Then,
 \[
u(s,t)=P_q^s \exp{\left(2\pi \int_0^s (\cov(u_-)+ m\eta+mf(q))ds + 2\pi i \cov (u_-) t\right)} u(0,0)
\]
where $q=\Pi(w)$, $w_-=([u_-,\bar u_-],\eta_-)$, and $P^s_q$ denotes the parallel transport along $q$ defined in \eqref{eq:parallel}.
\end{cor}
\begin{proof}
	As in \eqref{eq:u^0} and \eqref{eq:u^0_polar}, we write 
	\[
	(P^s_q)^{-1} u(s,t) = u^0(s,t)=e^{2\pi(\rho(s,t)+i\psi(s,t))}u(0,0)\,.
	\]
	Then $\rho\in\MM_\mathrm{KW}(\tau, \kappa)$ with $\kappa=m\om(\p_sq,j\p_sq)$.  As mentioned below  \eqref{eq:rho_psi}, since $q$ is independent of $t$, the second equation in \eqref{eq:rho_psi}  reduces to 
	$\p_t\rho+\p_s\psi=0$. By Lemma \ref{lem:kappa_t_indep}, $\rho$ is independent of $t$, which implies that $\psi$ is independent of $s$. Hence we have
\begin{equation}\label{eq:u^0_t}
(P^s_q)^{-1} u(s,t)=u^0(s,t) = e^{2\pi(\rho(s)+i\psi(t))}u(0,0)\,.	
\end{equation}
Moreover $\rho(0)=0$ and $\psi(0)=0$ since $P_q^0$ is the identity map. 
From the asymptotic condition, namely
\[
e^{2\pi \cov (u_-)it} u_-(0)=u_-(t)=\lim_{s\to-\infty}u(s,t)= e^{\frac{1}{2}\ln\frac{\tau}{m\pi} + 2\pi i\psi(t)}  \lim_{s\to -\infty } P^s_q u(0,0)\,,
\]
we deduce that $\psi(t)$ equals $\cov (u_-)t$  up to constant. Then $\psi(0)=0$ implies $\psi(t)=\cov(u_-)t$. Note in this case that $\cov (u_-)=\cov(u_+)$, see Proposition \ref{prop:winding}.(d). Since the first equation in \eqref{eq:rho_psi} simplifies to $\p_s\rho=\cov (u_-)+m(\eta+f(q))$, we obtain
 \[
 u^0(s,t)= \exp{\left(2\pi \int_0^s (\cov(u_-)+ m\eta+mf(q))ds + 2\pi i \cov (u_-) t\right)} u(0,0)
 \]
 and the claim follows from \eqref{eq:u^0_t}. 
\end{proof}

\begin{rem}
	The above results for $t$-independent $q$ can be proved without appealing to the integro Kazdan-Warner equation. Arguing as in \cite[Appendix A]{AF16}, one can see that the operator obtained by linearizing $\overline{\p}^0$ in \eqref{eqn:Floer_eqn_fiber} is always surjective. Moreover the explicit form of $u$ in Corollary \ref{cor:rfh_lift}  can be obtained directly. We write 
	\[
	u^0(s,t)= \sum_{k\in\Z}e^{2\pi ikt}a_k(s)\,,\qquad a_k(s)\in E_{q(0,0)}=\C\,.
	\] 
	Then the first equation in $\overline{\p}^0(u^0,\eta)=0$, see \eqref{eqn:Floer_eqn_fiber}, translates into 
	\[
	\p_s a_k=2\pi (k+m\eta +mf(q))a_k\qquad \forall k\in\Z\,,
	\] 
	and in turn
	\[ 
	a_k(s)= \exp{\left(2\pi \int_0^s (k+m\eta+mf(q))ds\right)} a_k(0)\qquad \forall k\in\Z\,.
	\]  
	Since $a_k$ converges asymptotically, $a_k(0)$ is nonzero only when 
	\[
	0= k + \lim_{s\to \pm\infty} m(\eta+f(q)) = k - \cov (u_\pm)\,,
	\]
	where we used \eqref{eq:eta=cov} in the last equality. Hence,
	\[
	\begin{split}
	u^0(s,t) &=  e^{2\pi i\cov (u_-)t}a_{\cov(u_-)}(s) \\
	& = \exp{\left(2\pi\int_0^s (\cov(u_-)+m\eta+mf(q))ds + 2\pi i\cov (u_-)t\right)} a_{\cov(u_-)}(0)\,,
	\end{split}
	\]
	where $a_{\cov(u_-)}(0)=u^0(0,0)=u(0,0)$.
\end{rem}

\subsection{Transversality with $J\in\JJ_\mathrm{diag}$}\label{sec:trans_diag}

In this section, we prove transversality results for Rabinowitz Floer cylinders for $J\in\JJ_\mathrm{diag}$ connecting critical points with the same winding number.
Let $w_\pm=([u_\pm,\bar u_\pm],\eta_\pm)\in\Crit\AA^\tau_f$ with $\w(w_-)=\w(w_+)$. We fix $q\in\widehat\NN(\Pi(w_-),\Pi(w_+),\a_f,j)$ and define 
\[
\CC_q(S_{w_-},S_{w_+}):= \{(u,\eta)\in \CC (S_{w_-},S_{w_+}) \mid \wp\circ u=q\}\,.
\]
In particular,
\[
T_w\CC_q(S_{w_-},S_{w_+}) = T_w^\mathrm{v}\CC(S_{w_-},S_{w_+}) = W^{1,p}_\delta(\R\x S^1, q^*E)\oplus V_-\oplus V_+\oplus W_\delta^{1,p}(\R,\R) \,.	
\]
We note that the requirement \eqref{eq:u_Gamma_0} is automatically fulfilled since $q$ satisfies the corresponding identity and $\Gamma_E\cong\Gamma_M$. 
We also consider the Banach bundle $\EE^\mathrm{v}\to \CC_q(S_{w_-},S_{w_+})$, see Section \ref{sec:functional}.
For $w=(u,\eta)\in\CC_q(S_{w_-},S_{w_+})$, let $u^0:\R\x S^1\to E_{q(0,0)}$ be the parallel transport of $u$ defined in \eqref{eq:u^0}. We denote the positive and negative asymptotic orbits of $u^0$ by $u^0_+$ and $u^0_-$ respectively, and set
\[
S_{w_\pm^0}:=\{(e^{2\pi i \theta}u_\pm^0,\eta_\pm)\mid \theta\in S^1\}\subset C^\infty(S^1,E_{q(0,0)})\x \R\,.
\]  
Let $\CC^0(S_{w_-^0},S_{w_+^0})$ denote the space of maps 
\[
w^0=(u^0,\eta) \in W^{1,p}_{loc}(\R\x S^1,E_{q(0,0)})\x W^{1,p}_{loc}(\R,\R)
\]
such that 
\[
\eta-\eta_-\in W^{1,p}_\delta((-\infty,0),\R)\,,\quad \eta-\eta_+\in W^{1,p}_\delta((0,+\infty),\R)
\]
and, for some $\theta_\pm\in S^1$,
\[
u^0-e^{2\pi i\theta_-} u^0_-\in W^{1,p}_\delta((-\infty,0)\x S^1, E_{q(0,0)})\,,\quad u^0-e^{2\pi i\theta_+}u^0_+\in W^{1,p}_\delta((0,+\infty)\x S^1, E_{q(0,0)})\,.
\]
Its tangent space is
\[
T_{w^0} \CC^0(S_{(u^0_-,\eta_-)},S_{(u^0_+,\eta_+)}) = W^{1,p}_\delta(\R\x S^1,E_{q(0,0)})\oplus V_-^0 \oplus V_+^0 \oplus  W^{1,p}_\delta(\R,\R)
\]
where $V_\pm^0$ are defined as for $V_\pm$. We also consider the Banach bundle $\EE^0\to \CC^0(S_{w_-^0},S_{w_+^0})$ with fibers 
\[
\EE^\mathrm{0}_{w^0}=L^p_\delta (\R\x S^1, E_{q(0,0)}) \x L^p_\delta(\R,\R)\,.
\]
The two maps $\overline{\p}^\mathrm{v}$ and $\overline{\p}^0$ in \eqref{eqn:Floer_eqn_fiber1} and \eqref{eqn:Floer_eqn_fiber} are related though the diffeomorphism 
\[
\begin{split}
\mathcal P_{q}: \CC_q(S_{w_-},S_{w_+})  &\longrightarrow \CC^0(S_{w_-^0},S_{w_+^0}) \\
w= (u,\eta)&\longmapsto w^0=(u^0,\eta)
\end{split}
\]
induced by parallel transport in \eqref{eq:u^0}. That is, the following diagram commutes:
\begin{equation}\label{eq:conjugate}
	\begin{tikzcd}[row sep=2.5em]
	\EE^\mathrm{v} \arrow{rr}{d\PP_q} && \EE^0 \\
	\CC_q(S_{w_-},S_{w_+}) \arrow{u}{\overline{\p}^\mathrm{v}}  \arrow{rr}{\PP_q} && \CC^0(S_{w_-^0},S_{w_+^0})  \arrow{u}{\overline{\p}^0}
	\end{tikzcd}
\end{equation}
Here note that the differential $d\PP_q$ is indeed defined between $L^p$-sections.

\begin{prop}\label{prop:transversality}
Let $J=\begin{pmatrix}
 i & 0\\
 0 & j
\end{pmatrix}$ for $j\in\mathfrak{j}_\mathrm{reg}$. For a given $K>0$, let $\tau_0>0$ be the associated constant given by Theorem \ref{thm:bijection}. Let $\tau\in(0,\tau_0)$. Then for every $w_\pm\in\Crit\AA^\tau_f$ with 
\[
\w(w_-)=\w(w_+)\,,\qquad \mathfrak{a}_f(\Pi(w_-))-\mathfrak{a}_f(\Pi(w_+))\leq K\,,
\] 
 the moduli space $\widehat\MM(S_{w_-},S_{w_+},\AA^\tau_f,J)$ is cut out transversely, i.e.~$D_w$ is surjective for every $w\in \widehat\MM(S_{w_-},S_{w_+},\AA^\tau_f,J)$.
\end{prop}
\begin{proof}
We recall the horizontal-vertical splitting of $D_w$ in \eqref{eq:lin_block}, and the horizontal part $D_w^\textrm{h}=\mathfrak{d}_{q}$ is surjective since $j\in\mathfrak{j}_\mathrm{reg}$. Thus it suffices to show that $D_w^\textrm{v}$, which we view as the vertical differential of $\overline{\p}^\mathrm{v}$ at $w$, is surjective. Due to \eqref{eq:conjugate}, $D_w^\mathrm{v}$ is conjugated to the the vertical differential $D_{w^0}^0$ of $\overline{\p}^0$ at $w^0=(u^0,\eta)$, and we will show that $D^0_{w^0}$ is surjective. The precise formula of $D^0_{w^0}:T_{w^0} \CC^0(S_{w^0_-},S_{w^0_+}) \to  \EE^0_{w^0}$ is 
\[
D^0_{w^0}(\hat u,\hat\eta)=
\left(\;
\begin{aligned}
&\p_s\hat u+i\p_t\hat u+2\pi\big((\chi_2-m\eta-mf(q)\mathrm{id}-\chi_1i\big)\hat u-2m\pi\hat\eta u^0\\[.5ex]
&\p_s\hat\eta-2m\pi\int_0^1\langle u^0,\hat u\rangle dt
\end{aligned}\;\right)\,,
\]
where $\langle \cdot,\cdot\rangle$ denotes the Hermitian metric on the bundle $\wp:E\to M$. By Lemma \ref{lem:index_vert}, $D_{w^0}^0$ has index one. Thus surjectivity of $D_{w^0}^0$ is equivalent to $\dim\ker D_{w^0}^0=1$. Suppose $(\hat u,\hat\eta)\in\ker D_{w^0}^0$. We  write $u^0=e^{2\pi(\rho+i\psi)}$ as in \eqref{eq:u^0_polar} and $\hat u=2\pi(\hat\rho+i\hat\psi)e^{2\pi(\rho+i\psi)}$, where 
\[
\hat\rho\in W^{1,p}_\delta(\R\x S^1,\R)\,,\qquad \hat\psi\in W^{1,p}_{loc}(\R\x S^1,\R)
\]
such that there exist constants $\hat\psi_\pm\in\R$ satisfying 
\[
\hat\psi-\hat\psi_-\in W^{1,p}_\delta((-\infty,0)\x S^1,\R)\,,\qquad \hat\psi-\hat\psi_+\in W^{1,p}_\delta((0,+\infty)\x S^1,\R)\,.
\]
Then we obtain
\begin{equation}\label{eq:D_{w^0}^0}
	0=D_{w^0}^0(\hat u,\hat\eta)=
\left(\;
\begin{aligned}
& 2\pi e^{2\pi(\rho+i\psi)}\big(\p_s\hat\rho-\p_t\hat\psi-m\hat\eta+i(\p_s\hat\psi+\p_t\hat\rho)\big)\\[.5ex]
&\p_s\hat\eta-4m\pi^2\int_0^1e^{4\pi\rho}\hat\rho\, dt
\end{aligned}\;\right)\,,
\end{equation}
where we used $\overline{\p}^0w^0=0$, see \eqref{eq:rho_psi}. This in particular yields 
\begin{equation}\label{eq:hat_phi_psi}
\left\{ 
\begin{aligned}
	&\;\p_s\hat\rho-\p_t\hat\psi-m\hat\eta=0\\[0.5ex]
	&\;\p_s\hat\psi+\p_t\hat\rho=0\,.
\end{aligned}
\right.	
\end{equation}
By elliptic regularity, $(\hat u,\hat \eta)$ and  $(\hat\rho,\hat\psi)$ are smooth. 
Differentiating the first equation in \eqref{eq:hat_phi_psi} with respect to $s$ and the second one with respect to $t$ and summing them up, we obtain
\[
0=\Delta\hat\rho-m\p_s\hat\eta=\Delta\hat\rho-4m^2\pi^2\int_0^1e^{4\pi\rho}\hat\rho\, dt\,.
\]
Thus $\hat\rho\in\ker L_\rho$, see \eqref{eq:linearization_of_KW}. 
As mentioned in the proof of Proposition \ref{prop:uniqueness_rho}, if $\tau$ is as small as in Theorem \ref{thm:bijection}, then $\hat\rho=0$. Therefore $\hat\eta=0$ and $\p_s\hat\psi=\p_t\hat\psi =0$ by \eqref{eq:D_{w^0}^0} and \eqref{eq:hat_phi_psi}. This proves that $\hat\psi=\hat\psi_-=\hat\psi_+$ is constant, and $\ker D^0_{w^0}$ is spanned by $(\hat\rho,\hat\psi)=(0,1)$. This corresponds to the vector field $R$ along $u^0$ which, for symmetry reasons, necessarily lies in  $\ker D^0_{w^0}$, see Lemma \ref{lem:index_vert}. The proof is complete.
\end{proof}

 Let $J=\begin{pmatrix}
 i & 0\\
 0 & j
\end{pmatrix}$ for $j\in\mathfrak{j}_\mathrm{HS}$. We denote by  
\begin{equation}\label{eq:simple_E}
 \widehat\MM_s(S_{w_-},S_{w_+},\AA^\tau_f,J)\subset \widehat\MM(S_{w_-},S_{w_+},\AA^\tau_f,J)
\end{equation}
 the space consisting of $w=(u,\eta)$ with simple $\Pi(w)$, i.e.~$\Pi(w)\in\widehat\NN_s(\Pi(w_-),\Pi(w_+),\mathfrak{a}_f,j)$ defined in Section \ref{sec:time_indep_j}.

\begin{prop}\label{prop:transversality2}
Let $J=\begin{pmatrix}
 i & 0\\
 0 & j
\end{pmatrix}$ for $j\in\mathfrak{j}_\mathrm{HS}$. For a given $K>0$, let $\tau_0>0$ be the associated constant given by Theorem \ref{thm:bijection}. Let $\tau\in(0,\tau_0)$. Then for every $w_\pm\in\Crit\AA^\tau_f$ with 
\[
\w(w_-)=\w(w_+)\,,\qquad \mathfrak{a}_{f}(\Pi(w_-))-\mathfrak{a}_{f}(\Pi(w_+))\leq K\,,
\] 
and for every $w\in \widehat\MM_s(S_{w_-},S_{w_+},\AA^\tau_{f},J)$, the operator $D_w$ is surjective. 
\end{prop}
\begin{proof}
Since $j\in\mathfrak{j}_\mathrm{HS}$ and $\Pi(w)=q$ is simple, the horizontal part $D_w^\mathrm{h}=\mathfrak{d}_q$ is surjective, see Proposition \ref{prop:j_HS}.(a).
As shown in the proof of Proposition \ref{prop:transversality}, the vertical part $D^\mathrm{v}_{w}$ is also surjective. Hence $D_w$ is surjective. 
\end{proof}

\begin{prop}\label{prop:transversality3}
Let $J=\begin{pmatrix}
 i & 0\\
 0 & j
\end{pmatrix}$ for $j\in\mathfrak{j}_\mathrm{HS}$, and let $\tau>0$ be arbitrary. For every  $w_\pm\in\Crit\AA^\tau_f$ with $\w(w_-)=\w(w_+)$ and $w\in \widehat\MM(S_{w_-},S_{w_+},\AA^\tau_{f},J)$ with $t$-independent $\Pi(w)$, the operator $D_w$ is surjective.
\end{prop}
\begin{proof}
	The horizontal part $D_w^\mathrm{h}$ is surjective due to Proposition \ref{prop:j_HS}.(c). It suffices to show that the vertical part $D_w^\mathrm{v}$ is surjective. This follows from the proof of Proposition \ref{prop:transversality} together with Lemma \ref{lem:kappa_t_indep}.
\end{proof}

The following lemma ensures that moduli spaces we will use are cut out transversely due to Proposition \ref{prop:transversality2}.  
\begin{lem}\label{lem:simple_w}
Suppose that every $A\in\pi_2(M)$ with $\om(A)\neq 0$ satisfies either $c_1^{TM}(A)\geq1$ or $c_1^{TM}(A)\leq-\frac{1}{2}\dim M$. 
	Let $J=\begin{pmatrix}
 i & 0\\
 0 & j
\end{pmatrix}$ with $j\in\mathfrak{j}_\mathrm{HS}$, let $\tau>0$ be arbitrary, and let $w_\pm\in\Crit\AA^\tau_f$ with $\mu_\RFH(w_-)-\mu_\RFH(w_+)\leq 3$.
Then, for every $w\in\widehat\MM(S_{w_-},S_{w_+},\AA^\tau_{f},J)$, the Floer cylinder $\Pi(w)$ is either $t$-independent or simple. In the latter case, $c_1^{TM}([\Pi(w)])\geq1$. 
\end{lem}
\begin{proof}
We first observe that $\w(w_-)\leq\w(w_+)$ by Proposition \ref{prop:positivity_of_intersection}, and thus 
\begin{equation}\label{eq:cpt_ind}
\mu_\FH(\Pi(w_-))-\mu_\FH(\Pi(w_+))	= \mu_\RFH(w_-)-\mu_\RFH(w_+) - 2(\w(w_+)-\w(w_-))\leq 3
\end{equation}
due to \eqref{eq:indices_and_winding}. 
	Suppose that $q=\Pi(w)$ depends on $t$. We denote by $q_\mathrm{sim}$ the underlying simple cylinder of $q$ defined in \eqref{eq:simple}. Then $\om([q_\mathrm{sim}])>0$ by \eqref{eq:om-energy}. If $c_1^{TM}([q_\mathrm{sim}])\leq-\frac{1}{2}\dim M$, the moduli space containing $q_\mathrm{sim}$ has dimension at most zero, see Proposition \ref{prop:j_HS}.(a). This is absurd due to the free $\R\times S^1$-action on $q_\mathrm{sim}$, see \eqref{eq:RxS1_action}. Next, suppose that $c_1^{TM}([q_\mathrm{sim}])\geq1$. If $q$ is not simple, then $c_1^{TM}([q])\geq2$ holds. This implies that the moduli space containing $q_\mathrm{sim}$ has dimension at most 1 by Proposition \ref{prop:j_HS}.(a) and \eqref{eq:cpt_ind}. This contradicts again  the existence of the free $\R\times S^1$-action on $q_\mathrm{sim}$. Alternatively one can  apply Proposition \ref{prop:j_HS}.(d) to derive a contradiction that $q_\mathrm{sim}$ is $t$-independent.
\end{proof}

\begin{rem}
The surjectivity of $D_w$ as in Proposition \ref{prop:transversality}, Proposition \ref{prop:transversality2}, and Proposition \ref{prop:transversality3} implies that the corresponding moduli spaces are smooth manifolds of dimension $\mu_\RFH(w_-)-\mu_\RFH(w_+)+1$, cf.~Proposition \ref{prop:transversality_cylinder}. 	
\end{rem}

\subsection{Compactness with $J\in\JJ_\mathrm{diag}\cup\JJ^\BB_\mathrm{reg}$}\label{sec:compact_J_diag}
In this section we prove a compactness result for $J\in\JJ_\mathrm{diag}\cup\JJ^\BB_\mathrm{reg}$ similar to the corresponding compactness result for $J\in\JJ_\mathrm{reg}$ in Section \ref{sec:compactness_J}. In contrast to the case of $J\in\JJ_\mathrm{reg}$, here we do not need any assumption on winding number. Conditions (i) and (ii) in the proposition below are the standing assumptions (A1) and (A2). 

\begin{prop}\label{prop:compactness2}
We assume that one of the following condition is fulfilled.
\begin{enumerate}[(i)]
	\item $\om$ vanishes on $\pi_2(M)$.
	\item $c_1^{TM}=\lambda\om$ on $\pi_2(M)$ for some $\lambda\in\R$ satisfying either $\lambda\nu\leq -\frac{1}{2}\dim M$ or $\lambda\nu\geq 2$ where $\nu\in\N$ is defined by $\om(\pi_2(M))=\nu\Z$. 
\end{enumerate}
Let $J\in\JJ_\mathrm{diag}\cup\JJ^\BB_\mathrm{reg}$. 
Then for every $\tau>0$ and $w_\pm\in\Crit \AA^\tau_f$ satisfying
\[
\mu_\RFH(w_-)-\mu_\RFH(w_+)\leq 3\,,
\]  
the moduli space $\widehat\MM(S_{w_-},S_{w_+},\AA^\tau_{f},J)$ is compact in the $C^\infty_{loc}$-topology. 
\end{prop}

\begin{proof}
Let $w_\nu=(u_\nu,\eta_\nu)$, $\nu\in\N$ be a sequence in $\widehat\MM(S_{w_-},S_{w_+},\AA^\tau_f,J)$. As  in the proof of Proposition \ref{prop:compactness1}, a uniform $L^\infty$-bound on $\eta_\nu$ is established in \cite[Proposition 6.4]{Fra08}, and it suffices to prove that the derivative of $u_\nu$ is uniformly bounded. Once we show that the horizontal derivative of $u_\nu$ is uniformly bounded, the vertical derivative of $u_\nu$ is also uniformly bounded since otherwise the usual bubbling-off analysis yields a pseudo-holomorphic sphere entirely contained in a single fiber of $\wp:E\to M$, which is absurd. To establish that the horizontal derivative of $u_\nu$ is uniformly bounded, we show that the sequence $q_\nu:=\Pi(w_\nu)\in\widehat\NN(\Pi(w_-),\Pi(w_+),\mathfrak{a}_{f},j)$ has uniformly bounded derivatives. Suppose that a $j_t$-holomorphic sphere $v$ bubbles off from $q_\nu$. If $v$ has negative first Chern number, then so does the underlying simple curve $v_\mathrm{sim}$, and thus, by hypothesis, $c_1^{TM}([v_\mathrm{sim}])\leq -\frac{1}{2}\dim M$. This yields that $\MM([v_\mathrm{sim}],j)/\mathrm{PSL}(2;\C)$ has negative dimension, see the definitions of $\mathfrak{j}_\mathrm{reg}$ and $\j_\mathrm{HS}$ in \eqref{eq:j_reg} and Proposition \ref{prop:j_HS}, and thus this case does not occur. If $v$ has positive first Chern number, it contributes at least $2c_1^{TM}([v])\geq 4$ to the Fredholm index. This contradicts $\mu_\FH(\Pi(w_-))-\mu_\FH(\Pi(w_+))\leq 3$, see \eqref{eq:cpt_ind}, and the surjectivity of the linearized operator in \eqref{eq:d_q} for $j\in\mathfrak{j}_\mathrm{reg}$. In the case of $j\in\mathfrak{j}_\mathrm{HS}$, we use instead Proposition \ref{prop:j_HS} and Lemma \ref{lem:simple_w}. This proves that the derivatives of $q_\nu$, and hence the horizontal derivatives of $u_\nu$, are uniformly bounded.
\end{proof}

\begin{rem}\label{rem:hope}
The assertion in Proposition \ref{prop:compactness2} is true under the weaker condition $\lambda\nu\geq1$ in the case that the horizontal part of $J\in\JJ_\mathrm{diag}\cup\JJ^\BB_\mathrm{reg}$ is $j\in\j_\mathrm{HS}$. Indeed, arguing as in the above proof, suppose that a sequence of Floer cylinders $q_\nu$ with  index 3 converges to a Floer cylinder $q$ together with a $j$-holomorphic sphere $v$. If $c_1^{TM}(v)= 1$, then the index of $q$ is at most $1$. This does not happen due to Proposition \ref{prop:j_HS}.(e). We also note that $\lambda\nu\geq1$ is sufficient to achieve necessary transversality results, see Lemma \ref{lem:simple_w}.

It is unclear to us whether the proposition still holds under the hypothesis $\lambda\nu\geq 0$, like \cite{HS95,Ono95}. For our purpose, it is enough to show compactness for $\widehat\MM(w_-,w_+,\AA^\tau_f,J)$ when $\mu_\RFH^h(w_-)-\mu_\RFH^h(w_+)\leq 2$. It could be that this gives additional one-dimensional constraints on $q_\nu$ and the aforementioned weaker assumption is sufficient. But we do not know how to make this idea rigorous.  
\end{rem}

\subsection{Coherent orientations with $J\in\JJ_\mathrm{diag}$}\label{sec:orientation}
In this section, let $J=\begin{pmatrix}
 i & 0\\
 0 & j
\end{pmatrix}$ with $j\in\j_\mathrm{reg}\cup\mathfrak{j}_\mathrm{HS}$. We assume that $\tau>0$ is sufficiently small so that the transversality results in Proposition \ref{prop:transversality} and Proposition \ref{prop:transversality2} hold for a given action bound $K>0$. 
 We fix coherent orientations for moduli spaces $\widehat\NN(\q_-,\q_+,\a_f,j)$ and describe below how to equip $\widehat\MM(S_{w_-},S_{w_+},\AA^\tau_f,J)$ with coherent orientations. In the case of $j\in\mathfrak{j}_\mathrm{HS}$, we only consider simple solutions, see \eqref{eq:simple_E}.
 
Let $w=(u,\eta)\in\widehat\MM(w_-,w_+,\AA^\tau_f,J)$ and denote $\q_\pm=\Pi(w_\pm)$ and $q=\Pi(w)$. Recall from \eqref{eq:lin_block} that $D_w=
\left(\begin{aligned}
 D_w^\mathrm{v} && D_w^\mathrm{hv}\\
 0 && D_w^\mathrm{h}
\end{aligned}\right)$
and that $D_w^\mathrm{v}$, $D_w^\mathrm{h}$, and $D_w$ are surjective. We also recall $\ker D_w^\mathrm{h}\cong \ker\mathfrak{d}_{q}=T_{q}\widehat\NN(\q_-,\q_+,\a_f,j)$ and $\ker D_w^\mathrm{v}$ is 1-dimensional vector space spanned by $(R(u),0)$. We have the exact sequence 
\[
0\to \ker D_w^\mathrm{v} \to \ker D_w \to \ker  D_w^\mathrm{h} \to0
\]
given by inclusion and projection.
Thus 
\begin{equation}\label{eq:ori}
\begin{split}
\Det D_w&=\Lambda^{\max} \ker D_w \otimes (\Lambda^{\max} \coker D_w)^*	\\
&\cong\Lambda^{\max} \ker D_w^\mathrm{v} \otimes \Lambda^{\max}\ker D_w^\mathrm{h} \otimes \R^*\\
&\cong\ker D_w^\mathrm{v} \otimes \Lambda^{\max} \ker\mathfrak{d}_{q} \otimes \R^*.
\end{split}
\end{equation}
The isomorphism in the second line is canonical, see \cite{FH93}. We orient $\ker D^\mathrm{v}_w$ by $(R(u),0)$. This, together with the orientation on $\ker\mathfrak{d}_q$ we fixed above, endows $\Det D_w=\Lambda^{\max}\ker D_w$ with an orientation. The orientation on $\widehat\MM(S_{w_-},S_{w_+},\AA^\tau_f,J)$ induced in this way is coherent with respect to the gluing operation. We will show below that the boundary operator constructed with these orientations corresponds to the boundary operator for the Floer homology of $\Sigma$ studied in Section \ref{sec:time_indep_j}. This provides an independent proof of coherence of orientations. 

For every $w\in\Crit\AA^\tau_f$, we orient the unstable manifold $W^u(\hat w)$ and $S_w$ by $R$, and $W^u(\check w)$ by $+1$. Then moduli spaces of flow lines with cascades $\widehat\MM(w_-,w_+,\AA^\tau_f,J)$ are oriented according to the fibered sum rule in Remark \ref{rem:ori_rule}.

From now on, let $J$ be $t$-independent, i.e.~$j\in\j_\mathrm{HS}$. We recall from Proposition \ref{prop:quantum_gysin_simple}.(b) that, for simple $q\in\widehat\NN_s(\q_-,\q_+,\a_f,j)$ with $\mu_\FH(\q_-)-\mu_\FH(\q_+)=2$ and $\om([q])>0$, the sign $\epsilon(q)\in\{-1,+1\}$ is determined in a way that $\ker\mathfrak{d}_q$ is oriented by $\epsilon(q)\p_sq\wedge\p_tq$.

\begin{lem}\label{lem:sign1}
Let $w_\pm\in\Crit\AA^\tau_f$ with 	$\mu_\RFH(w_-)-\mu_\RFH(w_+)=2$ and $\w(w_-)=\w(w_+)$. Suppose that $\mathbf{w}\in\widehat\MM(\check{w}_-,\hat{w}_+,\AA^\tau_f,J)$ has only one cascade $w\in\widehat\MM_s(S_{w_-},S_{w_+},\AA^\tau_f,J)$ with  simple $\Pi(w)=q$. Then the sign $\epsilon(\mathbf{w})$ agrees with $\epsilon(q)$.  
\end{lem}
\begin{proof}
We suppress $\AA^\tau_f$ and $J$ from notation below and denote $w_\pm=([u_\pm,\bar u_\pm],\eta_\pm)$ as usual.
The space $\ker D_w=T_w\widehat\MM(S_{w_-},S_{w_+})$ is spanned by $(R(u),0)$, $\p_s w$, and $\p_t w$, see also the proof of Lemma \ref{lem:chern_number=min_to_max} below, and our choice of orientation mentioned above yields that $\ker D_w$ is oriented by $\epsilon(q)(R(u),0)\wedge\p_s w\wedge \p_t w$. We have
 \[
 w\in \ev_-^{-1}(\check w_-)\cap \ev _+^{-1}(\hat w_+)\subset \widehat\MM(S_{w_-},S_{w_+})
 \] 
and view $w$ also as an element of
\begin{equation}\label{eq:fiber_sign}
\widehat\MM^1(\check{w}_-,\hat{w}_+):=W^u(\check w_-) _{\,\,\iota\!\!}  \x_{\ev_-} \widehat\MM(S_{w_-},S_{w_+}) _{\,\,\,\ev_+}  \!\!\x_{\,\iota} W^s(\hat w_+)
\end{equation}
to compute the sign. Here $\iota$ refers to the inclusions $W^u(\check w_-)\into S_{w_-}$ and $W^s(\hat w_+)\into S_{w_+}$. For the first fiber product we consider 
\[
d_{\check w_-}\iota-d_w\ev_-:T_{\check w_-} W^u(\check w_-) \oplus T_w\widehat\MM(S_{w_-},S_{w_+}) \longrightarrow T_{\check w_-} S_{w_-}\,.
\]
 Then it holds that
\[
\begin{split}
T_{(\check w_-,w)}\Big(W^u(\check w_-) _{\,\,\iota\!\!} \x_{\ev_-} \widehat\MM(S_{w_-},S_{w_+})\Big)&=\ker(d_{\check w_-}\iota-d_w\ev_-)\\
&=\ker\Big( -d_w\ev_-: T_w\widehat\MM(S_{w_-},S_{w_+})\to T_\theta S_{w_-}\Big)\,,	
\end{split}
\]
where the last line is due to $T_{\check w_-} W^u(\check w_-)=\{0\}$. We keep the minus sign in $-d_w\ev_-$ in order to record the orientation convention. One can readily see that $\p_s w$ and $(-\cov(u_-)R(u)+\p_tu,0)$ belong to $\ker (-d_w\ev_-)$ and form a basis of this 2-dimensional space. We orient $\ker (-d_w\ev_-)$ by $\epsilon(q)\p_s w\wedge (-\cov(u_-)R(u)+\p_tu,0)$. Then the quotient space $\frac{T_w\widehat\MM(S_{w_-},S_{w_+})}{\ker(-d_w\ev_-)}$ is oriented by $(R(u),0)+\ker(-d_w\ev_-)$ to make the isomorphism 
\[
\ker(-d_w\ev_-)\oplus \frac{T_w\widehat\MM(S_{w_-},S_{w_+})}{\ker(-d_w\ev_-)}\cong T_w\widehat\MM(S_{w_-},S_{w_+})
\]
orientation preserving. Therefore the isomorphism 
\[
\frac{T_w\widehat\MM(S_{w_-},S_{w_+})}{\ker(-d_w\ev_-)} \cong  T_{\check w_-} S_{w_-}
\]
induced by $-d_w\ev_-$ is orientation reversing, which is consistent with the fibered sum rule in Remark \ref{rem:ori_rule}
 since $\widehat\MM(S_{w_-},S_{w_+})$ is 3-dimensional and $S_{w_-}$ is 1-dimensional. This verifies that the orientation on $\ker (-d_w\ev_-)$ that we have chosen  obeys the fibered sum rule.

Next we study the second fiber product in \eqref{eq:fiber_sign}. We consider  
\[
d_w\ev_+|_{\ker(- d_w\ev_-)}-d_{\hat w_+}\iota: \ker(- d_w\ev_-)\oplus T_{\hat w_+}W^s(\hat w_+) \longrightarrow T_{\hat w_+}S_{w_+}
\]
so that 
\[
\begin{split}
T_w\widehat\MM^1(\check{w}_-,\hat{w}_+) &= \ker(d_w\ev_+|_{\ker(- d_w\ev_-)}-d_{\hat w_+}\iota) \\
&=\ker 	\Big(d_w\ev_+|_{\ker(- d_w\ev_-)}:\ker(- d_w\ev_-)\to  T_{\hat w_+}S_{w_+}\Big)\\
&= \ker (-d_w\ev_-)\cap \ker (d_w\ev_+)\,,
\end{split}
\]
where we used $T_{\hat w_+} W^w(\hat w_+)=\{0\}$. We orient $T_w\widehat\MM^1(\check{w}_-,\hat{w}_+)$ by $\epsilon(q)\p_s w$. This induces the orientation $(-\cov(u_-)R(u)+\p_tu,0)+T_w\widehat\MM(w_-,w_+)$ on the quotient space  $\frac{\ker (-d_w\ev_-)}{T_w\widehat\MM(w_-,w_+)}$ making the isomorphism
\[
T_w\widehat\MM^1(\check{w}_-,\hat{w}_+)\oplus \frac{\ker (-d_w\ev_-)}{\widehat\MM^1(\check{w}_-,\hat{w}_+)} \cong \ker (-d_w\ev_-)
\]
orientation preserving. Then the isomorphism 
\[
\frac{\ker (-d_w\ev_-)}{T_w\widehat\MM^1(\check{w}_-,\hat{w}_+)} \cong T_{\hat w_+}S_{w_+}
\]
induced by $d_w\ev_+$ gives the orientation $(-\cov(u_-)+\cov(u_+))R=-c_1^E([q])R$ on $T_{\hat w_+}S_{w_+}$, see Proposition \ref{prop:winding}.(d). Since $-c_1^E([q])=m\om([q])>0$, it agrees with the orientation we fixed on $S_{w_+}$. Thus the isomorphism is orientation preserving, as we wished according to the fibered sum rule since $W^s(\hat w_+)$ is 0-dimensional. We have checked that the orientation $\epsilon(q)\p_s w$ on $T_w\widehat\MM^1(\check{w}_-,\hat{w}_+)$ is indeed the one given by the fibered sum rule, and therefore $\epsilon(q)=\epsilon(\mathbf{w})$, see above \eqref{eq:sign_count_rfh}. This finishes the proof.  
\end{proof}

\begin{lem}\label{lem:chern_number=min_to_max}
Let $w_\pm\in\Crit\AA^\tau_f$ with 	$\mu_\RFH(w_-)-\mu_\RFH(w_+)=2$ and $\w(w_-)=\w(w_+)$. For every $q\in\widehat\NN_s(\Pi(w_-),\Pi(w_+),\a_f,j)$, we have
\[
\#\left\{[\mathbf{w}]=[w]\in \MM^1(\check{w}_-,\hat{w}_+,\AA^\tau_f,J) \mid \Pi(w)=q_\theta \textrm{ for some }\theta\in S^1\right\}=-\epsilon(q) c_1^E([q])
\]
where $q_\theta$ is defined by $q_\theta(s,t):=q_\theta(s,t+\theta)$.  
\end{lem}
\begin{proof}
Due to Theorem \ref{thm:bijection}, there exists a unique $w=(u,\eta)\in \widehat\MM(S_{w_-},S_{w_+},\AA^\tau_f,J)$ up to the $S^1$-action such that $\Pi(w)=q$. 
Since $\widehat\MM(S_{w_-},S_{w_+},\AA^\tau_f,J)$ is 3-dimensional and $q$ is simple, the connected component containing $(u,\eta)$ is diffeomorphic to $\R\x S^1\x S^1$ parametrized by
\[
\big(e^{2\pi i\vartheta}u(\cdot +\tau,\cdot+\theta),\eta(\cdot +\tau)\big)\,,\qquad (\tau,\theta,\vartheta)\in \R\x S^1\x S^1\,,
\]
which should be compared to \eqref{eq:moduli_q}. 
We write $w_\pm=([u_\pm,\bar u_\pm],\eta_\pm)$ as usual. We define
\[
u_\theta(s,t):=e^{-2\pi i\theta\cov(u_-)}u(s,t+\theta)
\]
and observe
\[
\begin{split}
\lim_{s\to-\infty}u_\theta(s,t) = e^{-2\pi i\theta\cov(u_-)} \lim_{s\to-\infty} u(s,t+\theta) = e^{-2\pi i\theta\cov(u_-)} u_-(t+\theta)=u_-(t)\,.
\end{split}
\]
We claim that 
\[
\lim_{s\to+\infty}u_\theta(s,t)=u_+(t)
\]
holds exactly when $\theta(\cov(u_+)-\cov(u_-))\in\Z$. 
Indeed, this follows from
\[
\begin{split}
\lim_{s\to+\infty}u_\theta(s,t)=e^{-2\pi i\theta\cov(u_-)}u_+(t+\theta)=e^{2\pi i \theta(\cov(u_+)-\cov(u_-))}u_+(t)\,.
\end{split}
\]
Hence $(u_\theta,\eta)\in\widehat\MM^1(\check{w}_-,\hat{w}_+,\AA^\tau_f,J)$ if and only if $\theta(\cov(u_+)-\cov(u_-))\in\Z$. This implies that the cardinality of the connected component of $\MM^1(\check{w}_-,\hat{w}_+,\AA^\tau_f,J)$ containing $[w]$ is $\cov(u_+)-\cov(u_-)$, which equals $-c_1^E([q])=m\om([q])>0$ by Proposition \ref{prop:winding}.(d). Moreover each element in that component has sign $\epsilon(q)$ as proved in Lemma \ref{lem:sign1}. This completes the proof. 
\end{proof}

Recall that we have chosen two auxiliary perfect Morse functions $h:\Crit\tilde f\to\R$ in \eqref{eq:ftn_h} and $h:\Crit\AA^\tau_f\to\R$ in \eqref{eq:perfect}. We are free to choose these functions as long as the transversality property for evaluation maps are satisfied, see \eqref{eq:ev_minus_plus} and \eqref{eq:ev_minus_plus2}. To ease computations in the following lemma, we require these two functions to have the property that $u(0)$ is a critical point of $h:\Crit\tilde f\to\R$ if and only if $([u,\bar u],\eta)$ is a critical point of $h:\Crit\AA^\tau_f\to\R$, see \eqref{eq:crit_diffeo}.

\begin{lem}\label{lem:equal_count}
Let $w_\pm\in\Crit\AA^\tau_f$ with $\w(w_-)=\w(w_+)$. Let $\q_\pm=[q_\pm,\bar q_\pm]=\Pi(w_\pm)$.
\begin{enumerate}[(a)]
	\item If $\mu_\RFH(w_-)-\mu_\RFH(w_+)=1$, then
		\[
		\begin{split}
			\#\MM(\hat w_-,\hat w_+,\AA^\tau_f,J)&=\#\NN(\hat q_-,\hat q_+,\tilde f,\tilde g)=-\#\NN(q_-,q_+,f,g)\\
			&=-\#\NN(\check q_-,\check q_+,\tilde f,\tilde g)=-\#\MM(\check w_-,\check w_+,\AA^\tau_f,J)\,.
		\end{split}
		\]
	\item If $\mu_\RFH(w_-)-\mu_\RFH(w_+)=2$, then
	 	\[
	 	\begin{split}
	 	&\#\MM^1_{\om=0}(\check w_-,\hat w_+,\AA^\tau_f,J)=\#\NN^1(\check q_-,\hat q_+,\tilde f,\tilde g)\,,\\[1ex]
	 	&\#\MM^2(\check w_-,\hat w_+,\AA^\tau_f,J)=\#\NN^2(\check q_-,\hat q_+,\tilde f,\tilde g)\,.	
	 	\end{split}
	 	\]
\end{enumerate}
\end{lem}
\begin{proof}
	We first consider (a). We show only the first identity. The last one follows from the same argument, and the rest are established in Lemma \ref{lem:moduli_1-1}. We recall that $h:\Crit\AA^\tau_f\to\R$ is chosen so that $\hat q_\pm=\hat u_\pm(0)$ where $\hat u_\pm$ are the loop components of $\hat w_\pm$ respectively.
	 By the assumption on index and winding number together with Proposition \ref{prop:j_HS}.(d) and Lemma \ref{lem:one-to-one_index}, we have $\widehat\NN(\q_-,\q_+,\a_f,j)=\widehat\NN(q_-,q_+,f,g)$. Thus Proposition \ref{prop:bijection} gives a bijection 
	\[
	\widehat\MM(S_{w_-},S_{w_+},\AA^\tau_f,J) \to \widehat\NN(S_{q_-},S_{q_+},\tilde f,\tilde g)\,,\qquad w=(u,\eta)\mapsto \tilde q\\
	\]
	where $\tilde q$ is the horizontal lift of $q=\Pi(w)$ with $\ev_-(\tilde q)=\ev_-(u)(0)$. Both spaces are oriented using the vector field $R$ and the orientation on $\widehat\NN(q_-,q_+,f,g)$. 
	Furthermore Corollary \ref{cor:rfh_lift} yields that the argument (in the fiber) of $\tilde q$ and that of $u(\cdot,0)$ agree. In particular we also have $\ev_+(\tilde q)=\ev_+(u)(0)$. Therefore all ingredients necessary for $\#\MM(\hat w_-,\hat w_+,\AA^\tau_f,J)$ and $\#\NN(\hat q_-,\hat q_+,\tilde f,\tilde g)$ (critical manifolds, (un)stable manifolds, cascades, etc.) are in one-to-one correspondence with orientation.  This proves (a). 
	
	The first identity in (b) can be shown in the same way for (a). We can argue similarly also for the second  identity in (b) using the fact that, for every $(w^1,w^2)\in\widehat\MM^2(\check w_-,\hat w_+,\AA^\tau_f,J)$ with $\mu_\RFH(w_-)-\mu_\RFH(w_+)=2$, both $\Pi(w^1)$ and $\Pi(w^2)$ are $t$-independent. 
\end{proof}

\section{Rabinowitz Floer homology with zero winding number}\label{sec:RFH}
In this section, we take $J$ from $\JJ_\mathrm{reg}$, $\JJ^\BB_\mathrm{reg}$, or $\JJ_\mathrm{diag}$. 
If $J\in\JJ_\mathrm{reg}^\BB\cup\JJ_\mathrm{diag}$, then we assume either (A1) or (A2) given in the introduction. As mentioned in the introduction, if we take the horizontal part $j$ of $J$ from $\j_{\mathrm{HS}}$, the condition $\lambda\nu\geq 2$ in (A2) can be weakened to $\lambda\nu\geq1$, see Remark \ref{rem:hope}. 
If we work with $J\in\JJ_\mathrm{reg}$, we assume either (A1) or (A3).
\begin{rem}\label{rem:finitely_many}
	Any of conditions (A1), (A2), and (A3) implies condition (H) in \eqref{eq:H}. Thus, by Remark \ref{rem:H} and Lemma \ref{lem:one-to-one_index}, we have the following consequences.
\begin{enumerate}[(i)]
	\item For given $k,\ell\in\Z$, there exist at most finitely many critical points $w\in\Crit h\subset\Crit\AA^\tau_f$ with $\w(w)=k$ and $\mu_\RFH^h(w)=\ell$.
	\item There exists $K>0$ such that for any $w_\pm\in\Crit h$ with $\mu_\RFH^h(w_-)-\mu_\RFH^h(w_+)\leq 2$, it holds that $\a_f(\Pi(w_-)-\a_f(\Pi(w_+))\leq K$. 
\end{enumerate}
\end{rem}
If $J\in\JJ_\mathrm{diag}$, we take $\tau\in(0,\tau_0)$ where $\tau_0>0$ is the constant given by Theorem \ref{thm:bijection} associated with $K$ from Remark \ref{rem:finitely_many}.(ii). Otherwise $\tau>0$ can be arbitrary.

\subsection{Rabinowitz Floer complex with zero winding number}\label{sec:RFH_zero_wind}
Let $-\infty\leq a<b\leq +\infty$. We define the $\Z$-module 
\begin{equation}\label{eq:RFC}
\RFC_*^{\w_0,(a,b)}(\AA^\tau_f)\,,\qquad *\in\Z	
\end{equation}
generated by all $w\in\Crit h\subset\Crit\AA^\tau_f$ with $\mu_\RFH^h(w)=*$, $
\w(w)=0$, and $\AA^\tau_f(w)\in(a,b)$. We do not include $h$ in the notation as its role is minor. We continue to require that $h$ is a perfect Morse function just to simplify explanation and computations. 
 We define the homomorphism
\begin{equation}\label{eq:boundary}
\p^{\w_0}=\p^{\w_0}_J:\RFC_*^{\w_0,(a,b)}(\AA^\tau_f)\longrightarrow \RFC_{*-1}^{\w_0,(a,b)}(\AA^\tau_f)	
\end{equation}
as linear extension of 
\[
\p^{\w_0}w_-:=\sum_{w_+}\#\MM(w_-,w_+,\AA^\tau_f,J)\,w_+\,,
\]
where the sum is taken over all $w_+\in\Crit h$ with $\w(w_+)=0$, $\mu_\RFH^h(w_+)=*-1$, and $\AA^\tau_f(w_+)\in(a,b)$. The signed count $\#\MM(w_-,w_+,\AA^\tau_f,J)\in\Z$ is well-defined thanks to the transversality and compactness results in Section \ref{sec:J}.

\begin{prop}\label{prop:boundary_operator}
The homomorphism \eqref{eq:boundary} is a boundary operator, i.e.~$\p^{\w_0}\circ\p^{\w_0}=0$.
\end{prop}
\begin{proof}
The proof follows from the usual scheme in Floer homology if we ensure that solutions of the Rabinowitz Floer equation \eqref{eq:Floer_eqn_for_A} between critical points with zero winding number only break along  critical points with zero winding number.
Let $w^n=(u^n,\eta^n)$ be  a sequence of finite energy solutions of  \eqref{eq:Floer_eqn_for_A} breaking to solutions $w'=(u',\eta')$ and $w''=(u'',\eta'')$. Then, for sufficiently large $n\in\N$, the cylinder component $u^n$ is homotopic relative to the asymptotic ends to the concatenation $u'\#u''$ of $u'$ and $u''$ along the common asymptotic orbit, and hence 
\[
u'\cdot\OO_E+u''\cdot\OO_E = (u'\# u'')\cdot\OO_E =  u^n\cdot \OO_E\,.
\]
Therefore  $u^n\cdot\OO_E=0$ if and only if $u'\cdot\OO_E=u''\cdot\OO_E=0$ due to Proposition \ref{prop:positivity_of_intersection}.
It suffices to consider only this case since every cascade involved in $\widehat\MM(w_-,w_+,\AA^\tau_f,J)$ with $\w(w_\pm)=0$ does not intersect the zero section $\OO_E$ due to Proposition \ref{prop:winding}.(d) and Proposition \ref{prop:positivity_of_intersection}. Hence compactness results in Proposition \ref{prop:compactness1} and Proposition \ref{prop:compactness2} together with standard gluing analysis prove the proposition.
\end{proof}

We refer to Remark \ref{rem:boundary_operator} below for another viewpoint on Proposition \ref{prop:boundary_operator}.

\begin{rem}
Without excluding bubbling-off of pseudo-holomorphic spheres one cannot prevent the following scenario in the proof of Proposition \ref{prop:boundary_operator}. There might be a sequence $w^n=(u^n,\eta^n)$ with $u^n\cdot\OO_E=0$ that breaks to $w'=(u',\eta')$ and $w''=(u'',\eta'')$ together with bubbling-off of a pseudo-holomorphic sphere $v:S^2\to\OO_E\subset E$ having the following intersection numbers with $\OO_E$:
\[
u^n\cdot\OO_E=u'\cdot\OO_E+u''\cdot\OO_E+v\cdot\OO_E=0+1+(-1)=0\,.
\]
\end{rem}

Due to Proposition \ref{prop:boundary_operator}, we can define the homology
\[
\RFH^{\w_0,(a,b)}_*(\AA^\tau_f,J):= \H_*\big(\RFC^{\w_0,(a,b)}(\AA^\tau_f),\p^{\w_0}_J\big)\,.
\]
We will see in the next section that $\RFH_*^{\w_0,(a,b)}(\AA^\tau_f,J)$ is invariant under the change of $J$, so we sometimes omit $J$ from the notation. We will also see that $\RFH_*^{\w_0,(-\infty,+\infty)}(\AA^\tau_f)$ is independent of the choice of $\tau$. When specific choices of $f$ and $\tau$ are irrelevant, we simply write  
\begin{equation}\label{eq:wind_rfh}
\RFH_*^{\w_0}(E,\Sigma)=\RFH_*^{\w_0}(E,\Sigma_\tau)=\RFH_*^{\w_0}(\AA^\tau_f)=\RFH_*^{\w_0,(-\infty,+\infty)}(\AA^\tau_f)\,.	
\end{equation}
We will observe in Remark \ref{rem:Lambda_module_str} that $\RFH_*^{\w_0}(E,\Sigma)$ carries a module structure over the Novikov ring $\Lambda$ defined in \eqref{eq:nov}. This completes the proof of Theorem \ref{thm:main}.(a). 

For any $-\infty\leq a'<a<b<b'\leq+\infty$, there are chain homomorphisms induced by canonical inclusions and projections
\[
\begin{split}
&\RFC_*^{\w_0,(a,b)}(\AA^\tau_f)\hookrightarrow \RFC_*^{\w_0,(a,b')}(\AA^\tau_f)\,,\\[0.5ex]
&\RFC_*^{\w_0,(a',b)}(\AA^\tau_f)\twoheadrightarrow \RFC_*^{\w_0,(a,b)}(\AA^\tau_f)\,.	
\end{split}
\]
These induce action filtration homomorphisms 
\begin{equation}\label{eq:direct_system_RFH}
\begin{split}
	&\RFH_*^{\w_0,(a,b)}(\AA^\tau_f)\longrightarrow \RFH_*^{\w_0,(a,b')}(\AA^\tau_f)\,,\\[0.5ex]
	&\RFH_*^{\w_0,(a',b)}(\AA^\tau_f)\longrightarrow \RFH_*^{\w_0,(a,b)}(\AA^\tau_f)\,,
\end{split}	
\end{equation}
which form a bidirect system. Since $\RFH_*^{\w_0,(a,b)}(\AA^\tau_f)$ stabilizes for a large action-window in each degree $*\in\Z$ due to Remark \ref{rem:finitely_many}.(i), we have
\begin{equation}\label{eq:wind_lim}
\RFH_*^{\w_0}(E,\Sigma)= \varinjlim_{b\uparrow+\infty}\varprojlim_{a\downarrow-\infty}\RFH_*^{\w_0,(a,b)}(\AA^\tau_f)= \varprojlim_{a\downarrow-\infty}\varinjlim_{b\uparrow+\infty}\RFH_*^{\w_0,(a,b)}(\AA^\tau_f)\,.	
\end{equation}

\begin{rem}\label{rem:CY}
	We wonder whether it is possible to define $\RFH^{\w_0}(\AA^\tau_f)$ for more general class of line bundles $E\to M$, e.g.~those with $c_1^{TM}=0$, see Remark \ref{rem:hope}. If one succeeds in this generalization, then  
		 Remark \ref{rem:finitely_many}.(i) is not true anymore. Then we may define various versions of $\RFH^{\w_0}(\AA^\tau_f)$ by e.g.~completing the action-window either at the chain level as in \eqref{eq:tilde_f_complete} or at the homology level as in \eqref{eq:wind_lim}. We then encounter the same question as in Remark \ref{rem:ML2}.
\end{rem}

\subsection{Invariance properties}\label{sec:invariance}
In this subsection we study invariance properties of $\RFH^{\w_0}$ under the change of $J$, $h$, $f$, or $\tau$. In all cases, continuation homomorphisms give isomorphisms as usual. We refer to \cite[Theorem A.18]{Fra04} for invariance properties in the finite dimensional Morse-Bott setting. One issue is a uniform bound on the Lagrange multiplier $\eta$ for solutions appearing in the construction of continuation homomorphisms. For the invariance problem for $J$, $h$, or $f$, this bound can be achieved in a similar way to obtain a uniform bound for solutions of \eqref{eq:Floer_eqn_for_A} used in the boundary operator, see e.g.~\cite[Step 2 in p.284]{CF09} and \cite[Theorem 2.10]{AF10}. Thus we do not prove this here. The invariance property for the change of $\tau$ in the Liouville setting, when the action-window is the whole $(-\infty,+\infty)$, is established in \cite{CF09}. Since all ingredients used in defining $\RFH^{\w_0}(E,\Sigma_\tau)$ lie in the region $E\setminus\OO_E$ where $\Omega$ is exact, it is reasonable to expect that $\RFH^{\w_0}(E,\Sigma_\tau)$  is independent of $\tau$ up to isomorphism. As we also want to have an invariance property for finite action-window, we carry out details below.

\subsubsection*{Changing the almost complex structure}
 Let $-\infty\leq a<b\leq+\infty$. Let $J_0,J_1\in\JJ_\mathrm{reg}\cup\JJ^\BB_\mathrm{reg}\cup\JJ_\mathrm{diag}$. If $J_0$ or $J_1$ is diagonal, we take a sufficiently small $\tau>0$. If both $J_0$ and $J_1$ belong to $\JJ^\BB_\mathrm{reg}\cup\JJ_\mathrm{diag}$, we assume (A1) or (A2), otherwise we assume (A1) or (A3). We choose a smooth homotopy $\{J_r\}_{r\in[0,1]}\subset\JJ$ of almost complex structures connecting $J_0$ and $J_1$, see \eqref{eq:J_incl}. Let $\beta:\R\to[0,1]$ be a smooth function such that $\beta'\geq 0$, $\beta(-1)=0$, and $\beta(1)=1$.
As usual, we want to construct an isomorphism 
\begin{equation}\label{eq:continuation_J}
\RFH_*^{\w_0,(a,b)}(\AA^\tau_{f},J_0) \stackrel{\cong}{\longrightarrow} \RFH_*^{\w_0,(a,b)}(\AA^\tau_{f},J_1)\qquad \forall *\in\Z
\end{equation}
by counting finite energy solutions $w=(u,\eta):C^\infty(\R\x S^1,E)\x C^\infty(\R,\R)$ of 
\begin{equation}\label{eq:Floer_eqn_homotopy}
\left(\begin{aligned}
&\p_su+J_{\beta(s)}(\eta,t,u)\big(\p_tu-\eta X_{\mu_\tau}(u) - X_F(u)\big)\\[.5ex]
&\p_s\eta-\int_0^1\mu_\tau(u)dt\,
\end{aligned}\right)=0\,
\end{equation}
together with Morse gradient flow lines of $h$. 
For this, we choose a homotopy $\{J_r\}\subset\JJ$ such that a parametrized version of transversality and compactness results, that we have established in Section \ref{sec:J}, for solutions of \eqref{eq:Floer_eqn_homotopy} holds. If we assume (A2), then we need to choose a homotopy $\{J_r\}$ inside $\JJ^\BB$ to argue as in Proposition \ref{prop:compactness2}.
Then the homomorphism \eqref{eq:continuation_J} is indeed defined with one additional input, namely Proposition \ref{prop:positivity_of_intersection} holds also for solutions of \eqref{eq:Floer_eqn_homotopy}. An inverse homomorphism is constructed in the same way by interchanging the roles of $J_0$ and $J_1$.

\subsubsection*{Changing the functions $f$ and $h$}
Invariance properties with respect to the change of $f$ or $h$ are also standard. For later purpose we remark that we have chosen $h$ to be perfect for convenience, but this is not necessary. We also note that since $f$ is small, the action value $\AA^\tau_f(w)$ is close to $(1+\tau)\Z$  for every $w\in\Crit\AA^\tau_f$ with $\w(w)=0$ due to Lemma \ref{lem:one-to-one_index} and the fact that critical values of $\a_f$ are close to integers.
Let $f':M\to\R$ be another $C^2$-small Morse function, and let $h':\Crit\AA^\tau_{f'}\to\R$ be a Morse function. For $a,b\in(1+\tau)(\Z+\frac{1}{2})\cup \{-\infty,+\infty\}$, there is an isomorphism
\begin{equation}\label{eq:isom_f}
	\RFH^{\w_0,(a,b)}_*(\AA^\tau_f,h,J)\stackrel{\cong}{\longrightarrow} \RFH^{\w_0,(a,b)}_*(\AA^\tau_{f'},h',J)\qquad \forall *\in\Z
\end{equation}
which is constructed by counting flow lines with cascades for a homotopy between $f$ and $f'$ and a homotopy between $h$ and $h'$. Here we add $h$ and $h'$ in the notation for clarity.

\subsubsection*{Changing the radius}

Next, we fix $\tau_-,\tau_+>0$ and aim to construct an isomorphism 
	\[
	\RFH^{\w_0,(a_-,b_-)}_*(\AA^{\tau_-}_f) \cong \RFH^{\w_0,(a_+,b_+)}_*(\AA^{\tau_+}_f)
	\]
for suitable $a_\pm,b_\pm\in\R$ or $(a_\pm,b_\pm)=(-\infty,+\infty)$ following closely \cite{CF09}. Such an isomorphism can be obtained also from isomorphisms to the Floer homology of $\Sigma$ for small $\tau_-,\tau_+$ as shown in Proposition \ref{prop:cap_product}.(a) below. This invariance property for $\RFH^{\w_0}(E,\Sigma_\tau)$ under the change of  radius $\tau>0$ should be compared with the fact that the full Rabinowitz Floer homology of $(E,\Sigma_\tau)$ often depends on $\tau$, see Section \ref{sec:full}.

We choose any $\tau_0,\tau_1>0$ satisfying $\tau_{0}<\min\{\tau_-,\tau_+\}\leq\max\{\tau_-,\tau_+\}<\tau_{1}$ and define
	 \[
	 V:=\{\mu_{\tau_1}\leq 1\} \,,\quad 
c^\lambda:=\|\lambda|_{V}\|_{L^\infty}\,,\quad  c^F:=\|F|_{V}\|_{L^\infty}+c^\lambda\|X_F|_{V}\|_{L^\infty},\quad c_1:=c^\lambda+c^F\,.
	 \] 

\begin{lem}\label{lem:key}
There exists $\epsilon>0$ such that the following claim holds for every $\tau$ between $\tau_-$ and $\tau_+$. Let $w=([u,\bar u],\eta)\in\widetilde{\mathscr L}(E)\x \R$ with $u(S^1)\subset E\setminus\OO_E$ and $\w(u,\bar u)=0$. Suppose that $\|\nabla\AA^{\tau}_f(w)\|_{L^2}<\epsilon$. Then, 
\[
 c_2(\AA^\tau_f(w)-c_1)\leq \eta \leq c_3(\AA^\tau_f(w)+c_1)
\]
holds for $c_2,c_3\in\{(1+\tau_0)^{-1},(1+\tau_1)^{-1}\}$, depending on the sign of $\eta$, see the proof for details.
\end{lem}
\begin{rem}
The actual values of $c_2,c_3>0$ are not relevant in what follows. 
\end{rem}
\begin{proof}
	We first note that $\Omega=d\lambda$ on $E\setminus\OO_E$, where $\lambda=(\tfrac{1}{m}+\pi r^2)\alpha$, and $\AA^\tau_f(w)$ does not depend on the capping disk component $\bar u$ provided $\w(u,\bar u)=0$, see Section \ref{sec:zero_winding}. This allows us to apply Step 2 in the proof of \cite[Proposition 3.2]{CF09} which is carried out in the Liouville setting. Namely, for any  $\delta>0$, there exists $\epsilon>0$ such that the following implication holds: 
	\[
	\|\nabla\AA^{\tau}_f(w)\|_{L^2}\leq \epsilon  \quad \Longrightarrow \quad u(S^1)\subset  \mu_\tau^{-1}\big((-\delta,\delta)\big)\,.
	\] 
	As pointed out after the proof of \cite[Proposition 3.2]{CF09}, Step 2 indeed holds for all $\tau$ between $\tau_-$ and $\tau_+$ simultaneously. 
	 We take any $\delta\in(0,1)$ and choose $\epsilon\in(0,1)$ as above so that $u(S^1)\subset \mu_\tau^{-1}((-\delta,\delta))\subset V$. Recalling $\mu_\tau=m\pi r^2-\tau$ and $X_{\mu_\tau}=-m R$, we estimate
	\[
	\begin{split}
		\AA^{\tau}_f(w)&=-\int_0^1u^*\lambda-\eta\int_0^1\mu_\tau(u)\,dt-\int_0^1F(u)\,dt \\
		&=-\int_0^1 \lambda\big(\dot u-\eta X_{\mu_\tau}(u)-X_F(u)\big)\,dt +\eta\int_0^1 \lambda(mR(u))\,dt \\
		&\quad -\int_0^1\lambda(X_F(u))\,dt -\eta\int_0^1\mu_\tau(u)\,dt-\int_0^1F(u)\,dt\\
		&\geq -c^\lambda\|\nabla \AA^\tau_f(w)\|_{L^2}+(1+\tau)\eta -c^F \,.
	\end{split}
	\]
	Therefore,
	\[
	\eta\leq c_3(\AA^\tau_f(w)+ c_1)
	\]
	where $c_3=(1+\tau_0)^{-1}$ if $\eta\geq0$, and $c_3=(1+\tau_1)^{-1}$ if $\eta<0$. Similarly, we also obtain 
	\[
	\AA^{\tau}_f(w)\leq c^\lambda\|\nabla \AA^\tau_f(w)\|_{L^2}+(1+\tau)\eta  + c^F\,.
	\]
	This yields that 
 	\[
 	 c_3(\AA^\tau_f(w) -c_1)\leq \eta\,,
 	\]
 	where in this case $c_3=(1+\tau_1)^{-1}$ if $\eta\geq0$, and $c_3=(1+\tau_0)^{-1}$ if $\eta<0$. 
\end{proof}

Let $\beta:\R\to[0,1]$ be a smooth function with $\beta(-1)=0$, $\beta(1)=1$, and $0\leq\beta'\leq1$. For
\[
\tau_\beta(s):=(1-\beta(s))\tau_- + \beta(s)\tau_{+}\,,
\]
we consider the functional $\AA^{\tau_\beta}_f:\widetilde{\mathcal{L}}(E)\x\R\to\R$ depending on $s\in\R$. 
For $w_-\in\Crit\AA^{\tau_-}_f$ and $w_{+}\in\Crit\AA^{\tau_{+}}_f$ with $w_\pm=([u_\pm,\bar u_\pm],\eta_\pm)$ and $\w(w_\pm)=0$, we consider $w(s)=(u(s),\eta(s))$ solving
\begin{equation}\label{eq:s_dependent_gradient}
\p_sw(s)+\nabla\AA^{\tau_\beta(s)}_f(w(s))=0\,,\quad
\lim_{s\to \pm\infty}w(s) = (u_\pm,\eta_\pm)\,,\quad  [\bar u_-\#u\#\bar u_+^\mathrm{rev}]=0 \;\textrm{ in }\;\Gamma_E
\end{equation}
where $\nabla\AA^{\tau_\beta}_f$ denotes the gradient of $\AA^{\tau_\beta}_f$ with respect to $\mathfrak{m}$ defined in \eqref{eq:bilinear form}.

\begin{lem}\label{lem:s_dependent_action}
	Let $w$ be a solution of \eqref{eq:s_dependent_gradient}. For any $-\infty\leq s_-\leq s_+\leq+\infty$, we have 
	\[
	\begin{split}
	&\AA^{\tau_\beta(s_+)}_f(w(s_+)) \leq  \AA^{\tau_\beta(s_-)}_f(w(s_-)) + (\tau_+-\tau_-)\sup_{s\in[s_-,s_+]}\eta(s) \qquad \textrm{if }\; \tau_+\geq \tau_-\,,	\\
	&\AA^{\tau_\beta(s_+)}_f(w(s_+)) \leq  \AA^{\tau_\beta(s_-)}_f(w(s_-)) + (\tau_+-\tau_-)\inf_{s\in[s_-,s_+]}\eta(s) \qquad\, \textrm{if }\; \tau_+\leq \tau_-\,.
	\end{split}
	\]
	In particular, 
	\[
	\begin{split}
	&\AA^{\tau_+}_f(w_+) \leq  \AA^{\tau_-}_f(w_-)  \qquad \textrm{if }\; \tau_+\geq \tau_- \text{ and }\; \sup_\R\eta\leq0 \,,	\\
	&\AA^{\tau_+}_f(w_+) \leq  \AA^{\tau_-}_f(w_-)   \qquad\, \textrm{if }\; \tau_+\leq \tau_- \text{ and }\; \inf_\R\eta\geq0\,.
	\end{split}
	\]
\end{lem}
\begin{proof}
	The claim immediately follows from the estimate
\[
\begin{split}
0\leq \int_{s_-}^{s_+}\|\p_s w\|_{L^2}^2\,ds 
&= \int_{s_-}^{s_+}\mathfrak{m}(  -\nabla \AA^{\tau_\beta}_f(w), \p_s w ) \,ds \\
&= \int_{s_-}^{s_+}  -\frac{d}{ds}\AA^{\tau_\beta}_f(w)\,ds -  \int_{s_-}^{s_+} \eta\int_0^1 \p_s\mu_{\tau_\beta} (u)\,dtds\\
&= \AA^{\tau_\beta(s_-)}_f(w(s_-))  - \AA^{\tau_\beta(s_+)}_f(w(s_+)) + (\tau_+-\tau_-) \int_{s_-}^{s_+} \eta \beta' \,ds 
\end{split}
\]
where $\mu_{\tau_\beta}= m\pi r^2 - \tau_\beta$.
\end{proof}

Let $a,b\in\R$ be arbitrary. As in Proposition \ref{prop:compactness1} and \cite[Proposition 2.5, Lemma 2.11]{AS09}, a maximum principle ensures that every solution $w=(u,\eta)$ of \eqref{eq:s_dependent_gradient} with asymptotic orbits $w_\pm$ satisfying $\AA^{\tau_-}_f(w_-)\leq b$ and $\AA^{\tau_+}_f(w_+)\geq a$ has the cylinder component $u$ mapping into $\{x\in E\mid r(x)<r_{0}\}$ for some $r_{0}>0$ depending on $a,b\in\R$. We choose any $c_\mu>\max\{\tau_-,\tau_+,m\pi r_0^2\}$ so that, for every such solution $w=(u,\eta)$, we have
\begin{equation}\label{eq:p_eta}
	\|\p_s\eta\|_{L^\infty} =\left\|\int_0^1\mu_{\tau_\beta}(u)\,dt\right\|_{L^\infty}\leq \|\mu_{\tau_\beta}(u)\|_{L^\infty} \leq c_\mu\,,
\end{equation}
see \eqref{eq:Floer_eqn_for_A} for the first equality.

\begin{lem}\label{lem:bound_eta}
Let $w$ be a solution of \eqref{eq:s_dependent_gradient} connecting $w_-\in\Crit\AA^{\tau_-}_f$ and $w_+\in\Crit\AA^{\tau_+}_f$ satisfying $\AA^{\tau_-}_f(w_-)\leq b$ and $\AA^{\tau_+}_f(w_+)\geq a$ as above.  
Let $\epsilon,c_1,c_2,c_3>0$ be  constants such that the statement of Lemma \ref{lem:key} holds for all $\tau$ between $\tau_-$ and $\tau_+$. 
\begin{enumerate}[(a)]
	\item If $\tau_+\geq\tau_-$, the following holds with $c_4:=c_3+\frac{c_\mu}{\epsilon^2}$.
\[
\big(1-c_4(\tau_+-\tau_-)\big)\sup_{s\in\R}\eta(s)\leq c_3\big(\AA^{\tau_-}_f(w_-)+ c_1\big) +\frac{c_\mu}{\epsilon^2}\big(\AA^{\tau_-}_f(w_-)-\AA^{\tau_+}_f(w_+)\big)\,.
\]
 \item If $\tau_+\leq \tau_-$, the following holds with $c_5:=c_2+\frac{c_\mu}{\epsilon^2}$.
\[
\big(1+c_5(\tau_+-\tau_-)\big)\inf_{s\in\R}\eta(s)\geq c_2\big(\AA^{\tau_+}_f(w_+)-c_1\big)-\frac{c_\mu}{\epsilon^2}\big(\AA^{\tau_-}_f(w_-)-\AA^{\tau_+}_f(w_+)\big)\,.
\]
\item For any $\tau_-,\tau_+$, the following holds with $c_6:=\frac{1}{1+\tau_0}+\frac{c_\mu}{\epsilon^2}$.
\[
\big(1-c_6|\tau_+-\tau_-|\big)\|\eta\|_{L^\infty} \leq \frac{1}{1+\tau_0}\big(\max\{|a|,|b|\}+c_1\big) +\frac{c_\mu}{\epsilon^2}(b-a)\,.
\]
\end{enumerate}
\end{lem}
\begin{proof}
For each $\sigma\in\R$, we set
\[
\tau(\sigma):=\inf\big\{\tau\geq 0 \mid \|\nabla \AA^{\tau_\beta(\sigma+\tau)}_f\big(w(\sigma+\tau)\big)\|_{L^2}<\epsilon \big\}\,.
\]
As in the proof of Lemma \ref{lem:s_dependent_action}, we estimate 
\[
\begin{split}
\tau(\sigma)\epsilon^2 \leq \int_\sigma^{\sigma+\tau(\sigma)} \|\nabla \AA^{\tau_\beta}_f(w)\|_{L^2}^2\,ds 
&\leq\int_{-\infty}^{+\infty} \|\nabla \AA^{\tau_\beta}_f(w)\|_{L^2}^2\,ds  \\
&\leq  \AA^{\tau_-}_f(w_-)  - \AA^{\tau_{+}}_f(w_{+}) +  (\tau_+-\tau_-)\int_{-\infty}^\infty \eta \beta'\,ds \,.
\end{split}
\]
Using this, Lemma \ref{lem:key}, Lemma \ref{lem:s_dependent_action}, and estimate \eqref{eq:p_eta}, we deduce for $\tau_+\geq\tau_-$ 
\[
\begin{split}
	\eta(\sigma) &= \eta(\sigma+\tau(\sigma))-\int_\sigma^{\sigma+\tau(\sigma)}\p_s\eta\,ds \\
	&\leq  c_3\Big(\AA^{\tau_\beta(\sigma+\tau(\sigma))}_f(w(\sigma+\tau(\sigma))+c_1\Big) + \tau(\sigma)\|\p_s\eta\|_{L^\infty}\\
	&\leq c_3\Big(\AA^{\tau_-}_f(w_-) + (\tau_+-\tau_-)\sup_{\R}\eta + c_1\Big) + \frac{c_\mu}{\epsilon^2}\Big(\AA^{\tau_-}_f(w_-)-\AA^{\tau_+}_f(w_+)+(\tau_+-\tau_-)\sup_{\R}\eta\Big)\,.
\end{split}
\]
This proves (a). Similarly, for $\tau_+\leq\tau_-$, we have
\[
\begin{split}
	\eta(\sigma) &= \eta(\sigma+\tau(\sigma))-\int_\sigma^{\sigma+\tau(\sigma)}\p_s\eta\,ds \\
	&\geq c_2\Big(\AA^{\tau_\beta(\sigma+\tau(\sigma))}_f(w(\sigma+\tau(\sigma))-c_1\Big) -\tau(\sigma)\|\p_s\eta\|_{L^\infty}\\
	&\geq  c_2\Big(\AA^{\tau_+}_f(w_+) - (\tau_+-\tau_-)\inf_{\R}\eta -c_1\Big) -\frac{c_\mu}{\epsilon^2}\Big(\AA^{\tau_-}_f(w_-)-\AA^{\tau_+}_f(w_+) +(\tau_+-\tau_-)\inf_{\R}\eta\Big)\,.
\end{split}
\]
This proves (b). Statement (c) follows from 
\[
\begin{split}
	|\eta(\sigma)| &\leq |\eta(\sigma+\tau(\sigma))|+\int_\sigma^{\sigma+\tau(\sigma)}|\p_s\eta|\,ds \\
	&\leq \frac{1}{1+\tau_0}\Big(|\AA^{\tau_\beta(\sigma+\tau(\sigma))}_f(w(\sigma+\tau(\sigma))|+c_1\Big)+\tau(\sigma)\|\p_s\eta\|_{L^\infty}\\
	&\leq \frac{1}{1+\tau_0}\Big(\max\{|a|,|b|\}+\|\eta\|_{L^\infty}|\tau_+-\tau_-|+c_1\Big)+\frac{c_\mu}{\epsilon^2}\big(b-a+\|\eta\|_{L^\infty}|\tau_+-\tau_-|\Big)\,.
\end{split}
\]
This completes the proof. 
\end{proof}

We consider the discrete set, see Lemma \ref{lem:one-to-one_index}, 
\[
\mathrm{spec}^{\w_0}\AA^{\tau}_f:= \{\AA^\tau_f(w) \mid w\in\Crit\AA^\tau_f,\; \w(w)=0\}\subset \R\,.
\]

\begin{lem}\label{lem:continuation_action}
For a given $b\in\R\setminus \mathrm{spec}^{\w_0}\AA^{\tau_-}_f$, there exists $\delta>0$ such that if $|\tau_+-\tau_-|<\delta$, the implication 
\[
\AA^{\tau_-}_f(w_-)<b \quad \Longrightarrow\quad  \AA^{\tau_+}_f(w_+)<b
\]
holds for every solution $w$ of \eqref{eq:s_dependent_gradient}.
\end{lem}
\begin{proof}
Due to Lemma \ref{lem:s_dependent_action}, it suffices to analyze the case $\tau_+\geq\tau_-$ with $\sup_\R\eta>0$ and $\tau_+\leq\tau_-$ with $\inf_\R\eta<0$. We assume $\sup_\R\eta>0$ if $\tau_+\geq\tau_-$, and $\inf_\R\eta<0$ if $\tau_+\leq\tau_-$.
We choose $\epsilon>0$ so small that 
\begin{equation}\label{eq:spec_epsilon}
\mathrm{spec}^{\w_0}\AA^{\tau_-}_f\cap (b-\epsilon,b+\epsilon)=\emptyset\,.	
\end{equation}
Assume for a contradiction $\AA^{\tau_-}_f(w_-)<b\leq \AA^{\tau_+}_f(w_+)$. Then $\AA^{\tau_-}_f(w_-)<\AA^{\tau_+}_f(w_+)$ together with (a) and (b) in Lemma \ref{lem:bound_eta} implies
\[
\begin{split}
&\big(1-c_4(\tau_+-\tau_-)\big)\sup_{s\in\R}\eta(s)\leq c_3(\AA^{\tau_-}_f(w_-)+ c_1) \qquad \textrm{if }\; \tau_+\geq \tau_-\,, \\	
&\big(1+c_5(\tau_+-\tau_-)\big)\inf_{s\in\R}\eta(s) \,\geq c_2(\AA^{\tau_+}_f(w_+)-c_1) \qquad \textrm{if }\; \tau_+\leq \tau_- \,.
\end{split}
\]
Choosing $\delta<\min\{\frac{1}{2c_4},\frac{1}{2c_5}\}$ leads to 
\begin{equation}\label{eq:c_1,c_2}
\begin{split}
&\sup_{s\in\R}\eta(s)\leq 2c_3(\AA^{\tau_-}_f(w_-)+ c_1) \qquad \textrm{if }\; \tau_+\geq \tau_-\,, \\
&\inf_{s\in\R}\eta(s)\geq 2c_2(\AA^{\tau_+}_f(w_+)-c_1) \qquad \,\textrm{if }\; \tau_+\leq \tau_- \,.
\end{split}
\end{equation}
Combining Lemma \ref{lem:s_dependent_action} with \eqref{eq:c_1,c_2}, we obtain
\begin{equation}\label{eq:inv_est}
	\begin{split}
&\AA^{\tau_+}_f(w_+) \leq   \AA^{\tau_-}_f(w_-) + 2c_3(\tau_+-\tau_-)  (\AA^{\tau_-}_f(w_-)+c_1) \qquad \textrm{if }\; \tau_+\geq \tau_-\,, \\[.5ex]
&\AA^{\tau_+}_f(w_+) \leq   \AA^{\tau_-}_f(w_-) + 2c_2(\tau_+-\tau_-) (\AA^{\tau_+}_f(w_+)-c_1) \qquad \,\textrm{if }\; \tau_+\leq \tau_- \,.
	\end{split}
\end{equation}
We require $\delta$ to additionally satisfy $\delta<\min\{\frac{\epsilon}{2c_3|b+c_1|},\frac{\epsilon}{2c_2|c_1-b|}\}$.   
We first consider the case $\tau_+\geq \tau_-$ and use the first line of \eqref{eq:inv_est}. If $\AA^{\tau_-}_f(w_-)\leq -c_1$, this leads to the contradiction $\AA^{\tau_+}_f(w_+)\leq  \AA^{\tau_-}_f(w_-)<b$. We now assume $-c_1<\AA_f^{\tau_-}(w_-)<b$. Then by $\delta<\frac{\epsilon}{2c_3(b+c_1)}$ we arrive at the contradiction $\AA^{\tau_+}_f(w_+) \leq \AA^{\tau_-}_f(w_-)  + \epsilon <b$, where the latter inequality is by \eqref{eq:spec_epsilon}. 
Suppose $\tau_+\leq \tau_-$. In this case, we use the second line of \eqref{eq:inv_est}. For $\AA^{\tau_+}_f(w_+)\geq c_1$, we get the  contradiction $\AA^{\tau_+}_f(w_+)\leq \AA^{\tau_-}_f(w_-)<b$. If $b\leq \AA^{\tau_+}_f(w_+)<c_1$, then our choice $\delta<\frac{\epsilon}{2c_2(c_1-b)}$ leads to the contradiction $\AA^{\tau_+}_f(w_+) \leq   \AA^{\tau_-}_f(w_-) +\epsilon<b$ by \eqref{eq:spec_epsilon} again. This proves $\AA^{\tau_+}_f(w_+)<b$ in both cases. 
\end{proof}

In the following corollary and proposition, as usual, we assume (A1), (A2), or (A3) depending on which class of $J$ we are working with. We also assume $\tau_-,\tau_+>0$ to be small when working with $J\in\JJ_\mathrm{diag}$. 
\begin{cor}\label{cor:invariance_close}
Let $a,b\in \R\setminus\mathrm{spec}^{\w_0}\AA^{\tau_-}_f$. There exists $\delta>0$ such that if $|\tau_+-\tau_-|<\delta$, it holds that
	\[
	\RFH_*^{\w_0,(a,b)}(\AA^{\tau_-}_f) \cong \RFH_*^{\w_0,(a,b)}(\AA^{\tau_+}_f)\qquad\forall *\in\Z\,.
	\]
\end{cor}
\begin{proof}
Lemma \ref{lem:bound_eta}.(c) gives a uniform $L^\infty$-bound on $\eta$ for solutions of  \eqref{eq:s_dependent_gradient} for $\delta<1/c_6$. This guarantees necessary compactness results  as in Proposition \ref{prop:compactness1} or Proposition \ref{prop:compactness2} and allows us to construct a continuation homomorphism 
\[
\RFH_*^{\w_0,(a,b)}(\AA^{\tau_-}_f) \longrightarrow \RFH_*^{\w_0,(a',b')}(\AA^{\tau_+}_f)
\]
given by the count of solutions of \eqref{eq:s_dependent_gradient} not intersecting the zero section $\OO_E$. That this count defines a chain homomorphism is again due to Proposition \ref{prop:positivity_of_intersection} which continues to hold for solutions of \eqref{eq:s_dependent_gradient}.

In addition if we choose $\delta>0$ so small that Lemma \ref{lem:continuation_action} holds for both $a$ and $b$, we may take $a'=a$ and $b'=b$. Switching the roles of $\tau_-$ and $\tau_+$, we also obtain a continuation homomorphism going in the opposite direction. A standard homotopy of homotopies argument shows that these two continuation homomorphisms  at the homology level are inverse to each other. 
\end{proof}

\begin{prop}\label{prop:inv_radius}
	For given $a_-,b_-\in{\R}\setminus\mathrm{spec}^{\w_0}\AA^{\tau_-}_f$, we set  
	\[
	a_+:=\frac{(1+\tau_+)a_-}{1+\tau_-}\,,\qquad b_+:=\frac{(1+\tau_+)b_-}{1+\tau_-}\,.
	\]
	Then there is an isomorphism
	\[
	\RFH^{\w_0,(a_-,b_-)}_*(\AA^{\tau_-}_f) \cong \RFH^{\w_0,(a_+,b_+)}_*(\AA^{\tau_+}_f)\qquad\forall *\in\Z\,.
	\]
	In particular, $\RFH^{\w_0}_*(E,\Sigma_{\tau_-})\cong \RFH^{\w_0}_*(E,\Sigma_{\tau_+})$ for all $*\in\Z$. 
	\end{prop}
	
\begin{proof}
We first observe that if $|\tau_+-\tau_-|$ is sufficiently small, we have 
\[
\RFH^{\w_0,(a_-,b_-)}_*(\AA^{\tau_-}_f) \cong \RFH^{\w_0,(a_-,b_-)}_*(\AA^{\tau_+}_f)\cong \RFH^{\w_0,(a_+,b_+)}_*(\AA^{\tau_+}_f)\,.
\]
The first isomorphism is due to Corollary \ref{cor:invariance_close}, and the second one follows from the fact that the intervals $(a_-,a_+)$ and $(b_-,b_+)$ (or $(a_+,a_-)$ and $(b_+,b_-)$) do not intersect $\mathrm{spec}^{\w_0}\AA^{\tau_+}_f$ since by Lemma \ref{lem:one-to-one_index} 
\[
\mathrm{spec}^{\w_0}\AA^{\tau_+}_f=\big\{\tfrac{(1+\tau_+)c}{1+\tau_-}\mid c\in \mathrm{spec}^{\w_0}\AA^{\tau_-}_f\big\}\,.
\] 
For general $\tau_-$ and $\tau_+$, we may consider $\tau_i:=\tau_-+\frac{i}{\nu}(\tau_+-\tau_-)$ for some $\nu\in\N$ and $i\in\{0,\dots,\nu\}$. If $\nu$ is sufficiently large, the above argument constructs a chain of isomorphisms 
\[
\RFH^{\w_0,(a_i,b_i)}_*(\AA_f^{\tau_i}) \cong \RFH^{\w_0,(a_{i+1},b_{i+1})}_*(\AA_f^{\tau_{i+1}})\,,
\] 
with $a_i:=\frac{(1+\tau_i)a_-}{1+\tau_-}$ and $b_i:=\frac{(1+\tau_i)b_-}{1+\tau_-}$. We obtain the claimed isomorphism by composing these isomorphisms. The last assertion follows from Remark \ref{rem:finitely_many}.(i).
\end{proof}

\section{Floer Gysin sequence revisited}\label{sec:gysin_revisit}
\subsection{Gysin sequence in Rabinowitz Floer homology}
In this section, let $J\in\JJ_\mathrm{reg}^\BB\cup\JJ_\mathrm{diag}$, and we accordingly assume (A1) or (A2) from the introduction. 
As before, $\tau>0$ can be arbitrary if $J \in\JJ_\mathrm{reg}^\BB$ while we take small $\tau>0$ when working with $J\in\JJ_\mathrm{diag}$. We abbreviate
\[
\mathfrak C^{\w_0}(\AA^\tau_f):=\big\{w \mid w \in\Crit h\subset \Crit\AA^\tau_f\,,\; \w(w)=0\,\big\}\,.
\]
We recall that $h$ is a perfect Morse function which simplifies the construction in this section significantly. Let $\widehat{\mathfrak{C}}^{\w_0}(\AA^\tau_f)$ and $\widecheck{\mathfrak{C}}^{\w_0}(\AA^\tau_f)$ be the subset of $\mathfrak C^{\w_0}(\AA^\tau_f)$ consisting of maximum and minimum points of $h$ respectively. Then,
\[
\mathfrak C^{\w_0}(\AA^\tau_f)=\widehat{\mathfrak{C}}^{\w_0}(\AA^\tau_f) \sqcup \widecheck{\mathfrak{C}}^{\w_0}(\AA^\tau_f)\,.
\]
Let $-\infty\leq a<b\leq +\infty$. The module $\RFC_*^{\w_0,(a,b)}(\AA^\tau_f)$ defined in \eqref{eq:RFC} splits into
\[
\mathrm{RFC}^{\w_0,(a,b)}_*(\AA^\tau_f)=\widehat{\mathrm{RFC}}{}^{\w_0,(a,b)}_*(\AA^\tau_f)\oplus \widecheck{\mathrm{RFC}}^{\w_0,(a,b)}_*(\AA^\tau_f)
\]
where $\widehat{\mathrm{RFC}}{}^{\w_0,(a,b)}_*(\AA^\tau_f)$ and $\widecheck{\mathrm{RFC}}^{\w_0,(a,b)}_*(\AA^\tau_f)$ are the submodules generated by elements in $\widehat{\mathfrak{C}}^{\w_0}(\AA^\tau_f)$ and $\widecheck{\mathfrak{C}}^{\w_0}(\AA^\tau_f)$ respectively. In order to make the notation less cumbersome, we sometimes omit $\AA^\tau_f$ from the notation if the choices of $f$ and $\tau$ are irrelevant to the context.

\begin{lem}
 Let $w_\pm\in\Crit h$ with $\w(w_-)=\w(w_+)$ and $\mu_\RFH^h(w_-)-\mu_\RFH^h(w_+)=1$. Suppose that $\MM(w_-,w_+,\AA^\tau_f,J)$ is not empty. Then one of the following holds.
	\begin{enumerate}[(a)]
		\item $\mu_\FH(\Pi(w_-))-\mu_\FH(\Pi(w_+))=1$, and both $w_-$ and $w_+$ are either maximum points or minimum points of $h$.
		\item $\mu_\FH(\Pi(w_-))-\mu_\FH(\Pi(w_+))=2$, and $w_-$ is a minimum point and $w_+$ is a maximum point of $h$.
	\end{enumerate}
\end{lem}
\begin{proof}
	 Lemma \ref{lem:one-to-one_index} implies that $\mu_\FH(\Pi(w_-))-\mu_\FH(\Pi(w_+))\in\{0,1,2\}$. Suppose that this difference is $0$. Then $\Pi(w_-)=\Pi(w_+)$ and thus  $S_{w_-}=S_{w_+}$ where $w_-$ is the maximum point and $w_+$ is the minimum point of $h$ in $S_{w_-}$. In this case, $\MM(w_-,w_+,\AA^\tau_f,J)$ is empty, see Remark \ref{rem:zero_cascade}. The rest follows immediately from \eqref{eq:index_h} and Lemma \ref{lem:one-to-one_index}.
\end{proof}

According to the preceding  lemma, we can decompose the boundary operator $\p^{\w_0}_J$ into
\begin{equation}\label{eq:p^w_0}
\p^{\w_0}_J=\hat\p^{\w_0}_J + \check\p^{\w_0}_J + \p^{\w_0,c_1^E}_J	
\end{equation}
where $\hat\p^{\w_0}_J$ resp.~$\check\p^{\w_0}_J$ counts gradient flow lines from $\widehat{\mathfrak{C}}^{\w_0}$ resp.~$\widecheck{\mathfrak{C}}^{\w_0}$ to itself  while $\p^{\w_0,c_1^E}_J$ counts gradient flow lines from $\widecheck{\mathfrak{C}}^{\w_0}$ to $\widehat{\mathfrak{C}}^{\w_0}$, i.e.
\[
\begin{split}
\hat\p^{\w_0}_J&:\widehat{\mathrm{RFC}}{}^{\w_0,(a,b)}_*(\AA^\tau_f)\longrightarrow \widehat{\mathrm{RFC}}{}^{\w_0,(a,b)}_{*-1}(\AA^\tau_f)\,, \\[0.5ex]
\check\p^{\w_0}_J&:\widecheck{\mathrm{RFC}}^{\w_0,(a,b)}_*(\AA^\tau_f)\longrightarrow \widecheck{\mathrm{RFC}}^{\w_0,(a,b)}_{*-1}(\AA^\tau_f)\,, \\[0.5ex]
\p^{\w_0,c_1^E}_J&:\widecheck{\mathrm{RFC}}^{\w_0,(a,b)}_*(\AA^\tau_f)\longrightarrow \widehat{\mathrm{RFC}}{}^{\w_0,(a,b)}_{*-1}(\AA^\tau_f)\,.
\end{split}
\] 
As usual $-\infty\leq a<b\leq +\infty$. 
The equality $\p^{\w_0}_J\circ\p^{\w_0}_J=0$ is equivalent to  
\[
\hat\p^{\w_0}_J\circ\hat\p^{\w_0}_J=0\,,\quad \check\p^{\w_0}_J\circ \check\p^{\w_0}_J=0\,,\quad\text{and}\quad \p^{\w_0,c_1^E}_J\circ\check\p^{\w_0}_J+\hat\p^{\w_0}_J\circ \p^{\w_0,c_1^E}_J=0\,.
\]
Therefore 
$(\widehat{\mathrm{RFC}}{}^{\w_0,(a,b)}(\AA^\tau_f),\hat\p^{\w_0}_J)$ 
is a subcomplex of $(\mathrm{RFC}^{\w_0,(a,b)}(\AA^\tau_f),\p^{\w_0}_J)$ and the induced quotient complex is isomorphic to $(\widecheck{\mathrm{RFC}}^{\w_0,(a,b)}(\AA^\tau_f), \check\p^{\w_0}_J)$. We denote their homologies by
\[
\begin{split}
\widehat{\mathrm{RFH}}{}^{\w_0,(a,b)}_*(\AA^\tau_f,J)&:=\mathrm{H}_*\big(\widehat{\RFC}{}^{\w_0,(a,b)}(\AA^\tau_f), \hat\p^{\w_0}_J\big)\,,\\[1ex]
\widecheck{\mathrm{RFH}}^{\w_0,(a,b)}_*(\AA^\tau_f,J)&:=\mathrm{H}_*\big(\widecheck{\RFC}^{\w_0,(a,b)}(\AA^\tau_f), \check\p^{\w_0}_J\big)\,.
\end{split}
\]
These are invariant under the change of $J$ by Lemma \ref{lem:inv_radius_compatibility} or Proposition \ref{prop:one-to-one} below, so we often omit $J$ from the notation. We write 
\[
\begin{split} 
\widehat\RFH{}_*^{\w_0}(E,\Sigma)=\widehat{\mathrm{RFH}}{}^{\w_0}_*(\AA^\tau_f)=\widehat{\mathrm{RFH}}{}^{\w_0,(-\infty,+\infty)}_*(\AA^\tau_f) \,, \\[1ex]
\widecheck\RFH_*^{\w_0}(E,\Sigma)=\widecheck{\mathrm{RFH}}^{\w_0}_*(\AA^\tau_f)=\widecheck{\mathrm{RFH}}^{\w_0,(-\infty,+\infty)}_*(\AA^\tau_f) \,.
\end{split}
\]
As the notation indicates the above two homologies are invariant under the change of $\tau$, see again Lemma \ref{lem:inv_radius_compatibility} or Proposition \ref{prop:one-to-one}.

\begin{prop}\label{prop:les}
There is a long exact sequence 	
\[
\cdots\to \RFH_*^{\w_0,(a,b)}(\AA^\tau_f) \to\widecheck{\RFH}_*^{\w_0,(a,b)}(\AA^\tau_f) \stackrel{\delta}{\to}\widehat{\RFH}{}_{*-1}^{\w_0,(a,b)}(\AA^\tau_f) \to \RFH_{*-1}^{\w_0,(a,b)}(\AA^\tau_f) \to\cdots
\]
where $\delta$ is induced by $\p^{\w_0,c_1^E}$. For $(a,b)=(-\infty,+\infty)$, the above long exact sequence reads
\[
\cdots\to \RFH_*^{\w_0}(E,\Sigma) \to\widecheck{\RFH}_*^{\w_0}(E,\Sigma) \stackrel{\delta}{\to}\widehat{\RFH}{}_{*-1}^{\w_0}(E,\Sigma) \to \RFH_{*-1}^{\w_0}(E,\Sigma) \to\cdots\,.
\]
\end{prop}
\begin{proof}
The assertion follows since $({\mathrm{RFC}}{}^{\w_0,(a,b)}(\AA^\tau_f),\p^{\w_0})$ is isomorphic to the cone complex of $\partial^{\w_0,c_1^E}$. Alternatively,   
the short exact sequence of chain complexes
\[
	0\to \big(\widehat{\mathrm{RFC}}{}^{\w_0,(a,b)}(\AA^\tau_f),\hat\p^{\w_0}\big) \to \big(\mathrm{RFC}^{\w_0,(a,b)}_*(\AA^\tau_f),\p^{\w_0}\big)\to (\widecheck{\mathrm{RFC}}^{\w_0,(a,b)}(\AA^\tau_f), \check\p^{\w_0}\big)\to0
\]
induced by natural inclusion and projection also leads to the desired exact sequence.  
\end{proof}

\begin{lem}\label{lem:inv_radius_compatibility}
The isomorphism in Proposition \ref{prop:inv_radius} is compatible with the exact sequence in Proposition \ref{prop:les}, i.e.~the following diagram commutes
		\[
	\begin{tikzcd}[row sep=1.5em,column sep=1.3em]
\cdots \arrow{r} & \widehat{\mathrm{RFH}}{}^{\w_0,(a_-,b_-)}({\AA}^{\tau_-}_f) \arrow{r} \arrow{d}{\cong}  & {\mathrm{RFH}}^{\w_0,(a_-,b_-)}({\AA}^{\tau_-}_f)  \arrow{r} \arrow{d}{\cong} & \widecheck{\mathrm{RFH}}^{\w_0,(a_-,b_-)}_*({\AA}^{\tau_-}_f) \arrow{r}{\delta} \arrow{d}{\cong} & \cdots\\
\cdots \arrow{r} & \widehat{\mathrm{RFH}}{}^{\w_0,(a_+,b_+)}({\AA}^{\tau_+}_{f}) \arrow{r} & {\mathrm{RFH}}^{\w_0,(a_+,b_+)}({\AA}^{\tau_+}_{f})   \arrow{r} &  \widecheck{\mathrm{RFH}}^{\w_0,(a_+,b_+)}_*({\AA}^{\tau_+}_{f}) \arrow{r}{\delta} & \cdots
\end{tikzcd}
\]
where all vertical maps are isomorphisms, and $a_\pm,b_\pm$ are given as in Proposition \ref{prop:inv_radius}. 

Corresponding statements for the isomorphisms in \eqref{eq:continuation_J} and \eqref{eq:isom_f} for the change of $J$, $f$, or $h$ are also true.
\end{lem}
\begin{proof}
	We first recall that the isomorphism $\Psi$ in Proposition \ref{prop:inv_radius} is given by the composition of finitely many isomorphisms, each of which is induced by a chain level homomorphism given by the count of solutions $w$ of \eqref{eq:s_dependent_gradient} with asymptotic limits $w_\pm$ having the same $\mu^h_\RFH$-index. In particular, $\Psi$ comes from a chain map $\psi$. We claim that $\psi$ maps $\widehat{\mathfrak{C}}^{\w_0}$ to itself. Suppose that $w_-\in\widehat{\mathfrak{C}}^{\w_0}$ and $w_+\in\widecheck{\mathfrak{C}}^{\w_0}$. Then Lemma \ref{lem:one-to-one_index} yields $\mu_\FH(\Pi(w_-))+1=\mu_\FH(\Pi(w_+))$. This contradicts the existence of a solution $\Pi(w)$ of \eqref{eq:Floer_eq_M} from $\Pi(w_-)$ to $\Pi(w_+)$. This shows the claim. Hence, $\psi$ induces a chain homomorphism between two short exact sequences  
	\[
	\begin{tikzcd}[row sep=1.5em,column sep=1.3em]
0\arrow{r} & \widehat{\mathrm{RFC}}{}^{\w_0,(a_-,b_-)}({\AA}^{\tau_-}_f)  \arrow{r} \arrow{d} & \mathrm{RFC}^{\w_0,(a_-,b_-)}_*({\AA}^{\tau_-}_f) \arrow{r} \arrow{d}{\psi} & \widecheck{\mathrm{RFC}}^{\w_0,(a_-,b_-)}({\AA}^{\tau_-}_f) \arrow{r} \arrow{d} & 0 \\
0\arrow{r} & \widehat{\mathrm{RFC}}{}^{\w_0,(a_+,b_+)}({\AA}^{\tau_+}_f)  \arrow{r} & \mathrm{RFC}^{\w_0,(a_+,b_+)}_*({\AA}^{\tau_+}_f) \arrow{r} & \widecheck{\mathrm{RFC}}^{\w_0,(a_+,b_+)}({\AA}^{\tau_+}_f) \arrow{r} & 0
\end{tikzcd}
\]
where the last vertical map is $\psi$ followed by the projection to $\widecheck{\mathrm{RFC}}^{\w_0}$. Alternatively, one may think of the last vertical arrow as the quotient map induced by $\psi$. 
The map $\psi$ is a quasi-isomorphism, in fact there is a map $\varphi$ in the opposite direction so that the composition $\varphi\circ\psi$ is chain homotopic to the identity, see Proposition \ref{prop:inv_radius}, and similarly for $\psi\circ\varphi$. Since $\p|_{\widehat{\mathrm{RFC}}{}^{\w_0}}$ and $\varphi\circ\psi|_{\widehat{\mathrm{RFC}}{}^{\w_0}}$ automatically map into ${\widehat{\mathrm{RFC}}{}^{\w_0}}$, it follows that $\varphi\circ\psi|_{\widehat{\mathrm{RFC}}{}^{\w_0}}$ is chain homotopic to the identity also in $({\widehat{\mathrm{RFC}}{}^{\w_0}},\hat\p)$.
Therefore also the first vertical map is a quasi-isomorphism, and thus so is the last vertical map. This induces the commutative diagram in the statement. 

The claim for the isomorphisms in \eqref{eq:continuation_J} and \eqref{eq:isom_f} can be shown analogously.
\end{proof}

The above results are parallel to those in Section \ref{sec:quantum_gysin}. Indeed, the following propositions show that this is consistent.

\begin{prop}\label{prop:one-to-one}
There exist isomorphisms
\[
\widehat{\RFH}{}_*^{\w_0,(a,b)}(\AA^\tau_f,J)\cong \FH^{(\frac{a}{1+\tau},\frac{b}{1+\tau})}_{*-1}(f,j)\,,\quad \widecheck{\RFH}_*^{\w_0,(a,b)}(\AA^\tau_f,J)\cong \FH^{(\frac{a}{1+\tau},\frac{b}{1+\tau})}_{*}(f,j)\,.
\]
In particular, $\widehat{\RFH}{}_*^{\w_0}(E,\Sigma)\cong \FH_{*-1}(M)\cong \widecheck{\RFH}_{*-1}^{\w_0}(E,\Sigma)$.
\end{prop}
\begin{proof}
	We only show the isomorphism for $\widehat{\RFH}{}_*^{\w_0}$. The case for $\widecheck{\RFH}_*^{\w_0}$ can be shown analogously. Due to the invariance properties, \eqref{eq:isom} and Lemma \ref{lem:inv_radius_compatibility}, it suffices to obtain the isomorphism for small $\tau>0$ and $J=\begin{pmatrix}
 i & 0\\
 0 & j
\end{pmatrix}$ with $j\in\mathfrak{j}_\mathrm{HS}(f)$. 
By Lemma \ref{lem:one-to-one_index}, we have a bijection between $\widehat{\mathfrak C}^{\w_0}(\AA^\tau_{f})$ and $\Crit\a_{f}$ that relates actions and indices as in the statement of the proposition. Moreover, Proposition \ref{prop:j_HS}.(d)  and Lemma \ref{lem:equal_count}.(a) yield 
\[
\#\widehat\NN(\Pi(w_-),\Pi(w_+),\mathfrak{a}_{f},j)=-\#\widehat\MM(w_-,w_+,\AA^\tau_{f},J)
\] 
for all $w_\pm\in\widehat{\mathfrak C}^{\w_0}(\AA^\tau_f)$ with $\mu_\RFH(w_-)-\mu_\RFH(w_+)=1$. This proves that the following two chain complexes are isomorphic:
\begin{equation}\label{eq:Pi_isom}
\big(\widehat{\RFC}{}_*^{\w_0,(a,b)}(\AA^\tau_{f}),\hat\p^{\w_0}_J\big) \cong \big(\FC^{(\frac{a}{1+\tau},\frac{b}{1+\tau})}_{*-1}(f),-\p_j\big)\,.	
\end{equation}
This induced the desired isomorphism.
\end{proof}

The following proposition together with Proposition \ref{prop:quantum_gysin_simple}.(c) proves statements (b) and (c) in Theorem \ref{thm:main}, see also Remark \ref{rem:RFH_w_0=H_Sigma} below.

\begin{prop}\label{prop:cap_product}
$ $
\begin{enumerate}[(a)]
	\item There exists an isomorphism 
\[
\RFH_*^{\w_0,(a,b)}(\AA^\tau_{f},J)\cong \FH_{*}^{(\frac{a}{1+\tau},\frac{b}{1+\tau})}(\tilde f,j)\,.
\]
In particular, $\RFH_*^{\w_0}(E,\Sigma)\cong \FH_{*}(\Sigma)$.
	\item The exact sequence in Proposition \ref{prop:les} is isomorphic to the one in Proposition \ref{prop:quantum_gysin}, that is, we have the commutative diagram
	\[
	\begin{tikzcd}[row sep=1.5em,column sep=.5em]
\cdots \arrow{r} &  \RFH_*^{\w_0,(a,b)}(\AA^\tau_f) \arrow{r} \arrow{d}{\cong} & \widecheck{\RFH}_*^{\w_0,(a,b)}(\AA^\tau_f) \arrow{r} \arrow{d}{\cong} & \widehat{\RFH}{}_{*-1}^{\w_0,(a,b)}(\AA^\tau_f) \arrow{r} \arrow{d}{\cong} & \RFH_{*-1}^{\w_0,(a,b)}(\AA^\tau_f) \arrow{r}\arrow{d}{\cong} & \cdots \\
\cdots \arrow{r} & \FH_{*}^{(\frac{a}{1+\tau},\frac{b}{1+\tau})}(\tilde f) \arrow{r} & \FH_{*}^{(\frac{a}{1+\tau},\frac{b}{1+\tau})}(f) \arrow{r} & \FH_{*-2}^{(\frac{a}{1+\tau},\frac{b}{1+\tau})}(f)  \arrow{r} & \FH_{*-1}^{(\frac{a}{1+\tau},\frac{b}{1+\tau})}(\tilde f) \arrow{r} & \cdots
\end{tikzcd}
\]
 where all vertical maps are isomorphisms.
\end{enumerate}
\end{prop}

\begin{proof}
Due to Proposition \ref{prop:quantum_gysin_simple}.(a) and Lemma \ref{lem:inv_radius_compatibility}, it again suffices to show both (a) and (b) for small $\tau>0$ and $J=\begin{pmatrix}
 i & 0\\
 0 & j
\end{pmatrix}$ with $j\in\mathfrak{j}_\mathrm{HS}(f)$. We show that the following two chain complexes are isomorphic through the map $\Pi$:
\begin{equation}\label{eq:claim}
	\Pi_*:\big({\RFC}_*^{\w_0,(a,b)}(\AA^{\tau}_{f}),\p^{\w_0}_J\big) \stackrel{\cong}{\longrightarrow} \big( \FC^{(\frac{a}{1+\tau},\frac{b}{1+\tau})}_{*}(\tilde f), \p^\wp_j\big)\,.
\end{equation}
Lemma \ref{lem:one-to-one_index} shows that $\Pi_*$ induces an isomorphism of chain modules. Next we compare the boundary operator $\p^\wp=\hat\p+\check\p+\p^{c_1^E}$ for $\FC$ and the boundary operator $\p^{\w_0}=\hat\p^{\w_0}+\check{\p}^{\w_0}+\p^{\w_0,c_1^E}$ for $\RFC^{\w_0}$. We claim that up to conjugation by $\Pi_*$ it holds that
\begin{equation}\label{eq:differentials_same}
\hat\p = \hat\p^{\w_0}\,,\qquad \check\p = \check\p^{\w_0}\,,\qquad \p^{c_1^E} = \p^{\w_0,c_1^E}\,.		
\end{equation}
For the first two identities, see the proof of Proposition \ref{prop:one-to-one}. For the last one, we show that there is a sign-preserving bijection 
\[
\MM(\check w_-,\hat w_+,\AA^\tau_{f},J) \longrightarrow \NN(\q_-,\q_+,\mathfrak{a}_{f},j,N)
\]
for $\check w_-\in\widecheck{\mathfrak{C}}^{\w_0}(\AA^\tau_f)$ and $\hat w_+\in\widehat{\mathfrak{C}}^{\w_0}(\AA^\tau_f)$ with $\mu^h_\RFH(\check w_-)-\mu^h_\RFH(\hat w_+)=1$. Here $\q_\pm=[q_\pm,\bar q_\pm]=\Pi(\check w_\pm)$. In view of \eqref{eq:no_3} we have
\[
\begin{split}
\MM(\check w_-,\hat w_+,\AA^\tau_{f},J)&=\MM^1_{\om\neq 0}(\check w_-,\hat w_+,\AA^\tau_{f},J)\sqcup \MM^1_{\om= 0}(\check w_-,\hat w_+,\AA^\tau_{f},J)\\[.5ex]
&\quad \sqcup  \MM^2(\check w_-,\hat w_+,\AA^\tau_{f},J)\,,
\end{split}
\]
where the subscript $\om\neq 0$ indicates that the space consists of $[w]$ with $\om([\Pi(w)])\neq 0$, i.e.~$\Pi(w)$ is $t$-dependent by \eqref{eq:om-energy}, while the subscript $\om=0$ means the opposite as in \eqref{eq:om=0_moduli}. Using the same notation, we write
\[
\NN(\q_-,\q_+,\a_{f},j,N)= \NN_{\om\neq0}(\q_-,\q_+,\a_{f},j,N)\sqcup \NN_{\om=0}(\q_-,\q_+,\a_{f},j,N)\,.
\]
It follows from Proposition \ref{prop:quantum_gysin_simple}.(b) and Lemma \ref{lem:chern_number=min_to_max} that 
\[
\#\MM^1_{\om\neq 0}(\check w_-,\hat w_+,\AA^\tau_f,J)=\#\NN_{\om\neq0}(\q_-,\q_+,\a_f,j,N)\,.
\]
On the other hand if $[\bar q_-]=[\bar q_+]$, then 
\[
\begin{split}
\#\NN_{\om=0}(\q_-,\q_+,\a_f,j,N)&= \#\NN(q_-,q_+,f,g,N)\\
&=\#\NN^1(\check q_-,\hat q_+,\tilde f,\tilde g)+\#\NN^2(\check q_-,\hat q_+,\tilde f,\tilde g)\\
&=\#\MM^1_{\om= 0}(\check w_-,\hat w_+,\AA^\tau_f,J) +\#\MM^2(\check w_-,\hat w_+,\AA^\tau_f,J)
\end{split}
\]
where the first equality holds by \eqref{eq:om-energy} and the last line is due to Lemma \ref{lem:equal_count}.(b). The second equality follows from the proof of Proposition \ref{prop:quantum_gysin_simple}.(c). In the proof, we lift $g$ to $\tilde g'$ on $\Sigma$ using a special connection 1-form and have the equality for $\tilde g'$ with $\NN^2(\check q_-,\hat q_+,\tilde f,\tilde g')=\emptyset$. For $\tilde g$ used here, however, the space $\NN^2(\check q_-,\hat q_+,\tilde f,\tilde g)$ is not necessarily empty, but the second equality can be seen using a homotopy between $\tilde g$ and $\tilde g'$.
Moreover $\NN_{\om=0}(\q_-,\q_+,\a_f,j,N)=\emptyset$ if $[\bar q_-]\neq[\bar q_+]$ again by \eqref{eq:om-energy}.
This establishes \eqref{eq:differentials_same} and thus \eqref{eq:claim}. Hence (a) is proved. 

Statement (b) follows immediately from \eqref{eq:differentials_same} and the construction of the respective exact sequences. 
\end{proof}

\begin{rem}\label{rem:Lambda_module_str}
	The proof of Proposition \ref{prop:cap_product}.(a) implies that the chain complex $\RFC^{\w_0}(\AA^\tau_f)$ is isomorphic to the mapping cone of $\psi^{c_1^E}$, see \eqref{eq:cone} and \eqref{eq:mapping_cone}:
\[
\big(\RFC^{\w_0}(\AA^\tau_f),\p^{\w_0}\big)\cong \mathrm{Cone\,}(\psi^{c_1^E})\,.
\] 
As mentioned after \eqref{eq:floer_homology_h}, $\mathrm{Cone\,}(\psi^{c_1^E})$ is a chain complex of $\Lambda$-modules, where the Novikov ring $\Lambda$ is defined in \eqref{eq:nov}, see also Remark \ref{rem:H}.(ii). Thus, through the above isomorphism, we also have a $\Lambda$-module structure on $(\RFC^{\w_0}(\AA^\tau_f),\p^{\w_0})$. In view of Lemma \ref{lem:one-to-one}, this module structure can be interpreted geometrically as follows. For $w=([u,\bar u],\eta)\in\Crit\AA^\tau_f$, we write $u=u_q^\ell$, meaning that it is the $\ell$-fold cover of a simple Reeb orbit $u_q$ on $\Sigma_\tau$ over $q\in\Crit f$. Then an element $A\in\Gamma_M$ acts on $\Crit\AA^\tau_f$ by
\[
A\cdot w=A\cdot\big([u,\bar u],\eta\big):=\big([u_q^{\ell-c_1^E(A)},\bar u\#s],\eta'\big)\,,
\]
where $s:S^2\to E$ with $[\wp\circ s]=A$ and $\eta'\in\R$ is determined by \eqref{eq:eta=cov}. This action  preserves the winding number $\w$ and shifts the $\mu_\RFH$-index by $-2c_1^{TM}(A)$. Moreover the proof of Proposition \ref{prop:cap_product}.(a) yields that this $\Gamma_M$-action commutes with the boundary operator $\p^{\w_0}$, and therefore this naturally extends to a $\Lambda$-module action on $(\RFC^{\w_0}(\AA^\tau_f),\p^{\w_0})$. We refer to \cite{Ueb19} for related results.
\end{rem}

\begin{rem}\label{rem:RFH_w_0=H_Sigma}
In the case that $\om$ vanishes on $\pi_2(M)$, Proposition \ref{prop:quantum_gysin_simple}.(c) and Proposition \ref{prop:cap_product}.(a) imply $\RFH^{\w_0}(E,\Sigma) \cong \H(\Sigma;\Lambda)$. Let us explain this isomorphism directly. 
 To simplify the exposition we use $\AA^\tau_{f=0}$ where we take $f=0$ in the definition of the action functional. Then every connected component of $\Crit\AA^\tau_{f=0}$ is diffeomorphic to $\Sigma$, and we choose a Morse function $h$ on $\Crit\AA^\tau_{f=0}$. We note that before the role of $f$ was  solely to reduce the Morse-Bott setting to $S^1$, as opposed to $\Sigma$. 
The assumption $\om|_{\pi_2(M)}=0$ implies that the map $\pi_1(S^1)\to \pi_1(\Sigma)$ in \eqref{eq:homotopy_les} is injective. Therefore the loop component $u$ of $w=([u,\bar u],\eta)\in\Crit\AA^\tau_{f=0}$ with $\w(w)=0$ is  constant, mapping to a point in $\Sigma_\tau$ by Proposition \ref{prop:winding}.(f), and the boundary operator counts only gradient flow lines of $h$ for energy reasons. Therefore $(\RFC^{\w_0}(\AA^\tau_{f=0}),\p^{\w_0})$ is simply the Morse complex of $h$, and the claimed isomorphism follows. 
\end{rem}

\begin{rem}\label{rem:w_k_RFH}
It is possible to define the Rabinowitz Floer chain complex with winding number $k\in\Z$ and its homology, denoted by $\RFC^{\w_k}$ and $\RFH^{\w_k}$ respectively, exactly in the same way to construct $\RFH^{\w_0}$. Then arguments used in Proposition \ref{prop:cap_product} together with Lemma \ref{lem:one-to-one}, \eqref{eq:action_comparision}, and \eqref{eq:indices_and_winding} prove
\[
\RFH_{*-2k}^{\w_k,((1+\tau)a-\frac{\tau k}{m},(1+\tau)b-\frac{\tau k}{m})}(\AA^\tau_f)\cong\FH_{*}^{(a,b)}(\tilde f)  \,.
\]
Similarly $\widehat{\RFC}{}^{\w_k}$, $\widecheck{\RFC}{}^{\w_k}$, $\widehat{\RFH}{}^{\w_k}$, and $\widecheck{\RFH}^{\w_k}$ are defined as well and 
\[
	\widehat{\RFH}{}^{\w_k,((1+\tau)a-\frac{\tau k}{m},(1+\tau)b-\frac{\tau k}{m})}_{*-2k+1}(\AA^\tau_f)\cong\FH_{*}^{(a,b)}(f) \cong \widecheck{\RFH}^{\w_k,((1+\tau)a-\frac{\tau k}{m},(1+\tau)b-\frac{\tau k}{m})}_{*-2k}(\AA^\tau_f)
\]
as in Proposition \ref{prop:one-to-one}. 
Furthermore the exact sequence in Proposition \ref{prop:cap_product} should hold for these homologies with winding number $k$.
We also note that the composition of the isomorphisms 
\[
\RFC_{*}^{\w_0,(a,b)}(\AA^\tau_f)\to \FC_{*}^{(\frac{a}{1+\tau},\frac{b}{1+\tau})}(\tilde f)\to \RFC_{*-2k}^{\w_k,(a-\frac{\tau k}{m},b-\frac{\tau k}{m})}(\AA^\tau_f)
\] 
coincides with the map simply adding $k$ iterations to generators. 
\end{rem}

Next we prove Corollary \ref{eq:RFH_w_0_O}.
Let $\om_\mathrm{FS}$ be the Fubini-Study form on $\CP^n$. 
Then $c_1^{T\CP^n}=(n+1)[\om_\mathrm{FS}]$ and thus condition (A3) is met for all $n\in\N$. 

\begin{cor}\label{cor:rhf_CPn}
Let $E=\OO_{\CP^n}(-m)$ be the complex line bundle over $M=\CP^n$ with $c_1^{\OO_{\CP^n}(-m)}=-m[\om_\mathrm{FS}]$. Then $\Sigma$ is diffeomorphic to the lens space $L(m,1)=S^{2n+1}/\Z_m$ and 
\[
\RFH^{\w_0}_*(\OO_{\CP^n}(-m),L(m,1)) \cong \left\{
\begin{aligned} 
\Z_m  & \quad *\in2\Z+1\,, \\[1ex]
0\quad & \quad *\in 2\Z\,.
\end{aligned}
\right.
\]
\end{cor}
\begin{proof}
	This is a direct consequence of Proposition \ref{prop:CPn} and Proposition \ref{prop:cap_product}.
\end{proof}

\subsection{Transfer homomorphism revisited}\label{sec:transfer_revisited}
Let $E^m$ and $\Sigma^m$ denote the degree $m$ bundles as in Section \ref{sec:cyclic}. 
\begin{cor}\label{cor:transfer2}
There exist transfer and projection homomorphisms 
\[
T:\RFH^{\w_0}_*(E^m,\Sigma^m)\to \RFH^{\w_0}_*(E^1,\Sigma^1)\,,\quad P: \RFH^{\w_0}_*(E^1,\Sigma^1)\to  \RFH^{\w_0}_*(E^m,\Sigma^m)
\] 
such that both compositions $P\circ T$ and $T\circ P$ agree with the scalar multiplication by $m$. 

	This leads to the following consequences. We fix $\kappa\in\Z$ and assume  $\RFH_{\kappa}^{\w_0}(E^1,\Sigma^1)=0$. Then, for any $m\in\N$, $\RFH_{\kappa}^{\w_0}(E^m,\Sigma^m)$ only contains torsion classes of order $m$. Moreover, $\RFH_{\kappa}^{\w_0}(E^m,\Sigma^m)$ is torsion for some $m\in\N$ if and only if it is torsion for every $m\in\N$. In particular, if we use coefficients in a field instead of integers $\Z$, then $\RFH_{\kappa}^{\w_0}(E^m,\Sigma^m)=0$ for some $m\in\N$ if and only if $\RFH_{\kappa}^{\w_0}(E^m,\Sigma^m)=0$ for all  $m\in\N$.
\end{cor}
\begin{proof}
	This is a direct consequence of \eqref{eq:TP=PT=m}, Proposition \ref{prop:transfer}, and Proposition \ref{prop:cap_product}.(a).
\end{proof}
This establishes Theorem \ref{thm:main}.(d) and Corollary \ref{cor:transfer}. In the rest of this section, we provide a geometric interpretation of Corollary \ref{cor:transfer2}.

To emphasize that we work on the bundles $E^m$ and $\Sigma^m$ we denote by
\[
\AA^{\tau,m}_f:\widetilde{\mathcal{L}}(E^m)\x\R\longrightarrow\R
\]
 the Rabinowitz action functional for $E^m$ given in \eqref{eq:Rabinowitz_functional}, i.e.~$\RFH(\AA^{\tau,m}_f)=\RFH(E^m,\Sigma^m)$. Moreover we set  $\tilde f^m:=f\circ\wp:\Sigma^m\to\R$, i.e.~the lift of $f:M\to\R$ by $\wp:\Sigma^m\to M$.   We also denote by $\tilde g^m$ the lift of a Riemannian metric $g$ on $M$ to $\Sigma^m$. As before, we write $h^m:\Crit \tilde f^m\to\R$ and $h^m:\Crit\AA^{\tau,m}_f\to\R$ for the perfect Morse functions given in \eqref{eq:ftn_h} and \eqref{eq:perfect}. We pull back both functions by the $m$-fold covering map $\wp^m:\Sigma^1\to\Sigma^m$ to obtain 
\[
\tilde h^m:=h^m\circ\wp^m:\Crit \tilde f^1\to\R\,,\qquad\tilde h^m:=h^m\circ\wp^m: \Crit{\AA}^{\tau,1}_f\to\R
\]
where, for the latter, we actually mean $h^m$ composed with the map $\Crit{\AA}^{\tau,1}_f\to \Crit{\AA}^{\tau,m}_f$ induced by $\wp^m:\Sigma^1\to\Sigma^m$, see Lemma \ref{lem:cyclic1} for similar situation.
 Both $\tilde h^m$ are Morse functions with exactly $m$ maximum points and $m$ minimum points on each connected component.
We denote by 
\[
\{\tilde{\bar q}^1,\dots,\tilde{\bar q}^m\}\subset \Crit \tilde h^m\subset \Crit \tilde f^1\,,\qquad \{\tilde w^1,\dots,\tilde w^m\}\subset\Crit\tilde h^m \subset \Crit{\AA}^{\tau,1}_f
\]
the preimages of $\bar q\in\Crit h^m\subset \Crit\tilde f^m$ and of   $w\in\Crit h^m\subset \Crit{\AA}^{\tau,m}_f$ respectively. 
 We also take $J\in\JJ_\mathrm{diag}$ associated with $j\in\j_\mathrm{HS}$ on $E^m$ and note that the pullback $\widetilde J=(\wp^m)^*J$ by $\wp^m:E^1\to E^m$ is also diagonal associated with $j$.
Although $\tilde h^m$ is not a perfect Morse function for $m\geq2$, we still can define the chain complex $(\RFC^{\w_0,(a,b)}(\AA^{\tau,1}_f,\tilde h^m),\p^{\w_0}_{\widetilde J})$ exactly in the same manner as in Section \ref{sec:RFH} except that the boundary operator $\p^{\w_0}_{\widetilde J}$ now also counts zero cascades, i.e.~negative gradient flow lines of $\tilde h^m$, as opposed to the above situation where we only considered perfect Morse functions. The resulting homology is isomorphic to the one defined with the perfect Morse function $h^1$ via a continuation homomorphism:
\[
\H_*\big(\RFC^{\w_0,(a,b)}(\AA^{\tau,1}_f,\tilde h^m),\p^{\w_0}_{\widetilde J}\big)\cong\RFH^{\w_0,(a,b)}_*(\AA^{\tau,1}_f,h^1)\,.
\]
Let $-\infty\leq a<b\leq+\infty$, and let $\tau>0$ be sufficiently small as required at the beginning of Section \ref{sec:RFH}. There is a free $\Z_m$-action on $\RFC_*^{\w_0,(a,b)}(\AA^{\tau,1}_f,\tilde h^m)$ induced by the free $\Z_m$-action on $\Crit\tilde h^m$. We denote the quotient module by
\[
\RFC_*^{\w_0,(a,b),\Z_m}({\AA}^{\tau,1}_f,\tilde h^m):={\RFC_*^{\w_0,(a,b)}(\AA^{\tau,1}_f,\tilde h^m)}/{\Z_m}\,,\qquad *\in\Z\,.
\]
There is a module isomorphism 
\[
\RFC_*^{\w_0,(a,b),\Z_m}({\AA}^{\tau,1}_f,\tilde h^m)\cong \RFC_*^{\w_0,(a,b)}(\AA^{\tau,m}_f, h^m)
\] 
given by the action-preserving map $[\tilde w^1]=\cdots=[\tilde w^m]\mapsto w$, see again Lemma \ref{lem:cyclic1} for similar situation. 
We then define the boundary operator
\[
\p^{\w_0,\Z_m}_{\widetilde J}:\RFC_*^{\w_0,(a,b),\Z_m}(\AA^{\tau,1}_f,\tilde h^m)\to \RFC_{*-1}^{\w_0,(a,b),\Z_m}(\AA^{\tau,1}_f,\tilde h^m)\,,\qquad 
\p^{\w_0,\Z_m}_{\widetilde J} [\tilde w_-]:= [\p^{\w_0}_{\widetilde J}\tilde w_-]\,,
\]
which is indeed well-defined since $\p^{\w_0}_{\widetilde J}$ is $\Z_m$-equivariant. We claim that the above module isomorphism is in fact an isomorphism of chain complexes:
\begin{equation}\label{eq:equiv_quot}
\big(\RFC_*^{\w_0,(a,b),\Z_m}(\AA^{\tau,1}_f,\tilde h^m),\p^{\w_0,\Z_m}_{\widetilde J}\big) \cong  \big(\RFC_*^{\w_0,(a,b)}(\AA^{\tau,m}_f, h^m),\p^{\w_0}_{J}\big)\,.
\end{equation}
To show the claim, we establish a bijection between the moduli spaces involved in the definition of $\p^{\w_0,\Z_m}_{\widetilde J}$ and $\p^{\w_0}_{J}$. One is tempted to argue by projecting and lifting Rabinowitz Floer cylinders directly via the covering map $E^1\setminus\OO_{E^1}\to E^m\setminus\OO_{E^m}$, but this does not work because $\AA^{\tau,1}_f$ is not exactly the pullback of $\AA^{\tau,m}_f$, see Remark \ref{eq:Floer_eqn_for_A} below. Instead we obtain the desired bijection by relating the moduli spaces for $\p^{\w_0,\Z_m}_{\widetilde J}$ and for $\p^{\w_0}_{J}$ with the moduli spaces of Floer cylinders in $M$ as in the proof of Proposition \ref{prop:cap_product}.(a). 
Adapting the proof of Lemma \ref{lem:equal_count}.(a), we deduce 
\[
\sum_{\ell=1}^m\#\MM(\tilde{\hat w}_-^k,\tilde{\hat w}_+^\ell,\AA^{\tau,1}_f,\widetilde J)=-\#\NN(q_-,q_+,f,g) =\#\MM({\hat w}_-,{\hat w}_+,\AA^{\tau,m}_f,J)
\]
 for every $k\in\Z_m$ and  $\hat w_\pm\in\Crit h^m\subset\Crit\AA^{\tau,m}_f$ with $\mu_\RFH^h(\hat w_-)-\mu_\RFH^h(\hat w_+)=1$ and $\w(\hat w_-)=\w(\hat w_+)$. A corresponding identity  for $\check w_\pm$ holds as well. Now we consider $\check w_-,\hat w_+\in\Crit h^m\subset\Crit\AA^{\tau,m}_f$ with	$\mu_\RFH^h(\check w_-)-\mu_\RFH^h(\hat w_+)=1$ and $\w(\check w_-)=\w(\hat w_+)$.
Again by Lemma \ref{lem:equal_count}.(a),
\[
\#\MM_{\om=0}^1({\check w}_-,{\hat w}_+,\AA^{\tau,m}_f,J)=\#\NN^1(\check q_-,\hat q_+,\tilde f^m, \tilde g^m)\,,
\]
where $\check q_-,\hat q_+\in\Crit h^m\subset\Crit\tilde f^m$ are the points corresponding to $\check w_-,\hat w_+$, 
and a minor modification of its proof also yields
\[
\#\MM_{\om=0}^1(\tilde{\check w}_-^k,\tilde{\hat w}_+^\ell,\AA^{\tau,1}_f,\widetilde J)=\#\NN^1(\tilde{\check q}_-^k,\tilde{\hat q}_+^\ell,\tilde f^1, \tilde g^1)\,.
\]
In addition, for Morse trajectories, projecting and lifting via $\wp^m:\Sigma^1\to\Sigma^m$ implies that for every $k\in\Z_m$
\[
\begin{split}
\sum_{\ell= 1}^m\#\MM_{\om=0}^1(\tilde{\check w}_-^k,\tilde{\hat w}_+^\ell,\AA^{\tau,1}_f,\widetilde J)=\sum_{\ell= 1}^m\#\NN^1(\tilde{\check q}_-^k,\tilde{\hat q}_+^\ell,\tilde f^1, \tilde g^1)
&=\#\NN^1(\check q_-,\hat q_+,\tilde f^m, \tilde g^m)\\
&=\#\MM_{\om=0}^1({\check w}_-,{\hat w}_+,\AA^{\tau,m}_f,J)\,.	
\end{split}
\]
Arguing analogously, we also obtain that for any $k\in\Z_m$
\[
\sum_{\ell= 1}^m\#\MM^2(\tilde{\check w}_-^k,\tilde{\hat w}_+^\ell,\AA^{\tau,1}_f,
\widetilde J)=\#\MM^2({\check w}_-,{\hat w}_+,\AA^{\tau,m}_f,J)\,.
\]
Finally, by Lemma \ref{lem:chern_number=min_to_max}, we have that for every $q\in\widehat\NN_s(\Pi(w_-),\Pi(w_+),\a_f,j)$
\[
\#\left\{[\mathbf{w}]=[w]\in \MM^1(\check{w}_-,\hat{w}_+,\AA^{\tau,m}_f,J) \mid \Pi(w)=q_\theta \textrm{ for some }\theta\in S^1\right\}=-\epsilon(q) c_1^{E^m}([q])\,.
\]
It is easy to see that the assertion of Lemma \ref{lem:chern_number=min_to_max} also applies to $(\AA^{\tau,1}_f,\tilde h^m)$ even though $\tilde h^m$ is not perfect. That is, for any $k,\ell\in\Z_m$,
\[
\#\left\{[\mathbf{w}]=[w]\in \MM^1(\tilde{\check{w}}^k_-,\tilde{\hat{w}}^\ell_+,\AA^{\tau,1}_f,\widetilde J) \mid \Pi(w)=q_\theta \textrm{ for some }\theta\in S^1\right\}=-\epsilon(q)  c_1^{E^1}([q])\,.
\]
Since $mc_1^{E^1}=c_1^{E^m}$, the two equalities imply
\[
\sum_{\ell= 1}^m\#\MM^1_{\om\neq 0}(\tilde{\check w}_-^k,\tilde{\hat w}_+^\ell,\AA^{\tau,1}_f,\widetilde J)=\#\MM^1_{\om\neq 0}(\check w_-,\hat w_+,\AA^{\tau,m}_f,J)\,.
\]
Altogether this verifies \eqref{eq:equiv_quot}. Therefore the $\Z_m$-equivariant Rabinowitz Floer homology with zero winding number for $E^1$ is canonically isomorphic to the Rabinowitz Floer homology with zero winding number for $E^m$:
\[
\H_*\big(\RFC_*^{\w_0,(a,b),\Z_m}(\AA^{\tau,1}_f,\tilde h^m),\p^{\w_0,\Z_m}_{\widetilde J}\big) \cong \RFH_*^{\w_0,(a,b)}(\AA^{\tau,m}_f, h^m, J)\,.
\]

Next we define the transfer and projection homomorphisms by 
\[
\begin{split}
&T:\RFC_*^{\w_0,(a,b)}(\AA^{\tau,m}_f, h^m) \to \RFC_*^{\w_0,(a,b)}(\AA^{\tau,1}_f, \tilde h^m)\,,\qquad Tw:=\sum_{i\in\Z_m} \tilde w^i\\[0.5ex]
&P:\RFC_*^{\w_0,(a,b)}(\AA^{\tau,1}_f, \tilde h^m) \to \RFC_*^{\w_0,(a,b)}(\AA^{\tau,m}_f, h^m) \,,\qquad P \tilde w^i:= w \quad \forall i\in\Z_m
\end{split}
\]
respectively. In fact, these correspond to the transfer and projection homomorphisms in Section \ref{sec:transfer} via the isomorphism in Proposition \ref{prop:cap_product}.(a).
The above homomorphisms can be interpreted as maps between $\RFC_*^{\w_0,(a,b)}(\AA^{\tau,1}_f, \tilde h^m)$ and $ \RFC_*^{\w_0,(a,b),\Z_m}(\AA^{\tau,1}_f, \tilde h^m)$ via \eqref{eq:equiv_quot}, and we readily see that $T$ and $P$ are chain maps. Moreover the composition $P\circ T$ is the scalar multiplication by $m$, i.e.
\[
P\circ T:\RFC_*^{\w_0,(a,b)}(\AA^{\tau,m}_f)\to \RFC_*^{\w_0,(a,b)}(\AA^{\tau,m}_f)\,,\qquad w\mapsto mw\,.
\]
The other composition $T\circ P:\RFC_*^{\w_0,(a,b)}(\AA^{\tau,1}_f, \tilde h^m)\to \RFC_*^{\w_0,(a,b)}(\AA^{\tau,1}_f, \tilde h^m)$ maps generators $\tilde{w}^i$ to $\sum_{j=1}^m \tilde{w}^j$, and thus the induced map at the homology level also equals the scalar multiplication by $m$. This provides another proof of Corollary \ref{cor:transfer2}.

\begin{rem}\label{rem:covering}
 Lifting and projecting critical points and gradient flow lines by the covering map $E^1\setminus\OO_{E^1}\to E^m\setminus\OO_{E^m}$ as in Lemma \ref{lem:cyclic1} and Lemma \ref{lem:cyclic2} immediately produce transfer and projection homomorphisms between $\RFH_*^{\w_0,(a,b)}(\AA^{\tau,m}_f)$ and $\RFH_*^{\w_0,(a,b)}(\widetilde{\AA}^\tau_f)$, the latter of which uses the pullback symplectic form and the pullback functions on $E^1$ defined in \eqref{eq:tilde_om} and \eqref{eq:tilde_mu}. It is then natural to expect 
\begin{equation}\label{eq:tilde_isom}
\RFH_*^{\w_0,(a,b)}(\widetilde\AA^\tau_f)\cong\FH_*^{(\frac{a}{1+\tau},\frac{b}{1+\tau})}(\tilde f^1)\cong \RFC_*^{\w_0,(a,b)}(\AA^{\tau,1}_f)\,,
\end{equation}
for $\tilde f^1:\Sigma^1\to\R$. The latter isomorphism from Proposition \ref{prop:cap_product}.(a) is established by a correspondence between gradient flow lines of $\a_f$ and those of $\AA^{\tau,1}_f$. Indeed, an analogous correspondence  for the first two homologies in \eqref{eq:tilde_isom} should exist because using the symplectic form and the functions in \eqref{eq:tilde_om} and  \eqref{eq:tilde_mu} instead does not cause essential changes to equation \eqref{eq:Kazdan-Warner}.
 In this way, we would obtain a bijection between gradient flow lines of $\widetilde{\AA}^\tau_f$ and those of $\AA^{\tau,1}_f$, which might also be proved by adapting arguments in \cite{Fra22}. 
\end{rem}

\section{Contact Hamiltonians and orderability}\label{sec:contact}

It was discovered in \cite{AbM18} that the orderability problem for a contact manifold can be studied by means of the Rabinowitz Floer homology of Liouville fillings. Despite the nonexactenss of line bundles $(E,\Omega)$, we will see that this line of attack is still valid as $\RFH^{\w_0}(E,\Sigma)$ is essentially the Rabinowitz Floer homology of $E\setminus\OO_E$ on which $\Omega$ is exact.

To adapt the setting in \cite{AbM18} to our situation, we first recall our convention
\[
\Sigma=\Sigma_m=\{x\in E\mid m\pi r^2(x)=m\}\,,\qquad \Omega= d\lambda \;\;\textrm{on }\; E\setminus\OO_E\,,
\]
where $\lambda=\left(\tfrac{1}{m}+\pi r^2\right)\alpha$. 
The symplectization of $(\Sigma, \Theta:=(\frac{1}{m}+1)\alpha)$ is defined by 
\[
\Big((0,+\infty)\x\Sigma,d(\r\Theta)\Big)
\]
where $\mathfrak{r}$ denotes the coordinate on $(0,+\infty)$. Then the map 
	\begin{equation}\label{eq:symplectization}
	\Xi:E\setminus\OO_E\to \left(\frac{1}{m+1},+\infty\right)\x \Sigma,\qquad x\mapsto \left(\frac{\frac{1}{m}+\pi r^2(x)}{\frac{1}{m}+1},\frac{x}{\sqrt{\pi} r(x)}\right)\,,
	\end{equation}
satisfies $\Xi^*(\r\Theta)=\lambda$. In the following, we identify these two spaces.

Let $\Cont_0(\Sigma,\xi)$ be the identity component of the group of contactomorphisms on $(\Sigma,\xi)$ where $\xi=\ker\Theta=\ker\alpha$. Its universal cover $\widetilde{\Cont_0}(\Sigma,\xi)$ consists of equivalence classes $[\{\varphi_t\}]$ where $\{\varphi_t\}_{t\in[0,1]}$ is a path of contactomorphisms starting at $\varphi_0=\mathrm{id}_\Sigma$ and two such paths are equivalent if they are homotopic with fixed end points.
A smooth path $\{\varphi_t\}_{t\in[0,1]}$ in $\Cont(\Sigma,\xi)$ with $\varphi_0=\mathrm{id}_\Sigma$ defines a smooth function 
\[
k_t:\Sigma\to\R\,,\qquad k_t(\varphi_t(x)):=\Theta_{\varphi_t(x)}\left(\frac{d}{dt}\varphi_t(x)\right)\,.
\]
Following \cite[Proposition 2.3]{AF12}, we set
\begin{equation}\label{eq:hamiltonian_contact}
K_t:(0,+\infty)\x \Sigma\to\R\,,\qquad K_t(\r,x):= \r k_t(x)\,.
\end{equation}
Its Hamiltonian flow with respect to $d(\r\Theta)$ is 
\[
\phi_K^t(\r,x)=\left(\frac{\r}{\rho_t(x)},\varphi_t(x)\right):(0,+\infty)\x \Sigma\to (0,+\infty)\x \Sigma 
\]
where $\rho_t:\Sigma\to\R$ is defined by 
\begin{equation}\label{eq:rho}
(\varphi_t^*\Theta)_x=\rho_t(x)\Theta_x\,.
\end{equation}

\subsection{Perturbed Rabinowitz action functional}
For a given smooth path $\{\varphi_t\}_{t\in[0,1]}$ in $\Cont_0(\Sigma,\xi)$ with $\varphi_0=\mathrm{id}_\Sigma$, we set
\[
C(\{\varphi_t\}):=6\max_{t\in[0,1]}\left|\int_0^t \max_{x\in\Sigma} \frac{\dot\rho_{a}(x)}{\rho_{a}(x)^2}\,da\right|
\]
where $\rho_t$ is defined as in \eqref{eq:rho}. We fix 
\begin{equation}\label{eq:size_tau}
C>C(\{\varphi_t\})\quad \text{and}\quad  \tau>\frac{e^{2C}}{m+1}\,.	
\end{equation}
We choose a smooth cutoff function 
\[
\beta:\left(\frac{1}{m+1},+\infty\right)\to[0,1],\qquad \beta(\r)=
\left\{\begin{aligned} 
&1\qquad \r\in[e^{-C}\tau,e^{C}\tau]\,, \\[.5ex]
&0\qquad \r\notin[\tfrac{e^C}{m+1},e^{2C}\tau]\,.
\end{aligned}
\right.
\]
We also choose two auxiliary smooth functions with disjoint supports: 
\[
\nu:S^1\to\R,\qquad\supp\nu\subset(0,\tfrac{1}{2}),\qquad \int_0^1\nu(t)\,dt=1\,,
\]
and
\[
\chi:[0,1]\to[0,1]\,,\qquad \chi(\tfrac{1}{2})=0\,,\qquad\chi(1)=1\,,\qquad 0\leq \dot\chi(t)\leq 3\;\;\forall t\in[0,1]\,.
\]
We consider $\beta$ as a function on $(\frac{1}{m+1},+\infty)\x\Sigma$ via projection to the first component and denote the product of $\beta$ and $K_t|_{(\frac{1}{m+1},+\infty)\x\Sigma}$ for simplicity by $\beta K_t$.
Pulling-back $\beta K_t$ to $E\setminus\OO_E$ via the map $\Xi$ in \eqref{eq:symplectization} and extending  it over $\OO_E$ by zero, we think of $\beta K_t$ as a smooth function on $E$, i.e.~$\beta K_t:E\to\R$. We consider the perturbed Rabinowitz action functional
\[
\begin{split}
&\AA^\tau_K:\widetilde{\mathcal{L}}(E)\x\R\to\R,\\
&\AA^\tau_K([u,\bar u],\eta)=-\int_{D^2}\bar u^*\Omega-\eta\int_0^1\nu(t)\mu_\tau(u)dt-\int_0^1\dot\chi(t)\beta K_{\chi(t)}(u(t))dt\,.
\end{split}
\]
One can readily see that $([u,\bar u],\eta)$ is a critical point of $\AA^\tau_K$ if and only if
\[
\left\{\begin{aligned}
\dot u-\eta\nu X_{\mu_\tau}(u)-\dot\chi X_{\beta K_{\chi}}(u)=0\\[.5ex]
 \int_0^1\nu\mu_\tau(u)dt=0\,.\qquad \qquad\quad \;\;
\end{aligned}\;
\right.
\]
Due to the disjoint supports of $\nu$ and $\chi$, the above equations can be recast as 
\begin{equation}\label{eq:perturbed_crit}
\left\{
\begin{aligned}
&\mu_\tau(u(t))=0  & t\in[0,\tfrac{1}{2}]\,,\\[1ex]
&\dot u(t)=\eta\nu(t) X_{\mu_\tau}(u(t))\quad & t\in[0,\tfrac{1}{2}]\,,\\[1ex]
&\dot u(t)=\dot\chi(t)X_{\beta K_{\chi(t)}}(u(t))\qquad & t\in[\tfrac{1}{2},1]\,.
\end{aligned}
\right.
\end{equation}
Moreover, arguing exactly as in \cite[Proposition 2.5]{AM13}, we see that  $u$ is uniformly bounded and away from $\OO_E$, more precisely:
\begin{equation}\label{eq:bound_r}
r(u(t))\in \left[e^{-C/2}\tau,e^{C/2}\tau\right] \qquad \forall t\in S^1\,.	
\end{equation}
Thus the last equation in \eqref{eq:perturbed_crit} simplifies to $\dot u(t)=\dot\chi(t)X_{K_{\chi(t)}}(u(t))$. Since $\phi_{X_{\mu_\tau}}^{\eta\int_0^t\nu(a)da}$ and $\phi_{X_{K}}^{\chi(t)}$ are the Hamiltonian flows of $\eta\nu X_{\mu_\tau}$ and $\dot\chi X_{K_\chi}$ respectively, we have
\[
\left\{
\begin{aligned}
&u(\tfrac{1}{2})=\phi_{X_{\mu_\tau}}^{\eta\int_0^{1/2}\nu(a)da}(u(0))=\phi_{X_{\mu_\tau}}^{\eta}(u(0))=\phi_{R}^{-m\eta}(u(0))\\[1ex]
& u(0)=u(1)=\phi_{X_K}^{\chi(1)}(u(\tfrac{1}{2}))=\phi_{X_K}^1(u(\tfrac{1}{2}))\,.
\end{aligned}
\right.
\]
In particular if $([u,\bar u],\eta)\in\Crit\AA^\tau_K$, then $x=u(\frac{1}{2})\in\Sigma$ is a translated point of $\varphi_1$ with respect to $\alpha$, i.e.~it satisfies 
\begin{equation}\label{eq:trans_pt}
\varphi_1(x)=\phi_{R}^{m\eta}(x) \,,\qquad (\varphi_1^*\alpha)_x=\alpha_x\,.	
\end{equation}
We refer to \cite{San12,San13} for an account on this notion. 
We associate the RFH-index to $w=([u,\bar u],\eta)\in\Crit\AA^\tau_K$ as follows. The linearized flow $\phi^t_{\eta\nu X_{\mu_\tau}+\dot\chi X_{\beta K_\chi}}(u(0))$ for $t\in [0,1]$ in a trivialization given by $\bar u$ as in \eqref{eq:trivialization} defines a path of symplectic matrices $\Psi_w$, and we set 
\[
\mu_\RFH(w):=-\mu_\CZ(\Psi_w)\,.
\]
We also define the winding number of $w=([u,\bar u],\eta)\in\Crit\AA^\tau_K$ by 
\[
\w(w):=\w(u,\bar u)\,,
\]
which is well-defined due to \eqref{eq:bound_r}, see also \eqref{eq:wind_crit}. The action value for $w=([u,\bar u],\eta)\in\Crit\AA^\tau_f$ with $\w(w)=0$ computes to 
\begin{equation}\label{eq:action_K}
\AA^\tau_K(w)=-\int_0^1u^*\lambda - \int_0^1\dot\chi K_{\chi}(u)\,dt =(1+\tau)\eta
\end{equation}
using Proposition \ref{prop:winding}.(f), \eqref{eq:perturbed_crit}, and $\lambda(X_{K_\chi})=-K_\chi$, 
see also \cite[Lemma 3.2]{AF12}.

\subsection{Rabinowitz Floer homology for perturbed functionals}
We assume that $\varphi$ satisfies a genericity condition so that $\AA^\tau_K$ is Morse, see the proof of \cite[Theorem 1.4.(2)]{AM13}. 
Let $J$ be a family of smooth $\Omega$-compatible  almost complex structures on $E$ smoothly parametrized by $\R\x S^1$ such that it is diagonal near $\OO_E$ and outside a bounded subset as in \eqref{eq:diagonal_J}. That is, $J\in\JJ$ and in addition $\Omega$-compatible.
We consider the bilinear map $\mathfrak m$ in \eqref{eq:bilinear form} again, which for this choice of $J$ is a genuine metric. The gradient vector field of $\AA_K^\tau$ with respect to $\mathfrak m$ at $(u,\eta)$ is
\[
\nabla\AA^\tau_K(u,\eta)= \left(J(\eta,t,u)\big(\p_tu-\eta\nu X_{\mu_\tau}(u)-\dot\chi X_{\beta K_\chi}(u)\big),\,-\int_0^1\nu\mu_\tau(u)dt\right)\,.
\]
A smooth map $(u,\eta)\in C^\infty(\R\x S^1,E)\x C^\infty(\R,\R)$ is called a negative gradient flow line of $\AA^\tau_K$ with respect to $J$ if it is a solution of 
\begin{equation}\label{eqn:perturbed_Floer_eqn_for_A}
\left\{\begin{aligned}
&\p_su+J(\eta,t,u)\big(\p_tu-\eta\nu X_{\mu_\tau}(u) - \dot\chi X_{\beta K_\chi}(u)\big)=0\\[.5ex]
&\p_s\eta-\int_0^1\nu\mu_\tau(u)dt=0\,.
\end{aligned}\right.
\end{equation}
For $w_\pm=([u_\pm,\bar u_\pm],\eta_\pm)\in\Crit\AA^\tau_K$, we define the moduli space
\[
\widehat\MM(w_-,w_+,\AA^\tau_K,J)
\]
consisting of smooth maps $w=(u,\eta)\in C^\infty(\R\x S^1,M)\x C^\infty(\R,\R)$ solving \eqref{eqn:perturbed_Floer_eqn_for_A} and 
\[
\lim_{s\to\pm\infty}w(s,\cdot) =u_\pm \,,\qquad [\bar u_-\#u\# \bar u_+^\mathrm{rev}]=0 \;\textrm{ in }\;\Gamma_E\,.	
\]
Since $J$ is diagonal near $\OO_E$ and $K$ vanishes near $\OO_E$, we have the following corollary of Proposition \ref{prop:positivity_of_intersection}.
\begin{cor}\label{cor:positivity_intersection_K}
For any $(u,\eta)\in\widehat\MM(w_-,w_+,\AA^\tau_K,J)$, the intersection number $u\cdot\OO_E=\w(w_-)-\w(w_+)$ is nonnegative. Moreover $u\cdot\OO_E=0$ if and only if the image of $u$ is contained in $E\setminus\OO_E$.
\end{cor}

Next we state a necessary compactness and transversality result for defining the Floer homology of $\AA^\tau_K$. 
\begin{prop}\label{prop:compactness_K}
Let $(M,\om)$ satisfy either (A1) or (A3) in the introduction. Then for a generic $\Omega$-compatible $J\in\JJ$, the following holds. For every $w_\pm\in\Crit\AA^\tau_K$ with 
\[
\w(w_-)=\w(w_+)=0\,,\qquad \mu_\RFH(w_-)-\mu_\RFH(w_+)\leq 2\,,
\]
the moduli space $\widehat\MM(w_-,w_+,\AA^\tau_K,J)$ is cut out transversely and compact in the $C^\infty_{loc}$-topology. 
\end{prop}
The compactness part follows as in Proposition \ref{prop:compactness1} using Corollary \ref{cor:positivity_intersection_K}.
We point out that projecting solutions of the perturbed Rabinowitz Floer equation  \eqref{eqn:perturbed_Floer_eqn_for_A} to the base $M$ does not give Floer solutions. Therefore our method of proving compactness under the assumption (A2) in Section \ref{sec:compact_J_diag} does not work in the current setting. Similarly proving transversality with almost complex structures in $\JJ_\mathrm{diag}$ or $\JJ^\BB$, see Section \ref{sec:J_in_JJ}, does not work. Instead we choose to work with $\Omega$-compatible $J\in\JJ$ for which standard transversality arguments are applicable.

  Unless $w_-= w_+$, there is a free $\R$-action on $\widehat\MM(w_-,w_+,\AA^\tau_K,J)$ by translating solutions in the $s$-direction, and we denote the quotient space by
\[
\MM(w_-,w_+,\AA^\tau_K,J):=\widehat\MM(w_-,w_+,\AA^\tau_K,J)/\R\,.
\]
If $\mu_\RFH(w_-)-\mu_\RFH(w_+)=1$, then $\MM(w_-,w_+,\AA^\tau_K,J)$ is a finite set and we denote by $\#\MM(w_-,w_+,\AA^\tau_K,J)\in\Z$ its cardinality counted with sign with respect to  coherent orientation.

Now we are in a position to define the Floer complex for the perturbed action functional. For this, we recall the setting. 
We assume that $(M,\om)$ satisfies either (A1) or (A3) in the introduction. For a given generic path $\{\varphi_t\}_{t\in[0,1]}$ in $\Cont_0(\Sigma,\xi)$ with $\varphi_0=\mathrm{id}_\Sigma$, let $K_t$ be the associated Hamiltonian defined in \eqref{eq:hamiltonian_contact}. Let $\tau>0$ be a constant associated with $\{\varphi_t\}$ satisfying \eqref{eq:size_tau}. Finally we choose a generic $\Omega$-compatible $J\in\JJ$ to Proposition \ref{prop:compactness_K} hold. Then for $-\infty<a<b<+\infty$, we define the $\Z$-module 
\[
\RFC^{\w_0,(a,b)}_*(\AA^\tau_K)\,,\qquad *\in\Z
\] 
generated by elements $w\in\Crit\AA^\tau_K$ with $\mu_\RFH(w)=*$, $\w(w)=0$, and  $\AA^\tau_K(w)\in(a,b)$. We note that $\AA^\tau_K$ has at most finitely many critical points $w$ with $\mu_\RFH(w)=*$, $\w(w)=0$ and $\AA^\tau_K(w)\in(a,b)$ since it is Morse and there is a unique class of capping disks for each solution of \eqref{eq:perturbed_crit} with zero winding number modulo the action of ${\pi_2(E\setminus\OO_E)}/{\ker \wp^*c_1^{TM}}$ as observed in Section \ref{sec:zero_winding}. We define homomorphisms
\[
\begin{split}
& \p^{\w_0}_J=\p^{\w_0}:\RFC^{\w_0,(a,b)}_*(\AA^\tau_K)\longrightarrow  \RFC^{\w_0,(a,b)}_{*-1}(\AA^\tau_K)\\[1ex]
& \p^{\w_0}w_-:=\sum_{w_+}\#\MM(w_-,w_+,\AA^\tau_K,J)\,w_+\,.
\end{split}
\]
where the sum is taken over all $w_+\in\Crit\AA^\tau_f$ with $\mu_\RFH(w_+)=*-1$, $\w(w_+)=0$, and $\AA^\tau_K(w_+)\in(a,b)$. 
The same reasons as in Proposition \ref{prop:boundary_operator}, Corollary \ref{cor:positivity_intersection_K} and Proposition \ref{prop:compactness_K} imply that $\p^{\w_0}\circ \p^{\w_0}=0$ and we denote 
\[
\RFH^{\w_0,(a,b)}_*(\AA^\tau_K):= \H_*\big(\RFC^{\w_0,(a,b)}(\AA^\tau_K),\p^{\w_0}\big)\,.
\]
As in the unperturbed case \eqref{eq:direct_system_RFH}, we have a bidirect system induced by action filtration homomorphisms. We define
\[
\RFH^{\w_0}_*(E,\Sigma_\tau,\{\varphi_t\}):=\varinjlim_{b\uparrow+\infty}\varprojlim_{a\downarrow-\infty}\RFH^{\w_0,(a,b)}_*(\AA^\tau_K)\,.
\]

\subsection{Application to the orderability problem}

We continue to assume that $(M,\om)$ satisfies either (A1) or (A3) from the introduction.   
A continuation argument as in Section \ref{sec:invariance} proves that
\begin{equation}\label{eq:isom_K}
\RFH^{\w_0}(E,\Sigma_\tau,\{\varphi_t\})\cong \RFH^{\w_0}(E,\Sigma)=\RFH^{\w_0}(E,\Sigma_\tau)\,.	
\end{equation}
For this, it is crucial to choose a homotopy between $\AA^\tau_f$ and $\AA^\tau_K$ so that flow lines for the homotopy also have the positivity of intersection property with $\OO_E$ similar to the situation in Section \ref{sec:invariance}. As the precise value of $\tau$ is not important due to \eqref{eq:isom_K}, we simply write $\RFH^{\w_0}(E,\Sigma,\{\varphi\})$ in Theorem \ref{thm:main}.(e). 
 The following proposition  immediately follows from \eqref{eq:trans_pt} and \eqref{eq:isom_K}.

\begin{prop}\label{prop:translated_pt}
	If $\RFH^{\w_0}(E,\Sigma)$ is nonzero, then every $\varphi\in\Cont_0(\Sigma,\xi)$ has a translated point with respect to $\alpha$. 
\end{prop}

\begin{rem}
It is plausible that one can extend the proposition for \emph{every} contact form supporting $\xi$ by considering $\RFH^{\w_0}$ for a hypersurface in $E$ which arises as the graph of a function on $\Sigma$. 	
\end{rem}

All critical points of $\AA^\tau_K$ lie $E\setminus\OO_E$ by \eqref{eq:bound_r}. Moreover we only consider critical points with zero winding number and 
 all negative gradient flow lines connecting them are entirely contained in $E\setminus\OO_E$ by Corollary \ref{cor:positivity_intersection_K}. Thus we can rephrase the construction of $\RFH^{\w_0,(a,b)}(\AA^\tau_K)$ on the part of the symplectization $(\frac{1}{1+m},+\infty)\x \Sigma\cong E\setminus\OO_E$, see Section \ref{sec:zero_winding}.
This ensures that results in \cite{AbM18} carry over to our situation, which is essentially an exact setting as discussed above. Key inputs are \eqref{eq:action_K} and \eqref{eq:isom_K}.  
\begin{thm}\label{thm:orderablility}
If $\RFH^{\w_0}(E,\Sigma)$ is nonzero, then the group $\widetilde{\Cont_0}(\Sigma,\xi)$ is orderable in the sense of \cite{EP00}.
\end{thm}
Strictly speaking, it is shown in \cite{AM18} that the presence of a spectrally finite class implies orderability. A nonzero class $Z\in \RFH^{\w_0}(E,\Sigma)$ is said to be spectrally finite if 
\[
\inf \Big\{b\in\R \,\Big|\, Z\in \mathrm{im}\big( \iota^b: \varprojlim_{a\downarrow-\infty}\RFH^{\w_0,(a,b)}(E,\Sigma)\to \RFH^{\w_0}(E,\Sigma)\big)\Big\} \neq -\infty
\]
where $\iota^b$ is the action filtration homomorphism. 
However under the hypothesis (A1), (A2), or (A3) in the introduction, $\RFH^{\w_0}(E,\Sigma)$ is nonzero if and only if it has a spectrally finite class due to Remark \ref{rem:finitely_many}.(i). 

\begin{cor}
Suppose that the Floer cap product with $c_1^E$ on $\FH(M)$ is not an isomorphism.  
Then $\widetilde{\Cont_0}(\Sigma,\xi)$ is orderable. 
\end{cor}
\begin{proof}
	This follows from Proposition \ref{prop:nonzero_class}, Proposition \ref{prop:cap_product}, and  Theorem \ref{thm:orderablility}.
\end{proof}

This completes the proof of Theorem \ref{thm:main}.(e). Corollary \ref{cor:orderability} now follows from results in Section \ref{sec:nonvanishing} and Proposition \ref{prop:cap_product}. 
Finally, we present another proof for the computation for $\OO_{\CP^{n}}(-1)$ in Corollary \ref{cor:rhf_CPn}.
\begin{cor}
Let $E=\OO_{\CP^{n}}(-1)$ be as in Corollary \ref{cor:rhf_CPn}. Then 
\[
\RFH_*^{\w_0}(\OO_{\CP^{n}}(-1),S^{2n+1}) =0\qquad \forall *\in\Z\,.
\]
\end{cor}
\begin{proof}
	We note that in this case $(\Sigma,\xi)$ is contactomorphic to $S^{2n+1}$ with the standard contact structure which is known to be not orderable due to \cite[Theorem 1.10]{EKP06}. Therefore the assertion follows from Theorem \ref{thm:orderablility}.
\end{proof}

\begin{rem}
The full Rabinowitz Floer homology of $(E,\Sigma_\tau)$ studied in \cite{AK17}, that we revisit in Section \ref{sec:full} below, may not be appropriate to attack the orderability problem. Leaving aside applicability of the full Rabinowitz Floer homology to the orderability problem in the nonexact setting, we need a nonvanishing result for arbitrary $\tau>0$ in order to accommodate the condition in \eqref{eq:size_tau}, but the full Rabinowitz Floer homology, in general, does vary between vanishing and nonvanishing under the change of $\tau$. 
\end{rem}

\section{Full Rabinowitz Floer homology}\label{sec:full}
In this section, we define the Rabinowitz Floer homology of $\AA^\tau_f$ which takes all critical points, i.e.~with any winding numbers, into account. We sometimes call this the full Rabinowitz Floer homology to distinguish it from the homology with zero winding number we have studied so far. We work with $J\in\JJ^\BB_\mathrm{reg}$ and accordingly assume either (A1) or (A2) from the introduction. If the horizontal part $j$ of $J$ is from $\j_{\mathrm{HS}}$, we can slightly relax assumption (A2), see the beginning of Section \ref{sec:RFH}. 

\subsection{Two constructions of full Rabinowitz Floer homology}
For $-\infty < a<b < +\infty$, we define the $\Z$-module 
\[
\RFC_*^{(a,b)}(\AA^\tau_f)\,,\qquad *\in\Z 
\]
generated by critical points $w\in\Crit h\subset\Crit\AA^\tau_f$ with $\mu_\RFH^h(w)=*$ and $\AA^\tau_f(w)\in(a,b)$. We also define
\begin{equation}\label{eq:compl}
\begin{split}
\RFC_*^{(-\infty,b)}(\AA^\tau_f)&:=\varprojlim_{a\downarrow-\infty}\RFC_*^{(a,b)}(\AA^\tau_f) \,,\\
\RFC_*^{(-\infty,+\infty)}(\AA^\tau_f)&:=\varinjlim_{b\uparrow+\infty}\varprojlim_{a\downarrow-\infty}\RFC_*^{(a,b)}(\AA^\tau_f)\,,
\end{split}
\end{equation}
using canonical inclusions and projections, cf.~\eqref{eq:FC}. 
We recall from 	\eqref{eq:action_value_full} and \eqref{eq:mu_RFH_cov}  that for $w=([u,\bar u_\mathrm{fib}\#s],\eta)\in\Crit\AA^\tau_f$, 
\begin{equation}\label{eq:action_index}
\begin{split}
\AA^\tau_f(w) &= -\om([s])- \frac{\tau}{m}\cov(u)-(1+\tau)f(q)\,,\\
\mu_\RFH(w) &= -2\cov(u)-2(\lambda-m)\om([s])+\mu_{-f}(q)-\frac{1}{2}\dim M\,,
\end{split}	
\end{equation}
where $q=\wp\circ u$ as usual. 
Here $\lambda\in\R$ is the constant such that $c_1^E=\lambda\omega$ on $\pi_2(M)$ in (A2), and in the case of (A1) the term $2(\lambda-m)\omega([s])$ vanishes. In particular, the second identity in \eqref{eq:action_index} implies that for each $*\in\Z$ there are at most finitely many $w\in\Crit\AA^\tau_f$ with $\mu_\RFH(w)=*$ if (A1) is assumed. To study the case of (A2), let $c_1^E=\lambda\omega$ on $\pi_2(M)$. 
Abbreviating $e(w):=\mu_{-f}(q)-\frac{1}{2}\dim M$, we have
\begin{equation}\label{eq:full_action_winding}
	\AA^\tau_f(w)= \left(\frac{\tau}{m}(\lambda-m)-1\right)\om([s]) + \frac{\tau}{2m}\big(\mu_\RFH(w)-e(w)\big) -(1+\tau)f(q)\,.
\end{equation}
We recall our convention  $\frac{\tau}{m}=\pi r^2(u)$ and note that $e(w)$ and $(1+\tau)f(q)$ are uniformly bounded. If $\tau(\lambda-m)\neq m$, then for each $*\in\Z$ there are at most finitely many $w\in\Crit\AA^\tau_f$ with $\mu_\RFH(w)=*$ and $\AA^\tau_f(w)\in(a,b)$ by \eqref{eq:full_action_winding}, and thus $\RFC_*^{(a,b)}(\AA^\tau_f)$ has finite rank. On the other hand, \eqref{eq:full_action_winding} also shows that, if $\tau(\lambda-m)=m$, then for a given $*\in\Z$ there exist $a,b\in\R$ such that every $w\in\Crit\AA^\tau_f$ with $\mu_\RFH(w)=*$ has action $\AA^\tau_f(w)\in(a,b)$. Therefore $\RFC_*^{(a,b)}(\AA^\tau_f)$ can have infinite rank in degree $*\in\Z$ provided $\Gamma_E$ has infinitely many elements.
Using $\w(w)=\cov(u)-m\om([s])$, see Proposition \ref{prop:winding}.(a), we recast \eqref{eq:full_action_winding} as 
\begin{equation}\label{eq:full_action_winding2}
	\AA^\tau_f(w)= \left(\frac{\tau}{m}(\lambda-m)-1\right)\frac{2\w(w)+\mu_\RFH(w)-e}{-2\lambda}  + \frac{\tau}{2m}\left(\mu_\RFH(w)-e\right) -(1+\tau)f(q)\,.
\end{equation}
Now we define the homomorphism 
\[
\p_J=\p:\RFC_*^{(a,b)}(\AA^\tau_f)\longrightarrow \RFC_{*-1}^{(a,b)}(\AA^\tau_f)
\]
by the linear extension of 
\begin{equation}\label{eq:p_full}
\p w_-:=\sum_{w_+}\#\MM(w_-,w_+,\AA^\tau_f,J)\,w_+\,.	
\end{equation}
Here the sum is taken over all $w_+\in\Crit h$ with $\mu_\RFH^h(w_+)=*-1$ and $\AA^\tau_f(w_+)\in(a,b)$. To make sense of the above definition in the case of (A2) with $\tau(\lambda-m)=m$, we need to ensure that the sum in \eqref{eq:p_full} is finite. Suppose that this is not the case. Then we have a sequence $\mathbf{w}_\nu\in\MM(w_-,w_+^{\nu},\AA^\tau_f,J)$, $\nu\in\N$, where $w_+^{\nu}$ are mutually distinct critical points of $h$ satisfying the above conditions. Since $\mu_\RFH^h(w_+^\nu)=*-1$, in particular is independent of $\nu$, the covering number of the loop component of $w_+^\nu$ does not converge, see the index computation in \eqref{eq:action_index}. This contradicts the fact that the sequence $\mathbf{w}_\nu$ has energy bounded by $b-a$ and thus converges to a broken flow line with cascades. We will see this finiteness in an alternative way in Remark \ref{prop:vanishing}.

We have established necessary compactness and transversality results for $\p\circ\p=0$ in Proposition \ref{prop:transversality_J_B} and Proposition \ref{prop:compactness2}. Note that we do not assume any condition on winding numbers in these propositions. Hence we can define
\[
\RFH^{(a,b)}_*(\AA^\tau_f,J):= \H_*\big(\RFC^{(a,b)}(\AA^\tau_f),\p_J\big)\,.
\]
A standard continuation argument shows that it is invariant under the change of $J\in\JJ_\mathrm{reg}^\BB$, so we omit $J$ from the notation. We also write $\RFH_*(E,\Sigma_\tau)=\RFH_*^{(-\infty,+\infty)}(\AA^\tau_f)$. This proves Theorem \ref{thm:full_rfh}.(a). 

\begin{rem}\label{rem:Novikov}
The module $\RFC_*^{(-\infty,+\infty)}(\AA^\tau_f)$ in \eqref{eq:compl} is the one used in \cite{AK17}. Elements can be interpreted as formal linear combinations 
\[
\sum_{w\in\Crit h} a_w w\,,\qquad a_w\in\Z\,,\;\; \mu_\RFH^h(w)=*
\]
subject to the Novikov condition 
\begin{equation}\label{eq:novikov}
\forall \kappa\in\R\;:\; \#\{w \mid a_w\neq 0\,,\;\AA^\tau_f(w)\geq\kappa\}<\infty\,.	
\end{equation}
\end{rem}

\begin{rem}\label{rem:M-L}
Suppose that condition (A2) is fulfilled. The following discussion is trivial in the case of (A1). 
We can define the  Rabinowitz Floer homology for the action-windows $(-\infty,b)$ or $(-\infty,+\infty)$ as follows:
\[
\varprojlim_{a\downarrow-\infty}\RFH_*^{(a,b)}(\AA^\tau_f)\,,\qquad  \varinjlim_{b\uparrow+\infty}\varprojlim_{a\downarrow-\infty}\RFH_*^{(a,b)}(\AA^\tau_f)\,.
\]
As observed above, if $\tau(\lambda-m)=m$, there exist $a,b\in\R$ depending on $*\in\Z$ such that every $w\in\Crit\AA^\tau_f$ with $\mu_\RFH(w)=*$ satisfies $\AA^\tau_f(w)\in(a,b)$. Hence, 
\begin{equation}\label{eq:rfh=rfh_nov}
\RFH_*(E,\Sigma_\tau) = \varinjlim_{b\uparrow+\infty}\varprojlim_{a\downarrow-\infty}\RFH_*^{(a,b)}(\AA^\tau_f)\quad\textrm{for }\; \tau=\frac{m}{\lambda-m}\,,
\end{equation}
and furthermore switching the two limits does not make any difference. 
However it is not clear whether the two homologies in \eqref{eq:rfh=rfh_nov} are isomorphic for general $\tau>0$. To be precise, there is a surjective homomorphism 
\begin{equation}\label{eq:full_surj}
\RFH_*(E,\Sigma_\tau) \longrightarrow  \varinjlim_{b\uparrow+\infty}\varprojlim_{a\downarrow-\infty}\RFH_*^{(a,b)}(\AA^\tau_f)
\end{equation}
induced naturally by the universal properties of inverse and direct limits, and the exactness of direct limits. This is even an isomorphism if the inverse limit satisfies the Mittag-Leffler condition as discussed in Remark \ref{rem:ML} and Remark \ref{rem:ML2}. We emphasize again that this is indeed the case if we work with coefficients in a field, see \cite{CF11}.
\end{rem}

We remark that in the case of (A2) with $1\leq m\leq \lambda-1$, the  module $\RFC_*^{(-\infty,+\infty)}(\AA^\tau_f)$ changes dramatically at $\tau=\frac{m}{\lambda-m}$. Indeed the sign of the coefficient of $\om([s])$ in  \eqref{eq:full_action_winding} is determined by the sign of $\tau-\frac{m}{\lambda-m}$, and thus by \eqref{eq:ind_M}, the Novikov condition in \eqref{eq:novikov} is equivalent to 
\begin{equation}\label{eq:nov_ind}
	\begin{split}
	&\forall \kappa\in\R\;:\; \#\{w \mid a_w\neq 0\,,\;\mu_\FH(\Pi(w))\geq\kappa\}<\infty \quad  \textrm{if }\,\tau<\frac{m}{\lambda-m}\,,\\[.5ex]
	&\forall \kappa\in\R\;:\; \#\{w \mid a_w\neq 0\,,\;\mu_\FH(\Pi(w))\leq\kappa\}<\infty \quad  \textrm{if }\,\tau>\frac{m}{\lambda-m}\,.
	\end{split}
\end{equation}
In other words, if $\tau<\tfrac{m}{\lambda-m}$, then $\RFC_*^{(-\infty,+\infty)}(\AA^\tau_f)$ contains infinite sums only in the direction that the $\mu_\FH$-index in the base $M$ decreases. Exactly the opposite is true when $\tau>\tfrac{m}{\lambda-m}$. We mention again that $\RFC_*^{(-\infty,+\infty)}(\AA^\tau_f)$ for $\tau=\frac{m}{\lambda-m}$ consists only of finite sums. It is worth pointing out that if $m\geq\lambda$, then the first line in \eqref{eq:nov_ind} holds for all $\tau>0$ since the coefficient of $\om([s])$ in \eqref{eq:full_action_winding} is always negative.

\begin{prop}\label{prop:vanishing}
In the case of (A1) or (A2) with $\lambda \leq m$, let $\tau>0$ be arbitrary. In the remaining case, i.e.~(A2) with $1\leq m\leq\lambda-1$, let $0<\tau<\frac{m}{\lambda-m}$. Then we have
\[
\RFH_*(E,\Sigma_\tau)=0\qquad \forall *\in\Z\,.
\]
\end{prop}
	This proposition was proved in \cite[Theorem 1.2]{AK17}. Strictly speaking, the assumption on $(M,\om)$ in this paper is different from that in \cite{AK17}, and $\Z_2$-coefficient is used in \cite{AK17}, but the same proof remains applicable, see also Remark \ref{rem:p_0}. Nevertheless we present a proof as part of the proof of Proposition \ref{prop:id+Psi_surjective}.
	
	\begin{rem}\label{rem:full_RFH_invariance}
	The assumption $0<\tau<\frac{m}{\lambda-m}$ when $1\leq m\leq\lambda-1$ is not made in the statement of \cite[Theorem 1.2]{AK17} although this assumption together with the first case of \eqref{eq:nov_ind} was crucially used in the proof  to guarantee that boundaries of cycles that we explicitly construct indeed satisfy the Novikov condition. We had the wrong expectation that $\RFH_*(E,\Sigma_\tau)$ is invariant under the change of $\tau$. As pointed out by Sara Venkatesh \cite{Ven18}, this invariance property is not always true. We will come back to this below. 
	\end{rem}

\subsection{Two filtrations}\label{sec:two_filtrations}
For $w_\pm\in\Crit\AA^\tau_f$, equation \eqref{eq:indices_and_winding} yields
	\begin{equation}\label{eq:index_wind_intersection}
	\mu_\RFH(w_-) - \mu_\RFH(w_+) = 2(\w(w_+)-\w(w_-))+\mu_\FH(\Pi(w_-)) -\mu_\FH(\Pi(w_+))\,.
	\end{equation}
Now we introduce two filtrations of $\p$ suggested by the relation \eqref{eq:index_wind_intersection}.

\subsubsection*{Filtration by winding number}

For each $i\in\Z$, we define the homomorphism 
\[
\p^{\w_i}:\RFC^{(a,b)}_*(\AA^\tau_f)\longrightarrow\RFC^{(a,b)}_{*-1}(\AA^\tau_f)
\]
by the linear extension of
\[
\p^{\w_i} w_-:=\sum_{w_+}\#\MM(w_-,w_+,\AA^\tau_f,J)\,w_+\,,
\]
where the sum ranges over all $w_+\in\Crit h$ with $\mu_\RFH^h(w_+)=*-1$, $\AA^\tau_f(w_+)\in(a,b)$, and $\w(w_+)-\w(w_-)=i$. The restriction of $\p^{\w_0}$ to $\RFC^{\w_0,(a,b)}_*(\AA^\tau_f)$ recovers the previous definition of $\p^{\w_0}$ in \eqref{eq:boundary}. We also write $\p^{\w_0}$ in the current  situation. Suppose that $\MM(w_-,w_+,\AA^\tau_f,J)$ is nonempty for $w_\pm\in\Crit h$ with $\mu_\RFH^h(w_-)-\mu_\RFH^h(w_+)=1$. Since cascades in $\MM(w_-,w_+,\AA^\tau_f,J)$ project to Floer cylinders in $M$, see \eqref{eq:proj_cylinder}, we have $\mu_\FH(\Pi(w_-))\geq\mu_\FH(\Pi(w_+))$ due to our choice of $j\in\mathfrak{j}_\mathrm{reg}(f)\cup\mathfrak{j}_\mathrm{HS}(f)$. Moreover, by 
Proposition \ref{prop:positivity_of_intersection}, $\w(w_+)\geq \w(w_-)$. Hence, \eqref{eq:index_wind_intersection} together with \eqref{eq:index_h} yields that $\w(w_+)$ equals either $\w(w_-)$ or $\w(w_-)+1$, and we have in fact
\[
\p=\p^{\w_0} +\p^{\w_1}\,.
\]

\begin{rem}\label{rem:boundary_operator}
We note that the identity $\p\circ\p=0$ subsumes $\p^{\w_0}\circ\p^{\w_0}=0$. Indeed,
\[
0=\p\circ\p= \p^{\w_0}\circ\p^{\w_0}+(\p^{\w_1}\circ\p^{\w_0}+\p^{\w_0}\circ\p^{\w_1})+\p^{\w_1}\circ \p^{\w_1}
\]
implies that each of the three terms on the right-hand side vanishes individually since they map into different subspaces.
\end{rem}

\subsubsection*{Filtration by index on the base}
For each $i\in\Z$, we define the homomorphism
\[
\p_i:\RFC^{(a,b)}_*(\AA^\tau_f)\longrightarrow\RFC^{(a,b)}_{*-1}(\AA^\tau_f)
\]
by the linear extension of
\[
\p_{i} w_-:=\sum_{w_+}\#\MM(w_-,w_+,\AA^\tau_f,J)\,w_+\,.
\]
where the sum is taken over all $w_+\in\Crit h$ with $\mu_\RFH^h(w_+)=*-1$, $\AA^\tau_f(w_+)\in(a,b)$, and $\mu_\FH(\Pi(w_-))-\mu_\FH(\Pi(w_+))=i$.
By the discussion on \eqref{eq:index_wind_intersection} above, $\p_i$ for $i<0$ and $i>2$ are trivial, and hence 
\begin{equation}\label{eq:filtr}
\p=\p_0 + \p_1 +\p_2\,.	
\end{equation}
In fact, \eqref{eq:index_wind_intersection} relates the two filtrations as follows. We have  
\[
\p^{\w_0}=\p_1+\p_2\,,\qquad  \p^{\w_1}=\p_0\,.	
\]
For the latter equality, we recall from Remark \ref{rem:zero_cascade} that  flow lines with zero cascade do not contribute to $\p_0$, see also Remark \ref{rem:p_0} below. 
Similar to Section \ref{sec:gysin_revisit}, we abbreviate $\mathfrak C:=\Crit h\subset\Crit\AA^\tau_f$ and denote by $\widehat{\mathfrak{C}}$ and $\widecheck{\mathfrak{C}}$ the subsets composed of the maximum and minimum points of $h$ respectively:
\[
\mathfrak C =\widehat{\mathfrak{C}}  \sqcup \widecheck{\mathfrak{C}} \,.
\]
As in \eqref{eq:p^w_0} we decompose 
\[
\p^{\w_0}=\hat\p^{\w_0}+\check\p^{\w_0}+\p^{\w_0,c_1^E}
\]
where $\hat\p^{\w_0}$ counts solutions from $\widehat{\mathfrak{C}}$ to itself, $\check\p^{\w_0}$ counts solutions from $\widecheck{\mathfrak{C}}$ to itself, and $\p^{\w_0,c_1^E}$ counts solutions from $\widecheck{\mathfrak{C}}$ to $\widehat{\mathfrak{C}}$. Then it holds that 
\begin{equation}\label{eq:filtr2}
\p_1 = \hat\p^{\w_0}+\check\p^{\w_0} \,,\qquad \p_2=\p^{\w_0,c_1^E}\,.
\end{equation}
We sum up the domains and targets of $\p_i$:
\[
	\begin{split}
	\p_0:\widecheck{\mathfrak{C}}\to \widehat{\mathfrak{C}}\,,\quad \p_0:\widehat{\mathfrak{C}}\to 0\,,\quad 
	\p_1:\widehat{\mathfrak{C}}\to \widehat{\mathfrak{C}}\,,\quad \p_1:\widecheck{\mathfrak{C}}\to \widecheck{\mathfrak{C}}\,,\quad 
	\p_2:\widecheck{\mathfrak{C}}\to \widehat{\mathfrak{C}}\,,\quad \p_2:\widehat{\mathfrak{C}}\to 0\,.
\end{split}
\]

\begin{rem}\label{rem:p_0}
	The filtration by index played a key role in \cite{AK17} to show Proposition \ref{prop:vanishing}. 
As in Remark \ref{rem:boundary_operator}, $\p\circ \p=0$ implies $\p_0\circ\p_0=0$. Since $\p_0$ counts solutions $(u,\eta)$ with $u$ entirely contained in a single fiber of $E\to M$, see \cite[Proposition 3.3]{AK17}, the chain complex with $\p_0$ is  just copies of the Rabinowitz Floer complex for $(\C,S^1)$. We hence have
\[
\H_*\big(\RFC^{(-\infty,+\infty)}(\AA^\tau_f),\p_0\big)=0\,\qquad \forall *\in\Z\,,
\]
see \cite[Corollary 3.4]{AK17}. We even know that for $\check w_-=([u_-,\bar u_-],\eta_-)\in\widecheck{\mathfrak{C}}$
\[
\partial_0(\check w_-)=\hat w_+ 
\] 
where $\hat w_+=([u_+,\bar u_+],\eta_+)\in\widehat{\mathfrak{C}}$ is uniquely determined by $\Pi(\check w_-)=\Pi(\hat w_+)$ and $\cov(u_+)=\cov(u_-)+1$.

One can see that the sum in \eqref{eq:p_full} is finite even when  $\tau(\lambda-m)=m$ from the fact that $\p^{\w_0}w_-$ consists of a finite sum by Remark \ref{rem:finitely_many}.(i)  and $\p^{\w_1}w_-=\p_0w_-$ is zero or has a single term as mentioned above.
\end{rem}

\begin{rem}\label{rem:full_RFH_diagonal_J}
So far $J\in\JJ^\BB_\mathrm{reg}$ was taken. 
We can also define $\RFH_*(\AA^\tau_f)$ with $J\in\JJ_\mathrm{diag}$ but with the tools from this article only for small $\tau>0$. Transversality results for $\p^{\w_0}$ are established in Section \ref{sec:trans_diag}. Moreover as mentioned Remark \ref{rem:p_0}, $\p^{\w_1}=\p_0$ corresponds to the boundary operator of the Rabinowitz Floer complex for $(\C,S^1)$ where automatic transversality holds, see \cite[Appendix A]{AF16}. The latter fact is also used in the proof of Proposition \ref{prop:transversality_J_B}. Next we want to see $\p\circ\p=(\p^{\w_0}+\p^{\w_1})\circ (\p^{\w_0}+\p^{\w_1})=0$. Since $\p^{\w_0}\circ \p^{\w_0}=0$ and $\p^{\w_1}\circ\p^{\w_1}=0$, it remains to verify
\[
\p^{\w_0}\circ \p^{\w_1}+\p^{\w_1}\circ\p^{\w_0}=\hat\p^{\w_0}\circ \p^{\w_1}+\p^{\w_1}\circ\check\p^{\w_0}=0\,.
\]
This indeed holds since two chain complexes $\widehat\RFC(\AA^\tau_f)$ and $\widecheck\RFC(\AA^\tau_f)$ can be identified through $\p^{\w_1}=\p_0$ by Remark \ref{rem:p_0}, and up to this identification $\hat\p^{\w_0}$ and $-\check\p^{\w_0}$ coincide by Lemma \ref{lem:equal_count}. 
Due to the aforementioned transversality results for $J\in\JJ_\mathrm{diag}$, a  continuation argument shows that $\RFH_*(\AA^\tau_f)$ defined with $J\in\JJ_\mathrm{diag}$ is isomorphic to the one defined with $J\in\JJ_\mathrm{reg}^\BB$.

We also note that a continuation homomorphism $\varphi$ at the chain level with respect to any admissible change of data $f$, $h$, $\tau$, and $J$ respect winding numbers. Indeed, a continuation homomorphism counts solutions connecting critical points $w_-$ and $w_+$ with $\mu_\RFH^h(w_-)=\mu_\RFH^h(w_+)$. Then \eqref{eq:index_wind_intersection} yields that $\w(w_+)=\w(w_-)$. Therefore 
\[
\varphi\circ\p^{\w_0}=\p^{\w_0}\circ\varphi\,,\qquad \varphi\circ\p^{\w_1}=\p^{\w_1}\circ\varphi\,.
\]
\end{rem}

\subsection{Gysin sequence for full Rabinowitz Floer homology}
In this section, we assume that condition (A2) is fulfilled. However, the situation with (A1) can be dealt with exactly in the same way as in the case of (A2) with $\tau(\lambda-m)=m$.
Let $-\infty\leq a<b\leq+\infty$. As in Section \ref{sec:gysin_revisit}, we decompose 
\[
{\RFC}^{(a,b)}(\AA^\tau_f)= \widehat{\RFC}{}^{(a,b)}(\AA^\tau_f)\oplus \widecheck{\RFC}^{(a,b)}(\AA^\tau_f)
\]
where $\widehat{\RFC}{}^{(a,b)}(\AA^\tau_f)$ and $\widecheck{\RFC}^{(a,b)}(\AA^\tau_f)$ are the submodules generated by elements in $\widehat{\mathfrak{C}}$ and $\widecheck{\mathfrak{C}}$ respectively. Then $(\widehat{\RFC}{}^{(a,b)}(\AA^\tau_f),\hat \p^{\w_0})$
is a subcomplex since $\p$ and $\hat\p^{\w_0}$ coincide on $\widehat{\RFC}{}^{(a,b)}(\AA^\tau_f)$. The quotient complex is isomorphic to $(\widecheck{\RFC}^{(a,b)}(\AA^\tau_f),\check\p^{\w_0})$. We denote the homologies of these two complexes by $\widehat\RFH{}_{*}^{(a,b)}(\AA^\tau_f)$ and $\widecheck\RFH_{*}^{(a,b)}(\AA^\tau_f)$ respectively. As before, we omit $(a,b)$ if we consider the whole action-window $(-\infty,+\infty)$. We also write $\widehat\RFH{}_{*}(E,\Sigma_\tau)=\widehat\RFH{}_{*}(\AA^\tau_f)$, and likewise for $\widecheck\RFH{}_{*}(\AA^\tau_f)$. 
 In view of Remark \ref{rem:w_k_RFH}, for $a\neq -\infty$, we have
\[
\begin{split}
\widehat\RFH{}_{*}^{(a,b)}(\AA^\tau_f)&=\bigoplus_{k\in\Z}\widehat\RFH{}_*^{\w_k,(a,b)}(\AA^\tau_f)\cong \bigoplus_{k\in\Z}\FH_{*+2k-1}^{\big(\frac{1}{1+\tau}(a+\frac{\tau k}{m}),\frac{1}{1+\tau}(b+\frac{\tau k}{m})\big)}(f)\,,\\[1ex]
\widecheck\RFH_{*}^{(a,b)}(\AA^\tau_f)&=\bigoplus_{k\in\Z}\widecheck\RFH_*^{\w_k,(a,b)}(\AA^\tau_f)\cong \bigoplus_{k\in\Z}\FH_{*+2k}^{\big(\frac{1}{1+\tau}(a+\frac{\tau k}{m}),\frac{1}{1+\tau}(b+\frac{\tau k}{m})\big)}(f)\,.
\end{split}
\]
We now compute these homologies for the action-window $(a,b)=(-\infty,+\infty)$. The case $(-\infty,b)$ can be done analogously. If $\tau(\lambda-m)=m$, then by the discussion following \eqref{eq:full_action_winding}
\begin{equation}\label{eq:check_rfh}
\begin{split}
\widecheck\RFH_{*}(E,\Sigma_\tau)\cong \widehat\RFH{}_{*+1}(E,\Sigma_\tau)&=\bigoplus_{k\in\Z}\widehat\RFH{}_{*+1}^{\w_k}(E,\Sigma_\tau)\cong \bigoplus_{k\in\Z}\FH_{*+2k}(M)\,,
\end{split}
\end{equation}
where the last isomorphism respects $k\in\Z$.
Next we assume $\tau(\lambda-m)>m$. In this case, for a fixed $\mu_\RFH$-index, $\AA^\tau_f$ is negatively proportional to $\w$ by \eqref{eq:full_action_winding2}. Therefore 
\begin{equation}\label{eq:check_rfh2}
\begin{split}
\widecheck\RFH{}_{*}(E,\Sigma_\tau)&\cong \widehat\RFH{}_{*+1}(E,\Sigma_\tau) \\
& \cong \left\{\sum_{k\geq k_0} Z_k \mid Z_k\in \widehat\RFH{}_{*+1}^{\w_k}(E,\Sigma_\tau)\cong \FH{}_{*+2k}(M),\; k_0\in\Z \right\}\,.
\end{split}	
\end{equation}
Similarly, if $\tau(\lambda-m)<m$ and $\mu_\RFH$-index is fixed, then $\AA^\tau_f$ is positively proportional to $\w$, and thus
\[
\begin{split}
\widecheck\RFH{}_{*}(E,\Sigma_\tau)&\cong\widehat\RFH{}_{*+1}(E,\Sigma_\tau)\\
& \cong \left\{\sum_{k\leq k_0} Z_k \mid Z_k\in \widehat\RFH{}_{*+1}^{\w_k}(E,\Sigma_\tau)\cong \FH{}_{*+2k}(M),\; k_0\in\Z \right\}\,.
\end{split}
\]
This proves that
\[
\widecheck\RFH{}_{*}(E,\Sigma_\tau)\cong\widehat\RFH{}_{*+1}(E,\Sigma_\tau)\cong \mathop{\widetilde{\bigoplus}}_{k\in\Z}\FH_{*+2k}(M)
\]
where 
\[
\mathop{\widetilde{\bigoplus}}_{k\in\Z}\FH_{*+2k}(M) := \left\{
\begin{aligned} 
&\bigg\{\sum_{k\leq k_0} Z_k \mid Z_k\in \FH{}_{*+2k}(M),\; k_0\in\Z \bigg\}\qquad & \tau(\lambda-m)<m\,,\\[.5ex]
&\bigg\{\sum_{|k|\leq k_0} Z_k \mid Z_k\in \FH{}_{*+2k}(M),\; k_0\in\Z \bigg\}	  & \tau(\lambda-m)=m\,,  \\[.5ex]
& \bigg\{\sum_{k\geq k_0} Z_k \mid Z_k\in \FH{}_{*+2k}(M),\; k_0\in\Z \bigg\}  &  \tau(\lambda-m)>m\,,
\end{aligned}
\right.
\]  
was defined in Theorem \ref{thm:full_rfh}.(b). We note that these are 2-periodic in degree. 
The short exact sequence of chain complexes
\[
0\to \big(\widehat{\RFC}{}(\AA^\tau_f),\hat\p^{\w_0}\big)\stackrel{}{\to} \big(\RFC(\AA^\tau_f),\p\big)\to \big(\widecheck{\RFC}(\AA^\tau_f),\check\p^{\w_0}\big)\to 0\,,
\]
induces the long exact sequence 
\begin{equation}\label{eq:les_full2}
\cdots\to\RFH_*(E,\Sigma_\tau) \to \widecheck\RFH_{*}(E,\Sigma_\tau) \stackrel{\delta}{\to} \widehat\RFH{}_{*-1}(E,\Sigma_\tau)\stackrel{}{\to} \RFH_{*-1}(E,\Sigma_\tau) \to  \cdots
\end{equation}
where the connecting map $\delta$ is induced by $\p_0+\p_2$ due to \eqref{eq:filtr} and \eqref{eq:filtr2}. 
Thus vanishing of $\RFH_*(E,\Sigma_\tau)$ in all degrees $*\in\Z$ is equivalent to $\delta$ being an isomorphism. This happens, for instance, under the hypothesis of Proposition \ref{prop:vanishing}. 
	We write $\delta=\delta_0+\delta_2$ where $\delta_0$ and $\delta_2$ are the maps at the homology level induced by $\p_0$ and $\p_2$ respectively. Then $\delta_0$ gives an isomorphism $\widecheck{\RFH}^{\w_k}_*(E,\Sigma_\tau)\to \widehat{\RFH}{}^{\w_{k+1}}_{*-1}(E,\Sigma_\tau)$ for every $k\in\Z$ since $\p_0$ even identifies the respective chain complexes, see Remark \ref{rem:p_0} and also Remark \ref{rem:full_RFH_diagonal_J}. This corresponds to the identity map on $\FH_{*+2k}(M)$. Furthermore, $\delta_2$ corresponds to the Floer cap product $\Psi^{c_1^E}$ as observed in the proof of Proposition \ref{prop:cap_product}.(b). In other words, we have the following commutative diagrams for each $k\in\Z$: 
\begin{equation}\label{eq:delta_0,2}
	\begin{tikzcd}[row sep=1.5em,column sep=1.5em]
\widecheck{\RFH}^{\w_k}_*(E,\Sigma_\tau) \arrow{r}{\delta_0} \arrow{d}{\cong} &  \widehat{\RFH}{}^{\w_{k+1}}_{*-1}(E,\Sigma_\tau)   \arrow{d}{\cong}   \\
\FH_{*+2k}(M) \arrow{r}{\mathrm{id}} & \FH_{*+2k}(M)
\end{tikzcd}
\qquad
\begin{tikzcd}[row sep=1.5em,column sep=1.5em]
\widecheck{\RFH}^{\w_k}_*(E,\Sigma_\tau) \arrow{r}{\delta_2} \arrow{d}{\cong} &  \widehat{\RFH}{}^{\w_{k}}_{*-1}(E,\Sigma_\tau)   \arrow{d}{\cong}   \\
\FH_{*+2k}(M) \arrow{r}{\mathrm{\Psi^{c_1^E}}} & \FH_{*+2k-2}(M)\,.
\end{tikzcd}
\end{equation}
Hence the long exact sequence in \eqref{eq:les_full2} is isomorphic to 
\begin{equation}\label{eq:les_full3}
	\to \RFH_{*}(E,\Sigma_\tau) \to   \mathop{\widetilde{\bigoplus}}_{k\in\Z} \FH_{*+2k}(M) \stackrel{\delta}{\to} \mathop{\widetilde{\bigoplus}}_{k\in\Z}\FH_{*+2k}(M)\to \RFH_{*-1}(E,\Sigma_\tau) \to  \cdots
\end{equation}
with $\delta=\mathrm{id}+\Psi^{c_1^E}$. In order to investigate the map $\delta$, we need the following basic fact in linear algebra, which was also used for the computation of the symplectic homology of $E$ in \cite{Rit14}. For any endomorphism $L$ on a finite dimensional vector space $V$, the image $\mathrm{im\,} L^n$ stabilizes for $n\geq \dim V$, i.e.~$\mathrm{im\,} L^{\dim V}=\mathrm{im\,} L^n$ for every $n\geq \dim V$, and the restriction of $L$ to $\mathrm{im\,} L^{\dim V}$ is an automorphism. To apply this fact, we recall that $\bigoplus_{*\in\Z}\FH_*(M)$ is $\Lambda$-module, and $\Psi^{c_1^E}$ on $\bigoplus_{*\in\Z}\FH_*(M)$ is $\Lambda$-linear. Thus if we use coefficients in a field $\mathbb F$, the image of $(\Psi^{c_1^E})^n$ on the finite dimensional $\Lambda$-vector space $\bigoplus_{*\in\Z}\FH_*(M)$ stabilizes for  
\begin{equation}\label{eq:stabilize}
n\geq b_{(M;\F)}:= \dim_{\mathbb F} \bigoplus_{*\in\Z}\H_*(M;\mathbb F)= \dim_\Lambda \bigoplus_{*\in\Z} \FH(M)\,. 	
\end{equation}

\begin{lem}\label{lem:id+Psi_injective}
	The homomorphism 
	\[
	\delta=\mathrm{id}+\Psi^{c_1^E}:\mathop{\widetilde{\bigoplus}}_{k\in\Z} \FH_{*+2k}(M)  \longrightarrow \mathop{\widetilde{\bigoplus}}_{k\in\Z}\FH_{*+2k}(M)
	\]
	is injective when coefficients are taken from $\Z$ or a field $\F$. 
\end{lem}
\begin{proof}
	We first treat the case that $\tau(\lambda-m)\leq m$ and coefficients are taken from  $\Z$ or $\F$. Suppose that there is an element 	
	\[
	\sum_{k\leq k_0} Z_k  \in \ker\,(\mathrm{id}+\Psi^{c_1^E})
	\]
	for some $k_0\in\Z$ with $Z_{k}\in\FH_{*+2k}(M)$. In case of $\tau(\lambda-m)= m$, the sum actually ranges from $-k_0$ to $k_0$. Since $\Psi^{c_1^E}$ decreases degree by $-2$, 
	\[
		0 =(\mathrm{id}+\Psi^{c_1^E})\sum_{k\leq k_0} Z_k 
		= Z_{k_0}+\sum_{i=1}^{+\infty} \big(Z_{k_0-i}+\Psi^{c_1^E}(Z_{k_0-i+1})\big) 
	\]
	implies for degree reasons that each summand on the right-hand side vanishes. In particular, $Z_{k_0}=0$ which then implies $Z_{k_0-1}=0$ and inductively $Z_{k_0 -i}=0$ for all $i\geq 0$. Thus $\mathrm{id}+\Psi^{c_1^E}$ is injective.
		
	Next we consider the case $\tau(\lambda-m)> m$ with  coefficients from a field $\F$. Assume for a contradiction that, for some $k_0\in\Z$, there is a nonzero element 	
	\[
	\sum_{k\geq k_0} Z_k  \in \ker\,(\mathrm{id}+\Psi^{c_1^E})
	\]
	 with $Z_{k}\in\FH_{*+2k}(M)$ and $Z_{k_0}\neq0$. Then we have 
	\[
	0 =(\mathrm{id}+\Psi^{c_1^E})\sum_{k\geq k_0} Z_k 
	=  \Psi^{c_1^E}(Z_{k_0})+	\sum_{i=1}^{+\infty} \big(Z_{k_0+i-1}+\Psi^{c_1^E}(Z_{k_0+i})\big) \,.
	\]
	 As before, all summands on the right-hand side vanish individually for degree reasons. Thus it holds that
	\begin{equation}\label{eq:ker_relation}
		\Psi^{c_1^E}(Z_{k_0})=0\,,\qquad Z_{k_{0}}=-\Psi^{c_1^E}(Z_{k_0+1})=\cdots=(-1)^i(\Psi^{c_1^E})^i(Z_{k_0+i})=\cdots \,.	
	\end{equation}
	Let $b= b_{(M,\F)}$ be as in \eqref{eq:stabilize} such that $\Psi^{c_1^E}$ is an automorphism on $V:=\mathrm{im}(\Psi^{c_1^E})^b\subset\bigoplus_{*\in\Z}\FH_*(M)$. From this we derive a  contradiction as follows. Observe that on the one hand, by \eqref{eq:ker_relation},
	\[
	0\neq Z_{k_0}=(-1)^b(\Psi^{c_1^E})^b(Z_{k_0+n})\in V\,.
	\] 
	On the other hand, $\Psi^{c_1^E}(Z_{k_0})=0$ which contradicts that $\Psi^{c_1^E}$ is an automorphism on $V$.

	Finally let us study the case $\tau(\lambda-m)> m$ with $\Z$-coefficients.   We decompose 
	\[
	\bigoplus_{*\in\Z}\H_*(M;\mathbb Z)\cong\mathcal{F}\oplus\mathcal{T}\,,\qquad \bigoplus_{*\in\Z}\FH_*(M)\cong (\mathcal{F}\otimes\Lambda)\oplus (\mathcal{T} \otimes\Lambda)\,,
	\] 
	where $\mathcal{F}$ and $\mathcal{T}$ are the free and torsion part of $\bigoplus_{*\in\Z}\H_*(M;\mathbb Z)$ respectively. We set
	\[
	n:=\max \big\{\mathrm{rank}_{\Z}\mathcal F, \#\mathcal{T}\}<+\infty\,,
	\]
	where $\#\mathcal{T}$ denotes the cardinality of $\mathcal{T}$. Assume as before that there is a nonzero element 	
	\begin{equation}\label{eq:ker_id+Psi}
	\sum_{k\geq k_0} Z_k  \in \ker\,(\mathrm{id}+\Psi^{c_1^E})		
	\end{equation}
	for some $k_0\in\Z$ with $Z_{k}\in\FH_{*+2k}(M)$ and $Z_{k_0}\neq0$. Again conclusion  \eqref{eq:ker_relation} holds. We also have that $Z_k\neq0$ for all $k\geq k_0$. Indeed, if $Z_{k}=0$ for some $k$, we have $Z_{k_0}+Z_{k_0+1}+\cdots + Z_{k}\in \ker\,(\mathrm{id}+\Psi^{c_1^E})$. Then arguing as in the case of $\tau(\lambda-m)= m$ above, we arrive at the contradiction $Z_{k_0}=0$.   
	
	As observed in Remark \ref{rem:H}.(ii), $\Lambda$ is isomorphic to $\Z[t,t^{-1}]$ with $\deg t=-2c_M$.	
	We claim that $Z_{k_0+m}\in\mathcal{T}\otimes\Z[t,t^{-1}]$ for every $m\geq 0$. Assume for a contradiction that $Z_{k_0+m}\notin\mathcal{T}\otimes\Z[t,t^{-1}]$ for some $m\geq0$. Then $Z_{k_0+m}\otimes 1$ is in $(\mathcal{F}\otimes\Z[t,t^{-1}])\otimes\R \cong \mathcal{F}\otimes \R[t,t^{-1}]$, which is a vector space over $\R[t,t^{-1}]$ of dimension at most $n$. Tensoring with the identity map, we extend $\Psi^{c_1^E}$ to $\mathcal{F}\otimes \R[t,t^{-1}]$.  By \eqref{eq:ker_relation}, 
	\[
	Z_{k_0+m}\otimes 1=(-1)^n(\Psi^{c_1^E})^{n}(Z_{k_0+m+n}\otimes 1)\in V:=\mathrm{im}(\Psi^{c_1^E})^n\,.
	\] 
		Since $Z_{k_0+m}$ is nonzero and not a torsion element,  $Z_{k_0+m}\otimes 1$ is nonzero. 
		Since $\Psi^{c_1^E}$ is an isomorphism on $V$, we have $(\Psi^{c_1^E})^i(Z_{k_0+m}\otimes 1)\neq0$ for every $i\geq 0$. This contradicts $(\Psi^{c_1^E})^{m+1}(Z_{k_0+m}\otimes 1)=(-1)^m\Psi^{c_1^E}(Z_{k_0}\otimes 1)=0$, see \eqref{eq:ker_relation}. This proves the claim. 
	Recalling that $\deg Z_k=*+2k$ and $\deg t=-2c_M$, we write 
	\[
	Z_k=\sum_{\ell\in\Z} Z_k^\ell\otimes t^\ell\,,\qquad Z_k^\ell\in\mathcal{T}\,,\;\;\deg Z_k^\ell= *+2k+2\ell c_M\,.
	\]
	Since $\#\mathcal{T}\leq n$, there exist $0\leq i<j\leq 2^n$ such that $Z_{k_0+j}=Z_{k_0+i} \cdot t^{(i-j)/2c_M}$, where $i-j$ is automatically divisible by $2c_M$. Due to the $\Lambda$-linearity of $\Psi^{c_1^E}$, we obtain
	\[
	(-1)^{j-i}Z_{k_0+i}=(\Psi^{c_1^E})^{j-i}(Z_{k_0+j})=(\Psi^{c_1^E})^{j-i}(Z_{k_0+i} \cdot t^{(i-j)/2c_M})=(\Psi^{c_1^E})^{j-i}(Z_{k_0+i}) \cdot t^{(i-j)/2c_M}
	\]
	and therefore for every $\ell\in\N$ 
	\[
	(-1)^{\ell(j-i)}Z_{k_0+i}=(\Psi^{c_1^E})^{\ell(j-i)}(Z_{k_0+i}) \cdot t^{\ell(i-j)/2c_M}\,.
	\]
	However, due to \eqref{eq:ker_relation} the term on the right-hand side vanishes if $\ell(j-i)>i$ which contradicts the observation that $Z_k\neq0$ for all $k\geq k_0$. This completes the proof.
\end{proof}

Due to Lemma \ref{lem:id+Psi_injective}, the long exact sequence \eqref{eq:les_full3} splits into 
\[
	0 \to   \mathop{\widetilde{\bigoplus}}_{k\in\Z} \FH_{*+2k}(M) \stackrel{\delta}{\to} \mathop{\widetilde{\bigoplus}}_{k\in\Z}\FH_{*+2k}(M)\to \RFH_{*-1}(E,\Sigma_\tau) \to  0\,.
\]
We obtain the following corollary. 
\begin{cor}\label{cor:rfh}
We have isomorphisms
\[
\begin{split}
\RFH_{*-1}(E,\Sigma_\tau) &\cong \bigg(\mathop{\widetilde{\bigoplus}}_{k\in\Z}\FH_{*+2k}(M)\bigg) \,\Big/\, \mathrm{im}\bigg(\mathop{\widetilde{\bigoplus}}_{k\in\Z}\FH_{*+2k}(M) \stackrel{\delta}\to\mathop{\widetilde{\bigoplus}}_{k\in\Z}\FH_{*+2k}(M) \bigg) \\[1ex]
	& \cong\left.\left\{\sum Z_k \mid Z_k\in \FH{}_{*+2k}(M)  \right\} \right/ \big\langle -Z_k=\Psi^{c_1^E}(Z_k)\big\rangle
\end{split}
\]
where the sum $\sum$ in the second line means $\sum_{k\leq k_0}$, $\sum_{|k|\leq k_0}$, and $\sum_{k\geq k_0}$ for some $k_0\in\Z$ in the case of $\tau(\lambda-m)<m$, $\tau(\lambda-m)=m$, and $\tau(\lambda-m)>m$ respectively. Here the relation includes infinite sums, e.g.~$-\sum_{k\geq k_0}Z_k=\sum_{k\geq k_0}\Psi^{c_1^E}(Z_k)$.

If we take coefficients in a field $\mathbb F$, the case of $\tau(\lambda-m)=m$ yields
\[
\begin{split}
\RFH_{*-1}(E,\Sigma_\tau)
 \cong \bigoplus_{k\in\Z}\FH_{*+2k}(M) /\,\mathrm{im\,}(\mathrm{id}+\Psi^{c_1^E})
 \cong \FH_{*}(M)/\ker (\Psi^{c_1^E})^n
\end{split}
\]
for any $n\geq b_{(M;\F)}$ since the image of $(\Psi^{c_1^E})^n$ stabilizes.
\end{cor}

The previous discussion leads to the following vanishing result.  

\begin{prop}\label{prop:id+Psi_surjective}
	 Suppose that one the following holds.
	\begin{enumerate}[(a)]
	\item $\tau(\lambda-m)<m$.
	\item $\tau(\lambda-m)\geq m$ and $(\Psi^{c_1^E})^n=0$ for some $n\in\N$.
	\item $\tau(\lambda-m)> m$ and $\Psi^{c_1^E}$ is an isomorphism.
	\item $\tau(\lambda-m)>m$ and coefficients are taken from a field $\F$. 
	\end{enumerate}
	Then the homomorphism
	\[
	\delta=\mathrm{id}+\Psi^{c_1^E}:\mathop{\widetilde{\bigoplus}}_{k\in\Z} \FH_{*+2k}(M)  \longrightarrow \mathop{\widetilde{\bigoplus}}_{k\in\Z}\FH_{*+2k}(M)
	\]
	is an isomorphism, and hence $\RFH_*(E,\Sigma_\tau)$ vanishes for all $*\in\Z$.
\end{prop}
\begin{proof}
	 Let $\zeta:=\sum Z_k\in\mathop{\widetilde{\bigoplus}}_{k\in\Z}\FH_{*+2k}(M)$ be an arbitrary element where $\sum$ means the same as in Corollary \ref{cor:rfh}. Due to Lemma \ref{lem:id+Psi_injective}, it suffices to show that $\zeta\in\mathrm{im}(\mathrm{id}+\Psi^{c_1^E})$.
	In case (a), we take 
	\[
	\xi:=\sum_{k\leq k_0}\xi_k \in \mathop{\widetilde{\bigoplus}}_{k\in\Z}\FH_{*+2k}(M)\,,\qquad \xi_{k_0}:=Z_{k_0}\,,\;\;\xi_{k-1}:=\big(Z_{k-1}-\Psi^{c_1^E}(\xi_{k})\big)\,.
	\]
	This satisfies $(\mathrm{id}+\Psi^{c_1^E})\xi= \zeta$.
	For (b) with condition $\tau(\lambda-m)>m$, we take
	\[
	\xi:=\sum_{k\geq k_0} \xi_k \in \mathop{\widetilde{\bigoplus}}_{k\in\Z}\FH_{*+2k}(M)
	\,,\qquad \xi_k:=\sum_{i=0}^{n-1}(-1)^i(\Psi^{c_1^E})^i(Z_k)
	\]
	so that $(\mathrm{id}+\Psi^{c_1^E})\xi_k= Z_k$ and thus $(\mathrm{id}+\Psi^{c_1^E})\xi= \zeta$. We take the same $\xi$ with $\sum_{k\geq k_0}$ replaced by $\sum_{|k|\leq k_0}$ 
	 if $\tau(\lambda-m)=m$. In case (c),
	\[
	\xi:=\sum_{\ell=1}^{+\infty}(-1)^{\ell+1}(\Psi^{c_1^E})^{-\ell}(\zeta) \in \mathop{\widetilde{\bigoplus}}_{k\in\Z}\FH_{*+2k}(M)
	\]
	satisfies $(\mathrm{id}+\Psi^{c_1^E})\xi= \zeta$. Finally we treat case (d).
	We take $b=b_{(M;\F)}$ as in \eqref{eq:stabilize} so that $\Psi^{c_1^E}$ is an automorphism on $V:=\mathrm{im}(\Psi^{c_1^E})^b$. We set 
	\[
	\beta_k:=\sum_{i=0}^{b-1}(-1)^i(\Psi^{c_1^E})^i(Z_k)\,,\quad\sigma_k:=(-1)^{b-1}(\Psi^{c_1^E})^{b}(Z_k)\,,\quad \theta_k:=\sum_{\ell=1}^{+\infty}(-1)^{\ell+1}(\Psi^{c_1^E}|_V)^{-\ell}(\sigma_k)\,,
	\]
	where $(\Psi^{c_1^E}|_V)^{-\ell}$ denotes the $\ell$-th power map of the inverse of $\Psi^{c_1^E}|_V:V\to V$.
	Then $(\mathrm{id}+\Psi^{c_1^E})\beta_k=Z_k+\sigma_k$ and $(\mathrm{id}+\Psi^{c_1^E})\theta_k=\sigma_k$. Therefore, 
	\[
	\xi:=\sum_{k\geq k_0}(\beta_k-\theta_k) \in \mathop{\widetilde{\bigoplus}}_{k\in\Z}\FH_{*+2k}(M)\,,\qquad  (\mathrm{id}+\Psi^{c_1^E})\xi=\zeta\,.
	\]
	This finishes the proof of the proposition. 
\end{proof}

At this point we completely proved Theorem \ref{thm:full_rfh} from the introduction. 

\begin{rem}\label{rem:consecutive index}
Suppose that $M$ admits a Morse function $f:M\to\R$ without critical points having  consecutive Morse indices. By multiplying a small constant, we may assume that $f$ is $C^2$-small. The chain complex $\FC(f)$ has trivial boundary operator for index reasons and hence $\FC(f)=\FH(f)$. We fix $J\in\JJ^{\BB}_\textrm{reg}$ with $j\in\j_\mathrm{HS}(f)$ and describe the boundary operator $\p=\p_0+\p_1+\p_2$ as follows. First, $\p_0$ is explained in Remark \ref{rem:p_0}. Second, $\p_1=0$ since $\p_1$ projects to the boundary operator on $\FC(f)$. Finally, as observed in Remark \ref{rem:full_RFH_diagonal_J}, a continuation homomorphism $\varphi$ at the chain level preserves winding numbers and commutes with $\p_0$ and also with $\p_2$. Moreover, by our assumption on $f$ and \eqref{eq:index_wind_intersection}, the continuation homomorphism $\varphi$ preserves the $\mu_\FH$-index in the base $M$ and thus maps $\widecheck{\RFC}^{\w_k}=\widecheck{\RFH}^{\w_k}$ to itself, and likewise for $\widehat{\RFC}{}^{\w_k}=\widehat{\RFH}{}^{\w_k}$. In particular, a continuation homomorphism for the change of $J$ is the identity map. Similarly, for the change of $\tau$, a continuation homomorphism is  essentially the identity map in the sense that it only scales the radius of the loop component of critical points of $\AA^\tau_f$. Since $\p_2$ respects winding numbers, it restricts to $\p_2:\widecheck{\RFC}^{\w_k}_*(\AA^\tau_f)\to \widehat{\RFC}{}^{\w_{k}}_*(\AA^\tau_f)$. Up to continuation homomorphisms for $\widecheck{\RFC}^{\w_k}$ and $\widehat{\RFC}{}^{\w_k}$, we can identify $\p_2$ with the one defined for small $\tau>0$ and $J\in\JJ_\mathrm{diag}$. Therefore $\p_2$ corresponds to the Floer cap product $\psi^{c_1^E}:\FC_*(f)\to\FC_{*-2}(f)$ as studied in Section \ref{sec:gysin_revisit}. We expect that a similar argument works for a perfect Morse function.
\end{rem}

\section{Examples:~$\OO_{\CP^n}(-m)$}\label{sec:example}

To describe the next examples we need the following sets. For $m\in\N$ we define 
\begin{equation}\label{eq:Q_m}
	\begin{split}
		\Q_m&:=\left\{\sum_{k\geq k_0} a_{k}m^{k} \,\Big|\, a_k\in\Z\,,\; 0\leq a_k\leq m-1\,, \;k_0\in\Z \right\}\,,\\[0.5ex]
		\widetilde{\Q}_m&:=\left\{\sum_{|k|\leq k_0} a_{k}m^{k} \,\Big|\, a_k\in\Z\,,\; 0\leq a_k\leq m-1\,, \;k_0\in\Z \right\}
	\end{split}
\end{equation}
consisting of formal (Laurent) polynomials. We endow these sets with natural $\Z$-module structures given by 
\[
\sum_{k\geq k_0} a_km^k +\sum_{k\geq k_0} b_k m^k= \sum_{k\geq k_0} c_k m^k
\]
where the coefficients $0\leq c_k \leq m-1$ are inductively determined by 
\[
\begin{split}
c_{k_0}&=a_{k_0}+b_{k_0} \;\;(\textrm{mod $m$}) \\
c_{k_0+1}&=a_{k_0+1}+b_{k_0+1}+\Big\lfloor\frac{a_{k_0}+b_{k_0}}{m}\Big\rfloor  \;\;(\textrm{mod $m$})	 \\
c_{k_0+2}&=a_{k_0+2}+b_{k_0+2}+\Big\lfloor\frac{a_{k_0+1}+b_{k_0+1}+\lfloor\frac{a_{k_0}+b_{k_0}}{m}\rfloor}{m}\Big\rfloor  \;\;(\textrm{mod $m$})	 \\
& \;\vdots
\end{split}
\]
We note that, when $m=p$ is a prime number,  $\Q_p$ equals the field of $p$-adic numbers as a $\Z$-module.

\subsection{Results}
We recall that $c_1^{T{\CP}^n}=(n+1)[\om_\mathrm{FS}]$ and the complex line bundle $\OO_{\CP^n}(-m)\to\CP^n$ has  first Chern class $c_1^{\OO_{\CP^n}(-m)}=-m[\om_\mathrm{FS}]$. Now we prove Corollary \ref{cor:full_rfh_o_intro} from the introduction.

\begin{cor}\label{cor:full_RFH_O}
Let $n,m\in\N$ be arbitrary. For $*\in2\Z$, we have
\[
\RFH_*(\OO_{\CP^n}(-m),\Sigma_\tau)=0 \qquad \forall \tau>0\,.
\]
For $*\in2\Z+1$,  we have
\[
\RFH_*(\OO_{\CP^n}(-m),\Sigma_\tau)\cong \left\{
\begin{aligned} 
&\;0 \quad  && \quad \tau(n+1-m)<1\,, \\[.5ex]
&\;\Z &&\quad \tau(n+1-m)=1\,,\; m=1\,,\\[.5ex]
&\;\widetilde{\Q}_m \quad  && \quad \tau(n+1-m)=1\,,\;m\neq1\,, \\[.5ex]
&\;\Q_m \quad  && \quad \tau(n+1-m)>1\,,
\end{aligned}
\right.
\]
where the isomorphism is as $\Z$-modules. Note that $\Q_{m=1}=0$. 

If we take coefficients in a field $\mathbb F$, then we again have $\RFH_*(\OO_{\CP^n}(-m),\Sigma_\tau)=0$ for all $\tau>0$ and $*\in2\Z$. For every  $*\in2\Z+1$,
\[
\RFH_*(\OO_{\CP^n}(-m),\Sigma_\tau)\cong \left\{
\begin{aligned} 
&\;0 \quad  && \quad \tau(n+1-m)\neq1\,, \\[.5ex]
&\;\F &&\quad \tau(n+1-m)=1\,.
\end{aligned}
\right.
\]
\end{cor}

\begin{proof}
	The assertions follow from Corollary \ref{cor:rfh} and Proposition \ref{prop:id+Psi_surjective}. We only give a proof of the claim that for $m\geq 2$
	\[
	\RFH_*(\OO_{\CP^n}(-m),\Sigma_\tau)\cong\Q_m \qquad  *\in 2\Z+1\,,\; \tau(n+1-m)>1\,.
	\]
	By Corollary \ref{cor:rfh} together with the well-known computation of $\FH_*(\CP^n)$, see the proof of Proposition \ref{prop:CPn}, $\RFH_*(\OO_{\CP^n}(-m),\Sigma_\tau)$ is isomorphic to 
	\[
	\Xi_m:=\Big\{(\dots,a_{-1},a_0,a_{1},\dots)  \in \prod_{\Z}\Z \mid  \exists k_0\in\Z \textrm{ s.t. } a_k=0\;\forall k<k_0 \Big\}  \Big/ \sim
	\]
	where the equivalence relation is generated by 
\[
(\dots,0,\underbrace{m}_{i\textrm{-th}},0,0,\dots) \sim (\dots,0,0,\underbrace{1}_{i+1\textrm{-th}},0,\dots)\qquad \forall i\in\Z\,.
\]
Then every element in $\Xi_m$ has a representative $(\dots,0,b_{k_0},b_{k_0+1},\dots)$ with $0\leq b_i\leq m-1$ for all $i\in\Z$, and the map 
\[
\Xi_m\to \Q_m\,,\qquad \big[(\dots,0,b_{k_0},b_{k_0+1},\dots)\big]\mapsto \sum_{k\geq k_0} b_{k} m^{k}
\]
is a $\Z$-module isomorphism.
\end{proof}

\subsection{Chain complex for $\OO_{\CP^2}(-m)$}
In this section, we provide a complete description and illustrations of the chain complex $\RFC(\OO_{\CP^2}(-m),\Sigma_\tau)$ and revisit the homology computation made in Corollary \ref{cor:full_RFH_O}.
Let $\sigma$ be a generator of $\pi_2(\CP^2)\cong\Z$ such that $\om_\mathrm{FS}(\sigma)=1$. 
We choose a Morse function $f:\CP^2\to\R$ having exactly three critical points $q_i$, $i=0,1,2$ with $\mu_{-f}(q_i)=2i$.  Let $\gamma_i$ denote the simple Reeb orbit on $\Sigma_\tau$ over $q_i$. Any element  $w=([u,\bar u],\eta)\in\Crit\AA^\tau_f$ with $\wp(u)=q_i$ is equivalent to $(\gamma_i^\ell,k\sigma)$ where $\ell=\cov(u)$, $[\bar u]=[\bar u_\mathrm{fib}\#s]$ with $[\wp\circ s]=k\sigma$, and $k=\om_\mathrm{FS}([\wp\circ\bar u])\in\Z$. Therefore, we write elements in $\Crit h$ by
\[
\hat w=(\hat\gamma_i^\ell,k\sigma)\,,\qquad \check w=(\check\gamma_i^\ell,k\sigma)\,,
\]
where $i\in\{0,1,2\}$ and $\ell,k\in\Z$ are uniquely determined by $w$ as above.
As computed in Example \ref{ex:CP}, we have
\begin{equation}\label{eq:example_index}
\begin{split}	&\mu_\RFH^h(\hat\gamma_i^\ell,k\sigma)-1=\mu_\RFH^h(\check\gamma_i^\ell,k\sigma)=-2\ell-2(3-m)k+2i-2\,,\\[0.5ex]
	&\w(\hat\gamma_i^\ell,k\sigma)=\w(\check\gamma_i^\ell,k\sigma)=\ell-mk\,.
\end{split}
\end{equation}
Let $\xi\in\RFC_{2j}(\AA^\tau_f)$ for $j\in\Z$. We can express it as 
\[
\xi=\sum_{i=0}^2\sum_{k\in\Z} a_{i,k}\big(\check\gamma_i^{-j-(3-m)k+i-1},k\sigma\big)\,,\qquad a_{i,k}\in\Z
\]
subject to 
\begin{equation}\label{eq:CP_Nov}
\left\{
\begin{aligned} 
\exists k_0\in\Z \textrm{ such that } a_{i,k}=0 \,\;\forall i\in\{0,1,2\}\,\;\forall k<k_0  \;\;& \quad \textrm{if } \tau(3-m)<m\,, \\[1ex]
\exists k_0\in\Z \textrm{ such that } a_{i,k}=0 \,\;\forall i\in\{0,1,2\}\,\;\forall k>k_0   \;\;& \quad \textrm{if }\tau(3-m)>m\,, \\[1ex]
\exists k_0\in\Z \textrm{ such that } a_{i,k}=0 \,\;\forall i\in\{0,1,2\}\,\;\forall |k|>k_0  & \quad \textrm{if } \tau(3-m)=m\,,
\end{aligned}
\right.
\end{equation}
according to the Novikov condition, see Remark \ref{rem:Novikov} and also \eqref{eq:full_action_winding} and \eqref{eq:nov_ind}. We note that if $m\geq 3$,  only the first one occurs for all $\tau>0$. A corresponding statement  for $\xi\in\RFC_{2j+1}(\AA^\tau_f)$ is true with $\check\gamma_i$ replaced by $\hat\gamma_i$.

Moreover, due to the discussion in Remark \ref{rem:consecutive index}, we have for all $i$, $\ell$, and $k$
\[
\p_0(\check\gamma_i^\ell,k\sigma)=(\hat\gamma_{i}^{\ell+1},k\sigma)\,,\qquad \p_0(\hat\gamma_i^\ell,k\sigma)=0\,,\qquad \p_1=0\,,\qquad \p_2(\hat\gamma_i^\ell,k\sigma)=0\,,
\]
\[
\p_2(\check\gamma_2^\ell,k\sigma)=(\hat\gamma_1^\ell,k\sigma)\,,\quad \p_2(\check\gamma_1^\ell,k\sigma)=(\hat\gamma_0^\ell,k\sigma)\,,\quad  \p_2(\check\gamma_0^\ell,k\sigma)=(\hat\gamma_2^{\ell+m} ,(k+1)\sigma)\,.
\]
In particular, every maximum point of $h$ is a cycle. Moreover every cycle is generated only by maximum points of $h$. Indeed if a chain contains a minimum point, then to be a cycle it needs infinitely many terms in both directions of increasing and decreasing $k$, which is not allowed by \eqref{eq:CP_Nov}.

\subsubsection{The case $\OO_{\CP^2}(-1)\to\CP^2$}

\begin{figure}
\centering
\includegraphics[scale=0.6]{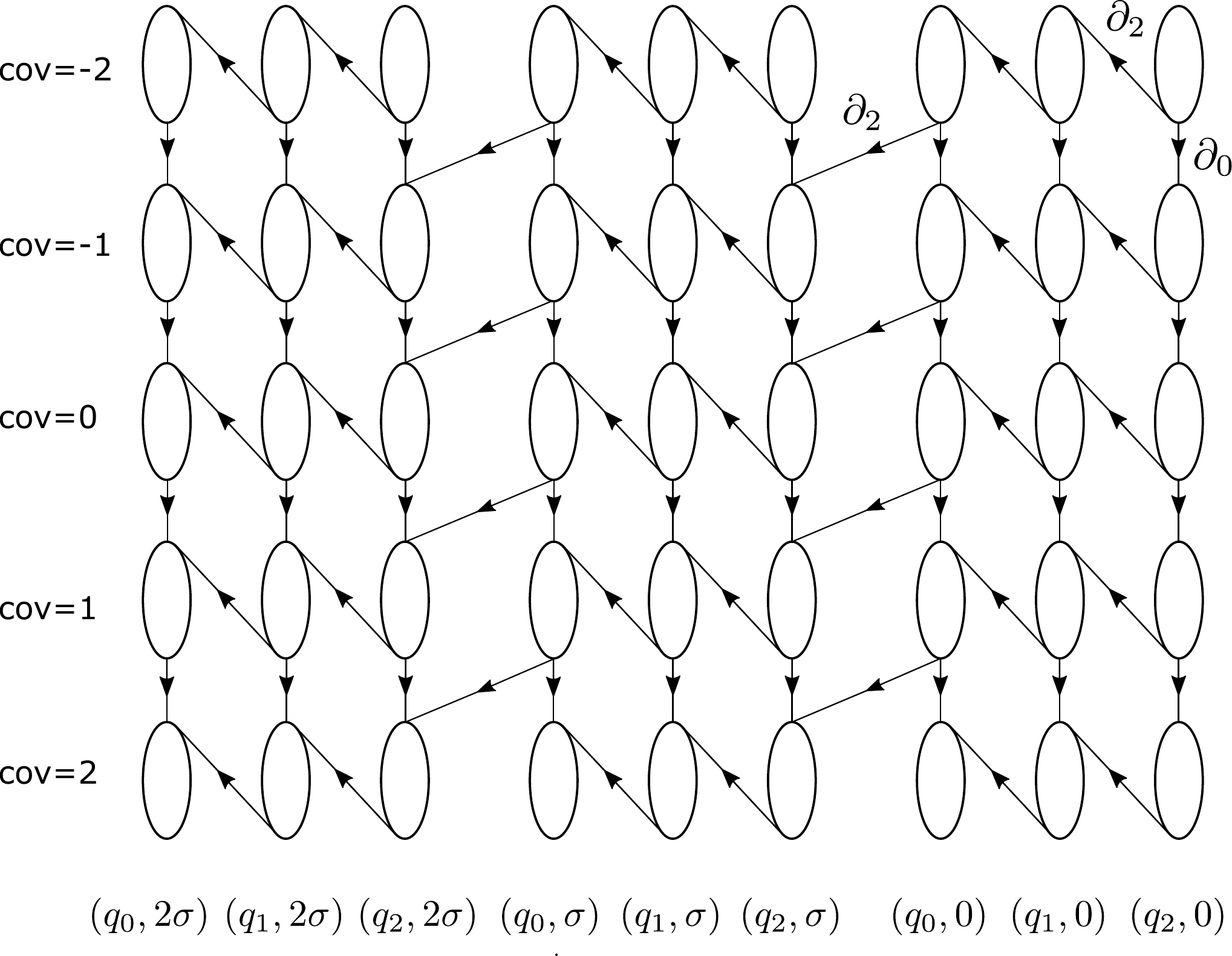}
\caption{$\OO_{\CP^2}(-1)\to\CP^2$}\label{fig:O(-1)}
\end{figure}

\noindent\underline{Case $\tau<\frac{1}{2}$}: Every cycle is a boundary and thus $\RFH_*(\OO_{\CP^2}(-1),\Sigma_{\tau})=0$ for all $*\in\Z$:
\[
\begin{split}
(\hat\gamma_0^\ell,k\sigma)=\p\Big(&(\check\gamma^{\ell-1}_0,k\sigma)-(\check\gamma_2^{\ell-1},(k+1)\sigma)+(\check\gamma_1^{\ell-2},(k+1)\sigma)- (\check\gamma_0^{\ell-3},(k+1)\sigma)\\
&+(\check\gamma_2^{\ell-3},(k+2)\sigma) - (\check\gamma_1^{\ell-4},(k+2)\sigma) + \cdots \Big)
\end{split}
\]
and similarly for $(\hat\gamma_1^\ell,k\sigma)$ and $(\hat\gamma_2^\ell,k\sigma)$.

\medskip

\noindent\underline{Case $\tau>\frac{1}{2}$}: Every cycle is boundary and thus $\RFH_*(\OO_{\CP^2}(-1),\Sigma_{\tau})=0$ for all $*\in\Z$:
\[
\begin{split}
(\hat\gamma_0^\ell,k\sigma)&=\p\Big((\check\gamma^{\ell}_1,k\sigma)-(\check\gamma_2^{\ell+1},k\sigma)+(\check\gamma_0^{\ell+1},(k-1)\sigma)- (\check\gamma_1^{\ell+2},(k-1)\sigma)+\cdots \Big)\\[1ex]
\end{split}
\]
and similarly for $(\hat\gamma_1^\ell,k\sigma)$ and $(\hat\gamma_2^\ell,k\sigma)$.

\medskip

\noindent\underline{Case $\tau=\frac{1}{2}$}: The cycles $(\hat\gamma_i^\ell,k\sigma)$ are not boundaries and homologous each other for suitably related $i$, $\ell$, and $k$ as follows:
\[
\big[(\hat\gamma_0^\ell,k\sigma)\big]=\big[(\hat\gamma_1^{\ell+1},k\sigma)\big]=\big[(\hat\gamma_2^{\ell+2},k\sigma)\big]=\big[(\hat\gamma_0^{\ell+2},(k-1)\sigma)\big]=\cdots \,.
\]
Together with the index computation in \eqref{eq:example_index}, this yields 
\[
\RFH_*(\OO_{\CP^2}(-1),\Sigma_{\tau}) \cong \left\{
\begin{aligned} 
\Z \quad  & \quad *\in2\Z+1\,, \\[1ex]
0 \quad & \quad *\in 2\Z\,.
\end{aligned}
\right.
\]

\subsubsection{The case $\OO_{\CP^2}(-2)\to\CP^2$}
\begin{figure}
\centering
\includegraphics[scale=0.6]{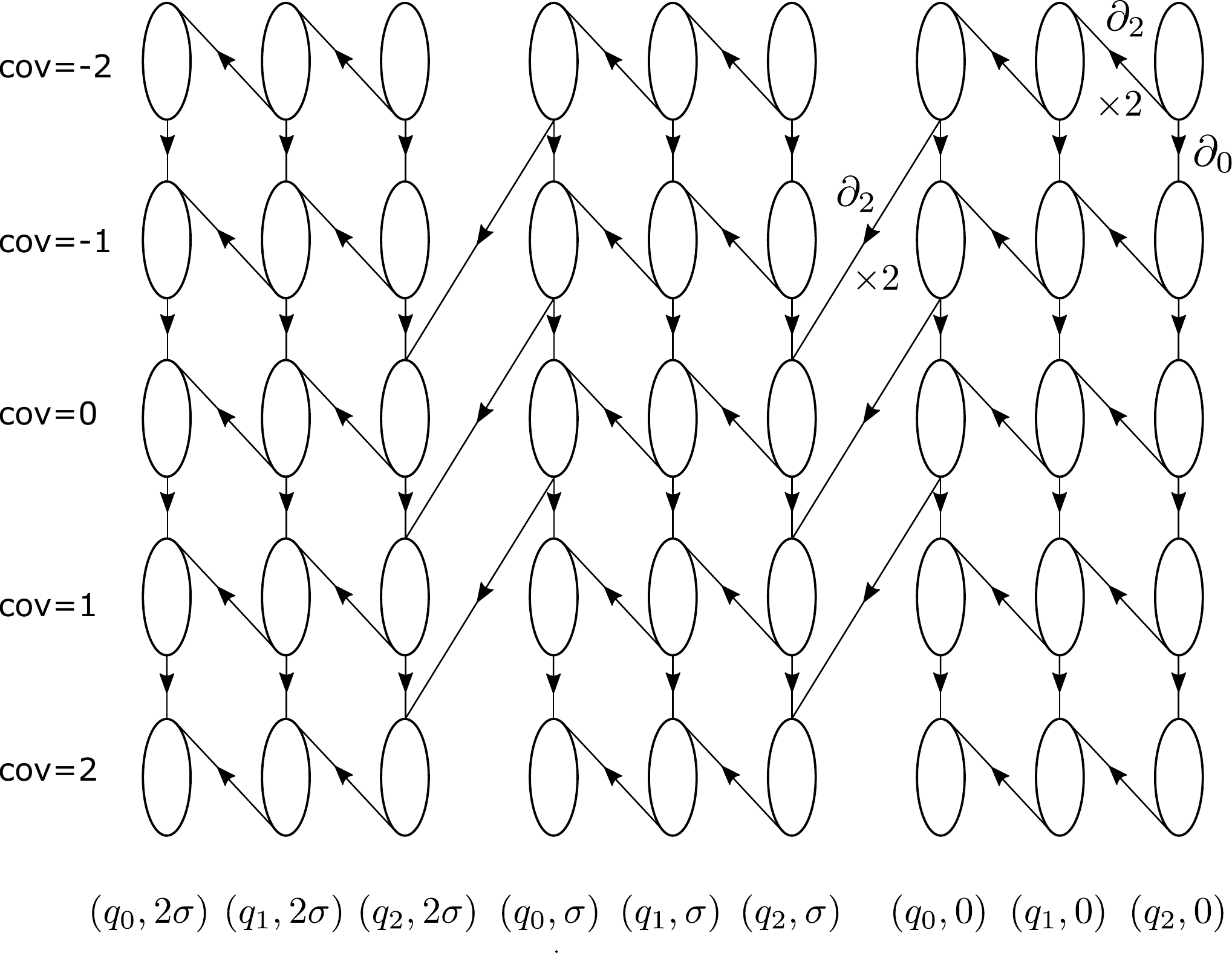}
\caption{$\OO_{\CP^2}(-2)\to\CP^2$}
\end{figure}

\noindent\underline{Case $\tau<2$}: Every cycle is a boundary and $\RFH_*(\OO_{\CP^2}(-2),\Sigma_{\tau})=0$ for all $*\in\Z$:
\[
\begin{split}
(\hat\gamma_0^\ell,k\sigma)=\p\Big(&\check\gamma^{\ell-1}_0,k\sigma)-2(\check\gamma_2^{\ell},(k+1)\sigma)+2^2(\check\gamma_1^{\ell-1},(k+1)\sigma)- 2^3(\check\gamma_0^{\ell-2},(k+1)\sigma)\\
&+2^4(\check\gamma_2^{\ell-1},(k+2)\sigma) - 2^5(\check\gamma_1^{\ell-2},(k+2)\sigma) +\cdots \Big)
\end{split}
\]
and similarly for $(\hat\gamma_1^\ell,k\sigma)$ and $(\hat\gamma_2^\ell,k\sigma)$.

\medskip

\noindent\underline{Case $\tau>2$}: The cycles $(\hat\gamma_i^\ell,k\sigma)$ are not boundaries and homologous as follows:
\[
\big[(\hat\gamma_0^{\ell+1},(k-1)\sigma)\big]=2\big[(\hat\gamma_2^{\ell+2},k\sigma)\big]=2^2\big[(\hat\gamma_1^{\ell+1},k\sigma)\big]=2^3\big[(\hat\gamma_0^\ell,k\sigma)\big]=\cdots \,.
\]
We note that, in this case, a cycle is allowed to have infinitely many terms in the direction of increasing $k$.
This leads to the computation
\[
\RFH_*(\OO_{\CP^2}(-2),\Sigma_{\tau}) \cong \left\{
\begin{aligned} 
\Q_2 \quad  & \quad *\in2\Z+1\,, \\[1ex]
0\; \quad & \quad *\in 2\Z\,.
\end{aligned}
\right.
\]

\medskip

\noindent\underline{Case $\tau=2$}: We have exactly the same relations between cycles as in the case of $\tau>2$ but this time cycles can have only finitely many terms. Therefore
\[
\RFH_*(\OO_{\CP^2}(-2),\Sigma_{\tau}) \cong \left\{
\begin{aligned} 
\widetilde{\Q}_2 \quad  & \quad *\in2\Z+1\,, \\[1ex]
0\; \quad & \quad *\in 2\Z\,.
\end{aligned}
\right.
\]

\subsubsection{The case $\OO_{\CP^2}(-m)\to\CP^2$ for $m\geq3$}

\begin{figure}
\centering
\includegraphics[scale=0.6]{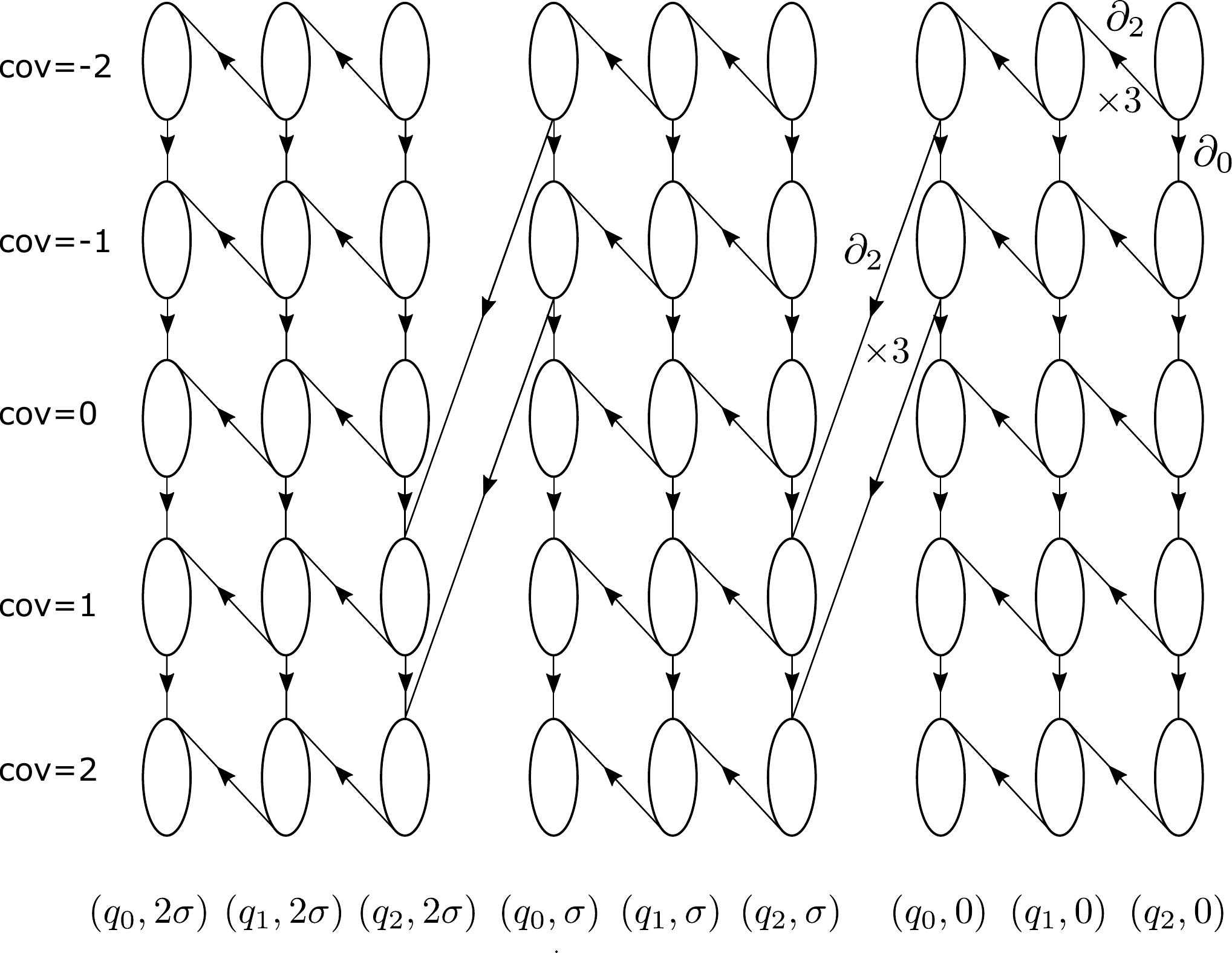}
\caption{$\OO_{\CP^2}(-3)\to\CP^2$}
\end{figure}
In this case, the first line in \eqref{eq:CP_Nov} holds for all $\tau>0$. Therefore we have 
\[
\begin{split}
(\hat\gamma_0^\ell,k\sigma)=\p\Big(&(\check\gamma^{\ell-1}_0,k\sigma)-m(\check\gamma_2^{\ell+m-2},(k+1)\sigma)+m^2(\check\gamma_1^{\ell+m-3},(k+1)\sigma) \\ 
& - m^3(\check\gamma_0^{\ell+m-4},(k+1)\sigma)
	 +m^4(\check\gamma_2^{\ell+2m-5},(k+2)\sigma)-\cdots \Big)\,,
\end{split}
\]
and similarly $(\hat\gamma_1^\ell,k\sigma)$ and $(\hat\gamma_2^\ell,k\sigma)$ are also boundaries. Hence, for every $m\geq3$, $\tau>0$, and $*\in\Z$, we have
\[
\RFH_*(\OO_{\CP^2}(-m),\Sigma_{\tau})=0\,.
\]

\newpage

\bibliographystyle{amsalpha}
\bibliography{RFH_wind.bib}
\end{document}